\documentclass[11pt]{article} 
\usepackage{hyperref}
\hypersetup{colorlinks, linkcolor=blue}
\usepackage[all]{xy}           
\usepackage{marginnote}  
\usepackage{float}
\usepackage{amsmath,amssymb,amsfonts,amsthm,amscd,graphicx,psfrag,epsfig,mathabx}
\usepackage{color}
\definecolor{blue}{rgb}{0,0,0.7}
\definecolor{red}{rgb}{0.75, 0, 0}
\usepackage[titletoc,toc]{appendix}

\newtheorem{theorem}{Theorem}[section]

\newtheorem{theorem-definition}[theorem]{Theorem-Definition}
\newtheorem{theorem-construction}[theorem]{Theorem-Construction}
\newtheorem{lemma-definition}[theorem]{Lemma--Definition}
\newtheorem{lemma-construction}[theorem]{Lemma--Construction}
\newtheorem{lemma}[theorem]{Lemma}
\newtheorem{proposition}[theorem]{Proposition}
\newtheorem{corollary}[theorem]{Corollary}
\newtheorem{conjecture}[theorem]{Conjecture}
\newtheorem{definition}[theorem]{Definition}
\newtheorem{example}[theorem]{Example}

\newenvironment{remark}[1][Remark.]{\begin{trivlist}
\item[\hskip \labelsep {\bfseries #1}]}{\end{trivlist}}

\newcommand{\old}[1]{}

\newcommand{\Z}{{\mathbb Z}}
\newcommand{\R}{{\mathbb R}}
\newcommand{\Q}{{\mathbb Q}}
\newcommand{\C}{{\mathbb C}}

\newcommand{\A}{{\rm A}}

\newcommand{\B}{{\rm B}}
\newcommand{\G}{{\rm G}}
\newcommand{\U}{{\rm U}}

\newcommand{\bS}{{{\Bbb S}}}

\newcommand{\lms}{\longmapsto}
\newcommand{\lra}{\longrightarrow}
\newcommand{\hra}{\hookrightarrow}
\newcommand{\ra}{\rightarrow}
\newcommand{\be}{\begin{equation}}
\newcommand{\ee}{\end{equation}}
\newcommand{\bt}{\begin{theorem}}
\newcommand{\et}{\end{theorem}}
\newcommand{\bd}{\begin{definition}}
\newcommand{\ed}{\end{definition}}
\newcommand{\bp}{\begin{proposition}}
\newcommand{\ep}{\end{proposition}}

\newcommand{\bl}{\begin{lemma}}
\newcommand{\el}{\end{lemma}}
\newcommand{\bc}{\begin{corollary}}
\newcommand{\ec}{\end{corollary}}
\newcommand{\bcon}{\begin{conjecture}}
\newcommand{\econ}{\end{conjecture}}
\newcommand{\la}{\label}

\begin{document}

\date{}

\title {Donaldson-Thomas transformations of moduli spaces of $\G$-local systems}
\author{Alexander Goncharov, Linhui Shen}

\maketitle

\tableofcontents
\begin{abstract}Kontsevich and Soibelman   
defined Donaldson-Thomas invariants of
 a 3d Calabi-Yau  category ${\cal C}$ equipped with a stability condition  \cite{KS1}. 
Any cluster variety gives rise to a family of such  categories. 
Their DT invariants are encapsulated in a  single formal 
automorphism of the cluster variety, called the {\it {\rm DT}-transformation}.

Let $\bS$ be an oriented surface with punctures,   
and a finite number of special points on the  boundary considered modulo isotopy.  
It give rise to a 
moduli space ${\cal X}_{{\rm PGL_m}, \bS}$, closely related to the moduli 
space of ${\rm PGL_m}$-local systems on $\bS$, which 
 carries a canonical cluster Poisson variety structure 
\cite{FG1}. 
For each puncture of $\bS$, 
there is a birational Weyl group action on the space ${\cal X}_{{\rm PGL_m}, \bS}$. 
We prove that it is given by cluster Poisson transformations. 
We prove a similar result for the involution $\ast$ of ${\cal X}_{{\rm PGL_m}, \bS}$ provided by 
dualising a local system on $\bS$.

Let  $\mu$ be 
the total number of punctures and special points, and $g(\bS)$ the genus of $\bS$. We assume that $\mu>0$.  We
 say that $\bS$ is admissible if $g(\bS) +\mu \geq 3$ and $\mu>1$ if $\bS$ has only punctures, and also when 
$\bS$ is an annulus with a special point on each boundary circle. 

Using a combinatorial characterization of a class of DT transformations due to B. Keller \cite{K13},  
we calculate the  {\rm DT}-transformation 
of  the  space ${\cal X}_{{\rm PGL_m}, \bS}$ for any admissible $\bS$.  
 
We show that the Weyl group and the involution  $\ast$ act by cluster transformations of
 the dual  moduli space ${\cal A}_{{\rm SL_m}, \bS}$, and calculate the  {\rm DT}-transformation 
of the space ${\cal A}_{{\rm SL}_m, \bS}$.

If $\bS$ admissible, combining the results above with the results of Gross, Hacking, Keel and Kontsevich \cite{GHKK}, we 
get a canonical basis in the space of regular functions on the cluster variety  ${\cal X}_{{\rm PGL_m}, \bS}$, and 
 in the Fomin-Zelevinsky upper cluster algebra with principal coefficients \cite{FZIV} related to the pair $({\rm SL_m}, \bS)$,   
 as predicted by Duality Conjectures \cite{FG2}. 

\end{abstract}

\section{Introduction} \la{sec1}

\subsection{Summary}

A {\it decorated surface} $\bS$ is an oriented topological surface  with $n$ {\it punctures} inside and 
a finite number of {\it special points} on the boundary, considered modulo isotopy. 
We assume that each boundary component has at least one special point, see Figure \ref{decsurf}. 
We define  {\it marked points} as either  punctures or special points. 
Denote by $\mu$ the number of marked points.  We assume that $\mu>0$. 
Filling the punctures by points, and the holes by discs, we get a compact closed surface. 
The genus $g(\bS)$ of $\bS$ is its genus. 
Denote by $\Gamma_\bS$ the mapping class group of $\bS$. 
 
\begin{figure}[ht]
\epsfxsize60pt
\centerline{\epsfbox{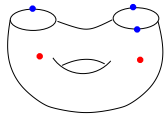}}
\caption{A decorated surface with three special points on the boundary and two punctures.}
\label{decsurf}
\end{figure}

Let $\G$ be a split semi-simple group over $\Q$. 
A pair $(\G, \bS)$  gives rise to 
a Poisson moduli space ${\cal X}_{\G, \bS}$, 
closely related to the moduli spaces of $\G$-local systems on $\bS$. 
If the center of $\G$ is trivial, the space ${\cal X}_{\G, \bS}$
has a 
natural cluster Poisson  structure, 
 defined for $\G={\rm PGL}_m$ in \cite{FG1}. 

There are three groups acting on the moduli space ${\cal X}_{\G, \bS}$:

1. The mapping class group $\Gamma_\bS$ acts by automorphisms of 
the moduli space ${\cal X}_{\G, \bS}$.

2. The group ${\rm Out}(\G)$ of outer automorphisms of $\G$  
acts  by automorphisms of  ${\cal X}_{\G, \bS}$.

3. The Weyl group $W^n$ acts by \underline{birational} automorphisms of the 
space ${\cal X}_{\G, \bS}$.

The actions of these three groups commute by the very definition. 

When do these groups act by cluster transformations of the space ${\cal X}_{{\rm PGL}_m, \bS}$? 
It was proved in \cite{FG1} that the action of the  group $\Gamma_\bS$ is cluster 
if  $\bS$ satisfies the following condition: 

\begin{itemize} 
\item[i)] One has $g(\bS) + \mu\geq 3$, or $\bS$ is an annulus with  two special points. 

\end{itemize}

We prove that the action of the group ${\rm Out}({\rm PGL}_m)$ is cluster  under the same 
assumptions, and that the action of the group $W^n$ is cluster if, in addition to i), 
$\bS$ has the following property:\footnote{There is one 
more minor exception: $\G = {\rm PGL}_2$ and $\bS$ is a sphere with three punctures. 
However in this case the longest element $(1,1,1) \in (\Z/2\Z)^3$ 
still acts by a cluster transformation.}  

\begin{itemize} 
\item[ii)] 
If $\bS$ has no special points, then it has more than one  puncture.

\end{itemize}

\bd \la{ADMDS}
A decorated surface $\bS$ is {\it admissible}, if it satisfies 
conditions $i)$ and $ii)$. 
\ed

We introduce a birational action of the group $W^n$ on the dual moduli space ${\cal A}_{{\rm SL}_m, \bS}$.

\bt \la{Theorem2} If $\bS$ is admissible, and  not a sphere with three punctures     
if $\G$ is of type  ${\rm A}_1$,   then the action of the group  $\Gamma_S \times W^n\times {\rm Out}(\G)$ 
 on  the spaces ${\cal X}_{{\rm PGL_m}, \bS}$ and ${\cal A}_{{\rm SL_m}, \bS}$ is cluster. 
\et

If $\bS$ has just one puncture, the $W$-action is not cluster at least if $\G = {\rm PGL}_2$. 

\vskip 3mm
We use Theorem \ref{Theorem2} in a crucial way to study the {\it Donaldson-Thomas transformations}.

Kontsevich and Soibelman   \cite{KS1}
defined Donaldson-Thomas invariants of
 a 3d Calabi-Yau  category equipped with a stability condition. 
Any cluster variety gives rise to a family of such categories. 
Their DT invariants are encapsulated in  single formal 
automorphism of the cluster variety, called the {\it {\rm DT}-transformation}. 

Let ${\bf w}_0\in W^n$ be the longest element, and  
${\bf r}$   the  clockwise rotation  of the  special points  on 
each boundary component of $\bS$ by one. The group ${\rm Out}(\G)$ contains a canonical involution  
$\ast$. 

\bt \la{THEmain.result.}
Let  $\bS$ be an  admissible decorated surface.  
Then the  {\rm DT}-transformation ${\rm DT}_{{\rm PGL}_m, \bS}$ 
of  the  space ${\cal X}_{{\rm PGL}_m, \bS}$ is  a cluster transformation. It is given by 
\be \la{DT=CGS}
{\rm DT}_{{\rm PGL}_m, \bS} = \ast \circ{\bf w}_0 \circ  {\bf r}. 
\ee 
The cluster transformation (\ref{DT=CGS}) is a cluster DT-transformation in the sense of Definition \ref{2.14.16.1}. 
\et

This implies that the {\rm DT}-transformation coincides with  
Gaiotto-Moore-Neitzke spectral generator, which encodes the count of BPS states 
in 4d ${\cal N}=2$ SUYM theories \cite{GMN1}-\cite{GMN5}. Thus, at least in certain cases, the   
{\rm DT}-invariants coincide with the GMN count of BPS states.  
We prove Theorem \ref{THEmain.result.} in Section \ref{secDT}.

\vskip 3mm

When $\bS$ is a triangle,  
we identify the tropicalized involution $\ast$ on the tropicalised space ${\cal A}_{{\rm SL_m}, \bS}$ with the 
Sch\"utzenberger involution.

\paragraph{An application to Duality Conjectures.} For any admissible  $\bS$, Theorems \ref{Theorem2} and \ref{THEmain.result.}, 
combined with  Theorem 0.10   of Gross, Hacking, Keel and Kontsevich \cite{GHKK}, deliver 
a canonical $\Gamma_\bS  \times W^n \times {\rm Out}(\G)$-equivariant linear basis in the  space of regular functions on the cluster Poisson variety  ${\cal X}_{{\rm PGL_m}, \bS}$, as predicted by Duality Conjectures \cite{FG2}. 

Precisely, the cluster Poisson variety structure on the moduli space 
${\cal X}_{{\rm PGL_m}, \bS}$ gives rise to a
$\Gamma_\bS$-equivariant  algebra 
${\cal O}_{\rm cl}({\cal X}_{{\rm  PGL_m}, \bS})$ of regular functions on the corresponding cluster  variety.\footnote{We abuse notation  denoting a moduli space and the corresponding cluster variety the same way,  although these are different geometric objects. 
The notation ${\cal O}_{\rm cl}({\cal Y})$ emphasises  that we  deal with the algebra of functions on a  
cluster variety ${\cal Y}$.} On the other hand, 
the cluster ${\cal A}$-variety structure on the moduli space ${\cal A}_{{\rm SL_m}, \bS}$ \cite{FG1} gives rise a $\Gamma_\bS$-equivariant set 
${\cal A}_{{\rm SL_m}, \bS}(\Z^t)$ 
of the integral tropical points. 

A cluster ${\cal A}$-variety ${\cal A}$ has a deformation  ${\cal A}_{\rm prin}$ over a  torus. 
 The algebra of regular 
function ${\cal O}_{\rm cl}({\cal A}_{\rm prin})$ 
 is the  Fomin-Zelevinsky upper cluster algebra with principal coefficients \cite{FZIV}. 
 We denote by ${\cal O}_{\rm cl}({\cal A}_{\rm prin}({{\rm SL_m}, \bS}))$ the one related to the pair 
 $({\rm SL_m}, \bS)$. 
  
\bt \la{THEmain.resultII.} Let  $\bS$ be an arbitrary admissible decorated surface.  
Then:

i) There is a canonical $\Gamma_\bS  \times W^n  \times {\rm Out}(\G)$-equivariant linear basis in the  space ${\cal O}_{\rm cl}({\cal X}_{{\rm PGL_m}, \bS})$ 
of regular functions on the cluster variety ${\cal X}_{{\rm PGL_m}, \bS}$, 
parametrized by the set  ${\cal A}_{{\rm SL_m}, \bS}(\Z^t)$.

ii) There is a canonical $  \Gamma_\bS  \times W^n \times {\rm Out}(\G)$-equivariant linear basis in the space ${\cal O}_{\rm cl}({\cal A}_{\rm prin}({{\rm SL_m}, \bS}))$. 
\et
A canonical basis in ${\cal O}_{\rm cl}({\cal X}_{{\rm PGL_m}, \bS})$ parametrized 
by the set ${\cal A}_{{\rm SL_m}, \bS}(\Z^t)$ just means that we have a canonical pairing
$$
{\bf I}: {\cal A}_{{\rm SL_m}, \bS}(\Z^t) \times {\cal X}_{{\rm PGL_m}, \bS} \lra {\Bbb A}^1.
$$
It assigns to a tropical point $l \in {\cal A}_{{\rm SL_m}, \bS}(\Z^t) $ a function on ${\cal X}_{{\rm PGL_m}, \bS}$, given by 
 the basis vector parametrized by $l$.  
The equivarianace  means that the pairing is $\Gamma_\bS  \times W^n  \times {\rm Out}(\G)$-invariant. 

Theorem \ref{THEmain.resultII.} follows immediately from Theorems \ref{Theorem2}, \ref{THEmain.result.} and  \ref{Th1.17}. 

A canonical basis in ${\cal O}_{\rm cl}({\cal X}_{{\rm PGL_2}, \bS})$ was defined by a different method in \cite[Section 12]{FG1}. 
\subsection{Definitions} \la{sec1.1}

\paragraph{The Poisson moduli space ${\cal X}_{\G, \bS}$.} 
The space ${\cal X}_{\G, \bS}$ parametrises $\G$-local systems 
 on $\bS$  with an additional data: a reduction to 
a Borel subgroup near every marked point, called a {\it framing}. 
As  the name suggests, it comes with additional structures: a 
 $\Gamma_\bS$-equivariant Poisson structure.

\paragraph{A birational automorphism ${\rm C}_{\G, \bS}$  of  the moduli space ${\cal X}_{\G, \bS}$.} 
A decorated surface has punctures and boundary components. 
We assume that each boundary component has at least one special point.

We introduce three  types of (birational) automorphisms of the spaces ${\cal X}_{\G, \bS}$.

\begin{enumerate} 

\item {\bf Punctures}. For each puncture on $\bS$ 
there is a birational action of the Weyl group  $W$ of $\G$ on the 
space  ${\cal X}_{\G, \bS}$, defined in \cite{FG1}.   
Namely, given a generic regular element $g \in \G$, 
the set of Borel subgroups containing $g$ is a principal homogeneous set of the Weyl group.  
So, given a puncture $p$  and a generic $\G$-local system ${\cal L}$ 
on $\bS$, the Weyl group acts  
by altering the 
reduction of ${\cal L}$ 
to a Borel subgroup near $p$, leaving the  $\G$-local system intact.

For example, for  a generic ${\rm SL}_m$-local system 
on $\bS$ the monodromy   around  $p$ has $m$ eigenlines. A reduction to Borel subgroup near 
$p$ just means that we order them. 
The symmetric group $S_m$ acts on the orderings.  

The actions at different punctures commute. So 
 the group $W^n$ acts birationally on ${\cal X}_{\G, \bS}$. 


\item {\bf Boundary components}. For each boundary component $h$ on $\bS$, consider an isomorphism $r_{{\bS}, h}$ of 
the moduli space ${\cal X}_{\G, \bS}$
provided by the  {\it rotation by one} 
of the special points on $h$ in the direction prescribed by the orientation of $\bS$. Namely,  we  rotate the surface near the boundary component, moving  
each special point 
to the next one, transporting framings.

The isomorphisms $r_{{\bS}, h}$ at different boundary components commute. 
We take their product over all boundary components:
$$
r_\bS:= \prod_h r_{{\bS}, h}.
$$
\begin{figure}[ht]
\epsfxsize50pt
\centerline{\epsfbox{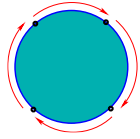}}
\caption{The rotation operator on a disc with four special points.}
\label{rotation}
\end{figure}

\item {\bf The involution $\ast$}. 
We define an involution $\ast$ acting on the space ${\cal X}_{\G, \bS}$. 
If $\G={\rm SL}_m$, it amounts to  dualising  
a local system on $\bS$, as well as  the framings at the marked points. 

For any group $\G$ it is provided by an outer automorphism of the group $\G$ defined as follows. 
 Let $\alpha_i$ $(i\in I)$ be simple positive roots. Let $i \to i^*$ be a Dynkin diagram automorphism such that 
$\alpha_{i^*}=-w_0(\alpha_i)$. Choose a pinning of $\G$. It  provides us with 
a Cartan subgroup ${\rm H} \in \G$ and one parametric subgroups $x_i(a)$ and $y_i(a)$, $i\in I$. 
These subgroups generate $\G$. It also provides 
 a lift of the Weyl group $W$ to $\G$,  $w \lms \overline w$,  
lifting the generator of each standard ${\rm SL}_2$ to the element
$\small{\begin{pmatrix}
0& 1 \\
-1 & 0\end{pmatrix}}$.  
 Then there is an involution $\ast: ~ \G\lra \G$: 
\be \la{inv}
\ast: ~ \G\lra \G, ~~~~
x_i(a)\lms x_{i^*}(a),\quad y_i(a)\lms y_{i^*}(a),\quad h\lms h^*=\overline{w}_0^{-1}h^{-1}\overline{w}_0, ~~\forall h\in {\rm H}.
\ee
For example, if $\G={\rm GL}_m$, then 
$
\ast(g)=\overline{w}_0 \cdot (g^t)^{-1}\cdot \overline{w}_0^{-1}$, where $g^t$ is the transpose of $g$. 
The involution $\ast$ preserves the Borel subgroup ${\rm B}$ generated by ${\rm H}$ and $\{x_i(a)\}_{i\in I}$. 
So it induces an involution of  the flag variety ${\cal B} \stackrel{\sim}{=}\G/{\rm B}$. Hence it acts on
 the moduli space  ${\cal X}_{\G, \bS}$.  
Abusing notation, all of them are denoted by $\ast$.


\end{enumerate}

\bd Let ${\bf w}_0= (w_0, ..., w_0)$ be the longest element of the Weyl group $W^n$. We set
\be \la{6.5.15.1}  
{\rm C}_{\G, \bS} := r_{\bS}\circ \ast \circ {\bf w}_0. 
\ee
\ed





Theorem \ref{2.27.15.1} below asserts that 
the transformation ${\rm C}_{{\rm PGL_m}, \bS}$ is a cluster transformation when $\bS$ 
is an admissible 
decorated surface. 
We observe that this is not the case when $\bS$ has a single puncture. Theorem \ref{2.27.15.1} is our main  tool to study DT-invariants. To  state it properly, 
let us review the background.  
 
 \subsection{Cluster nature of the Weyl group action and of 
the {$\ast$-involution}} \la{sec1.3dualII}



\paragraph{Quivers and quantum cluster varieties.} 
In this paper a {\it quiver} is an oriented graph without loops or 2-cycles,  
whose  vertices are labelled by a set ${\rm I}=\{1,\ldots, N\}$. 
See Figure \ref{quiver}.
\begin{figure}[ht]
\epsfxsize300pt
\centerline{\epsfbox{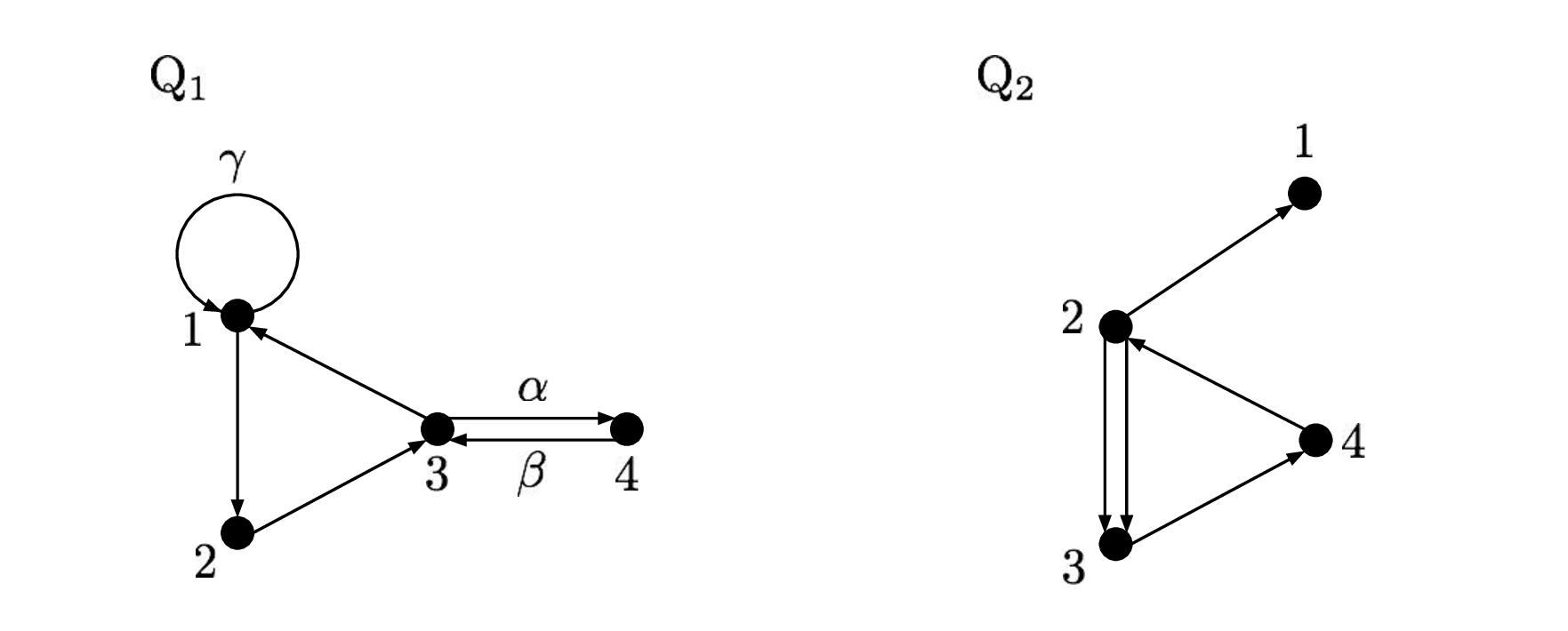}}
\caption{Graph ${\rm Q}_1$ has a 2-cycle $(\alpha, \beta)$ and a loop $\gamma$. 
It is not a quiver. Graph ${\rm Q}_2$ is a quiver. }
\label{quiver}
\end{figure}

A quiver  determines  a triple 
$$
(\Lambda,  \{e_v\}, (\ast,  \ast)),
$$ where  $\Lambda$  is a lattice  generated 
by  the vertices $\{v\}$ of the quiver, $\{e_v\}$ 
is the  basis parametrised 
by  the vertices, and $(\ast,  \ast)$ is a skewsymmetric integral bilinear form on $\Lambda$, 
uniquely defined by
$$
(e_v, e_w) := \#\{\mbox{arrows from $v$ to $w$}\}  - \#\{\mbox{arrows from $w$ to $v$}\}. 
$$
Vice verse, such a triple $(\Lambda, \{e_v\}, (\ast, \ast))$ 
determines a quiver, whose
 vertices are the basis vectors $e_v$, 
and vertices $v,w$ are related by an arrow with multiplicity $(e_v, e_w)$ if and only if $(e_v, e_w)>0$.

Any lattice $\Lambda$ with a bilinear skewsymmetric $\Z$-valued 
form $(\ast, \ast)$ gives rise to a {\it quantum torus algebra}  
${{\bf T}}_{\Lambda}$. It is an algebra over the ring of Laurent polynomials $\Z[q, q^{-1}]$ 
in  $q$ 
given by a free  $\Z[q, q^{-1}]$-module with a basis $ X_v$, $v \in \Lambda$, with the product relation
$$
X_{v_1} X_{v_2} = q^{(v_1, v_2)}X_{v_1+v_2}.
$$

Therefore a quiver ${\bf q}$ provides us with  a quantum torus algebra ${{\bf T}}_{\bf q}$. 
We think about it geometrically, as of the algebra of functions on a non-commutative space - a quantum torus ${{\rm T}}_{\bf q}$. 

\vskip 3mm
Any 
basis vector $e_w$ provides  a {\it mutated in the 
direction $e_w$} quiver ${\mathbf q'}$. 
The  quiver 
${\bf q}'$ is defined by changing the 
basis $\{e_v\}$ only. The lattice and the form stay intact. 
The new basis $\{e'_v\}$ is defined via halfreflection of the basis $\{e_v\}$ along the hyperplane 
$(e_{w}, \cdot)=0$:
\be \la{12.12.04.2aa}
e'_v := 
\left\{ \begin{array}{lll} e_v + (e_v, e_w)_+e_w
& \mbox{ if } &  v\not = w\\
-e_w& \mbox{ if } &  v = w.
\end{array}\right.
\ee
Here $\alpha_+:= \alpha$ if $\alpha\geq 0$ 
and $\alpha_+:=0$ otherwise.  
Quiver mutations in a coordinate form were  introduced by Seiberg \cite{Se95}, and independently by 
Fomin-Zelevinsky \cite{FZI}. 

\vskip 3mm
A quiver ${\bf q}$ gives rise to a dual pair $({\cal A}, {\cal X})$ of cluster varieties  \cite{FG2}. Their cluster coordinate systems are parametrised by the quivers obtained by  
mutations of the quiver ${\bf q}$. 
The algebra of regular functions on the cluster $K_2$-variety ${\cal A}$ is 
 Fomin-Zelevinsky's upper cluster algebra  \cite{FZI}.

The cluster Poisson variety ${\cal X}$ has a deformation, called quantum cluster variety, 
which depends on a parameter $q$. The quantum cluster coordinate systems  
of the quantum cluster variety are related by the {\it quantum cluster transformations}. 

The crucial part in their definition \cite{FG2} 
plays the quantum dilogarithm formal power series: 
 \be \la{psia}
{\bf \Psi}_q(x):= 
\frac{1}{(1+qx)(1+q^3x)(1+q^5x)(1+q^7x)\ldots}.
\ee
It is the unique formal power series starting from $1$ and satisfying a  
difference relation
\begin{equation} \label{11.19.06.20a}
{\bf \Psi}_q(q^2x) = (1+qx){\bf \Psi}_q(x). 
\ee
It has the  power series expansion, easily checked by using 
the difference relation:
\be \la{psib}
{\bf \Psi}_q(x)= \sum_{n=0}^{\infty}\frac{q^{n^2}x^n}{|{\rm GL}_n(F_{q^2})|}.
\ee
The logarithm of the power series 
${\bf \Psi}_q(x)$ is the $q$-dilogarithm power series:
\be \la{qdq}
 {\rm log}~{\bf \Psi}_q(x)=
\sum_{n\geq 1}\frac{(-1)^{n+1}}{n(q^{n}-q^{-n})} x^n. 
\ee
Indeed, it suffices to show that the difference relation \eqref{11.19.06.20a} holds for the right hand side:
\[
\displaystyle{\sum_{n\geq 1}}\frac{(-1)^{n+1}(q^2x)^n }{n(q^{n}-q^{-n})} -
\displaystyle{\sum_{n\geq 1}
\frac{(-1)^{n+1}x^n}{n(q^{n}-q^{-n})} } 
= \displaystyle{\sum_{n\geq 1}}\frac{(-1)^{n+1}(q^{2n}-1)x^n}{n(q^{n}-q^{-n})}= 
\displaystyle{\sum_{n\geq 1}}\frac{-(-qx)^n}{n} 
=\log(1+qx).
\]
In particular, in the quasiclassical limit we recover the classical dilogarithm power series:
\be \la{qdq1}
 \lim_{q \to 1}(q-q^{-1}){\rm log}~{\bf \Psi}_q(x) = -\sum_{n\geq 1}\frac{(-x)^{n}}{n^2}= - {\rm Li}_2(-x).
\ee

The quantum cluster transformations are automorphisms of the 
non-commutative fraction field 
 of the quantum torus algebra ${\bf T}_\Lambda$,  given by the 
conjugation by  {quantum dilogarithms}. Inspite of the fact that the quantum dilogarithms 
are power series, the conjugation is a rational transformation due to the 
difference relation (\ref{11.19.06.20a}). 
The quasiclassical limit when $q\to 1$ of the quantum cluster variety is 
the cluster Poisson variety ${\cal X}$. 
We carefully review the definition of quantum cluster varieties in Section 2. 

\vskip 3mm


There is a natural cluster Poisson structure on the space ${\cal X}_{{\rm PGL_m}, \bS}$, 
 introduced in \cite{FG1}. 
The group $\Gamma_\bS$ acts by 
cluster transformations of  ${\cal X}_{{\rm PGL_m}, \bS}$, provided that

\begin{itemize} \la{assumption}
\item[(i)] 
$g(\bS) + \mu\geq 3$, 
or that $\bS$ is an annulus with $2$ special points.
\end{itemize}
In particular,  the rotation $r_{\bS, h}$ is a cluster transformation. See Section \ref{cluster.mapping.class.group}.
\bt[{Theorems \ref{cluster.ast.6.16.16.07h},  \ref{2.17.22.42h}}] \la{2.27.15.1}
1) Assuming $(i)$, the  involution $\ast$ acts by a cluster transformation of  
${\cal X}_{{\rm PGL_m}, \bS}$. 

2) Let us assume $(i)$, exclude surfaces with $n=1$ puncture and no special points,  
and if $\G$ is of type ${\rm A}_1$, exclude a sphere with 3 punctures. 
Then the group $W^n$ acts by cluster transformations of the space ${\cal X}_{{\rm PGL_m}, \bS}$.    
\et


Let us recall the moduli space ${\cal A}_{{\G}, \bS}$. The group 
$\Gamma_\bS$ and the involution $\ast$ act naturally on the space ${\cal A}_{{\G}, \bS}$. 
In Section 4 we introduce an action of the Weyl group $W^n$ by birational automorphisms of the space 
${\cal A}_{{\G}, \bS}$. 
The space ${\cal A}_{{\rm SL_m}, \bS}$
 has a natural cluster structure of different kind, called a cluster $K_2$-structure, or 
cluster ${\cal A}$-variety structure; 
the  group $\Gamma_\bS$ acts by cluster transformations 
of the space ${\cal A}_{{\rm SL_m}, \bS}$ under the same assumptions as for the 
space ${\cal X}_{{\rm PGL_m}, \bS}$ \cite{FG1}. 

\bt \la{2.27.15.1aa} The involution $\ast$ and the Weyl group $W^n$ act on 
the space ${\cal A}_{{\rm SL_m}, \bS}$ 
by cluster transformations under the same assumptions  as in the parts 1) and 2) 
of  Theorem \ref{2.27.15.1}. 
\et

Our  main result determines the DT-transformation of the moduli space ${\cal X}_{{\rm PGL_m}, \bS}$ 
when $\bS$  is  an admissible  decorated surface. 
Let us formulate the question in the next subsection.

\subsection{Donaldson-Thomas transformations} \la{sec1.3dualII}

Kontsevich and Soibelman   \cite{KS1}, generalizing the original Donaldson-Thomas invariants \cite{DT}, 
defined Donaldson-Thomas invariants of
 a 3d Calabi-Yau  category ${\cal C}$ equipped with a stability condition. 
An important  class of 3d CY categories 
is provided by quivers with potentials. First we briefly recall the definitions and results following \cite{KS1, KS2}.

\vskip 3mm
A quiver ${\bf q}$ with a {\it generic} potential $W$  gives rise to a 3d CY category ${\cal C}({\bf q}, W)$, 
  defined as the derived category of certain representations of the Ginzburg DG algebra \cite{Gin} of the quiver with potential $({\bf q}, W)$. See \cite[Sect.7]{K12}, \cite{N10} for details.

The category {${\cal C}({{\bf q}, W})$} has a collection of spherical generators $\{S_v\}$, 
parametrized by the vertices  of the quiver ${\bf q}$, called a 
{\it cluster collection}. Their classes $[S_v]$  form a basis of the  Grothendieck group $K_0({\cal C}({\bf q},W))$. See \cite[Sect.8.1]{KS1}. 

For any 3d CY category ${\cal C}$, the lattice $K_0({\cal C})$ has 
a skew-symmetric integral bilinear form:
$$
([A], [B])_{\rm Euler}:= - \sum_{i=0}^3(-1)^{i} {\rm rk} ~{\rm Ext}^i(A,B), ~~~~[A],[B] \in K_0({\cal C}).
$$
The original quiver ${\bf q}$ is identified with the quiver assigned to the triple 
$$
\Bigl(K_0({\cal C}({\bf q}, W)), [S_v], (\ast, \ast)_{\rm Euler}\Bigr).
$$
\vskip 3mm

There is an open domain ${\cal H}_{\bf q}$ in the space of stability conditions on the category ${{\cal C}({{\bf q}, W})}$, described as follows. 
Let us consider a ``punctured upper halfplane'' 
$$
{\cal H} = \{z\in \C~|~ z = r e^{i\varphi}, ~0< \varphi \leq \pi, ~ r>0\}.
$$
The domain ${\cal H}_{\bf q}$ is identified with 
the product of ${\cal H}$'s over the set of vertices of the quiver ${\bf q}$:
$$
{\cal H}_{\bf q}:= \prod_v{\cal H} = \{z_v \in {\cal H}\}.
$$
The {\it central charge} of a stability condition ${\bf s} \in {\cal H}_{\bf q}$ is 
given by a group homomorphism
$$
Z: K_0({\cal C}({\bf q}, W)) \lra \C, ~~~~ [S_v] \lms z_v. 
$$
Define the positive cone  of the lattice $K_0({\cal C}({\bf q}, W))$ generated by the basis
\[
\Lambda_{\bf q}^+: =\oplus_v {\Z_{\geq 0}}[S_v].
\]
A central charge $Z$ is called {\it generic} if there are no two $\Q$-independent elements of $\Lambda^+_{\bf q}$ which are mapped by $Z$ to the same ray.

\paragraph{Quantum DT-series.} Consider a unital algebra   over the field $\Q((q))$ of Laurent 
power series in $q$, 
given by the $q$-commutative formal power series in $X_v$, 
where $v$ is in the positive cone: 
\[
\widehat{{\bf A}_{\bf q}}:= \Q((q)) [[X_v, ~v \in \Lambda_{\bf q}^+ ~|~ X_v X_w = q^{(v, w)}X_{v+w}]].
\]

Kontsevich-Soibelman \cite[Sect.8.3]{KS1} assigned to 
 the 3d CY category ${\cal C}({\bf q}, W)$ is a formal power series, 
called the {\it quantum {\rm DT}-series} of the category: 
\be
\la{refine.dt.invariant}
{\Bbb E}_{\bf q} =  1+ \mbox{higher order terms} \in \widehat{{\bf A}_{\bf q}}.
\ee 
For a generic potential $W$, the series ${\Bbb E}_{\bf q}$ depends on the quiver ${\bf q}$ only.
 
The following useful Lemma \ref{ULKS}, due to  \cite[Th.6]{KS}
see also \cite[Lemma 1.12]{MMNS}, 
expresses the series ${\Bbb E}_{\bf q}$ as a product, possibly infinite, 
of the ``$q$-powers'' of the quantum dilogarithm power series ${\bf \Psi}_q(X)$. 
Namely, given a formal Laurent series $\Omega(q) \in \Q((q))$, let us set 
\be \la{12.19.15.1}
\begin{split}
&{\bf \Psi}_q(X)^{\circ \Omega(q) } = {\rm exp}\Big(\sum_{n\geq 1}\frac{(-1)^{n+1}\Omega(q^n)}{n(q^{n}-q^{-n})} X^n\Big).
\end{split}
\ee
If $\Omega(q)$ is just an integer $\Omega$, then (\ref{12.19.15.1}) 
is the usual power ${\bf \Psi}_q(X)^{\Omega}$, as is clear 
 from the formula (\ref{qdq}) relating the logarithm of the power series 
${\bf \Psi}_q(X)$ to the $q$-dilogarithm power series.

\bl \la{ULKS} Given a  stability condition ${\bf s}$ with a generic central charge $Z$, 
 there exists a unique collection of rational functions 
 $\Omega_{\gamma}^{\bf s}(q) \in \Q(q)$, 
parametrized by the positive cone vectors $\gamma\in \Lambda^+_{\bf q}-\{0\}$ with $Z(\gamma)\in {\cal H}$, 
such that the quantum DT series ${\Bbb E}_{\bf q}$ are factorized as
\be
\la{refine.dt.invariant2} 
\begin{split}
&{\Bbb E}_{\bf q} = 
{\displaystyle \prod_{l\subset {\cal H}}^{\curvearrowright}} {\Psi}_l,  \hskip 6mm
{\Psi}_l:= \prod_{\gamma}{\bf \Psi}_q(X_\gamma)^{\circ \Omega_{\gamma}^{\bf s}(q) }.\\ 
\end{split}
\ee 

Here the first product is over all rays $l\subset {\cal H}$ in the clockwise order; 

The second product 
is over positive lattice vectors $\gamma\in \Lambda^+_{\bf q}-\{0\}$ such that $Z(\gamma)\in l$. 
\el
In particular, if $\Omega_{\gamma}^{\bf s}(q)$ are all integers, then ${\Bbb E}_{\bf q}$ is a product of quantum dilogarithm series.

The rational functions $\Omega_{\gamma}^{\bf s}(q)$ are called the {\it quantum {\rm (or refined) DT}-invariants} assigned to the stability condition ${\bf s}$, and the element $\gamma$ \cite[Def.6.4]{KS2}. 
So all the quantum DT-invariants are packaged into the single quantum DT-series ${\Bbb E}_{\bf q}$. 
They are uniquely determined by the  ${\Bbb E}_{\bf q}$. 

Evaluating at $q=-1$, we obtain the {\it numerical DT-invariant} $\Omega_{\gamma}^{\bf s}(-1)$.  \underline{Roughly speaking}, the numerical DT-invariant is the weighted Euler characteristic of the moduli space of all semistable objects 
 for the stability condition ${\bf s}$ with a given class $\gamma$ in $K_0$  ({\it cf}. \cite{KS1}).

The integrality conjecture \cite[Sect.7.6, Conj.6]{KS1} asserts that the numerical DT-invariants are integers for all generic ${\bf s}$.
Konstevich-Soibelman \cite[Sect.6.1]{KS2} proved that if ${\Bbb E}_{\bf q}$ is {\it quantum admissible} (in the sense of [{\it loc.cit.} Definition 6.3]), then $\Omega_{\gamma}^{\bf s}(q)$ is a Laurent polynomial with integral coefficients. 
The integrality conjecture follows directly then. 

Conversely, if for a given generic ${\bf s}$ we have $\Omega_{\gamma}^{\bf s}(q)\in \Z[q, q^{-1}]$ for all $\gamma\in \Lambda^+_{\bf q}$, then ${\Bbb E}_{\bf q}$ is  quantum admissible. Therefore  $\Omega_{\gamma}^{\bf s}(q)\in \Z[q, q^{-1}]$ for all generic ${\bf s}$.

\paragraph{DT-transformations.} 
Define  the algebra of formal power series along the negative cone
\[
\widecheck{{\bf A}_{\bf q}}:= \Q((q)) [[X_{-v}, ~v \in \Lambda_{\bf q}^+ ~|~ X_{-v} X_{-w} = q^{(v, w)}X_{-v-w}]].
\]
The conjugation ${\rm Ad}_{{\Bbb E}_{\bf q}}$ by the ${\Bbb E}_{\bf q}$ is  a formal power series transformation of $\widehat{{\bf A}_{\bf q}}$. 
Being composed with the reflection map $\Sigma$ acting by $$
{\Sigma}(X_{v})=X_{-v} ~~~~\forall v\in K_0({\cal C}({\bf q}, W)),
$$ 
we get a formal power series transformation, called   {\it {\rm DT}-transformation}:
$$
{\rm DT}_{\bf q} := {\rm Ad}_{{\Bbb E}_{\bf q}} \circ \Sigma : ~~ \widecheck{{\bf A}_{\bf q}}~\lra~ \widehat{{\bf A}_{\bf q}}.
$$
Note that ${\rm DT}_{\bf q}$  is 
an invariant of the quiver ${\bf q}$. It is an  ``infinite'' cluster transformation.

The conjugation ${\rm Ad}_{{\Bbb E}_{\bf q}}$ by the ${\Bbb E}_{\bf q}$ is not necessarily rational. 
If it is rational, then the {\rm DT}-transformation of the quantum torus ${\bf T}_{\bf q}$ is defined as
 \be \la{DTDEF}
 {\rm DT}_{\bf q} := {\rm Ad}_{{\Bbb E}_{\bf q}} \circ \Sigma : ~~ {\Bbb T}_{\bf q} ~\lra~  
{\Bbb T}_{\bf q}, ~~~~
 \ee
  where ${\Bbb T}_{\bf q}:= {\rm Frac}({\bf T}_{\bf q})$ is the non-commutative 
field of fractions of ${\bf T}_{\bf q}$.

 \paragraph{DT-invariants from DT-transformations.}
 Consider the symplectic double $(\Lambda_{\cal P}, \langle *,*\rangle_{\cal P})$ of the original lattice $\Lambda$ with the form $(*, *)$, given by 
 $$
 \Lambda_{\cal P}:= \Lambda \oplus {\rm Hom}(\Lambda, \Z), ~~~~\langle(v_1, f_1), (v_2, f_2)\rangle_{\cal P}:= (v_1, v_2) - (f_1, v_2) - (f_2, v_1). 
 $$
 Since $\Lambda \subset \Lambda_{\cal P}$, one can defined the DT-transformation  ${\rm DT}_{{\cal P}, \bf q}$ of the 
 quantum torus algebra related to the pair $(\Lambda_{\cal P}, \langle *,*\rangle_{\cal P})$ by the same formula (\ref{DTDEF}):
 \be \la{DTDEFa}
 {\rm DT}_{{\cal P}, \bf q} := {\rm Ad}_{{\Bbb E}_{\bf q}} \circ \Sigma_{\cal P}, ~~~~ \Sigma_{\cal P}: X_v \lms X_{-v} ~ ~\forall v\in \Lambda_{\cal P}. 
 \ee
 Remarkably,   
the DT-series ${\Bbb E}_{\bf q}$ and therefore all DT-invariants $\Omega_{\gamma}^{\bf s}$ are recovered from the DT-transformation (\ref{DTDEFa}). 
They are recovered from the DT-transformation (\ref{DTDEF}) if the form $(\ast, \ast)$ is non-degenerate. 

\begin{remark} 
 
Quantum cluster transformations were defined in \cite{FG2} by the conjugation by the 
quantum dilogarithm power series ${\bf \Psi}_q(X)$. The crucial fact that they are 
rational transformations follows from  difference relation (\ref{11.19.06.20a}), characterizing 
the  power series ${\bf \Psi}_q(X)$. 

On the other hand, the quantum dilogarithm ${\bf \Psi}_q(X)$ 
appeared in \cite{KS} story due to formula (\ref{qdq}),  
as well as  thanks to its power series expansion (\ref{psib}):
the coefficient in $x^n$ in (\ref{psib}) reflects counting the number of points of the stack 
$X^{\oplus n}/ {\rm Aut}(X^{\oplus n})$ over ${\Bbb F}_{q^2}$ for a simple object $X$.

The quantum dilogarithm power series (\ref{psia}) 
are convergent if $|q|<1$, but hopelessly divergent if $|q|=1$. 
The most remarkable feature of the quantum dilogarithm is that 
its ``modular double''
\be \la{12.6.15.1}
\Phi_{\hbar}(x) := \frac{{\bf \Psi}_q(e^x)}{{\bf \Psi}_{q^\vee}(e^{x/\hbar})}, ~~~~q= {\rm exp}(i\pi \hbar), 
~~q^\vee= {\rm exp}(i\pi /\hbar)
\ee
has wonderful analytic properties at all $q$, e.g. $|q|=1$. It has a beautiful  integral presentation
$$
\Phi_{\hbar}(x) = {\rm exp}\Bigl(-\frac{1}{4}\int_{\Omega}\frac{e^{ipx} dp}{sh(\pi  p)sh(\pi \hbar p)p}\Bigr).
$$
 The quantum dilogarithm function $\Phi_{\hbar}(x)$ is crucial in the 
quantization of cluster Poisson varieties, given by a $\ast$-representation in a Hilbert space of the 
$q$-deformed algebra of functions \cite[version 1]{FG2}, \cite{FG4}. 
Yet so far the function $\Phi_{\hbar}(x)$ has no role in the 
{\rm DT} theory. 
\end{remark}

\paragraph{Mutations of quivers with potentials.}
Mutations of quivers with potentials were  studied by 
{Derksen-Weyman-Zelevinsky \cite{DWZ}.}  
The mutation $\mu_k$ at the direction $k$ gives rise to a pair $({\bf q}' ,W')=\mu_k({\bf q}, W)$. 
The domains ${\cal H}_{\bf q}$ for the quivers with potentials obtained by mutations of the original quiver 
${\bf q}$ form a connected  open domain in the space of all stability conditions on the category 
${\cal C}({\bf q}, W)$. Moving in this domain, and thus mutating a quiver, 
 we get a different {\rm DT}-transformation of the same quantum torus.  
The Kontsevich-Soibelman wall crossing formula  
tells how quantum DT-series changes under quiver mutations, see \cite[Sect.8.4, Property 3]{KS1}.  

Indeed, let ${\bf q}=(\Lambda, \{e_i\}, (\ast, \ast))$. The mutated quiver $\mu_k({\bf q})$ is isomorphic to
\[
{\bf q}':=(\Lambda, \{e_i'\}, (\ast, \ast)), \hskip 7mm \mbox{where } e_i'= \left\{ \begin{array}{lll} 
-e_k& \mbox{ if }   i = k\\
e_i + [(e_i, e_k)]_+e_k
& \mbox{ otherwise. }
\end{array}\right.
\]
Note that it only changes the basis. The lattice and the form stay intact. Therefore one can identify the quantum tori
\be
\la{canonical.muk.2.4.11}
{\bf T}_{\bf q}\stackrel{\sim}{=} {\bf T}_{\Lambda} \stackrel{\sim}{=} {\bf T}_{{\bf q}'}.
\ee
Consider the intersection
\[
\widehat{{\bf A}_{\bf q}}\cap \widehat{{\bf A}_{\bf q'}}= \Q((q)) [[X_v, ~v \in \Lambda_{\bf q}^+\cap\Lambda_{\bf q'}^+ ~|~ X_v X_w = q^{(v, w)}X_{v+w}]].
\]
\bt[{\cite[p.138]{KS1}}]  
\la{KS.wall.crossing0h}
Under the identification \eqref{canonical.muk.2.4.11}, we have
\be
{\bf \Psi}_q(X_{e_k})^{-1} {\Bbb E}_{\bf q} = {\Bbb E}_{\bf q'} {\bf \Psi}_q(X_{e_k'})^{-1} ~~~ \in \widehat{{\bf A}_{\bf q}}\cap \widehat{{\bf A}_{\bf q}'}.
\ee
\et

\paragraph{Compatibility with cluster mutations.}
The {\it quantum cluster mutation}
\be
\Phi(\mu_k)= {\rm Ad}_{{\bf \Psi}_q(X_{e_k})} \circ i : ~~~ {\Bbb T}_{{\bf q}'} \lra {\Bbb T}_{{\bf q}}
\ee
is the composition of the isomorphism $i:{\Bbb T}_{{\bf q}'} \to {\Bbb T}_{{\bf q}}$ under \eqref{canonical.muk.2.4.11} and the conjugation $ {\rm Ad}_{{\bf \Psi}_q(X_{e_k})}$.

The following result is a direct consequence of Theorem \ref{KS.wall.crossing0h}. See also \cite[p.143]{KS1}.
\bt
\la{KS.wall.crossing}
If  ${\rm Ad}_{{\Bbb E}_{\bf q}}$ is rational,  
then the following diagram is commutative:
\begin{displaymath}
    \xymatrix{
        {\Bbb T}_{\bf q'} \ar[r]^{\Phi(\mu_k)} \ar[d]_{{\rm DT}_{\bf q'}} & {\Bbb T}_{{\bf q}} \ar[d]^{{\rm DT}_{{\bf q}}} \\
         {\Bbb T}_{\bf q'} \ar[r]^{\Phi(\mu_k)}       & {\Bbb T}_{{\bf q}} }
\end{displaymath}
\et
\begin{proof}
By Theorem \ref{KS.wall.crossing0h}, we have
$
{\rm Ad}_{{\bf \Psi}_q(X_{e_k})^{-1}}\circ  {\rm Ad}_{{\Bbb E}_{\bf q}} \circ i= i \circ {\rm Ad}_ {{\Bbb E}_{\bf q'}} \circ {\rm Ad}_{{\bf \Psi}_q(X_{e_k'})^{-1}}.
$ 
Therefore
\be
\la{proof.wall.cross.cor.2}
  {\rm Ad}_{{\Bbb E}_{\bf q}} \circ i \circ {\rm Ad}_{{\bf \Psi}_q(X_{e_k'})}= {\rm Ad}_{{\bf \Psi}_q(X_{e_k})}\circ i \circ {\rm Ad}_ {{\Bbb E}_{\bf q'}}. 
\ee
Note that $i(X_{e_k'})=X_{-e_k}$. Therefore 
\be
\la{proof.wall.cross.cor.1}
 i \circ {\rm Ad}_{{\bf \Psi}_q(X_{e_k'})} \circ \Sigma =  {\rm Ad}_{{\bf \Psi}_q(X_{-e_k})} \circ i \circ \Sigma = \Sigma \circ  {\rm Ad}_{{\bf \Psi}_q(X_{e_k})} \circ i. 
\ee
Therefore
\begin{align}
\Phi(\mu_k)\circ {\rm DT}_{{\bf q}'}&~= ~{\rm Ad}_{{\bf \Psi}_q(X_{e_k})}\circ i \circ {\rm Ad}_ {{\Bbb E}_{\bf q'}}\circ \Sigma 
\stackrel{\eqref{proof.wall.cross.cor.2}}{=} {\rm Ad}_{{\Bbb E}_{\bf q}} \circ i \circ {\rm Ad}_{{\bf \Psi}_q(X_{e_k'})} \circ \Sigma \nonumber\\
 &\stackrel{\eqref{proof.wall.cross.cor.1}}{=} {\rm Ad}_{{\Bbb E}_{\bf q}}\circ\Sigma \circ  {\rm Ad}_{{\bf \Psi}_q(X_{e_k})} \circ i 
 ~=~ {\rm DT}_{{\bf q}}\circ \Phi(\mu_k). 
\end{align}
\end{proof}

There is a similar interpretation of the DT-transformations ${\rm DT}_{{\cal P}, q}$, see (\ref{DTDEFa}),
 via the quantum cluster variety ${\cal A}_{\rm prin, q}$, discussed in the end of Section 2. 

\vskip 3mm
Theorem \ref{KS.wall.crossing} implies that 
the multitude of  {\rm DT}-transformations assigned to  quivers obtained by  mutations of an initial quiver ${\bf q}$  
are nothing but 
a single formal, i.e. given by formal power series, automorphism ${\rm DT}$ of the quantum cluster variety, written in different cluster coordinate 
systems assigned to these quivers.

\vskip 3mm

So a natural question arises: 
\be \la{question1}
\mbox{How to determine the {\rm DT}-transformation 
of a given (quantum) cluster variety?} 
\ee

\subsection{DT-transformations for  moduli spaces of local systems} \la{sec1.3dualIIa}

\paragraph{Cluster DT-transformation.}
Keller \cite{K11, K12, K13} using the 
work of Nagao \cite{N10}, proposed a simpler and more accessible, 
but much more restrictive combinatorial version of ${\rm DT}$-transformation. 
It  is a cluster transformation, which may not be defined, but when it does, it
 coincides with the Kontsevich-Soibelman {\rm DT}-transformation {\cite[Th 6.5]{K12}}. 
 We call it {\it cluster {\rm DT}-transformation}. 
It acts on any type of cluster variety, e.g. 
on the quantum cluster variety. We postpone a  definition of cluster DT-transformations till Section \ref{SecAA1}.

If ${\rm DT}_{\bf q}$ is a cluster DT-transformation, then Theorem \ref{KS.wall.crossing} follows directly from Theorem \ref{universal.dt} of the present paper. 
Furthermore, its quantum ${\rm DT}$-series ${\Bbb E}_{\bf q}$ can be presented as a finite product of quantum dilogarithm power series. As a Corollary of \cite[Prop 6.2]{KS2}, we have
\bp If ${\rm DT}_{\bf q}$ is a cluster DT-transformation, then ${\Bbb E}_{\bf q}$ is quantum admissible. Therefore for arbitary generic stability condition ${\bf s}$, the quantum DT-invariants $\Omega_{\gamma}^{\bf s}(q)\in \Z[q,q^{-1}]$. 
\ep
Conjecturally, $\Omega_{\gamma}^{\bf s}(q)$ has non-negative coefficients in this case.

\vskip 2mm
Even if a {\rm DT}-transformation of a  cluster variety is rational, 
it may not  be a cluster transformation. 
In Section \ref{SecAA1}, elaborated in Section \ref{SecAA}, we give  
a conjectural elementary characterization of rational 
 {\rm DT}-transformations of cluster varieties. It does not refer to  Kontsevich-Soibelman theory.  

\paragraph{Main result.}
Moduli spaces ${\cal X}_{\G, \bS}$ are important examples of 
cluster Poisson varieties. So there is a ${\rm DT}$-transformation acting as a single 
(formal, or if we are lucky, rational) transformation   of a moduli space ${\cal X}_{\G, \bS}$, encapsulating 
  the Donaldson-Thomas invariants 
of the corresponding 3d CY categories. 
This leads to the following questions:
\be \la{question1}
\mbox{What are the {\rm DT}-transformations 
of the moduli spaces ${\cal X}_{\G, \bS}$? Are they rational?} 
\ee

We use Theorem \ref{2.27.15.1} and Keller's characterization of cluster  {\rm DT}-transformations to determine the ${\rm DT}$-transformation 
${\rm DT}_{\G, \bS}$ of the 
space ${\cal X}_{\G, \bS}$ for $\G={\rm PGL}_m$.

\bt \la{2.27.15.2}
Let $\G={\rm PGL}_m$. If $\bS$ is admissible in the sense of Definition 
\ref{ADMDS}, then the 
${\rm DT}$-transformation ${\rm DT}_{\G, \bS}$ is a cluster transformation. It is  given by the following formula:
\be \la{FDT}
{\rm DT}_{\G, \bS} = {\rm C}_{\G, \bS}.
\ee
The cluster transformation (\ref{FDT}) is a cluster DT-transformation in the sense of Definition \ref{2.14.16.1}.
\et

\bcon \la{2.27.15.2*}
Formula  \eqref{FDT} is valid for any pair $(\G, \bS)$.
\econ

\paragraph{Example.} The $W$-action is not cluster if $\bS$ 
has a single puncture, no holes, and $\G={\rm PGL}_2$. 
The cluster DT transformation in this case is 
not defined. 
Yet when $G={\rm PGL}_2$ and $\bS$ is a punctured torus the formula ${\rm DT}_{\G, \bS} = {\rm C}_{\G, \bS}= w_0$ was proved by Kontsevich.  

\vskip 3mm

For surfaces with $n>2$ punctures
 a proof for $\G={\rm PGL}_2$ follows from the results of Bucher and Mills 
\cite{Bu}, \cite{BuM} who found green sequences of cluster transformations in these cases. 

For $\G={\rm PGL}_2$ we give two  
transparent  geometric proofs of formula (\ref{FDT}):

i) The first proof, presented in  Section \ref{sec3.2},  is based on 
the interpretation of  integral tropical points of the moduli space  ${\cal X}_{\G, \bS}$  
as integral laminations  \cite[Section 12]{FG1}. It  requires a calculation 
of the tropicalization of the $w_0$-action at the puncture given in \cite[Lemma 12.3]{FG1}.

ii) The second proof, presented in Section  \ref{sec3.3}, does not require any calculations at all. 
It uses an interpretation of integral tropical points 
as explicitly constructed divisors at infinity of the moduli space  ${\cal X}_{\G, \bS}$  
 \cite{FG3}. We worked out the details when $\bS$ has no punctures. 

None of them require a decomposition of 
the map ${\rm C}_{{\rm PGL}_2, \bS}$ 
into a composition of mutations. However we do not know how to generalise these proofs
to the higher rank groups. 
To find such generalizations is a very important problem. 

\vskip 3mm
For  $\G={\rm PGL}_m$ we give  another, high precision proof,  based on an explicit 
decomposition of the cluster transformation ${\rm C}_{\G, \bS}$ 
into a composition of mutations. This decomposition looks  pretty complicated, 
but it reveals a lot of valuable information about the 
{\rm DT}-invariants $\Omega_\gamma^{\bf s}$.

\subsection{{\rm DT}-transformations of cluster varieties and Duality Conjectures} \la{SecAA1}

The definition of   {\rm DT}-transformations is  complicated. It uses generic potentials, but  
in the end the {\rm DT}-transformation does not depend on it. So one wants an intrinsic 
definition of  {\rm DT}-transformations of cluster varieties, given just in terms of  cluster varieties.  
 
We state  a conjecture   relating  {\rm DT}-transformations of cluster varieties to Duality Conjectures. 
It implies an alternative conjectural definition of  
  the {\rm DT}-transformations of  cluster varieties,  
which characterizes them uniquely, and makes transparent their crucial properties.

\paragraph{Duality Conjectures \cite{FG2}.} Recall that a  quiver gives rise to a dual pair $({\cal A}, {\cal X})$ of cluster 
varieties of the same dimension, as well as a  Langlands dual pair of cluster varieties $({\cal A}^\vee, {\cal X}^\vee)$. The cluster modular group $\Gamma$ 
 acts  by their automorphisms.

The following objects, equipped with a $\Gamma$-action, are assigned to  any cluster variety 
${\cal Y}$:
\begin{itemize}

\item  The algebra  ${\cal O}({\cal Y})$  (respectively  $\widehat {\cal O}({\cal Y})$)  of regular (respectively  formal) 
 functions on  ${\cal Y}$.  
 
 \item  
  A set ${\cal Y}(\Z^t)$ of the integral tropical points of ${\cal Y}$. 

\end{itemize}

The set ${\cal Y}(\Z^t)$  is isomorphic, in many different ways, to $\Z^n$, where $n={\rm dim}{\cal Y}$. 

As  was shown in \cite{GHK}, the algebra 
${\cal O}({\cal X})$ could have smaller dimension then  ${\cal X}$.

Duality Conjectures \cite[Section 4]{FG2} predict  a  deep multifacet duality between   
cluster varieties ${\cal A}$ and ${\cal X}^\vee$. In particular, one should have 
canonical $\Gamma$-equivariant pairings 
\be \la{DCIa1}
\begin{split}
&{\bf I}_{\cal A}: {\cal A}(\Z^t) \times {\cal X}^\vee \lra {\Bbb A}^1, ~~~~{\bf I}_{\cal X}: {\cal A} \times {\cal X}^\vee(\Z^t) \lra {\Bbb A}^1. \\
\end{split}
\ee 
This means that each  $l \in {\cal A}(\Z^t)$, and each $m \in {\cal X}^\vee(\Z^t)$, give rise to 
functions  
$$
{\Bbb I}_{\cal A}(l):= {\bf I}_{\cal A}(l, \ast) ~~\mbox{on ${\cal X}^\vee$, ~~~~and }
 ~~~~{\Bbb I}_{\cal X}(m):= {\bf I}_{\cal X}(m, \ast)~~\mbox{on ${\cal A}^\vee$}.
$$  
In particular, in the formal setting we 
should have a pair of canonical $\Gamma$-equivariant maps 
\be \la{DCIII1}
\begin{split}
&{\Bbb I}_{\cal A}: {\cal A}(\Z^t) \lra  \widehat {\cal O}({\cal X}^\vee), ~~~~{\Bbb I}_{\cal X}: {\cal X}(\Z^t) \lra  \widehat {\cal O}({\cal A}^\vee).
\end{split}
\ee
Each map should parametrise 
a linear basis in its image on the space of  function on the target.

\paragraph{The involutions $i_{\cal A}$ and $i_{\cal X}$ \cite[Lemma 3.5]{FG4}.} Denote by $({\cal A^\circ}, {\cal X^\circ})$ the dual 
pair of cluster varieties assigned to the opposite quiver, obtained by changing the sign of the form.  
 Then there are isomorphisms of cluster varieties 
\be \la{12.4.15.1021}
i_{\cal A}: {\cal A}\lra {\cal A}^\circ, ~~~~
i_{\cal X}: {\cal X}\lra {\cal X}^\circ
\ee
which in any cluster coordinate systems $\{A_i\}$ on ${\cal A}$ and $\{X_i\}$ on ${\cal X}$ act as follows:
\be \la{12.4.15.1011}
i^*_{\cal A}: A^\circ_i \lms A_i, ~~~~i^*_{\cal X}: X^\circ_i \lms X_i^{-1}.
\ee
 
  \paragraph{DT-transformations and Duality Conjectures.} Our point is that 
the  duality ${\cal A} \longleftrightarrow {\cal X}^\vee$ is \underline{not compatible} with the 
isomorphisms $i_{\cal A}$ and $i_{{\cal X}^\vee}$! Furthermore, Conjecture \ref{MCDTTR1} suggests 
that the {\rm DT}-transformations ${\rm DT}_{\cal X}$ and 
${\rm DT}_{\cal A}$ of the cluster varieties ${\cal X}$ and ${\cal A}$ respectively measure the failure of the 
isomorphisms $i_{\cal A}$ and $i_{\cal X}$ to be compatible with the duality. Namely, set
\be \la{Tildeiso1}
{\rm D}_{\cal A}:= i_{\cal A} \circ{\rm DT}_{\cal A}, ~~~~
{\rm D}_{\cal X}:= i_{\cal X} \circ {\rm DT}_{\cal X}.
\ee
We show that these maps are involutions: ${\rm D}_{{\cal A}^\circ}\circ{\rm D}_{\cal A} = {\rm Id}_{\cal A}$, 
${\rm D}_{{\cal X}^\circ}\circ{\rm D}_{\cal X} = {\rm Id}_{\cal X}$.  
Then the duality should intertwine  
${\rm D}_{\cal A}$ with  $i_{{\cal X}^\vee}$, and  
$i_{\cal A}$ with  ${\rm D}_{{\cal X}^\vee}$. So 
we should have diagrams 
\begin{displaymath} \la{CDMS1}
    \xymatrix{
        {\cal A} \ar@{<->}[r] \ar[d]_{{\rm D}_{\cal A}} & {\cal X}^\vee   \ar[d]^{i_{{\cal X}^\vee}} \\
         {\cal A}^\circ  \ar@{<->}[r]      &  {{\cal X}^\vee}^\circ}
         ~~~~~~~~\xymatrix{
        {\cal A} \ar@{<->}[r] \ar[d]_{{i}_{\cal A}} & {\cal X}^\vee   \ar[d]^{{\rm D}_{{\cal X}^\vee}} \\
         {\cal A}^\circ \ar@{<->}[r]      &  {{\cal X}^\vee}^\circ}
         \end{displaymath}

They should give rise to commutative diagrams when one of the columns 
is tropicalised, and the other is replaced by the induced map 
of algebras of functions. 
The horizontal arrows become the canonical maps. Here is a precise statement.

\bcon\footnote{Conjecture \ref{MCDTTR1} is just one incarnation of the "commutative diagrams" above. Another incarnation is the one where 
 the horizontal arrows mean the mirror symmetry.} \la{MCDTTR1} Let $({\cal A}, {\cal X})$  be a  dual 
pair of cluster varieties.  
Then:

i) There are commutative diagrams
\begin{displaymath}
    \xymatrix{
        {\cal X}(\Z^t) \ar[r]^{{\Bbb I}_{\cal X}} \ar[d]_{{i}^t_{\cal X}} & \widehat {\cal O}({\cal A}^\vee) 
          \ar[d]^{{\rm D}^*_{{{\cal A}^\vee}^\circ}} \\
         {\cal X}^\circ(\Z^t) \ar[r]^{{\Bbb I}_{{\cal X}^\circ}} &  \widehat {\cal O}({\cal A}^{\vee\circ})}
         ~~~~~~~~\xymatrix{
        {\cal A}(\Z^t) \ar[r]^{{\Bbb I}_{{\cal A}}} \ar[d]_{{i}^t_{\cal A}} &\widehat {\cal O}( {\cal X}^\vee )  \ar[d]^{{\rm D}^*_{{{\cal X}^\vee}^\circ}} \\
         {\cal A}^\circ(\Z^t) \ar[r]^{{\Bbb I}_{{\cal A}^\circ}}      &  \widehat {\cal O}({\cal X}^{\vee \circ})}
         \end{displaymath}

ii) Assume that the {\rm DT}-transformations ${\rm DT}_{\cal A}$ and ${\rm DT}_{\cal X}$ are positive rational.  Then: 
\begin{itemize}

\item The canonical pairings are ${\rm DT}$-equivariant:
\be\la{CANP1a1}
\begin{split}
&{\bf I}_{\cal A}({\rm DT}_{\cal A}^t(a), {\rm DT}_{{\cal X}^\vee}(x)) = 
{\bf I}_{\cal A}(a, x),~~~~ {\bf I}_{{\cal X}^\vee}({\rm DT}_{\cal A}(x), {\rm DT}^t_{{\cal X}^\vee}(x)) = 
{\bf I}_{{\cal X}^\vee}(a, x).\\
\end{split}
\ee

\item  There are commutative diagrams
\begin{displaymath}
    \xymatrix{
        {\cal X}(\Z^t) \ar[r]^{{\Bbb I}_{\cal X}} \ar[d]_{{\rm D}^t_{\cal X}} & \widehat {\cal O}({\cal A}^\vee) 
          \ar[d]^{i^*_{{{\cal A}^\vee}^\circ}} \\
         {\cal X}^\circ(\Z^t) \ar[r]^{{\Bbb I}_{{\cal X}^\circ}} &  \widehat {\cal O}({\cal A}^{\vee\circ})}
         ~~~~~~~~\xymatrix{
        {\cal A}(\Z^t) \ar[r]^{{\Bbb I}_{{\cal A}}} \ar[d]_{{\rm D}^t_{\cal A}} &\widehat {\cal O}( {\cal X}^\vee )  \ar[d]^{{i}^*_{{{\cal X}^\vee}^\circ}} \\
         {\cal A}^\circ(\Z^t) \ar[r]^{{\Bbb I}_{{\cal A}^\circ}}      &  \widehat {\cal O}({\cal X}^{\vee \circ})}
         \end{displaymath}

\end{itemize}
\econ

Let us recall the most basic feature of Duality Conjectures. A quiver determines a cluster coordinate system $\{A_i\}$ 
 on ${\cal A}^\vee$, and  a cluster coordinate system  on ${\cal X}$.  Duality Conjectures predict that a tropical point $l^+ \in {\cal X}(\Z^t)$ 
with non-negative coordinates $(x_1, ..., x_n)$ in the tropicalised cluster coordinate system  on ${\cal X}$ gives rise to a cluster monomial  
$A_1^{x_1} \ldots A_n^{x_n}$ on ${\cal A}^\vee$:
$$
{\Bbb I}_{\cal X}(l^+) = A_1^{x_1} \ldots A_n^{x_n}, ~~~~ l^+ = (x_1, ..., x_n).
$$
 Let $l^- \in {\cal X}(\Z^t)$ be the tropical point with the coordinates $(-x_1, \ldots , -x_n)$ in the same  coordinate system. Then 
 (\ref{12.4.15.1011}) imlies that the tropicalised transformation ${\rm DT}^t_{{\cal X}}$ has the following property: 
\be \la{GSMPDTCL1}
{\rm DT}_{{\cal X}}^t(l^+)= l^-.
\ee
 In particular, let $l^+_i \in {\cal X}(\Z^t)$ be the tropical point with the coordinates $(0, ... , 1, ..., 0)$: all the coordinates but the i-th one are zero. 
  Specializing (\ref{GSMPDTCL1}) to these tropical points we get  
 \be \la{GSMPDTCL2}
{\rm DT}_{{\cal X}}^t(l_i^+)= l_i^-. 
\ee
It follows easily  from Duality Conjectures that if there is a 
 cluster transformation ${\bf K}$ such that 
 \be \la{KDKD}
 {\bf K}^t(l_i^+) = l_i^-
 \ee  then it is unique  (Proposition \ref{BLDC}). 
 
  \bd \la{2.14.16.1} A cluster transformation ${\bf K}$  such that (\ref{KDKD}) holds in a single cluster coordinate system is called a {\it cluster DT-transformation}. 
    \ed

 Let ${\bf K}_{\cal A}$ and ${\bf K}_{\cal X}$ be the cluster transformations of the 
 ${\cal A}$ and ${\cal X}$ spaces provided by such a cluster transformation ${\bf K}$. If  DT-transformations are rational, Conjecture \ref{MCDTTR1}  implies  that 
  \be \la{KTAX}
{\rm DT}_{\cal A} = {\bf K}_{\cal A}, ~~~~  {\rm DT}_{\cal X} = {\bf K}_{\cal X}.
 \ee
 Indeed, thanks to (\ref{GSMPDTCL2}) - (\ref{KDKD}), the ${\rm DT}_{\cal A}$ and ${\bf K}_{\cal A}$ must  act the same way on  the cluster variables $A_i$, 
 and hence on the cluster algebra.  Therefore they coincide. 
 Using Dulaity Conjectures again, this implies the second claim. See the proof of Proposition \ref{BLDC} for details.

    Keller proved  unconditionally \cite{K11}, 
\cite[Th 6.5, Sect 7.11]{K12}, although in a different 
formulation which used crucially the cluster nature of ${\bf K}$,  see Theorem \ref{basiclamDT},   even a stronger  claim:  
 $$
 \mbox{a cluster DT-transformation is a Kontsevich-Soibelman DT-transformation}.
 $$    
 
 Unlike Keller's definition, Condition (\ref{KDKD}) 
 makes sense for any positive rational transformation ${\bf K}$ of ${\cal X}$. Conjecture \ref{MCDTTR1}  implies that 
  ${\bf K}$ must coincide with the Kontsevich-Soibelman DT-transformation.    
 Condition (\ref{KDKD}) is an efficient way to find a cluster DT-transformation  of a cluster variety, which we use. 
 Yet, placed out of the context of Conjecture \ref{MCDTTR1}, it looks  enigmatic. 
  
 

\paragraph{DT-transformations and Duality Conjectures revisited.} The following conjecture links rationality of  
DT-transformations to the  existence of  regular canonical bases  on cluster varieties. 

 \bcon \la{2.9.16.2a}
 The map ${\rm DT}_{\cal X}$  is  rational if and only if the formal canonical basis in $\widehat {\cal O}({\cal X})$ 
  lies  in ${\cal O}({\cal X})$.  The same is true for the ${\cal A}$-space.  
  \econ

 \bt \la{Th1.17}
  Suppose that the map ${\rm DT}_{\cal X}$ is a cluster DT-transformation. Then
  
  i)  There is a canonical $\Gamma$-equivariant  basis in the space 
  ${\cal O}({\cal X})$, parametrized by the set of the integral tropical points of ${\cal A}^\vee(\Z^t)$ of the Langlands dual cluster ${\cal A}$-variety.   
  
  ii)   There is a canonical $\Gamma$-equivariant  basis in the upper cluster algebra with principal coefficients  
  ${\cal O}({\cal A}_{\rm prin})$.    \et
  
  \begin{proof} 
  This  follows immediately from    \cite[Theorem 0.10, Proposition 8.25]{GHKK} and the following observation. 
  Given a quiver ${\bf q}$, 
  consider  two cones:
  $$
  \Delta^+_{\bf q} \subset {\cal X}(\Z^t), ~~~~ \Delta^-_{\bf q} \subset {\cal X}(\Z^t).  
  $$ 
  The cone $\Delta^+_{\bf q}$ (respectively $\Delta^-_{\bf q}$) consists of all integral tropical points of ${\cal X}$ which have non-negative (respectively non-positive) 
  coordinates in the cluster coordinate system provided by the quiver ${\bf q}$. Following \cite{FG3}, take the union $\Delta^+$ of all cones $\Delta^+_{\bf q}$, 
  as well as  the union $\Delta^-$ of all cones $\Delta^-_{\bf q}$,
   when ${\bf q}$ runs through all quivers obtained from 
  a given one by  mutations: 
  $$
  \Delta^+ \subset {\cal X}(\Z^t), ~~~~ \Delta^- \subset {\cal X}(\Z^t).
   $$
  Equivalently, the  $\Delta^+$ (respectively $\Delta^-$) consists of all  points of ${\cal X}(\Z^t)$ which have non-negative (respectively non-positive) 
  coordinates in one of the  cluster coordinate systems.  Since 
  the map ${\rm DT}_{\cal X}$ is a cluster DT-transformation, we have
  $$
  \{l_1^+, ..., l_n^+, l_1^-, ..., l_n^-\} \in \Delta^+. 
  $$
  Indeed, $l_i^+ \in \Delta^+_{\bf q}$ by the  definition, and $l_i^- = {\rm DT}_{\cal X}^t(l_i^+)$ by (\ref{GSMPDTCL2}) and Definition \ref{2.14.16.1}. 
  Therefore, since ${\rm DT}_{\cal X}$ is cluster, $l_i^-\in \Delta^+$.   Evidently the convex hull of the points $\{l_i^+, l_j^-\}$ in the linear structure of ${\cal X}(\Z^t)$ given by the cluster coordinate system assigned to the quiver ${\bf q}$ 
  coincides with the ${\cal X}(\Z^t)$. This is exactly the condition in  \cite[Theorem 0.10, Proposition 8.25]{GHKK} needed to get a canonical basis in ${\cal O}({\cal X})$ as well as in 
  ${\cal O}({\cal A}_{\rm prin})$.  
    \end{proof}
  
  We notice that a  cluster transformation ${\bf K}$ is a cluster DT-transformation if and only if
  $$
  {\bf K}(\Delta^+_{\bf q} ) = \Delta^-_{\bf q}.
   $$
           
  Conjecture \ref{2.9.16.2a} tells that if DT-transformations are rational, we should have canonical maps
  \be \la{DCIII1x}
\begin{split}
&{\Bbb I}_{\cal A}: {\cal A}(\Z^t) \lra   {\cal O}({\cal X}^\vee), ~~~~{\Bbb I}_{\cal X}: {\cal X}(\Z^t) \lra  {\cal O}({\cal A}^\vee).
\end{split}
\ee

The functions  $ \{{\Bbb I}_{\cal A}(l)\}$ in ${\cal O}({\cal X}^\vee)$ should be linearly independent, and 
form a canonical 
linear basis in the linear span of the image. Similarly for the functions $ \{{\Bbb I}_{\cal X}(m)\}$ in ${\cal O}({\cal A}^\vee)$. 

So the next question is what are the images. The equivariance under the DT-transformations in Conjecture \ref{MCDTTR1} 
implies that any function $F$ in the linear span of the image of  each of the  two maps (\ref{DCIII1x})  
must  remain regular under   arbitrary powers of the corresponding 
DT-transformations. Conjecture \ref{2.14.16.10} claims that this is the only extra condition on the image.

    \bcon \la{2.14.16.10}
Assume that the DT-transformations of  cluster varieties $({\cal A}, {\cal X})$ are rational. 
Then the linear span of the images of  canonical maps (\ref{DCIII1x})  
 consist of all regular functions which remain regular under  the arbitrary powers of the corresponding 
DT-transformations. 
\econ

Let $\bS$ be a decorated surface with a single puncture and no boundary. Then the condition in Conjecture \ref{2.14.16.10} is  essential for the dual pair $({\cal A}_{\G, \bS}, {\cal X}_{\G^L, \bS})$, see the end of  Section \ref{SecAA}. 


\subsection{Ideal bipartite graphs on surfaces and 3d CY categories \cite{G}} 

To define the {\rm DT}-invariants Kontsevich-Soibelman 
start with a 3d CY category. However for a generic cluster variety 
there is no natural 3d CY category assigned to it. Indeed, one uses a 
quiver with generic potential $({\bf q}, W)$ as an input, 
and the 3d CY category does depend on $W$. 

\vskip 3mm
It turns out that for the moduli space 
${\cal X}_{\G, \bS}$ the situation is much better. 
Among the quivers describing its cluster structure there is a particularly nice subclass provided by {\it rank $m$ 
ideal bipartite graphs}, introduced in 
\cite{G}. A {\it bipartite graph} is a graph with vertices of two kinds,  so that  
each edge connects vertices of different kinds.  
Let us recall crucial examples. 

Let $T$ be an ideal triangulation of $\bS$, i.e. a triangulation of $\bS$ with the vertices at the marked points. Given a triangle of $T$, we subdivide it into $m^2$ small triangles 
by drawing three families of $m$ equidistant lines, parallel to 
 the sides of the triangle, 
as shown on the left  of Figure \ref{gra10i}. 

For every triangle  of $T$ there are two kinds of small triangles: 
the ``up'' and ``down'' triangles. We put a $\bullet$-vertex
 into the center of each of the 
``down'' triangles. 
Let us color  in red all ``up'' triangles.  
Consider the obtained red domains -- some of them are unions of red triangles, and 
put a $\circ$-vertex into the center of each of them. 
A $\circ$-vertex and a $\bullet$-vertex 
 are {\it neighbors} if the corresponding domains share an edge. 
Connecting the neighbors,  
 we get a bipartite surface graph $\Gamma_{{\rm A}_{m}}(T)$, see Figure \ref{gra10i}.

\begin{figure}[ht]
\centerline{\epsfbox{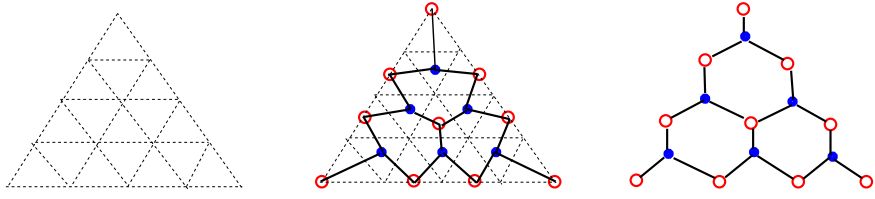}}
\caption{A bipartite graph associated with a 4-triangulation of a triangle.}
\label{gra10i}
\end{figure}

The dual graph to a bipartite surface graph $\Gamma$ is a quiver ${\bf q}_\Gamma$. 
Its vertices are the faces of $\Gamma$, and its edges are 
oriented so that the $\bullet$-vertex is on the left, see Figure \ref{bc0}. 
The quiver ${\bf q}_\Gamma$  was introduced in \cite{FG1}, where it was shown 
that it gives rise to a cluster coordinate system on  the moduli space ${\cal X}_{\G, \bS}$. 
We review these coordinate systems in Section 4.  

The quiver ${\bf q}_\Gamma$  is equipped with a \underline{canonical potential}. 
Namely, for each vertex of the graph $\Gamma$ there is a unique cycle on the quiver ${\bf q}$ 
going around the vertex. We sum all these cycles, 
with the $+$ sign for the $\circ$-vertices, and $-$ sign for the $\bullet$ vetrices: 
$$
W_{\bf q_\Gamma} := \sum_{\mbox{$\circ$-vertices of $\Gamma$}}\mbox{Cycles in ${\bf q}_\Gamma$ around $\circ$-vertices} - 
\sum_{\mbox{$\bullet$-vertices of $\Gamma$}}\mbox{Cycles in ${\bf q}_\Gamma$ around $\bullet$-vertices}. 
$$
\begin{figure}[ht]
\centerline{\epsfbox{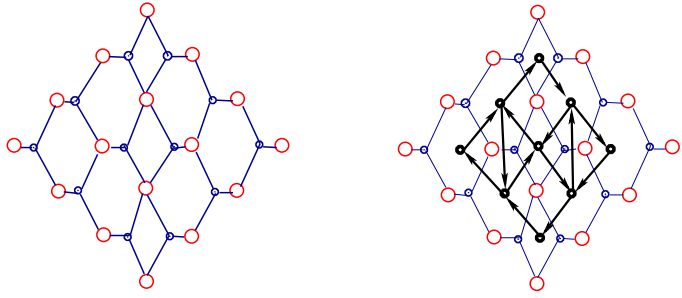}}
\caption{$\G={\rm PGL}_4$. A
 bipartite graph associated with a quadrilateral, and the related quiver.}
\label{bc0}
\end{figure}

It was proved in \cite{G} that any two bipartite graphs assigned to ideal triangulations of 
$\bS$  are related by special moves of bipartite graphs, called {\it two by two moves}, see Figure \ref{gc2a}. 
\begin{figure}[ht]
\centerline{\epsfbox{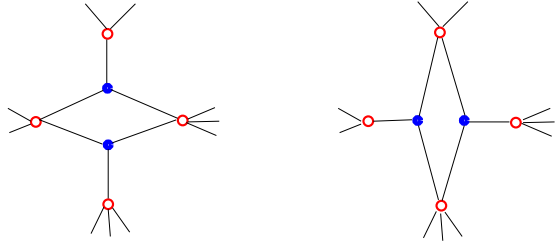}}
\caption{A two by two move. 
Flipping the colors of vertices delivers another two by two move. }
\label{gc2a}
\end{figure} 
The corresponding moves of the associated quivers are the {mutations}. The two by two moves 
keep us in the class of bipartite graphs on $\bS$, introduced in {\it loc. cit.} and called 
{\it ideal bipartite graphs}. In particular, 
this class  is $\Gamma_\bS$-invariant. 
 The crucial fact is that a two by two move $\Gamma \to \Gamma'$ 
 transforms the potential $W_{{\bf q}_\Gamma}$ 
to one $W_{\bf q_{\Gamma'}}$.

Therefore we arrive at  a  3d CY category ${\cal C}_{m, \bS}$
 with an array of cluster collections of generating objects. For $\G={\rm PGL}_2$ it was studied by Labardini-Fragoso 
in \cite{LF08}. 
So we conclude that 
$$
\mbox{There is a combinatorially defined 3d {\rm CY} category ${\cal C}_{m, \bS}$ 
``categorifying'' moduli spaces ${\cal X}_{\G, \bS}$.} 
$$

\paragraph{A conjectural  realization of the 3d {\rm CY} category ${\cal C}_{m, \bS}$ as a Fukaya category \cite{G}.}
Let $\Omega_P(\Sigma)$ be the sheaf of meromorphic differentials on a Riemann surface $\Sigma$ 
with the set of punctures $P$, 
with poles of order $\leq 1$ at $P$.  
A point ${\bf t}$ of the Hitchin's base of  $(\Sigma, P)$ is given by a data
$$
{\bf t} = (\Sigma, P; t_2, t_3, \ldots t_m), \quad t_k \in \Omega_P(\Sigma)^{\otimes k}. 
$$
The universal Hitchin base is the family of Hitchin bases over 
the moduli space ${\cal M}_{g,n}$. 

The spectral curve assigned to ${\bf t}$ is a curve in $T^*\Sigma$ given by  the  following equation:
\be \la{eqsc}
\Sigma_{\bf t}:= 
\{\lambda \in T^*\Sigma~|~ \lambda^m + t_2\lambda^{m-2} + \ldots + t_{m-1}\lambda +t_m=0\} \subset T^*\Sigma. 
\ee
The projection $T^*\Sigma\to \Sigma$ provides the spectral cover 
$\pi_{\bf t}:\Sigma_{\bf t} \to \Sigma$. It is an $m:1$ ramified cover. 

\vskip 3mm
 
The pair 
$(m, \bS)$ gives rise to a family 
of  open CY threefolds ${\cal Y}_{m, \bS}$ over the universal Hitchin base 
\cite{DDP}, \cite{KS3}.  For example, in the ${\rm A}_1$ case it is given by 
$$
\{({\bf t}, x\in \Sigma, \alpha_1, \alpha_2, \alpha_3)~~~|~ ~~ \alpha_1^2 + \alpha_2^2 + \alpha_3^2  = {\bf t}\}, ~~~~
{\bf t} \in \Omega_P(\Sigma)^{\otimes 2}; ~~\alpha_i(x) \in T^*_x(\Sigma). 
$$
The intermediate Jacobians of the fibers ${\cal Y}_{\bf t}$ provide Hitchin's integrable system. 

The category ${\cal C}_{m, \bS}$ should be equivalent to a Fukaya category of the open CY threefold 
given by the generic fiber ${\cal Y}_{\bf t}$ of the family ${\cal Y}_{m, \bS}$:

\bcon[{\cite{G}}] \la{FukC} 
For a generic point ${\bf t}$ of the Hitching base, there is a fully faithful functor 
$$
\varphi: {\cal C}_{m, \bS} \lra {\cal F}({\cal Y}_{\bf t}).
$$ 
It transforms  cluster collections in ${\cal C}_{m, \bS}$ to the ones provided by special Lagrangian spheres. 
\econ
 For $m=2$ this is known thanks to  the works of  
Bridgeland and Smith \cite{BrS}, \cite{S}. 

Therefore formula (\ref{FDT}),  conjecturally, describes the {\rm DT}-invariants of this Fukaya category.


\subsection{Physics perspective}

\paragraph{Unification diagram.}Assume that the group $\G$ is simply laced. 
In a series of works \cite{GMN1}-\cite{GMN5}, 
 Gaiotto, Moore and Neitzke  studied  4d ${\cal N}=2$ SUYM 
theories of class ${\cal S}$ related to a Riemann surface with punctures $\Sigma$. 
The ${\cal S}$ alludes to ``six dimensional'': the theories are ``defined'' as compactifications of the 
hypothetical $(2,0)$ 
theories ${\cal X}_\G$, related to ADE Dynkin diagrams,
 on the Riemann surface $\Sigma$, with defects at the punctures.

The origins of the 4d theory can be perceived as follows. 
Let $\Gamma_\G$ be a finite subgroup of ${\rm SU}(2)$ 
corresponding to $\G$ by the McKay correspondence. 
The theory ${\cal X}_\G$ itself is ``defined'' 
 as a compactification of the ten dimensional type {\rm IIB} superstring theory 
on the Klein singularity  $\C^2/\Gamma_\G$. 
One should have  a ``commutative diagram'':
\begin{displaymath}
    \xymatrix{
      \mbox{10d type {\rm IIB} superstring theory} \ar[dddd]_{\big{[Y_{\G, \Sigma, {\bf t}}}\big]}\ar[ddr]^{\big{[\C^2/\Gamma_{\G}}\big]} &      \\
       &\\
                                 & \mbox{6d theory}~{\cal X}_{\G} \ar[ddl]^{\big{[\Sigma}\big]} \\
                                 & \\
                     \mbox{$4d$ ${\cal N}=2$ {\rm SUYM} ~class ${\cal S}$ theory}~ {\cal T}_{\G, \Sigma} &            
                                 }
\end{displaymath}
The vertical arrow is the compactification of the $10d$ type {\rm IIB} superstring theory 
on the six-dimensional Riemannian  
manifold given by the complex open CY threefold $Y_{\G, \Sigma, {\bf t}}$. 
The point ${\bf t}$ belongs to the Hitchin base. The latter  
is the Coulomb branch of the space of vacua 
of the  4d theory ${\cal T}_{\G, \Sigma}$.

\paragraph{Gaiotto-Moore-Neitzke count of BPS states.} 
The Hilbert space ${\cal H}$ of the $4d$ ${\cal N}=2$ {\rm SUYM} class ${\cal S}$ theory 
is a huge representation of the ${\cal N}=2$ Poincare super Lie algebra 
${\cal P} = {\cal P}_0 \oplus {\cal P}_1$. Its even part ${\cal P}_0$ is the Poincare Lie algebra 
of the flat Minkowski space $\R^{3,1}$ plus a one dimensional center with the generator ${\cal Z}$. 
The odd part is 8-dimensional. As a representation of the Lorenz group, it is a sum of two copies of 
the spinor representation $S_+\oplus S_-$. 

The Hilbert space ${\cal H}$ is the symmetric algebra of a 1-particle Hilbert space ${\cal H}_{1}$. 
The $n$-th symmetric power of ${\cal H}_{1}$ is the ``n-particle part''. The 
${\cal H}_{1}$  should have a discrete spectrum, i.e. be a sum rather then integral of 
unitary representations of the Poincare super Lie algebra ${\cal P}$. 

The Hilbert space ${\cal H}$ has the following structures,  inherited on
 the subspace ${\cal H}_{1}$:

1. It depends on a point ${\bf t}$ of the Hitchin base. The Hitchin base  
 is the Coulomb branch of the moduli space of vacua in the theory. 

2. It is graded by a {\it charge lattice $\Gamma$}. In particular, there is a decomposition  
$$
{\cal H}_{1} = \oplus_{\gamma \in \Gamma}{\cal H}_{1, \gamma}.
$$
The lattice $\Gamma$ is equipped with an integral valued skew symmetric bilinear form $\langle \ast, \ast \rangle$. 

Irreducible unitary representations of the super Lie algebra ${\cal P}$ 
are parametrized by three parameters: the mass $M \in [0, \infty)$, the spin 
$j\in \{0, \frac{1}{2}, 1, \frac{3}{2}, ...\}$, and 
the central charge $Z\in \C$. 

The pairs $(M, j)$ parametrize ``positive'' 
unitary representations of the Lorenz group. 

The crucial fact is the inequality 
$M \geq |Z|$. We are interested in the 
BPS part ${\cal H}^{\rm BPS}_{1}$ of the space, defined by the $M = |Z|$ condition. 
Let $n_j$ be the multiplicity of the irreducible representations of spin $j$ in ${\cal H}^{\rm BPS}_{1}$. 
The integers $\Omega^{\rm GMN}_{\bf t}(\gamma)$, ``counting the BPS states'' of the central charge $\gamma$, 
are not the integers $n_j$ but rather the `` second helicity 
supertrace'': 
$$
\Omega^{\rm GMN}_{\bf t}(\gamma):= \sum_{j}(-1)^{2j}(2j+1)n_j. 
$$

Let us now discuss the Gaiotto-Moore-Neitzke  approach to calculate these numbers. 
A point ${\bf t}$ of the Hitchin base determines a spectral curve 
$\Sigma_{\bf t} \subset T^*\Sigma$, and a spectral cover
 $\pi_{\bf t}: \Sigma_{\bf t} \to \Sigma$. It determines a lattice $\Gamma_{\bf t}$. 
When $\Sigma$ is compact it is given by 
\be \la{8.5.15.1}
\Gamma_{\bf t}= {\rm Ker}\Bigl(H_1(\Sigma_{\bf t}, \Z) \stackrel{\pi_{\bf t}}{\lra} H_1(\Sigma, \Z)\Bigr).
\ee
Then   $\gamma\in \Gamma_{\bf t}$. 
 Integrating the canonical 1-form $\alpha$ on 
$T^*\Sigma$ over the homology classes from (\ref{8.5.15.1}) 
we get a linear map, called the central charge map:
$$
Z_{\bf t}: \Gamma_{\bf t} \lra \C, ~~~~ \gamma \lms \int_\gamma \alpha. 
$$

Gaiotto, Moore and Neitzke introduced a spectral network related to  a generic  ${\bf t}$. They use it to 
develop an  algorithm to calculate the numbers $\Omega_{\bf t}(\gamma)$. 
The algorithm  
has some mathematical issues for higher rank groups.\footnote{One of them is a possibility of having 
an infinite number of ``two side roads'' in a spectral network, making the algorithm problematic.}  
{Let us assume that  they are resolved. }

The Gaiotto-Moore-Neitzke {\it spectral generator} is a  transformtion of 
the Hitchin moduli space. It tells the cumulative result of the wall crossings 
which one encounters rotating a Higgs field $\Phi$ projecting to a point ${\bf t}$ by $e^{i\theta}\Phi$, 
with $0 \leq \theta \leq  \infty$.
It turned out\footnote{We thank Davide Gaiotto who 
pointed this to us at the 6d conference at Banff, and to Andy Neitzke 
for providing some details.} that our map ${\rm C}_{{\rm PGL_m}, \bS}$ acting on the moduli space 
${\cal X}_{{\rm PGL_m}, \bS}$ 
coincides with the result of calculation of  the spectral generator.

So formula (\ref{FDT}) implies that 
\be \la{5.7.15.1a}
\mbox{The {\rm DT}-transformation ${\rm DT}_{m, \bS}$  ~=~    The Gaiotto-Moore-Neitzke spectral generator}. 
\ee

Let us assume that a point ${\bf t}$ of the universal Hitchin base 
determines a quiver ${\bf q}$, and that the lattice $\Lambda_{\bf q}$ 
of this quiver is identified with the lattice $\Gamma_{\bf t}$. 
This is known for $\G={\rm SL}_2$ \cite{GMN2}. 
Examples 
were worked out in \cite{GMN5} for $\G=SL_m, m\leq 9$. They produce quivers of the type 
discussed above. Then  the central charge map 
$Z_{\bf t}$ translates into a central charge map  
$$
Z_{\bf q}: \Lambda_{\bf q} \lra \C. 
$$


Let $\{e_i\}$ be the basis of $\Lambda$ provided by the quiver ${\bf q}$. 
Assuming that $Z_{\bf q}(e_i) \in {\cal H}$, we arrive at a stability condition ${\bf s}$ determined by ${\bf t}$.

\vskip 3mm
So if all  mentioned above assumptions were satisfied, the formula \eqref{5.7.15.1a} would imply that 
\be \la{5.7.15.1b}
\mbox{KS numerical
{\rm DT}-invariants $\Omega^{\rm KS}_{\bf s}(\gamma)$ =  GMN invariants $\Omega^{\rm GMN}_{\bf t}(\gamma)$}.
\ee

Let us stress that the origins and definitions of the two sides of (\ref{5.7.15.1b}) are entirely different. 

The numbers $\Omega^{\rm KS}_{\bf s}(\gamma)$ came from  
\underline{3d CY categories} related to quivers ${\bf q}$.  

The  numbers 
$\Omega^{\rm GMN}_{\bf t}(\gamma)$ came from a \underline{quantum field theory}, and 
calculated using the geometry of a Riemann surface $\Sigma$. 

\vskip 3mm

Gaiotto-Moore-Neitzke algorithm for counting the numbers $\Omega_{{\bf t}}(\gamma)$ can be interpreted as 
a count of certain type of branes in 10d type IIB superstring theory on 
\be \la{threeused}
{\cal Y}_{{\bf t}}\times \R^{3,1}.
\ee
 These are the $D3$-branes supported on 
$L \times l$ where $L\subset {\cal Y}_{{\bf t}}$ is a special Lagrangian sphere,  and 
$l \subset \R^{3,1}$ is the world line of a BPS particle. 
The mass $M$ of the particle is the Riemannian volume of $L$. Its central charge $Z$ 
is the integral $Z(L) = \int_L\Omega$  over $L$ of the holomorphic 3-form $\Omega$ on ${\cal Y}_{{\bf t}}$. So the BPS condition $M=Z$ just means that we count {\it special} Lagrangian spheres. 

\vskip 3mm
Conjecture \ref{FukC} connects  the two approaches. Namely, among the ``combinatorial''  
3d CY categories assigned to a quiver ${\bf q}$ there is a distinguished one,  
${\cal C}_{m, \bS}$,  provided by the canonical potentials on ideal bipartite graphs on $\bS$. 
Conjecture \ref{FukC} predicts that the category ${\cal C}_{m, \bS}$
 has a geometric realisation as a Fukaya category of an open CY threefold ${\cal Y}_{{\bf t}}$. 
And this is threefold needed to get 
the $4d$ ${\cal N}=2$ {\rm SUYM} ~theory  from the {\rm IIB} superstring theory in (\ref{threeused}). 

So now the combinatorial 3d CY categories and the 4d theories are linked directly to each other. 
The common structure visible in both is the cluster structure.

\vskip 3mm
It is interesting to note that 
formula (\ref{FDT}) involves both the 3d CY category used to define 
the ${\rm DT}$-transformation and the moduli space ${\cal X}_{\G, \bS}$ on which the 
element ${\rm C}_{\G, \bS}$ acts naturally. 

Notice  that although the moduli space ${\cal X}_{\G, \bS}$ is closely related to the moduli space of 
$\G$-local systems on $\bS$, the element ${\rm C}_{\G, \bS}$ does not act 
on the  latter.

In general there is  no 
canonical moduli space linked to 
the {\rm DT}-transformation.  
We have only a cluster variety:  
it describes a lot of features of the  space,   yet it is different.


\subsection{Other ramifications and applications}

 \paragraph{Duality conjectures for moduli spaces of $\G$-local systems  and Weyl group actions.} 
We assume for simplicity that $\bS$ has no boundaries, only punctures.\footnote{When $\bS$ has boundaries / special points, the correct analog of the ${\cal X}$-moduli space 
is the moduli space ${\cal P}_{\G^L, \bS}$ from \cite{GS}, which we do not discuss here. In particular, 
${\rm dim}{\cal A}_{\G, \bS} = {\rm dim}{\cal P}_{\G^L, \bS}$. However 
${\rm dim}{\cal A}_{\G, \bS} > {\rm dim}{\cal X}_{\G^L, \bS}$ if $\bS$ has holes.
{In the other direction, each boundary interval ${\rm I}$ on a boundary component of $\bS$ gives rise a natural map $\alpha_{\rm I}$ from ${\cal A}_{\G, \bS}$ to the Cartan subgroup ${\rm H}$. Let ${\cal A}_{\G, \bS}'$ be the subspace of ${\cal A}_{\G, \bS}$ consisting of points $a$ such that $\alpha_{\rm I}(a)= {\rm Id}$ for all boundary intervals {\rm I}. Then we have ${\rm dim}{\cal A}_{\G, \bS}' = {\rm dim}{\cal X}_{\G^L, \bS}$. }}
 Duality Conjectures   suggest a  mirror duality between the Langlands  dual  moduli spaces 
${\cal A}_{\G, \bS}$ and ${\cal X}_{\G^L, \bS}$ \cite[Section 12]{FG1}, \cite{GS}. 
In particular, each integral tropical point $l$ of one space 
corresponds to a regular function ${\Bbb I}(l)$ on the other space. 
The functions ${\Bbb I}(l)$ should be linearly independent. 
We observe  that if $\bS$ has punctures, then $\{{\Bbb I}(l)\}$ can not span the space of  regular functions on the space ${\cal A}_{\G, \bS}$. 

Indeed,  each puncture $p$ on $\bS$ gives rise to  a regular 
$\Gamma_\bS$-invariant function 
${\cal W}_p$ on ${\cal A}_{\G, \bS}$, the potential \cite{GS}, see also Section \ref{sec1}. 
However, one can argue that the only $\Gamma_\bS$-invariant 
finite subset in the set ${\cal X}_{\G^L, \bS}(\Z^t)$ is the  zero point. 
This is obvious if $\G$ is of type ${\rm A}_1$ due to the interpretation 
of the integral tropical points as laminations on $\bS$: the $\Gamma_\bS$-orbit of any
 non-empty integral lamination is infinite. 
The zero point $0 \in {\cal X}_{\G^L, \bS}(\Z^t)$ 
maps under the duality to the constant function on ${\cal A}_{\G, \bS}$. 
So the ${\cal W}_p$ can not be  a finite linear compbination of the functions 
${\Bbb I}(l)$, where $l \in {\cal X}_{\G^L, \bS}(\Z^t)$. 

We conjecture that the space the functions ${\Bbb I}(l)$ 
span consists of all regular functions  remaining regular 
under the $W^n$-action. The latter condition is forced by the following conjecture. 

\bcon \la{WGAC}
There is a $\Gamma_\bS\times W^n$-equivariant duality between the spaces 
${\cal A}_{\G, \bS}$ and ${\cal X}_{\G^L, \bS}$. 
\econ 
\noindent

If the $W^n$-action  is not cluster, the $W^n$-transformations 
play a role of cluster transformations.

\paragraph{Periodicity of ${\rm DT}$-transformations.}
The periodicity of ${\rm DT}$-transformations is closely related to  the periodicity conjecture of Zamolodchikov \cite{Z}. Keller solved the cases of  square products of Dynkin quivers \cite{K}. As a consequence of Theorem \ref{2.27.15.2}, we immediately have
\bt  If ${\bS}$ has only punctures, then $({\rm DT}_{\G, \bS})^2={\rm Id}$.
 
 If ${\bS}$ is a disk/punctured disc with $k$ 
special points on its boundary, then $({\rm DT}_{\G, \bS})^{{\rm lcm}(2, k)}={\rm Id}$.
\et

\paragraph{Compatibility  of the ${\rm C}_{{\rm G}, \bS}$-transformations with covers.} 
The transformation  ${\rm C}_{{\rm G}, \bS}$  has many nice geometric properties. 

Let $\pi: \widetilde{\bS} \lra \bS$ be a finite cover of decorated surfaces. By pull back, it induces a natural positive embedding $\pi^*: {\cal X}_{\G, \bS}\lra {\cal X}_{\G, \widetilde{\bS}}$. The following  diagram is commutative:
\begin{displaymath}
    \xymatrix{
        {\cal X}_{\G, \widetilde\bS} \ar[r]^{{\rm C}_{\G, \widetilde\bS}}  &   {\cal X}_{\G, \widetilde\bS}  \\
        {\cal X}_{\G, \bS} \ar[r]^{{\rm C}_{\G, {\bS}}}   \ar[u]^{\pi^*}    & {\cal X}_{\G, \bS} \ar[u]_{\pi^*} }
\end{displaymath}
It would be interesting to establish a categorical interpretation of this diagram in the ${\rm DT}$-theory.

\paragraph{Compatibility  of the maps ${\rm C}_{{\rm G}, \bS}$ with the cluster ensemble structure.} The transformation ${\rm C}_{{\rm PGL_m},\bS}$ is birational. We define
 a similar map ${\rm C}_{{\rm SL_m},\bS}$ on the space ${\cal A}_{{\rm SL_m}, \bS}$ using quite
 different construction. Remarkably, the two maps can be presented by the same sequence of cluster (per)-mutations. 
So viewing the pair $({\cal A}_{{\rm SL_m}, \bS}, {\cal X}_{{\rm PGL_m}, \bS})$ as a cluster ensemble, we have the
 following commutative diagram:
 \begin{displaymath}
    \xymatrix{
        {\cal A}_{{\rm SL}_m, \bS} \ar[r]^{{\rm C}_{{\rm SL_m}, \bS}}  \ar[d]_{p} &   {\cal A}_{{\rm SL_m,} \bS}   \ar[d]_{p}\\
        {\cal X}_{{\rm PGL}_m, \bS} \ar[r]^{{\rm C}_{{\rm PGL}_m, {\bS}}}      & {\cal X}_{{\rm PGL}_m, \bS} }
\end{displaymath}
Therefore the  birational map ${\rm C}_{{\rm PGL}_m,\bS}$ is regular on the image of ${\cal A}_{{\rm SL}_m, \bS}$.

\paragraph{Configurations of points in ${\Bbb C}{\Bbb P}^m$: the DT-transformation $=$ the parity conjugation.} Denote by ${\rm Conf}_n({\Bbb C}{\Bbb P}^m)$ 
the moduli space of configurations, that is ${\rm PGL_{m+1}}$-orbits,  of points $(z_1, ..., z_m)$ in ${\Bbb C}{\Bbb P}^m$. It has a cluster Poisson variety 
structure invariant under the cyclic shift ${\rm C}: (z_1, ..., z_m) \lra (z_2, ..., z_m, z_1)$.  Let $h_k:= (z_{k-m}, ..., z_{k-1})$ be a hyperplane in ${\Bbb C}{\Bbb P}^m$  spanned by the points 
$z_{k-m}, ..., z_{k-1}$, where the indices are modulo $n$. Then there is a birational map
$$
{\rm P}: {\rm Conf}_n({\Bbb C}{\Bbb P}^m) \lra {\rm Conf}_n({\Bbb C}{\Bbb P}^m),  ~~~~(z_1, ..., z_n) \lms (h_1, ..., h_n).
$$ 
It was argued in \cite{GGSVV} that  ${\rm P}$ is a cluster transformation. 
Just recently  D. Weng \cite{We} proved
\bt \la{PCC} The map ${\rm P}$ is the cluster DT-transformation for the cluster  variety ${\rm Conf}_n({\Bbb C}{\Bbb P}^m)$. 
\et
When  $m=3$, the map ${\rm C}^2{\rm P}$ is  the  {\it parity conjugation},  which 
plays an essential role in the theory scattering amplitudes in the ${\cal N}=4$ super Yang-Mills. Namely, let $\overline {\rm P}$  be the composition of 
${\rm P}$ with the complex conjugation. Scattering amplitudes are functions / forms on  ${\rm Conf}_n({\Bbb C}{\Bbb P}^m)$ which are 
invariant under the subgroup generated by the cyclic shift and the  map  $\overline {\rm P}$. 


\paragraph{Double Bruhat cells: the   DT-transformation $\stackrel{?}{=}$ the twist map.}
Let $\G$ be a split semi simple group. 
Let us fix a pinning $({\rm B}, {\rm B}^-, x_i, y_i; I)$ of $\G$, where $\B$ is a Borel subgroup of $\G$, ${\rm B}^-$ is an opposite Borel subgroup, so that 
${\rm H}:= {\rm B}^- \cap {\rm B}$ is a Cartan subgroup, and the pair $x_i$, $y_i$ give rise to a homomorphism $\gamma_i : {\rm SL}_2 \to \G$ for each $i\in I$ such that
\[
\gamma_i   \begin{pmatrix} 
      1 & a \\
      0 & 1 \\
   \end{pmatrix}= x_i(a),  
 \hskip 7mm \gamma_i  \begin{pmatrix} 
      1 & 0 \\
      a & 1 \\
   \end{pmatrix}= y_i(a), \hskip 7mm 
   \gamma_i \begin{pmatrix} 
      a & 0 \\
      0 & a^{-1} \\
   \end{pmatrix}=\alpha_i^\vee(a).
\]

Each pair of Weyl group elements  $u, v \in W$  gives rise to a {\it double Bruhat cell}
\[
{\rm G}^{u, v} := \B u \B ~\cap ~\B^{-} v \B^{-}. 
\] 

There are several maps involving the double Bruhat cells.

\begin{itemize}  
\item 

The involution $i$ of the $\G$:
\[
i :  \G \lra \G, \hskip 7mm h\lms h^{-1}, \hskip 3mm x_i(a)\lms y_i(a), \hskip 3mm y_i(a)\lms x_i(a).
\]
The involution $i$ exchanged $\B$ and $\B^-$. Therefore it induces an involution
\be
i:  \G^{u, v}\lra \G^{v, u}.
\ee
\item 
 The transposition $\G\lra \G, ~ g \lms g^T$, defined as an anti-automorphism of $\G$ such that
\[
h^T=h, \hskip 5mm x_i(a)^T = y_i(a), \hskip 5mm y_i(a)^T =x_i(a). 
\]
It induces a similar involution
\be
 \G^{u,v} \lra \G^{v^{-1}, u^{-1}}, \hskip 5mm g \lms g^T.
\ee
\item

For each $w\in W$, we define two  representatives $\overline{w}, \overline{\overline{w}} \in \G$. If $s_i$ is a simple refection, then 
\[
\overline{s_i}=\gamma_i\begin{pmatrix} 
      0 & -1 \\
      1 & 0 \\
   \end{pmatrix} , \hskip 7mm \overline{\overline{s_i}}={\gamma_i \begin{pmatrix} 
      0 & 1 \\
      -1 & 0 \\
   \end{pmatrix}}. 
\]
The $\{ \overline{s_i} \}$ and $\{\overline{\overline{s_i}}\}$ satisfy braid relations. So given a reduced decomposition $w=s_{i_1}\ldots s_{i_k}$, 
\[
\overline{w}=\overline{s_{i_1}}\ldots \overline{s_{i_k}}, \hskip 7mm \overline{\overline{w}}=\overline{\overline{s_{i_1}}}\ldots\overline{ \overline{s_{i_k}}}.
\] 

Let ${\rm U}:=[\B, \B]$, and  ${\rm U}^-:=[\B^{-}, \B^{-}]$. Using the Gaussian decomposition ${\rm G}_0 = {\rm U}{\rm H}{\rm U}^-$, any $g\in \G_0$ can be written as 
$g = [g]_+ [g]_0 [g]_-$. 
Consider the twist map
\be
\eta: ~~ \G^{u,v}\lra \G^{u,v}, ~~~~~\hskip 7mm \eta(g) =  \Big([\overline{v^{-1}}g]_+^{-1} \overline{v^{-1}}g   \overline{\overline{u^{-1}}} [g \overline{\overline{u^{-1}}}]^{-1}_- \Big)^T.
\ee
\end{itemize}

The map $\eta$ is a slight modification of the Fomin-Zelevinsky twist map \cite[Section 1.5]{FZ98}, which 
plays crucial role in the factorization formulas \cite{FZ98}. 
It is well-defined for all $g\in \G^{u,v}$.

If $\G$ is simply connected, then  the algebra ${\cal O}(\G^{u,v})$ has a cluster algebra structure \cite{BFZ}.
 The generalized minors are cluster variables. So $\G^{u,v}$ is a cluster ${\cal A}$-variety.
 
If $\G$ has trivial center, then $\G^{u,v}$ has a cluster ${\cal X}$-variety structure \cite{FG5}. Taking quotients by the action of the Cartan group  ${\rm H}$ on both sides, we get another ${\cal X}$-variety
\be
{\cal X}_{u,v}:={\rm H}\backslash {\rm G}^{u, v} / {\rm H}.
\ee
Alternatively, it is obtained by   deleting all the frozen variables of the original ${\cal X}$-variety $\G^{u,v}$.

It is easy to check that
\[
\eta(h_1 \cdot g \cdot h_2) =u (h_2) \cdot  \eta(g)\cdot v^{-1}(h_1)\hskip 7mm \forall h_1, h_2\in {\rm H}, \hskip 3mm \forall g \in {\rm G}^{u,v}.
\]
Thus the twist map can be reduced to an isomorphism of ${\cal X}_{u,v}$. We still denote it by $\eta$.

It is easy to see that the map $i$ on the space ${\cal X}_{u,v}$ coincides with our involution $i_{\cal X}$. 

\bcon The twist map $\eta$ is the ${\rm DT}$-transformation of ${\cal X}_{u,v}$.
\econ 

\paragraph{Quantized {\rm DT}-transformations and $\Gamma$-invariant 
bilinear forms.} 
Quantizing a cluster Poisson variety  ${\cal X}$, 
we get a 
Hilbert space
${\cal H}_{\cal X}$ with a scalar product 
$\langle \ast, \ast \rangle_{\cal X}$. 
Each  cluster transformation ${\rm C}$ of  ${\cal X}$ is quantized into a  
unitary operator $\widehat {\rm C}$ in   ${\cal H}_{\cal X}$. 
In particular we get a unitary projective representation of the  cluster modular group $\Gamma$ 
in the Hilbert space ${\cal H}_{\cal X}$  \cite{FG4}.

Suppose that the {\rm DT}-transformation ${\rm DT}_{\cal X}$ of  
${\cal X}$ is  cluster.  
Recall   the map   
${\rm D}_{\cal X}= i_{\cal X} \circ {\rm DT}_{\cal X}$,   see (\ref{Tildeiso1}). It is involutive: ${\rm D}_{\cal X^\circ}\circ {\rm D}_{\cal X} = {\rm Id}_{\cal X}$. 
So we get  a  $\Gamma$-invariant 
``symmetric'' bilinear form 
\be \la{12.5.15.1ac}
\begin{split}
&{\rm B}_{\cal X}: {\cal H}_{\cal X} \times {\cal H}_{\cal X^\circ}  \lra \C, ~~~~
 {\rm B}_{\cal X}(v, w):= \langle \widehat {\rm D}_{\cal X} u, w\rangle_{\cal X^\circ}, \\
&{\rm B}_{\cal X}(v, w) = \mu\cdot \overline {{\rm B}_{\cal X}(w, v)}, ~~|\mu|=1. 
\end{split}
\ee

Applying this to the moduli space ${\cal X}_{{\G}, \bS}$, $\G = {{\rm PGL}_m}$, 
we get a 
Hilbert space
${\cal H}_{\G, \bS}$ \cite{FG1, FG4}. It is conjecturally the space of conformal blocks for 
the higher Liouville theory related to $\G$. 
Let $\bS^\circ$ be the surface $\bS$ with the opposite orientation. 
Then  pairing (\ref{12.5.15.1ac}) is a  $\Gamma_\bS$-invariant 
 form 
\be \la{12.5.15.1a}
{\rm B}_{\G, \bS}: {\cal H}_{\G, \bS} \times {\cal H}_{\G, \bS^\circ}  \lra \C. 
\ee

When $\bS$ has punctures {only}, the cluster transformation ${\rm C}_{\G, \bS}$ is already 
 involutive: ${\rm C}^2_{\G, \bS} = {\rm Id}$. So we get a new $\Gamma_\bS$-invariant Hermitian symmetric form 
in the Hilbert space $({\cal H}_{\G, \bS}, \langle \ast, \ast\rangle_{\G, \bS})$:
\be \la{12.5.15.1}
{\rm B}'_{\G, \bS}: {\cal H}_{\G, \bS} \times \overline {{\cal H}_{\G, \bS}}
  \lra \C, ~~~~{\rm B}'_{\G, \bS}(u, v):= \langle \widehat {\rm C}_{\G, \bS} ~u, v\rangle_{\G, \bS}.
\ee
{There might be 
 a potential similarity between  this form 
and the Drinfeld-Wang   invariant bilinear form  \cite{DW} in a space of automorphic forms 
 on $\G({\Bbb A}_F)/\G(F)$ where $\G = {\rm SL}_2$,  $F$ is a global field.} 

\paragraph{Acknowledgments.} 
This work was 
supported by the  NSF grant DMS-1301776. 

A.G. 
is grateful to IHES 
for the hospitality and support during the Summer of 2015.

We are grateful to  
 Davide Gaiotto,   Maxim Kontsevich, and Andy Neitzke 
for many illuminating conversations. 


\section{Quantum cluster varieties}\la{SecQCV}

In Section \ref{SecQCV}, borrowed mostly from \cite{FG2},  
we present a careful definition of quantum cluster transformations and 
quantum cluster varieties. 
Quantum cluster varieties have a quasiclassical limit,  called 
{\it cluster Poisson variety}, introduced in {\it loc. cit.}  under the name 
cluster ${\cal X}$-variety. 

\vskip 3mm
A cluster Poisson variety is obtained by gluing a collection of split algebraic tori, 
each equipped with a Poisson structure, by birational transformations 
called {\it cluster Poisson transformations}. It is not quite a variety: it is rather a 
prescheme,  possibly non-separated. Yet it has a lot of important geometric 
features, not available  for general 
varieties: a well defined set of points with values in any semifield, 
canonical cluster variety divisors at infinity, etc.  

The algebra of regular functions on each of the cluster Poisson tori has a canonical $q$-deformation, 
to a quantum torus algebra. Cluster Poisson transformations are quasiclassical limits 
of certain isomorphisms of the non-commutative fields of fractions of quantum torus algebras, called 
{\it quantum cluster transformations}. Quantum cluster transformations are the primary objects of study. 
They are given by isomorphisms of quantum torus algebras, 
followed by the birational transformations provided by the 
composition of conjugations by the quantum dilogarithm of certain cluster coordinates. 
Quantum cluster variety is just a collection of these 
non-commutative fields related by quantum cluster transformations.

For the convenience of the reader we present all definitions in the 
``simply laced'' case.


\paragraph{Quiver mutations.} 

\bd \la{QDEF} A quiver is  a data 
$$
{\bf q} = (\Lambda,  \{e_i\}_{i \in {\rm I}}, (\ast, \ast)),
$$ 
 where $\Lambda$ is a lattice; $\{e_i\}$ is a basis  of the lattice $\Lambda$ parametrized by a 
given set ${\rm I}$; and 
$(\ast, \ast)$ is a skewsymmetric $\Z$-valued bilinear form 
on the lattice $\Lambda$.

The $\varepsilon$-matrix associated to ${\bf q}$ is a matrix $\varepsilon_{\bf q}:=(\varepsilon_{ij})$, $\varepsilon_{ij}:=(e_i,e_j)$. Two quivers are called {\it isomorphic} if their $\varepsilon$-matrices are equal. \footnote{An isomorphism class of quivers from Definition \ref{QDEF} is the same thing as a geometric quiver without loops and length two cycles, whose vertices are parametrized by a set ${\rm I}$. Namely, every such a geometric quiver ${\bf q}$ determines a matrix $\varepsilon_{\bf q}:=(\varepsilon_{ij})$, 
$
\varepsilon_{ij}:= \#\{\mbox{arrows from $i$ to $j$}\}-  \#\{\mbox{arrows from $j$ to $i$}\},
$
and vice versa.}
\ed
Every basis vector $e_k$ provides  a {\it mutated in the 
direction $e_k$} quiver ${\mathbf q'}$. 
The  quiver 
${\bf q}'$ is defined by changing the 
basis $\{e_i\}$ only. The lattice and the form stay intact. 
The new basis $\{e'_i\}$ is defined via halfreflection of the basis $\{e_i\}$ along the hyperplane 
$(e_k, \cdot)=0$:
\begin{equation} \label{12.12.04.2a}
\mu_k(e_i) = e'_i := 
\left\{ \begin{array}{lll} 
-e_k& \mbox{ if }   i = k\\
e_i + [(e_i, e_k)]_+e_k
& \mbox{ otherwise. }
\end{array}\right.
\end{equation}
Here $[\alpha]_+:= \alpha$ if $\alpha\geq 0$ 
and $[\alpha]_+:=0$ otherwise.

{\paragraph{Relation with the Fomin-Zelevinsky quiver mutations.} 
Formula (\ref{12.12.04.2a}) 
implies the Fomin-Zelevinsky formula,  which 
also appeared in Seiberg's work in physics \cite{Se95},  
telling how the $\varepsilon$-matrix changes under mutations: 
\begin{equation} \label{epsilon.mutate}
 \varepsilon'_{ij} := \left\{ \begin{array}{lll} 
- \varepsilon_{ij} & \mbox{ if $k \in \{i,j\}$} \\ 
\varepsilon_{ij} - \varepsilon_{ik}{\rm min}\{0, -{\rm sgn}(\varepsilon_{ik}) \varepsilon_{ kj}\}& 
 \mbox{ if $k \not \in \{i,j\}.$}\end{array}\right.
\end{equation} }

Mutations of a given quiver can be  encoded by the elements of the set ${\rm I}$. 
Performing the mutation at an element $k \in {\rm I}$ twice, we get a basis  
\be \la{SR}
e_i'':= \mu_{k} \circ \mu_{k}(e_i)= e_i + (e_i, e_k)e_k. 
\ee
 So in general $\{e_i''\}$ is a different basis 
than $\{e_i\}$. 
However, 
it preserves the 
$\varepsilon$-matrix:
\be \la{SR1}
(e_i'', e_j'') =(e_i, e_j) + (e_i, e_k)(e_k, e_j) + (e_i,e_k)(e_j, e_k)  = (e_i, e_j).
\ee


Each quiver can be mutated in $n$ directions, where $n = {\rm rk}~\Lambda$, 
and mutations can be repeated indefinitely. 
Thanks to  (\ref{SR1}), the double mutation in the same direction preserves 
the isomorphism class of a quiver. So we can picture
  the quivers obtained by 
mutations of an original quiver ${\bf q}$, and \underline{considered up to  isomorphisms}, 
 at the vertices of an infinite $n$-valent tree ${\rm T}_n$.

Precisely, consider a tree ${\rm T}_n$ such that the edges incident 
to a given vertex are 
parametrized by the set 
${\rm I}$.  
We assign to each vertex ${a}$ of the tree ${\rm T}_n$ a quiver ${\bf q}_a$   
{considered up to an isomorphism}, so that the  quivers ${\bf q}_a$ and ${\bf q}_b$ assigned to 
the vertices of an edge $a \stackrel{k}{-} b$ of ${\rm T}_n$ labeled by an element $k \in {\rm I}$ 
 are related by the mutations assigned to this edge: 
\be \la{muasse}
{\bf q}_b = \mu_{k}({\bf q}_a), ~~~~ {\bf q}_a = \mu_{k}({\bf q}_b). 
\ee
Equivalently, the vertices of ${\rm T}_n$ are  assigned $\varepsilon$-matrices which are related by \eqref{epsilon.mutate}.

Summarizing, we arrive at the following definition.

\bd A decorated tree $T_n$ is a tree whose edges are labeled by the elements of a given 
set ${\rm I}$, and whose vertices are decorated by the isomorphism classes of quivers  so that 
\begin{itemize}

\item The set of the edges incident to a given vertex is identified with the set ${\rm I}$.

\item The  quivers ${\bf q}_a$ and ${\bf q}_b$ at  
the vertices of any edge $a \stackrel{k}{-} b$  
 are related by mutations   (\ref{muasse}).

\end{itemize}
\ed

\paragraph{Negative mutations.} The halfreflection (\ref{12.12.04.2a}) is not the only natural way 
to describe mutations of isomorphism classes of quivers. 
There is another transformation of quivers, acting on the basis vectors only, given by   
\begin{equation} \label{12.12.04.2ab}
\mu_k^-(e_i) := 
\left\{ \begin{array}{lll} -e_k& \mbox{ if }  i = k\\
e_i + [-(e_i, e_k)]_+e_k
& \mbox{ otherwise}.\\
\end{array}\right.
\end{equation}
The negative mutation $\mu_k^-$ is the \underline{inverse} of 
$\mu_k$ \underline{on the nose}, not only up to an isomorphism. Indeed, $\mu_k^-\circ \mu_k (e_k) = e_k$, and 
the composition $\mu_k^-\circ \mu_k (e_i)$ is computed as 
\begin{align}
e_i ~~&\lms ~~ e_i'= e_i + [(e_i, e_k)]_+e_k~~ \nonumber \\
&\lms ~~ e'_i +
[-(e'_i, e'_k)]_+(e'_k) = 
e_i + [(e_i, e_k)]_+e_k -  [(e_i, e_k)]_+e_k = e_i.  \nonumber
\end{align}
The negative  mutation $\mu^-_k$ acts on the isomorphism classes of quivers in the same way as 
 $\mu_k$. 
Note that changing the sign of the form $(\ast, \ast)$ amounts to changing 
the $\mu_k$ to the  $\mu^-_k$.

\begin{remark}
Let $t_k:=\mu_k\circ \mu_k$ be the transformation of bases $\{e_i\} \to \{e_i''\}$, see  \eqref{SR}-\eqref{SR1}. We have
\be
\la{tmucomtriv}
\mu_k=t_k \circ \mu_k^{-}=\mu_k^- \circ t_k.
\ee
It is easy to check that $\{t_k\}$ satisfy the braid relations
\begin{align}
t_j \circ t_k &=t_k \circ t_j ,  &\mbox{if } (e_j, e_k) &=0. \nonumber\\
t_j \circ t_k \circ t_j &= t_k \circ t_j \circ t_k,   &\mbox{if } (e_j, e_k) &=\pm 1. 
\end{align}
Therefore the braid group acts on the bases in $\Lambda$, preserving the
  isomorphism class of quivers.
  
Moreover, one has
\be
\la{tmucom}
t_i = 
\left\{ \begin{array}{lll}  \mu_j \circ t_i \circ \mu_j^- & \mbox{ if }  (e_i, e_j) \geq 0, \\
 \mu_j^- \circ t_i \circ \mu_j ~~
&  \mbox{ if }  (e_i, e_j)  < 0.\\
\end{array}\right.
\ee
It relates the braid group actions assigned to different vertices of  $T_n$.

In fact, after categorification, the $\{t_k\}$ correspond to the 
Seidal-Thomas twist functors \cite{ST}, which generate braid group 
actions on the category ${\cal C}({\bf q}, W)$. 
\end{remark}

From now on, all the quivers are considered up to isomorphisms, unless otherwise stated. Abusing notation, we write ${\bf q}={\bf q'}$ if ${\bf q}$ and ${\bf q'}$ are isomorphic.

\paragraph{Quantum torus algebra.} A lattice $\Lambda$ with a form 
$(\ast, \ast): \Lambda \wedge\Lambda \to \Z$ gives rise to a {\it quantum torus algebra}  
${{\bf T}}_{\Lambda}$, which is 
a free  $\Z[q, q^{-1}]$-module with a basis $ X_v$, $v \in \Lambda$, and the product
$$
q^{-(v_1, v_2)}X_{v_1} X_{v_2} = X_{v_1+v_2}.
$$
There is an involutive antiautomorphism making it into a $\ast$-algebra: 
$$
\ast: {{\bf T}}_{\Lambda}\lra {{\bf T}}_{\Lambda}, \qquad  
\ast(X_{v}) = X_{v}, ~
\ast(q) = q^{-1}.
$$

A quiver ${\bf q}=({\Lambda}, \{e_i\}, (\ast, \ast))$ gives rise to  a quantum torus algebra ${{\bf T}}_{\bf q} := {{\bf T}}_{\Lambda}$ equipped with a set $\{X_{e_i}\}$ of algebra generators  corresponding to the basis vectors $\{e_i\}$.
Thanks to  parametrization of the basis vectors, if two quivers ${\bf q}$ and ${\bf q}'$ are isomorphic, then there is a \underline{unique} isomorphism identifying the associated quantum torus algebras 
\be
\la{unique.iso.quantum.torus.hhh}
{\bf T}_{\bf q} \stackrel{\sim}{\lra} {\bf T}_{{\bf q}'}, \hskip 7mm X_{e_i} \lms X_{e_i'}.
\ee

Given a decorated tree $T_n$, each vertex $a$ of $T_n$ 
is decorated by a quiver $({\Lambda}, \{e_i\}, (\ast, \ast))$, 
considered up to isomorphisms.
Therefore, each vertex $a$ gives rise to a quantum torus algebra ${{\bf T}}(a) := {{\bf T}}_{\Lambda}$ with generators $\{X_{e_i}\}$, 
well-defined up to the isomorphism \eqref{unique.iso.quantum.torus.hhh}.

Denote by 
${\Bbb T}(a)$ the non-commutative fraction field of the quantum torus algebra ${{\bf T}}(a)$. 
Our goal is to assign to any pair of  vertices $(a, b)$ of the tree a unique 
{\it quantum cluster transformation} 
\[
\Phi(a, b) : {{\Bbb T}}(b) \lra {{\Bbb T}}(a). 
\]
First, we assign a quantum cluster transformation to each oriented edge of the tree. 
Choose  an oriented path ${\bf i}$ on $T_n$ connecting $a$ and $b$. Denote by $\Phi({\bf i}): {{\Bbb T}}(b) \to {{\Bbb T}}(a)$ the composition of  the cluster transformations assigned to ${\bf i}$ in the reversed order. 
We show that the composition of the two quantum cluster transformations assigned to the path 
$a \to a' \to a$ is the identity map.
Therefore $\Phi(a, b):=\Phi({\bf i})$ is independent of the choice of ${\bf i}$. 

\vskip 3mm
{\it Notation.} We denote by ${{\bf T}}(a)$ the quantum torus algebra assigned to 
a vertex $a$ of the tree $T_n$, and by ${{\Bbb T}}(a)$ its 
non-commutative fraction field. At the same time, we have a quiver ${\bf q}_{a}$ assigned to 
the vertex $a$, well defined up to an isomorphism. We denote by ${{\bf T}}_{\bf q}$ the quantum torus algebra assigned to a quiver ${\bf q}$, and by ${\Bbb T}_{\bf q}$ its 
fraction field. 
They come with canonical generators $\{X_{e_i}\}$ provided by the basis $\{e_i\}$ of the quiver. 
The quantum torus algebra itself depends on the lattice $\Lambda$ with the form only, and 
denoted also by ${{\bf T}}_{\Lambda}$.

\vskip 3mm
 Our crucial tool to define quantum cluster transformations is the quantum dilogarithm.

\paragraph{The quantum dilogarithm formal power series.} 
Let us recall its definition:  
 \be \la{psi}
{\bf \Psi}_q(x):= \prod_{a=1}^{\infty}(1+q^{2a-1}x)^{-1}.
\ee
It is the unique formal power series starting from $1$ and satisfying a  
difference relation
\begin{equation} \label{11.19.06.20}
{\bf \Psi}_q(q^2x) = (1+qx){\bf \Psi}_q(x). 
\ee
It is useful to note its equivalent form: 
\be \label{11.19.06.20eq}
{\bf \Psi}_q(q^{-2}x) = (1+q^{-1}x)^{-1}{\bf \Psi}_q(x).
\end{equation} 
It has the  power series expansion, easily checked by using 
the difference relation:
\be
{\bf \Psi}_q(x)= \sum_{n=0}^{\infty}\frac{q^{n}x^n}
{(q^2-1) (q^4-1) \ldots (q^{2n}-1)} = \sum_{n=0}^{\infty}\frac{q^{n^2}x^n}{|{\rm GL}_n(F_{q^2})|}.
\ee
By \eqref{qdq}, it has an exponential expression
\be
\la{quantum.dilog.exp}
{\bf \Psi}_q(x)=
{\rm exp}\Big(\sum_{n\geq 1}\frac{(-1)^{n+1}}{n(q^{n}-q^{-n})} x^n\Big).
\ee
As a direct consequence, we get
\begin{equation} \label{1ssw}
{\bf \Psi}_q(x)^{-1} = {\bf \Psi}_{q^{-1}}(x).
\end{equation} 
Below we skip $q$ in the notation, setting ${\bf \Psi}(x):= {\bf \Psi}_q(x)$.

\paragraph{Quantum tori mutations.} 
Take a  decorated tree $T_n$.  Each vertex $a$ of the tree $T_n$ 
gives rise to 
a quantum torus algebra ${\bf T}({a})$ with a set of generators $X_{e_i}$, 
considered up to an isomorphism.  
Given an oriented edge $a \stackrel{k}{\to} a'$, 
there is a unique isomorphism \footnote{Note that the isomorphism is in the reversed direction.}
\be \la{Imut1}
i_{a \stackrel{}{\to} a'}: {\bf T}({a'}) \lra {\bf T}(a), ~~~~ X_{e_i'}\lms X_{\mu_k(e_i)},
\ee
transforming the generator $X_{e_i'}$ of the algebra ${\bf T}({a'})$ 
to the one $X_{\mu_k(e_i)}$ of the algebra ${\bf T}({a})$. 
Abusing notation, we also denote by $i_{a \stackrel{}{\to} a'}$ the induced isomorphism of the 
fraction fields. 

\bd \la{12.5.15.23} The quantum mutation at an edge $a \stackrel{k}{\to} a'$ is an isomorphism 
of fraction fields 
$$
\Phi(a \to a'):~{\Bbb T}({a'}) \lra {\Bbb T}(a)
$$
 defined as  the composition  of the isomorphism 
$i_{a \stackrel{}{\to} a'}$  with the conjugation by the quantum dilogarithm 
${\bf \Psi}(X_{e_k})$:
\be
\la{quantum.mut.1.26hh}
\Phi(a \stackrel{}{\to} a'):= 
{\rm Ad}_{{\bf \Psi}(X_{e_k})}\circ i_{a \stackrel{}{\to} a'}, 
\hskip 1cm
Y \lms {\bf \Psi}(X_{e_k}) i_{a \stackrel{}{\to} a'}(Y) 
{\bf \Psi}^{-1}(X_{e_k}).
\ee
\ed

It is a remarkable fact, following from  difference equations (\ref{11.19.06.20}) - (\ref{11.19.06.20eq}), that 
the conjugation by the the quantum dilogarithm ${\bf \Psi}(X_{e_k})$ is a rational transformation. 
One can look at Definition \ref{12.5.15.23} as follows. The classical Scolem-Noether theorem tells that 
any automorphism of a simple central algebra is inner. If 
the form $(\ast, \ast)$ on $\Lambda$ is non-degenerate, 
the quantum torus algebra is an infinite dimensional simple central algebra. 
So the Scolem-Noether theorem can not be applied. However one can get a birational automorphism 
of the quantum torus algebra by the conjugation with the dilogarithm power series  ${\bf \Psi}_q(X)$. 
Although the ${\bf \Psi}_q(X)$ does not belong to the algebra, the induced automorphism deserves to be viewed as 
``inner''.

\paragraph{Quantum cluster transformations.} 
Consider a path ${\bf i}$ on the tree ${T_n}$, presented as a sequence of oriented edges  
labeled by the elements of the set ${\rm I}$:
\be \la{CPATH}
{\bf i}:  ~~a= a_0 \stackrel{k_1}{\lra} a_1 \stackrel{k_2}{\lra}  \ldots \stackrel{k_{m-1}}{\lra} a_{m-1} \stackrel{k_{m}}{\lra}  a_m=b.
\ee
The quantum cluster transformation $\Phi({\bf i})$ is 
 the composition of mutations in the reversed order: 
$$
\Phi({\bf i}):= \Phi(a_{0}\to a_{1})\circ \ldots \circ\Phi(a_{m-1}\to a_{m-2}) \circ \Phi(a_{m-1}\to a_{m}): \hskip 5mm {\Bbb T}(b){\lra} {\Bbb T}(a).
$$
Below we present $\Phi({\bf i})$ as a composition 
of an isomorphism of quantum tori algebras with a sequence of conjugations by 
quantum dilogarithms. 

 Let $\{e_i\}$ be the basis for the quiver ${\bf q}_{a}$ at the vertex $a$. 
Consider a composition of mutations
\be \la{fvect2}
\mu_{\bf i}:= \mu_{k_{m}} \circ \ldots \circ \mu_{k_1}.
\ee
It changes the basis $\{e_i\}$ 
to a basis $\{\mu_{\bf i}(e_i)\}$ of the same lattice $\Lambda$ for the quiver ${\bf q}_{a}$: 
\be \la{SEQV}
\{e_i\} = \{e^{(0)}_i\} \stackrel{\mu_{k_1}}{\lra} 
\{e^{(1)}_i\} \stackrel{\mu_{k_2}}{\lra} \ldots 
\stackrel{\mu_{k_{m}}}{\lra} \{e^{(m)}_i\} = \{\mu_{\bf i}(e_i)\}
\ee
 The basis $\{\mu_{\bf i}(e_i)\}$ 
defines a quiver isomorphic to ${\bf q}_{b}$. Denote by  
$\{e'_i\}$ the basis for the quiver ${\bf q}_{b}$. 
There is a unique isomorphism of quantum torus algebras identifying the generators:
$$
i({\bf i}): {\bf T}(b) \lra {\bf T}(a), ~~~~
X_{e'_j} \lms 
X_{\mu_{\bf i}(e_j) }.
$$
 
Let us define vectors $f_1, ..., f_{m}$  of the lattice $\Lambda$ for the quiver ${\bf q}_{a}$ 
by setting 
\be \la{fvect0}
f_{s}:=e_{k_s}^{(s-1)},
 ~~~~s = 1, ..., m.
\ee

\bp One has 
\be \la{sepf}
\Phi({\bf i}) = {\rm Ad}_{{\bf \Psi}(X_{f_1})}\circ \ldots \circ {\rm Ad}_{{\bf \Psi}(X_{f_{m}})}\circ i({\bf i}).
\ee
\ep

\begin{proof} Follows from the very definitions.
\end{proof}

\bl \la{clust.inv.hh}
The composition of cluster mutations 
$$
\Phi(a \stackrel{k}{\to} a' \stackrel{k}{\to} a):= 
\Phi(a \stackrel{k}{\to} a')\circ \Phi(a' \stackrel{k}{\to} a): ~~
{\Bbb T}(a) \lra {\Bbb T}({a'})\lra {\Bbb T}(a) ~~~~\mbox{\it is the identity map}. 
$$
\el

\begin{proof} Let $v\in \Lambda$ be the basis vector which we use to define the mutation 
$a \to a'$. Then $i_{a\to a'}\circ i_{a'\to a}$ is the reflection map $w \to w + (w, v)v$. 
The following lemma calculates  
the ``quantum dilogarithm part'' of the composition $\Phi(a\to a')\circ \Phi(a'\to a)$. 
\bl
One has 
\be \la{conjf}
{\rm Ad}_{{\bf \Psi}(X_{v})}{\rm Ad}_{{\bf \Psi}(X_{-v})} (X_w) =   X_{w - (w, v)v}.
\ee
\el

\begin{proof} The general case reduces to the case when $(v,w)=1$. 
Assuming $(v,w)=1$, we have 
$$
{\bf \Psi}(X_v){\bf \Psi}(X_{-v})X_w =  X_w{\bf \Psi}(q^{2}X_v){\bf \Psi}(q^{-2}X_{-v}) \stackrel{\eqref{11.19.06.20}
 \eqref{11.19.06.20eq}}{=\joinrel=\joinrel=} 
$$
$$
 X_w (1+qX_v){\bf \Psi}(X_{v})(1+q^{-1}X_{-v})^{-1} {\bf \Psi}(X_{-v}) =
qX_wX_v {\bf \Psi}(X_{v}) {\bf \Psi}(X_{-v}). 
$$
Since $(v,w)=1$ implies that $qX_wX_v = X_{v+w}$, and $w-(w,v)v = w+v$,  we get (\ref{conjf}).  
\end{proof}
Therefore 
\be \la{3.17.15.1}
\Phi(a\to a')\circ \Phi(a'\to a)={\rm Ad}_{{\bf \Psi}(X_{v})}\circ {\rm Ad}_{{\bf \Psi}(X_{-v})} \circ i_{a\to a'}\circ 
i_{a'{\to} a} =   {\rm Id}.
\ee
\end{proof}

\paragraph{Remark 1.} It is tempting to write 
\be \la{conjfa}
{\rm Ad}_{{\bf \Psi}(X_{v})}{\rm Ad}_{{\bf \Psi}(X_{-v})}  \stackrel{?}{=} {\rm Ad}_{{\bf \Psi}(X_{v}){\bf \Psi}(X_{-v})}. 
\ee
However the product ${\bf \Psi}(X^{-1}){\bf \Psi}(X)$ does not make sense 
as a power series. The formula starts  to make sense if we replace 
the quantum dilogarithm power series by their modular double, given by the quantum dilogarithm function 
$\Phi_{\hbar}(x)$, see (\ref{12.6.15.1}). 

\paragraph{Remark 2.} Formula \eqref{sepf} is a composition of two transformations. 
 The first one 
is an isomorphism $i({\bf i}): {\bf T}(b) \to {\bf T}(a)$. 
The second one is a birational automorphism 
\be \la{PRPFI}
{\rm Ad}_{{\bf \Psi}(X_{f_1})}\circ 
\ldots \circ {\rm Ad}_{{\bf \Psi}(X_{f_{m}})}:  ~{\Bbb T}(a) \lra {\Bbb T}(a).
\ee
By Lemma \ref{clust.inv.hh}, there is a unique 
rational map assigned to any pair $a,b$ of vertices of $T_n$, called the {\it quantum cluster transformation map}:  
\be \la{qcltra}
\Phi(a,b): {\Bbb T}(b) \lra {\Bbb T}(a). 
\ee

Another approach is to 
view mutations of quivers as transformations of bases in a \underline{given} 
lattice $\Lambda$. Then the lattices assigned to the vertices 
$a_0, ... , a_m$ of the path ${\bf i}$ are identified with $\Lambda$. So 
cluster transformations can be understood as birational 
automorphisms  (\ref{PRPFI}) of the quantum torus algebra ${\bf T}_\Lambda$.
 However then 
$\mu_k \circ \mu_k$ is no longer the identity map, 
and therefore the cluster transformation depends on the path ${\bf i}$ rather then 
on the vertices it starts and ends.  
Yet the advantage is that $\mu_k \circ \mu_k$ is identified with the symplectic reflection  $t_k$, 
discussed in the beginning of this Section, 
incorporating the braid group action into the cluster transformation story. 

\paragraph{Alternative formulas for quantum cluster transformations.} Recall the negative mutation \eqref{12.12.04.2ab}.
Given an oriented edge $a \stackrel{k}{\to} a'$, 
there is an isomorphism
\be \la{Imut}
i^-_{a {\to} a'}: {\bf T}({a'}) \lra {\bf T}(a), \hskip 7mm X_{e_i'}\lms X_{\mu^-_k(e_i)}. 
\ee
Since $i_{a'\to a}^{-1}= i^{-}_{a\to a'}$, and ${\bf \Psi}(X_{v})$ commutes with 
${\bf \Psi}(X_{-v})$,  formula (\ref{3.17.15.1}) is equivalent to
\be \la{3.18.15.1}
{\rm Ad}_{{\bf \Psi}(X_{v})} \circ i_{a\to a'} =  {\rm Ad}_{{\bf \Psi}(X_{-v})^{-1}}\circ i^{-}_{a\to a'}.
\ee
Therefore
the quantum cluster mutation $\Phi(a \stackrel{}{\to} a')$ can be defined by a different formula
\be
\Phi^-(a \stackrel{}{\to} a') := {\rm Ad}_{{\bf \Psi}(X_{-v})^{-1}}\circ i^{-}_{a\to a'}, \hskip 7mm \Phi(a \stackrel{}{\to} a') = 
\Phi^-(a \stackrel{}{\to} a'). 
\ee
So for any sequence of signs $\varepsilon_s \in \{\pm 1\}$ we can write the quantum cluster transformation 
as 
$$
\Phi({\bf i}) = \Phi^{\varepsilon_{1}}(a_0 \stackrel{}{\to} a_{1})\circ \ldots \circ \Phi^{\varepsilon_{m}}(a_{m-1} \stackrel{}{\to} a_{m}).
$$

Recall the bases $\{e_i\}$ and $\{e'_i\}$ for the quivers ${\bf q}_{a}$ and ${\bf q}_{b}$. We consider the sequence of mutations along the path ${\bf i}$:
\be \la{fvect2}
\{e_i\} = \{e^{(0)}_i\} \stackrel{\mu_{k_1}^{\varepsilon_{1}}}{\lra} 
\{e^{(1)}_i\} \stackrel{\mu_{k_2}^{\varepsilon_{2}}}{\lra} \ldots 
\stackrel{\mu_{k_{m}}^{\varepsilon_{m}}}{\lra} \{e^{(m)}_i\} = \{\mu_{\bf i}^\varepsilon(e_i)\}
\ee
There is 
an isomorphism  of quantum torus algebras: 
$$
i^\varepsilon({\bf i}): {\bf T}(b) \lra {\bf T}(a), ~~~~
X_{e'_i} \lms 
X_{\mu^{\varepsilon}_{\bf i}(e_i) }.
$$
 
Set
\be \la{fvect1}
f^{\varepsilon}_{s}:=\varepsilon_s \cdot e_{k_s}^{(s-1)}= \varepsilon_s\cdot\mu^{\varepsilon_{s-1}}_{k_{s-1}} \circ \ldots \circ 
\mu^{\varepsilon_1}_{k_1}(e_{k_{s}}), ~~~~s = 1, ..., m.
\ee 
The same quantum cluster transformation $\Phi({\bf i})$ can be written in a different form as 
\be \la{sepf1}
\Phi({\bf i}) = {\rm Ad}_{{\bf \Psi}(X_{f^{\varepsilon}_1})^{\varepsilon_1}}\circ 
\ldots \circ {\rm Ad}_{{\bf \Psi}(X_{f^{\varepsilon}_{m}})^{\varepsilon_{m}}}\circ 
i^\varepsilon({\bf i}).
\ee
 As we will see in Section 2.4, there is a unique sequence of signs $\varepsilon_s \in \{\pm 1\}$ 
for which all power series ${\bf \Psi}(X_{f^{\varepsilon}_s})^{\varepsilon_s}$ in (\ref{sepf1}) 
 lie in the same completion of the quantum torus algebra ${\bf T}_\Lambda$. 

\paragraph{Cluster modular groupoid.} 
Let $\pi$ be an arbitrary permutation of the set ${\rm I}$. It gives rise to a new quiver which does not necessarily preserve the isomorphism class of the original quiver:
\[
{\bf q}'=\pi({\bf q}):=\{\Lambda, \{e_i'\}, (\ast, \ast)\}, \hskip 7mm \mbox{where } e_i':=e_{\pi^{-1}(i)}, ~\forall i \in {\rm I}.
\]
There is an isomorphism between their associated non-commutative fraction fields
\[
\Phi(\pi): {\Bbb T}_{{\bf q}'}\lra {\Bbb T}_{\bf q}, \hskip 7mm X_{e_i'}\lms X_{e_{\pi^{-1}(i)}}.
\]

A {\em quiver cluster transformation} is a 
composition of quiver mutations and {permutations}. 
It induces a {\em quantum cluster transformation} of the associated non-commutative fraction fields.


\bd \la{Def2.7}
If  two quiver cluster transformations {$\sigma_1,\sigma_2: {\bf q}\to {\bf q'}$} 
induce the same quantum cluster transformation, i.e,
\[\Phi(\sigma_1)=\Phi(\sigma_2): {\Bbb T}_{\bf q'}\lra {\Bbb T}_{\bf q},\]
then we say $\sigma_1$ and $\sigma_2$ are equivalent, denoted by $\sigma_1=\sigma_2$.
\ed

Two quivers are
{\em equivalent} if they are related 
by a quiver cluster transformation. 

\bd
The cluster modular groupoid ${\cal G}_{{\bf q}}$ is 
a groupoid whose objects are quivers equivalent to ${\bf q}$, 
and morphisms are  quiver cluster transformations modulo equivalence. 
The fundamental group $\Gamma_{\bf q}$ of the 
groupoid  at ${\bf q}$ is  the cluster modular group.
\ed

Below we call both quiver cluster transformations and  quantum cluster transformations just 
cluster transformations, and use similar convention for mutations.

\paragraph{Cluster Poisson transformations.} Setting $q=1$, the quantum cluster transformation 
(\ref{qcltra}) becomes a birational transformation preserving the Poisson structure given by the quasiclassical limit 
of the commutator in the quantum torus algebra. It is called the {\it cluster Poisson map}.\footnote{Kontsevich and 
Soilbelman considered another specilization $q=-1$.} 
\footnote{Notice that 
it is important to present first the map (\ref{qcltra}) as a rational transformation, and only then set $q=1$. 
Indeed, setting $q=1$ first we get a commutative algebra, so the conjugation becomes the identity map. }

To write  it explicitly, let us consider the quiver mutation $\mu_k: {\bf q} \to {\bf q}'$.
We assign to ${\bf q}$ a set of {\it cluster Poisson coordinates} $\{X_i\}_{i\in {\rm I}}$. 
Denote  by $\{X'_i\}$ the cluster Poisson coordinates assigned  to ${\bf q'}$.  Setting $q=1$, the quantum cluster transformation \eqref{quantum.mut.1.26hh} becomes the cluster Poisson map
\begin{equation} \label{5.11.03.1x}
X'_{i}  \lms \left\{\begin{array}{ll} X_k^{-1}& \mbox{if }  i=k \\
 X_i(1+X_k^{-{\rm sgn} (\varepsilon_{ik})})^{-\varepsilon_{ik}} &   \mbox{\rm otherwise}.
\end{array} \right.
\end{equation}
Note that \eqref{5.11.03.1x} is subtraction free. Such a transformation is called {\it positive}.  
Its tropicalization  is
\begin{equation} \label{5.11.03.1xtr}
x'_{i} \lms \left\{\begin{array}{ll} -x_k& \mbox{if }  i=k \\
 x_i-\varepsilon_{ik}{\rm min}\{0, -{\rm sgn} (\varepsilon_{ik})x_k\} & \mbox{\rm otherwise}.
\end{array} \right.
\end{equation}

\paragraph{Tropical points of cluster Poisson varieties.}
A collection of quivers  $\{{\bf q}\}$ related by quiver cluster transformations   determines a cluster Poisson variety ${\cal X}$. 
It is given by a collection of cluster Poisson tori 
\be \la{trainit}
{\cal X}_{\bf q}:= {\rm Hom}(\Lambda, {\Bbb G}_m)
\ee
glued by cluster Poisson maps. 
A cluster Poisson variety ${\cal X}$ gives rise to a set ${\cal X}(\Z^t)$, 
called the set of integral tropical points  of ${\cal X}$, 
equipped with an action of the cluster modular group $\Gamma$. Namely, 
for each quiver ${\bf q}$ there is a set 
$$
{\cal X}_{\bf q}(\Z^t):= {\rm Hom}({\Bbb G}_m, {\cal X}_{\bf q}) = \Lambda^\vee = {\rm Hom}(\Lambda, \Z). 
$$
A mutation $\sigma: {\bf q} \to {\bf q'}$ gives rise to an isomorphism of sets 
$\varphi^t(\sigma): {\cal X}_{\bf q}(\Z^t) \lra {\cal X}_{\bf q'}(\Z^t)$ given in coordinates by the transformation 
(\ref{5.11.03.1xtr}). 

\bd An integral tropical point $l \in {\cal X}(\Z^t)$ as a collection 
of  $l_{\bf q}\in {\cal X}_{\bf q}(\Z^t)$ related by mutations: $\varphi^t(\sigma)(l_{\bf q}) = 
l_{\bf q'}$.
\ed


\paragraph{Frozen variables and integral tropical points of cluster Poisson variaties.} 
 Proposition \ref{9:14:04:1} below is borrowed from 
  \cite[ArXive version 2, Proposition 2.44]{FG2}. 

Given a quiver  ${\bf q} = ( \Lambda,  \{e_i\}, (\ast, \ast))$, take a lattice 
$$
\widehat \Lambda:= \Lambda \oplus \Z e_0 
$$
generated by $\Lambda$ and a new basis vector $e_0$.  
It has a basis $\{e_i\} \cup e_0$.

\begin{proposition} \label{9:14:04:1} 
There is a canonical bijection between the extensions of the skew-symmetric 
form $(\ast, \ast)$ from $\Lambda$ to  $\widehat \Lambda$, 
and the set of integral tropical points ${\cal X}(\Z^t)$. 
\end{proposition}
\begin{proof}  We encode a quiver ${\bf q}$ by a skew-symmetric $\Z$-valued function $\varepsilon_{ij}$ 
on ${\rm I} \times {\rm I}$. The rank one extensions of the quiver are 
parametrised by similar functions on  $({\rm I} \cup \{0\})^2$, i.e.  
 by the integers $\{\varepsilon_{i0}\}_{i \in {\rm I}}$. The desired bijection is then given by 
$$
\{\varepsilon_{i0}\} \lms \{x_i\}  \in {\cal X}_{{\bf q}}(\Z^t), \quad x_i := \varepsilon_{i0}.
$$
We have to show that the numbers $\{\varepsilon_{i0}\}$ and coordinates $\{x_i\}$ 
of a $\Z$-tropical point 
$x\in {\cal X}(\Z^t)$ in the tropical coordinate system assigned to the  quiver ${\bf q}$ 
change under mutations the same way. Indeed, under the mutation in the direction $k\in {\rm I}$ one has 
$$
x_i'= \left\{ \begin{array}{ll}-x_i& \mbox{if } i = k, \\
x_i - \varepsilon_{ik} {\rm min}\{0, -{\rm sgn}(\varepsilon_{ik})x_k\}&  
\mbox{\rm otherwise}.
\end{array}\right. 
$$
On the other hand, we have
$$
\varepsilon_{i0}' 
= 
\left\{ \begin{array}{ll}
- \varepsilon_{i0} & \mbox{if } i = k,\\
\varepsilon_{i0} - \varepsilon_{ik} {\rm min}\{0, -{\rm sgn}(\varepsilon_{ik})\varepsilon_{k0}\}&  
\mbox{\rm otherwise}.
\end{array}\right.
$$
The two formulas  coincide under the assumption
 $x_i=\varepsilon_{i0}$. 
\end{proof}

\paragraph{Quantum cluster algebras with principle coefficients.} Given a quiver ${\bf q}$, the  basis $\{e_i\}$ of the lattice $\Lambda$ provides a dual basis $\{f_i\}$ of 
the dual lattice $\Lambda^{\circ}:= {\rm Hom}(\Lambda, \Z)$. The basis $\{f_i\}$  mutates as follows:
\be \la{qdbas}
f'_i := 
\left\{ \begin{array}{lll} -f_k + \sum_{j\in I}[-\varepsilon_{kj}]_+f_j
& \mbox{ if } &  i = k\\
f_i& \mbox{ if } &  i \not = k.\end{array}\right.
\ee
We need a lattice 
$$
\Lambda_{\cal P}:= \Lambda \oplus \Lambda^{\circ}.
$$
   
   Let $[ \ast, \ast ]: \Lambda \times \Lambda^\circ\to\Z$ be the canonical pairing. 
Togerther with the skew symmetric form $(\ast, \ast)$ on $\Lambda$, it provides 
the lattice $\Lambda_{\cal P}$ with a skew symmetric bilinear form
$\langle\ast, \ast\rangle_{\Lambda_{\cal P}}$: 
$$
\langle(e,f), (e',f')\rangle_{\Lambda_{\cal P}}:= (e,e') + [ e, f' \rangle ] - [ e', f ].
$$
It gives rise to a  quantum torus $\ast$-algebras ${\bf T}_{\Lambda_{\cal P}}$. 
 The basis $\{e_i, f_j\}$ of the lattice $\Lambda_{\cal P}$ provides   its 
 generators $\{B_i, X_j\}$, where $B_i:=B_{f_i}$ and $X_j:=X_{e_j}$. 
The relations are the following:
\begin{equation} \label{4.28.03.11x}
q  B_iX_i = 
q^{-1}X_iB_i, \quad  B_i X_j = 
X_jB_i, ~i \not = j, \quad  
q^{-\varepsilon_{ij}} X_i X_j = 
q^{-\varepsilon_{ji}} X_jX_i, \quad B_iB_j=B_jB_i.
\end{equation}

Denote by ${\bf T}_{\cal P}(a)$  the quantum torus algebra assigned to a vertex $a$ of the tree ${\Bbb T}_n$, and by ${\Bbb T}_{\cal P}({a})$ 
its   fraction field.  Given oriented edge $a \stackrel{k}{\to} a'$, 
there is a unique isomorphism of algebras
\be \la{Imut1a}
i_{a \stackrel{}{\to} a'}: {\bf T}_{\cal P}({a'}) \lra {\bf T}_{\cal P}(a), ~~~~ X_{e_i'}\lms X_{\mu_k(e_i)}, ~~B_{f_i'}\lms B_{\mu_k(f_i)}.
\ee
Abusing notation, we  also denote by $i_{a \stackrel{}{\to} a'}$ the induced isomorphism of the 
fraction fields. 

\bd \la{12.5.15.23a} The quantum mutation at an edge $a \stackrel{k}{\to} a'$ is an isomorphism 
of fraction fields 
$$
\Phi_{\cal P}(a \to a'):~{\Bbb T}_{\cal P}({a'}) \lra {\Bbb T}_{\cal P}(a)
$$
\be
\la{quantum.mut.1.26hh}
\Phi_{\cal P}(a \stackrel{}{\to} a'):= 
{\rm Ad}_{{\bf \Psi}(X_{e_k})}\circ i_{a \stackrel{}{\to} a'}, 
\hskip 1cm
Y \lms {\bf \Psi}(X_{e_k}) i_{a \stackrel{}{\to} a'}(Y) 
{\bf \Psi}^{-1}(X_{e_k}).
\ee
\ed

In coordinates,  
\begin{equation} \label{11.18.06.1}
{\rm Ad}_{{\bf \Psi}(X_{e_k})}: B_i\lms \left\{\begin{array}{lll} B_i& \mbox{ if } & i\not =k, \\
    B_k(1+q_kX_k) & \mbox{ if } &  i= k. \\
\end{array} \right.
\end{equation} 
 Indeed,   the relation 
$
X_kB_i = q^2B_iX_k
$ 
implies 
$
\Psi_q(X_k)B_i = 
B_i\Psi_q(q^{2}X_k).
$
It remains to use   (\ref{11.19.06.20}).  Since the $B$-variables commute, we can set 
\be \la{setB}
{\Bbb B}_k^+:= 
\prod_{j\in I}B_j^{[\varepsilon_{kj}]_+}, \qquad {\Bbb B}_k^-:= 
\prod_{j\in I}B_j^{[-\varepsilon_{kj}]_+}.
\ee
Then formula  (\ref{qdbas}) translates into monomial transformations
\begin{equation} \label{11.18.06.10hr}
i_{a \stackrel{}{\to} a'}^*: B'_{i} \lms \left\{\begin{array}{lll} B_i& \mbox{ if } & i\not =k, \\
    {\Bbb B}_k^-/B_k
 & \mbox{ if } &  i= k.   \\
\end{array} \right.
\end{equation}
Formulas for the 
action  on the generators $X_i$ are the same as for 
the ${\cal X}$-space.  Set
\be \la{wideY}
\widetilde X_i:= X_i\cdot\prod_{j\in I}B_j^{\varepsilon_{ij}} = X_i\frac{{\Bbb B}_k^+}{{\Bbb B}_k^-}.
\ee
Then  the $\widetilde X_i$ commutes with the $X_j$. 
So   mutations act on the $\widetilde X_i$ by  monomial transformations. 

Let  us work out the formulas  for the mutations of $B$-coordinates in the $q\to 1$ limit. The conjugation (\ref{11.18.06.1})  preserves ${\Bbb B}_k^+$ and ${\Bbb B}_k^-$. 
So we get 
\begin{equation} \label{1.7.10.3}
\Phi_{\cal P}^*: B'_{i} \lms \left\{\begin{array}{lll} B_i& \mbox{ if } & i\not =k, \\
    \frac{{\Bbb B}_k^-}{B_k(1+X_k)}
 & \mbox{ if } &  i= k. \\ 
\end{array} \right. 
\end{equation}

Let us set $A_i:= B_i^{-1}$. Then $
 X_i= \widetilde X_i\frac{{\Bbb A}_k^+}{{\Bbb A}_k^-}.$ So the mutation 
formula can be written as 
\begin{equation} \label{1.7.10.3a}
\Phi_{\cal P}(a \stackrel{}{\to} a')^*: A'_{k} \lms  
    \frac{({\Bbb A}_k^-+ \widetilde X_k{\Bbb A}_k^+)}{A_k}, ~~~~ \Phi_{\cal P}(a \stackrel{}{\to} a')^*: A'_{i} \lms A_i, ~~i \not = k.
    \end{equation}
So  mutation formulas of the coordinates $(A_i, \widetilde X_j)$ in the $q\to 1$ limit coincide
 with the mutation formulas \cite{FZIV} for the cluster algebra with  cluster variables $A_i$ and 
 principle coefficients $\widetilde X_k$.

We denote by ${\cal A}_{\rm prin, q}$ the quantum cluster variety with principal coefficients obtained by gluing the symplectic tori assigned to the lattices $\Lambda_{\cal P}$ by the 
mutations $\Phi_{\cal P}(a \stackrel{}{\to} a')$.  Denote by ${\rm T}$ the split torus with the group of characters $\Lambda$, and by 
${\rm T}_q$ the corresponding quantum torus. Then the quantum space ${\cal A}_{\rm prin, q}$ projects canonically to the product of the quantum torus ${\rm T}_q$ and  the quantum cluster  variety ${\cal X}_q$, and the fiber of the map $\pi_{\rm T}$ is  the cluster variety ${\cal A}$:
$$
{\cal A} \stackrel{j}{\hra} {\cal A}_{\rm prin, q} \stackrel{\pi_{\rm T}\times \pi_{\cal X}}{\lra} {\rm T}_q \times {\cal X}_q, ~~~~\pi_{\rm T}^*X_i:= \widetilde X_i, ~~\pi_{\cal X}^*X_i:=X_i, ~~~~
j({\cal A})=\pi_{\rm T}^{-1}(e).
$$

The subalgebra $\pi_{\rm T}^*({\cal O}( {\rm T}_q))$, that is the subalgebra generated by the $\widetilde X_i$'s,  see (\ref{wideY}), is the subalgebra of "coefficients", explaining the name. 
The   $A_i$'s are the cluster algebra generators. 

The quantum symplectic double ${\cal D}_q$ defined in \cite[Definition 3.1]{FG3}  is similar to the 
 double ${\cal A}_{\rm prin, q}$. The difference is that 
the mutation automorphisms  are defined differently:
$$
{\cal D}_q: \mbox{We use the conjugation by the ratio of two quantum dilogarithms: 
${\bf \Psi}_{q}(X_k)/{\bf \Psi}_{q}(\widetilde X_k)$}. 
$$
$$
{\cal A}_{\rm prin, q}: \mbox{We use the conjugation by single quantum dilogarithm: ${\bf \Psi}_{q}(X_k)$}.
$$

\section{DT-transformations of cluster varieties and Duality Conjecturs}

In Section \ref{sec2.2a} we discuss a basic question: 
when do two quantum cluster transformations coincide? 
A considerable part of Section \ref{sec2.2a} is due to 
Nagao \cite{N10} and Keller \cite{K11,K12,K13}, 
although we present the story from a different perspective, 
emphasizing the role of certain integral tropical points of cluster Poisson variaties, 
called  {\it basic positive laminations}, 
rather then using the  C-matrices. 
Proposition \ref{9:14:04:1} provides the dictionary relating the two points of view.

In Section \ref{SSec1.2} we recall  cluster 
${\rm DT}$-transformations following Keller \cite{K12}. 
We interpret them as quantum cluster transformations, and prove that  
a cluster ${\rm DT}$-transformation is a central element of the cluster modular group. 
In Section \ref{iandF} we recall the isomorphism $i$  from \cite[Sect.3.2]{FG4}. 
In Section \ref{SecAA} we relate the cluster {\rm DT}-transformations to Duality Conjectures.

\subsection{When do two  cluster transformations coincide?} \la{sec2.2a}

A quiver cluster transformation $\sigma: {\bf q} \to {\bf q}'$ induces
 a quantum cluster transformation 
\[
\Phi(\sigma): {\Bbb T}_{{\bf q'}} \lra {\Bbb T}_{\bf q}.  
\] 
 Setting $q=1$, we get a positive birational isomorphism of the Poisson tori 
\[
\varphi(\sigma): {\cal X}_{\bf q} \lra {\cal X}_{\bf q'}.
\] 
Its tropicalization is 
a piecewise-linear map of the set of tropical points 
\[
\varphi^t(\sigma): {\cal X}_{{\bf q}}(\Z^t) \lra {\cal X}_{\bf q'}(\Z^t).  
\] 

If two quiver cluster transformations $\sigma_1, \sigma_2: {\bf q} \to {\bf q}'$ 
induce the same quantum cluster transformations $\Phi(\sigma_1) = \Phi(\sigma_2)$, then their tropicalisations evidently 
coincide:  $\varphi^t(\sigma_1) = \varphi^t(\sigma_2)$.

Remarkably, the converse is true:
$$
\mbox{$\varphi^t(\sigma_1) = \varphi^t(\sigma_2)$ implies that $\Phi(\sigma_1) = \Phi(\sigma_2)$.}
$$ 
It follows from a stronger Theorem \ref{basiclam}, proved by Keller 
\cite{K11}, 
\cite[Sect.7]{K12} and Nagao \cite{N10} in a different formulation. 
It also follows from Duality Conjectures, as we show in  Section \ref{SecAA}. 


\paragraph{Basic ${\cal X}$-laminations.} Recall that an equivalent class of quivers gives rise to a cluster Possion variety ${\cal X}$.
Fix a quiver ${\bf q}$. Each vertex $i\in \{1,\ldots, N\}$ of ${\bf q}$ corresponds to a rational function $X_i$ on ${\cal X}$. 
The set ${\bf c}_{\bf q}:=\{X_1,\ldots, X_N\}$ is a rational cluster Poisson coordinate system on ${\cal X}$. 
Its tropicalization identifies the set ${\cal X}(\Z^t)$ of $\Z$-tropical points of ${\cal X}$ with 
$\Z^N$:
\[
{\bf c}_{\bf q}^t: ~ {\cal X}(\Z^t) \stackrel{\sim}{\lra} \Z^N, \hskip 7mm l\lms (X_1^t(l),\ldots, X_{N}^t(l)).
\]
Let $e_i=(0,...,1,...,0)$ be the $i$-th unit element of $\Z^N$.
\bd 
\la{def.basic.laminations.t} 
The $\Z$-tropical points $l_{{\bf q}, i}^+$ (respectively $l_{{\bf q}, i}^-$) of ${\cal X}$ such that
\be
{\bf c}_{\bf q}^t(l_{{\bf q},i}^+) = e_i,
\hskip 9mm
{\bf c}_{\bf q}^t(l_{{\bf q}, i}^-) = -e_i
\ee
are called basic positive (respectively negative) ${\cal X}$-laminations  associated to the quiver ${\bf q}$.
\ed

We usually call basic positive ${\cal X}$ laminations just basic laminations, or basic ${\cal X}$-laminations. We also frequently write $l_{i}^\pm$ instead of $l_{{\bf q}, i}^\pm$ when there is no confusion.

\paragraph{Basic laminations and cluster transformations.} 
Using basic laminations, we can state now the strongest version of the 
criteria determining when two cluster transformations coincide. 

\bt \la{basiclam}
Let $\sigma_1, \sigma_2: {\bf q} \to {\bf q}'$ be two cluster transformations between the same quivers. 
The following are equivalent
\begin{itemize}
\item[1.] $\Phi(\sigma_1) = \Phi(\sigma_2)$. 
\item[2.] $\varphi^t(\sigma_1)(l^+_{{\bf q},i}) = \varphi^t(\sigma_2)(l^+_{{\bf q}, i})$ for all $i \in {\rm I}$.
 \end{itemize}
\et

It looks surprising  that such a strong statement is true, and even more surprising  
that  basic laminations play  key role in the formulation. 
The proof of Proposition \ref{BLDC} below explains both. 

\bp \la{BLDC}
Theorem \ref{basiclam} follows from  Duality Conjectures \cite{FG2}. 
\ep
We prove Proposition \ref{BLDC} 
in Section \ref{SecAA}, after a discussion of Duality Conjectires.

\subsection{Cluster Donaldson-Thomas transformations} \la{SSec1.2}


\bt \la{basiclamDT} 
Let $\sigma: {\bf q} \to {\bf q'}$ be a cluster transformation such that 
\[
\varphi^t({\sigma})(l^+_{{\bf q},i}) = l^-_{{\bf q'},i}, \quad \forall i \in {\rm I}.
\]
Then the quivers ${\bf q}$ and ${\bf q'}$ are isomorphic.
The quantum cluster transformation $\Phi(\sigma)$,  unique by Theorem \ref{basiclam},
 coincides with the Kontsevich-Soibelmam  ${\rm DT}$-transformation.  
\et
Theorem \ref{basiclamDT} suggests the following definition. 
\bd \la{Donaldson.Thomas.transformation}
A cluster transformation  ${\bf K}: {\bf q}\to {\bf q}$ is called a cluster Donaldson-Thomas transformation if 
\be
\varphi^t({\bf K})(l_{{\bf q},i}^+)=l_{{\bf q},i}^-, \quad \forall i\in {\rm I}.
\ee
\ed
A cluster {\rm DT}-transformation ${\bf K}$ 
may not exist. If it does, it is unique by 
 Theorem \ref{basiclam}.   

Theorems \ref{basiclam} and \ref{basiclamDT}  were proved by Keller  \cite{K11}, 
\cite[Th 6.5, Sect 7.11]{K12} in a different 
formulation, using 
 {\it $c$-vectors} and {\it $C$-matrices} \cite{FZIV}, which we review in Section \ref{sssec2.4}, 
 rather than the tropical points of 
cluster Poisson varieties. 
See an exposition in a  nice short 
 paper \cite{K13}. 
The equivalence of two points of view follows from Proposition \ref{9:14:04:1}. 
One of the benefits of using the tropical points is that then  Definition \ref{Donaldson.Thomas.transformation} 
make sense for any positive rational transformation, not necessarily a cluster one.


\bt \la{universal.dt}
Let $\sigma: {\bf q} \to {\bf q}'$ be a cluster transformation. If ${\bf K}:{\bf q}\to {\bf q}$ is a cluster {\rm DT}-transformation, then so is $\sigma\circ {\bf K}\circ \sigma^{-1}$ .\et
Theorem \ref{universal.dt} is proved in Section \ref{sec10}. It asserts that the cluster {\rm DT}-transformation is independent of  the choice of ${\bf q}$. 
Therefore  it  associates a \underline{canonical cluster transformation} to  the cluster variety ${\cal X}$, which is independent of the choice of coordinate system ${\bf c}_{\bf q}$.

\bc The cluster {\rm DT}-transformation ${\bf K}:{\bf q}\ra {\bf q}$ is in the  center of the cluster modular group $\Gamma_{{\bf q}}$.
\ec

\begin{proof} Let $\sigma \in \Gamma_{\bf q}$. By Theorem \ref{universal.dt},  $\sigma\circ {\bf K}\circ \sigma^{-1}$ is a DT-transformation. By the uniqueness of a DT-transformation,  $\sigma\circ {\bf K}\circ \sigma^{-1}={\bf K}$. 
\end{proof}

\subsection{The isomorphism $i$, the contravariant functor $F$, and  DT-transformations.}
\la{iandF}
We recall the isomorphism $i$ following \cite[Sect.3.2]{FG4}. It gives rise to a contravariant functor $F$.
\vskip 2mm

Let $-{\bf q}$ be the quiver obtained by 
reversing the sign of the form in 
the quiver ${\bf q}$. 
Equivalently, it is obtained by reversing the arrows in the geometric quiver ${\bf q}$. We use the notation 
$(u,v)_\circ:= -(u,v)$ for the form. 
The quantum torus algebra ${\bf T}_{-{\bf q}}$ has generators $X_u^\circ$, $u\in \Lambda$ satisfying
\be
\la{relation.opposite.torus.ii}
X_u^\circ X_v^\circ = q^{(u,v)_\circ} X_{u+v}^\circ.
\ee 
There is a natural ``antilinear" isomorphism 
\be
\la{anti.iso,quantum.i}
i^*: {\bf T}_{-{\bf q}} \lra {\bf T}_{\bf q}, \hskip 5mm X_v^\circ\lms X_{-v},~~ q\lms q^{-1}. 
\ee
Indeed, the isomorphism $i^*$ sends the relation \eqref{relation.opposite.torus.ii} to 
$X_{-u}X_{-v}= q^{(-u, -v)}X_{-u-v}$:
$$ 
i^*(X^{\circ}_uX^{\circ}_v) = X_{-u}X_{-v}, ~~~~
i^*(q^{(u,v)_\circ}X^\circ_{u+v}) =
q^{(u,v)}X_{-u-v} = q^{(-u,-v)}X_{-u-v}.
$$

\bl \la{Lemma3.8}
\la{compatiblity.of.i.hhh}
The isomorphism $i$ commutes with the quantum cluster transformations.
\el

\begin{proof}
Clearly $i$ commutes with the permutations of basis. It suffices to show that $i$ commutes with the mutations.

Recall the isomorphisms \eqref{Imut1}, \eqref{Imut}.
The mutation $\mu_k: -{\bf q} \to -{\bf q}'$  acts on the basis vectors in the same way as the negative mutation $\mu_k^-: {\bf q}\to {\bf q}'$. So  the following diagram commutes
\begin{displaymath}
    \xymatrix{
        {\bf T}_{-{\bf q}'} \ar[r]^{i^*}  \ar[d]_{i_{-{\bf q}\to -{\bf q}'}} &   {\bf T}_{\bf q'} \ar[d]^{{{i_{{\bf q}\to {\bf q}'}^-}}}  \\
        {\bf T}_{-{\bf q} } \ar[r]^{i^*}       & {  \bf T}_{\bf q} }
\end{displaymath}
By (\ref{1ssw}), we get
\be
i^* ({\bf \Psi}_{q}(X^\circ_v)) ={\bf \Psi}_{q^{-1}} (X_{-v}) ={\bf \Psi}_{q} (X_{-v})^{-1}
\ee
Therefore the following diagram commutes
\begin{displaymath}
    \xymatrix{
        {\Bbb T}_{-{\bf q}} \ar[r]^{i^*}  \ar[d]_{{\rm Ad}_{{\bf \Psi}(X^\circ_{e_k})}} &   {\Bbb T}_{\bf q} \ar[d]^{{\rm Ad}_{{\bf \Psi}(X_{-e_k})^{-1}}}  
        \\
        {\Bbb T}_{-{\bf q} } \ar[r]^{i^*}       & {  \Bbb T}_{\bf q} }
\end{displaymath}
By \eqref{3.18.15.1}, one can combine the above two diagrams, getting the following commutative diagram:
\begin{displaymath}
    \xymatrix{
        {\Bbb T}_{-{\bf q'}} \ar[r]^{i^*}  \ar[d]_{\Phi(\mu_k)} &   {\Bbb T}_{\bf q'} \ar[d]^{\Phi(\mu_k)}  
        \\
        {\Bbb T}_{-{\bf q} } \ar[r]^{i^*}       & {  \Bbb T}_{\bf q} }
\end{displaymath}
\end{proof}

Thanks to Lemma \ref{compatiblity.of.i.hhh}, we get a functor $I: {\cal G}_{\bf q}\to {\cal G}_{-\bf q}$, which assigns to a quiver ${\bf q}$ the quiver $-{\bf q}$, and to a (per)-mutation $\sigma: {\bf q}\ra {\bf q'}$ the one $\sigma: -{\bf q}\ra -{\bf q'}$. 

\vskip 2mm

It is useful to introduce a ``contravariant" version of the functor $I$.

\bd The contravariant functor $F: {\cal G}_{\bf q}\lra {\cal G}_{-\bf q}$ assigns to quiver ${\bf q}$ the quiver $-{\bf q}$, and to a cluster transformation $\sigma: {\bf q}\to {\bf q'}$ the one $I(\sigma^{-1}): -{\bf q'}\to -{\bf q}$. 
\ed

This allows to us to state the following result, which we prove in Section \ref{sec10}.
\bt \la{DT.contrav.2.1.hh}
A cluster transformation ${\bf K}$ is a cluster {\rm DT}-transformation if and only if $F({\bf K})$ is a cluster {\rm DT}-transformation.
\et

\subsection{{\rm DT}-transformations of cluster varieties and Duality Conjectures} \la{SecAA}

Definition \ref{Donaldson.Thomas.transformation} 
of  cluster {\rm DT}-transformations looks mysterious: 
it refers to a particular cluster Poisson coordinate system, and uses 
 positive and negative basic laminations 
in this coordinate system, which seem  out of the blue. 
Independence  
of a cluster coordinate system looks surprising.

A {\rm DT}-transformation of a cluster variety is always defined as a formal 
automorphism. 
However even if it is rational, it may not be a cluster transformation, 
even in the most basic cases, for example when $\G = PGL_2$ 
and $\bS$ is a surface of positive genus  with a single puncture. 

We suggest, using formal Duality Conjectures, 
 a conjectural property of  {\rm DT}-transformations of  cluster varieties 
which characterizes them uniquely, makes their crucial properties obvious, 
and in the  case when it is a cluster transformation implies immediately that it is the 
cluster {\rm DT}-transformation. 

We believe that this is the ``right'' definition of {\rm DT}-transformations of  cluster varieties, 
while Definition \ref{Donaldson.Thomas.transformation} is a convenient technical
characterization of those {\rm DT}-transformations of  cluster varieties which are 
cluster transformations. 

We formulate a conjecture  relating rationality of DT-transformations to existence of canonical bases in the space of regular functions on cluster varieties. 

Let us recall first  some features of  Duality Conjectures.

\paragraph{Duality Conjectures \cite[Section 4]{FG2}.} 
For any cluster variety 
${\cal Y}$, a {\it regular  function on ${\cal Y}$ }
is a function which, 
in any cluster coordinate system,  is a Laurent polynomial in  the cluster coordinates 
with positive integral coefficients. We denote by ${\cal O}({\cal Y})$ the algebra of regular
 functions on  ${\cal Y}$.

A {\it formal function on ${\cal Y}$} assigns to 
each cluster coordinate system  a Laurent  series 
with  integral coefficients 
in the cluster coordinates, related by  the cluster transformations. 
We denote by $\widehat {\cal O}({\cal Y})$ the algebra of formal functions on  ${\cal Y}$.
There is a canonical map 
\be \la{phimap}
\varphi: {\cal O}({\cal Y}) \lra \widehat {\cal O}({\cal Y}).
\ee

A  quiver gives rise to a dual pair of cluster 
varieties of the same dimension: a $K_2$-cluster variety ${\cal A}$, and a cluster Poisson variety 
${\cal X}$, as well as the Langlands dual cluster varieties ${\cal A}^\vee$ and 
 ${\cal X}^\vee$ \cite{FG2}. 
In the ``simpli-laced'' case 
${\cal A}^\vee = {\cal A}$ and ${\cal X}^\vee = {\cal X}$. The cluster modular group $\Gamma$ 
 acts  by their automorphisms.

The algebra ${\cal O}({\cal A})$ is closely related to the 
cluster algebra. By the Laurent phenomenon theorem \cite{FZ} the algebra 
${\cal O}({\cal A})$ is ``big'': every cluster coordinate $A_i$ lies in 
the ${\cal O}({\cal A})$. So the dimension of the spectrum of the algebra ${\cal O}({\cal A})$
 equals to the  dimension of  ${\cal A}$. 
As was shown in \cite{GHK}, the algebra 
${\cal O}({\cal X})$ could have smaller dimension then  ${\cal X}$. 
Yet for generic quiver, e.g. with a non-degenerate form $(\ast, \ast)$ on the lattice,  
the algebra ${\cal O}({\cal X})$ is ``big''.

Duality Conjecture \cite[Section 4]{FG2} predict  a  duality between   
cluster varieties ${\cal A}$ and ${\cal X}^\vee$. In particular, one should have 
canonical $\Gamma$-equivariant pairings 
\be \la{DCIa}
\begin{split}
&{\bf I}_{\cal A}: {\cal A}(\Z^t) \times {\cal X}^\vee \lra {\Bbb A}^1,\\ 
&{\bf I}_{\cal X}: {\cal A} \times {\cal X}^\vee(\Z^t) \lra {\Bbb A}^1. \\
\end{split}
\ee 
This means that each  $l \in {\cal A}(\Z^t)$ and each $m \in {\cal X}^\vee(\Z^t)$ give rise to 
functions  
$$
{\Bbb I}_{\cal A}(l):= {\bf I}_{\cal A}(l, \ast) ~~\mbox{on ${\cal X}^\vee$, and }
 ~~~~{\Bbb I}_{\cal X}(m):= {\bf I}_{\cal X}(m, \ast)~~\mbox{on ${\cal A}^\vee$}.
$$  
Since we can consider either formal or regular functions, 
there are two kinds of canonical pairings. 
In the formal setting we 
should have canonical $\Gamma$-equivariant maps 
\be \la{DCIII}
\begin{split}
&{\Bbb I}_{\cal A}: {\cal A}(\Z^t) \lra  \widehat {\cal O}({\cal X}^\vee), \\
&{\Bbb I}_{\cal X}: {\cal X}(\Z^t) \lra  \widehat {\cal O}({\cal A}^\vee).
\end{split}
\ee
In a quite general setting, pairings (\ref{DCIa}) should produce 
 $\Gamma$-equivariant maps to regular functions: 
\be \la{DCIIIa}
\begin{split}
&{\Bbb I}_{\cal A}: {\cal A}(\Z^t) \lra   {\cal O}({\cal X}^\vee), \\
&{\Bbb I}_{\cal X}: {\cal X}(\Z^t) \lra   {\cal O}({\cal A}^\vee).
\end{split}
\ee
Being composed with the embedding (\ref{phimap}), they produce the maps (\ref{DCIII}). 

The main feature of the maps (\ref{DCIII}) / (\ref{DCIIIa}) is that they should parametrise canonical 
linear bases in the corresponding space of  functions on the target space. 
Below we discuss two 
properties of canonical maps (\ref{DCIII}) / (\ref{DCIIIa}) relevant to our story. 

\paragraph{1. Positive tropical points and cluster algebras 
\cite[Conjecture 4.1, part 2)]{FG2}.}
The first basic property is this. If a tropical point $l\in {\cal X}(\Z^t)$ has non-negative coordinates $(l_1, ..., l_N) \in \Z^N_{\geq 0}$  in a cluster coordinate system 
assigned to a quiver ${\bf q}$, then the function ${\Bbb I}_{\cal X}(l)$ on ${\cal A}^\vee$  is  a monomial in the 
cluster ${\cal A}$-coordinates assigned to the same quiver:
\be \la{PTPCAM}
{\Bbb I}_{\cal X}(l) = \prod_{i\in{\rm I}}A_i^{l_i}.
\ee
By the Laurent Phenomenon theorem  \cite{FZ}, one has 
${\Bbb I}_{\cal X}(l)\in {\cal O}({\cal A}^\vee)$. By the very definition, the functions 
${\Bbb I}_{\cal X}(l)$ generate the cluster algebra related to  ${\cal A}^\vee$.

This immediately implies Proposition \ref{BLDC}.

\paragraph{Proof of Proposition \ref{BLDC}.} 
The cluster transformation of quivers $\sigma^{-1}_1\sigma_2: {\bf q} \to {\bf q}$ 
acts identically on the basic positive laminations. So by (\ref{PTPCAM}), it acts 
as the identity  
on the cluster algebra. It preserves the canonical 2-form on the 
spectrum of cluster algebra:  
$$
\Omega = \sum_{i,j \in {\rm I}} (e_i, e_j)d\log (A_i) \wedge d\log (A_j). 
$$ 
Therefore 
$\sigma^{-1}_1\sigma_2$ preserves the form $(\ast, \ast)$. 
This means that the map $\sigma^{-1}_1\sigma_2: {\bf q} \to {\bf q}$ is an isomorphism of quivers.  
Since it acts as the identity on the cluster coordinates,  it acts as the identity on the set ${\cal A}^\vee(\Z^t)$.  
Therefore, thanks to the formal Dulaity Conjecture,  it acts as the identity on the canonical formal basis on ${\cal X}$. Therefore it is the identity map of ${\cal X}$. 
The claim that the corresponding quantum cluster transformation is also 
the identity follows then by using arguments from \cite{FG4}. Alternatively, one can just use the quantum formal Dulaity Conjecture.

\paragraph{2. The parametrization of canonical bases \cite[Conjecture 4.1, part 1)]{FG2}.} This is the second basic property. 
It  tells how to 
recover the integral tropical point 
$l\in {\cal A}(\Z^t)$ parametrizing a canonical basis vector $F$ on ${\cal X}^\vee$. In the cluster 
coordinate system assigned to a quiver ${\bf q}$, the $F$ is given by 
 a Laurent polynomial / series $F_{\bf q}(X_1, ..., X_n)$. Let us right 
it as 
$$
F_{\bf q}(X_1, ..., X_n) = \prod_{i\in {\rm I}}X_i^{a_i} + \mbox{lower order terms}. 
$$
Then the exponents  $(a_1, ..., a_n)$ are the coordinates of a tropical 
point $l\in {\cal A}(\Z^t)$ in the cluster coordinate system assigned to the quiver ${\bf q}$. 
In other words, the exponents of the upper term of $F_{\bf q}(X_1, ..., X_n)$ change under the 
cluster transformations as the coordinates of an integral tropical point of ${\cal A}$. 
 This way one should get a bijection between the canonical basis 
elements and the set $l\in {\cal A}(\Z^t)$. This is the ``upper'' parametrisation 
of the canonical basis 
on ${\cal X}^\vee$.

\paragraph{The involutions $i_{\cal A}$ and $i_{\cal X}$ \cite[Lemma 3.5]{FG4}.} 
There are isomorphisms of cluster varieties 
\be \la{12.4.15.102}
i_{\cal A}: {\cal A}\lra {\cal A}^\circ, ~~~~
i_{\cal X}: {\cal X}\lra {\cal X}^\circ
\ee
which in any cluster coordinate system act as follows:
\be \la{12.4.15.101}
i^*_{\cal A}: A^\circ_i \lms A_i, ~~~~i^*_{\cal X}: X^\circ_i \lms X_i^{-1}.
\ee
The maps  (\ref{12.4.15.101}) are compatible with mutations, which means that they define
isomorphisms (\ref{12.4.15.102}). They are also compatible 
with the canonical projection 
$p: {\cal A}\lra {\cal X}$. 
By Lemma \ref{Lemma3.8}, the map $i_{\cal X}$ 
is the classical limit of an isomorphism of quantum cluster varieties
\be \la{12.4.15.21}
i_{\cal X}: {\cal X}_q \lra {\cal X}^{\circ}_{q}
\ee
which in any cluster coordinate system is given by an ``antilinear'' isomorphism 
of $\ast$-algebras
\be \la{12.4.15.20}
i^*_{\cal X}: {\cal O}_{q}({\cal X}^{\circ}) \lra  {\cal O}_{q}({\cal X}), 
~~~~i^*_{\cal X}(X^\circ_i) =  X^{-1}_i, ~~i^*_{\cal X}(q) = q^{-1}. 
\ee


\paragraph{The lower parametrization of canonical bases.} 
One can also parametrize canonical basis elements $F$  on ${\cal X}^\vee$ 
by the exponents of the lowest term  by writing 
$$
F_{\bf q}(X_1, ..., X_n) = \prod_{i\in {\rm I}}X_i^{b_i} + \mbox{higher order terms}. 
$$
Namely, assigning to $F$ the exponents $(b_1, ..., b_n)$  one should get a   
well defined integral tropical point of ${\cal A}$. 
This is the ``lower'' parametrization of the canonical basis 
on ${\cal X}^\vee$. 

\bl
The  existence of lower parametrization follows from  
 the existence of the upper. 

\el

\begin{proof} 
The isomorphism $i_{\cal X}$ transforms a canonical basis on ${\cal X}^\vee$ 
to a canonical basis on ${{\cal X}^\vee}^\circ$. 
Evidently, in any cluster coordinate system one has 
$$
\mbox{The lower term of $i^*_{{\cal X}^\vee}(F)$}  = \mbox{The upper term of $F$}.
$$
\end{proof}

The  
duality between  cluster varieties ${\cal A}$ and ${\cal X}^\vee$ is \underline{not compatible} with the 
isomorphisms $i_{\cal A}$ and $i_{{\cal X}^\vee}$. Conjecture \ref{MCDTTR} suggests 
that the {\rm DT}-transformations ${\rm DT}_{\cal X}$ and 
${\rm DT}_{\cal A}$  of  the cluster varieties ${\cal A}$ and  ${\cal X}$ tell the failure of the 
isomorphisms $i_{\cal A}$ and $i_{\cal X}$ to be compatible with the duality. Precisely, set
\be \la{Tildeiso}
{\rm D}_{\cal A}:=  i_{\cal A}\circ {\rm DT}_{{\cal A}}, ~~~
{\rm D}_{\cal X}:= i_{\cal X} \circ {\rm DT}_{\cal X}.
\ee
The duality should intertwine  
${\rm D}_{\cal A}$ with  $i_{{\cal X}^\vee}$, and  
$i_{\cal A}$ with  ${\rm D}_{{\cal X}^\vee}$. So, very schematically, 
we should have diagrams
\be
\begin{array}{cccccccccccc}
{\cal A}&\stackrel{}{\longleftrightarrow}& {\cal X}^\vee 
&&&&&&{\cal A}&\stackrel{}{\longleftrightarrow}& {\cal X}^\vee\\
&&&&&&&&\\
{\rm D}_{\cal A}\downarrow  &&\downarrow i_{{\cal X}^\vee}&&&&&&
i_{\cal A}\downarrow  &&\downarrow {\rm D}_{{\cal X}^\vee}\\
&&&&&&&&\\
{\cal A}^\circ& \stackrel{}{\longleftrightarrow}&{{\cal X}^\vee}^\circ&&&&&&
{\cal A}^\circ& \stackrel{}{\longleftrightarrow}&{{\cal X}^\vee}^\circ
\end{array}
\ee
They become commutative diagrams when one of the columns 
is tropicalised, and the other is replaced by the induced map 
of algebras of functions. 
So there are four commutative diagrams. The horizontal arrows are the canonical maps, 
going in the direction ``from the tropical column''. 
Let us state this precisely. 


\bcon \la{MCDTTR} Let $({\cal A}, {\cal X})$  be a  dual 
pair of cluster varieties satisfying  formal Duality Conjectures. Then

i) There are commutative diagrams
\be \la{COMMD1a}
\begin{array}{cccccccccccc}
{\cal X}(\Z^t)&\stackrel{{\Bbb I}_{\cal X}}{\lra}& \widehat {\cal O}({\cal A}^\vee) 
&&&&&&
{\cal A}(\Z^t)&\stackrel{{\Bbb I}_{\cal A}}{\lra}& \widehat {\cal O}({\cal X}^\vee)\\
&&&&&&&&&&\\
i^t_{\cal X}\downarrow  &&\downarrow {\rm D}^*_{{{\cal A}^\vee}^\circ}&&&&&&
i^t_{\cal A}\downarrow  &&\downarrow {\rm D}^*_{{{\cal X}^\vee}^\circ}\\
&&&&&&&&&&\\
{\cal X}^\circ(\Z^t)& \stackrel{{\Bbb I}_{{\cal X}^\circ}}{\lra}&\widehat {\cal O}({{\cal A}^\vee}^\circ)
&&&&&&
{\cal A}^\circ(\Z^t)& \stackrel{{\Bbb I}_{{\cal A}^\circ}}{\lra}&\widehat {\cal O}({{\cal X}^\vee}^\circ)
\end{array}
\ee

ii) Assume that the 
transformations ${\rm DT}_{\cal A}$ and 
${\rm DT}_{\cal X}$ are  positive rational maps, e.g. cluster transformations, so
 their tropicalisations ${\rm DT}^t_{\cal A}$ and ${\rm DT}^t_{\cal X}$ are defined. Then they have the following properties: 
\begin{itemize}

\item The canonical pairings are ${\rm DT}$-equivariant, that is
\be\la{CANP1a}
\begin{split}
&{\bf I}_{\cal A}: {\cal A}(\Z^t) \times {\cal X}^\vee \lra {\Bbb A}^1, ~~~~
{\bf I}_{\cal A}({\rm DT}_{\cal A}^t(a), {\rm DT}_{{\cal X}^\vee}(x)) = 
{\rm I}_{\cal A}(a, x),\\
&{\bf I}_{{\cal X}^\vee}: {\cal A} \times {\cal X}^\vee(\Z^t)\lra {\Bbb A}^1, ~~~~
{\bf I}_{{\cal X}^\vee}({\rm DT}_{\cal A}(x), {\rm DT}^t_{{\cal X}^\vee}(x)) = 
{\rm I}_{{\cal X}^\vee}(a, x).\\
\end{split}
\ee

\item Recall the maps   (\ref{Tildeiso}).  
Then there are commutative diagrams
\be \la{COMMD1}
\begin{array}{cccccccccccc}
{\cal X}(\Z^t)&\stackrel{{\Bbb I}_{\cal X}}{\lra}& \widehat {\cal O}({\cal A}^\vee) 
&&&&&&
{\cal A}(\Z^t)&\stackrel{{\Bbb I}_{\cal A}}{\lra}& \widehat {\cal O}({\cal X}^\vee)\\
&&&&&&&&&&\\
{\rm D}^t_{\cal X}\downarrow  &&\downarrow i^*_{{{\cal A}^\vee}^\circ}&&&&&&
{\rm D}^t_{\cal A}\downarrow  &&\downarrow i^*_{{{\cal X}^\vee}^\circ}\\
&&&&&&&&&&\\
{\cal X}^\circ(\Z^t)& \stackrel{{\Bbb I}_{{\cal X}^\circ}}{\lra}&\widehat {\cal O}({{\cal A}^\vee}^\circ)
&&&&&&
{\cal A}^\circ(\Z^t)& \stackrel{{\Bbb I}_{{\cal A}^\circ}}{\lra}&\widehat {\cal O}({{\cal X}^\vee}^\circ)
\end{array}
\ee

\end{itemize}
\econ

Few comments are in order.

\begin{enumerate}
\item
The right commutative diagram in (\ref{COMMD1a}) just means that, for any element $F$ of the 
canonical basis,  
the upper parametrization of ${\rm DT}^*_{\cal X}(F)$ = the 
lower parametrization of $F$.

It tells that the canonical basis on ${\cal X}^\vee$ is essentially\footnote{``Essentially'' 
reflect the fact that they are canonical bases on different spaces.}  invariant under the involution 
${\rm D}_{\cal X}$. 
\item The transformation ${\rm DT}_{\cal X}$ satisfies the property  
characterizing  \underline {cluster} DT-transformations: 
\be \la{GSMPDTCL}
{\rm DT}_{{\cal X}}^t(l_i^+)= l_i^-. 
\ee
Indeed, since  $i_{\cal A}^*(A_i) = A_i$, 
we have, using the left  diagram (\ref{COMMD1}):
 $$
{\rm D}_{\cal X}^t(l_i^+) = i_{\cal X}^t \circ {\rm DT}_{{\cal X}}^t(l_i^+)= l_i^+. 
$$
Applying to this the map $i_{\cal X^\circ}^t$, and using $i_{\cal X^\circ}\circ i_{\cal X} = {\rm Id}$ and 
$i_{\cal X}^t(l_i^+) = l_i^-$, we get (\ref{GSMPDTCL}). 

So if the transformation ${\rm DT}_{{\cal X}}$ is cluster, it is  the 
cluster DT-transformation.

\item The {\rm DT}-equivariance  (\ref{CANP1a})  can be stated as follows:
\be \la{EQDTX}
\begin{split}
&{\Bbb I}_{\cal A}({\rm DT}_{\cal A}^t(a)) = {\rm DT}_{{\cal X}^\vee}^*({\Bbb I}_{\cal A}(a)), ~~~
\mbox{\rm i.e.}~~~{\Bbb I}_{\cal A}\circ {\rm DT}_{\cal A}^t = 
{\rm DT}_{{\cal X}^\vee}^*\circ {\Bbb I}_{\cal A}.\\
&{\Bbb I}_{{\cal X}^\vee}({\rm DT}_{{\cal X}^\vee}^t(x)) = {\rm DT}_{\cal A}^*({\Bbb I}_{{\cal X}^\vee}(x)), ~~~
\mbox{\rm i.e.}~~~{\Bbb I}_{{\cal X}^\vee}\circ {\rm DT}_{{\cal X}^\vee}^t = {\rm DT}_{\cal A}^*\circ 
{\Bbb I}_{{\cal X}^\vee}.
\end{split}
\ee

\item The transformation    ${\rm DT}_{\cal A}$ 
is uniquely determined by (\ref{GSMPDTCL}) and (\ref{EQDTX}), since its  action 
on the cluster coordinates is determined by these conditions. 

\item Commutative diagrams (\ref{COMMD1}) plus (\ref{EQDTX}) imply  the commutative diagrams 
(\ref{COMMD1a}). 
Indeed, the  maps $i_{\cal X}$ and $i_{\cal A}$ are involutive, in the sense that 
$$
i_{\cal X^\circ}\circ i_{\cal X} = {\rm Id}_{\cal X}, ~~~~i_{\cal A^\circ}\circ i_{\cal A} = {\rm Id}_{\cal A}.
$$
Diagrams (\ref{COMMD1}) commute, so 
the  maps ${\rm D}^t_{\cal X}$ and ${\rm D}^t_{\cal A}$ are involutive in the same sense. 
Thus  
\be\la{2.8.16.1}
i_{{\cal X}}^t = (i^t_{\cal X^\circ})^{-1} = {\rm DT}^t_{{\cal X}^\circ}\circ i_{\cal X}^t\circ {\rm DT}^t_{\cal X}.
\ee 
Using this, we have:
\be
\begin{split}
&{\Bbb I}_{{{\cal X}}^\circ}\circ i_{{\cal X}}^t \stackrel{(\ref{2.8.16.1})}{=} 
{\Bbb I}_{{\cal X}^\circ}\circ {\rm DT}^t_{{\cal X}^\circ}\circ i_{\cal X}^t\circ {\rm DT}^t_{\cal X} \stackrel{(\ref{EQDTX})}{=} 
{\rm DT}_{{{\cal A}^\vee}^\circ}^*\circ {\Bbb I}_{{\cal X}^\circ} \circ  i_{\cal X}^t\circ {\rm DT}^t_{\cal X} \stackrel{(\ref{Tildeiso})}{=}\\
&{\rm DT}_{{{\cal A}^\vee}^\circ}^*\circ {\Bbb I}_{{\cal X}^\circ} \circ  
{\rm D}_{\cal X}^t \stackrel{(\ref{COMMD1})}{=}
{\rm DT}_{{{\cal A}^\vee}^\circ}^*\circ i^*_{{{\cal A}^\vee}^\circ}\circ {\Bbb I}_{{\cal X}} \stackrel{(\ref{Tildeiso})}{=}
{\rm D}_{{{\cal A}^\vee}^\circ}^*\circ  {\Bbb I}_{{{\cal X}}}. 
\end{split}
\ee
The argument for the second diagram is similar. 

\item Since the  maps $i_{\cal X}$ and $i_{\cal A}$ are involutive 
and diagrams (\ref{COMMD1a}) commute, the maps ${\rm D}_{\cal X}$ and ${\rm D}_{\cal A}$ must be involutive: 
\be \la{Dinv}
{\rm D}_{{\cal X}^\circ} \circ {\rm D}_{\cal X} = {\rm Id}_{\cal X}, ~~~~
{\rm D}_{{\cal A}^\circ}\circ {\rm D}_{\cal A} = {\rm Id}_{\cal A}.
\ee

\item Using ${\rm DT}_{{\cal X}, {\rm cl}}\circ i_{\cal X}$ 
instead of ${\rm D}_{\cal X}$ in the left diagram (\ref{COMMD1}) 
would not make the diagram  commute for basic positive laminations. 
So we have no choice but to use the tropicalised 
$i_{\cal X}\circ {\rm DT}_{\cal X}$ in the left diagram (\ref{COMMD1}) rather than the one 
${\rm DT}_{\cal X}\circ i_{\cal X}$. 

 \end{enumerate}

\paragraph{DT-transformations and Duality Conjectures for the double.} Recall the  cluster variety ${\cal A}_{\rm prin}$, see the end of Section 2.   
The algebra of regular functions ${\cal O}({\cal A}_{\rm prin})$ is the upper cluster algebra with principal coefficients 
\cite{FZIV}. The cluster variety  ${\cal A}_{\rm prin}$ contains the cluster variety ${\cal A}$ and projects onto the cluster Poisson variety ${\cal X}$:
$$
{\cal A} \stackrel{j}{\hra} {\cal A}_{\rm prin} \stackrel{\pi}{\lra} {\cal X}.
$$
There is a canonical involution, compatible with the involutions $i_{\cal A}$ and $i_{\cal X}$ in the obvious way:\footnote{We use the subscript ${\cal P}$ - "principal" - 
for maps related to the cluster variety ${\cal A}_{\rm prin}$, the subscript ${\cal P}^\vee$ for the ${\cal A}^\vee_{\rm prin}$, etc.}
$$
i_{\cal P}: {\cal A}_{\rm prin}   \lra  {\cal A}_{\rm prin}. 
$$

Duality Conjectures  can be casted as  a  duality between   
cluster varieties $ {\cal A}_{\rm prin}$ and $ {\cal A}_{\rm prin}^\vee$. In the "simply laced" case, which we mostly focus on in this paper, 
${\cal A}_{\rm prin}^\vee = {\cal A}_{\rm prin}$. 

In particular, one should have 
canonical $\Gamma$-equivariant pairings 
\be \la{DCIaa}
\begin{split}
&{\bf I}_{\cal P}: {\cal A}_{\rm prin}(\Z^t) \times {\cal A}_{\rm prin}^\vee \lra {\Bbb A}^1.
\end{split}
\ee 
This means that each  $l \in {\cal A}_{\rm prin}(\Z^t)$ give rise to 
functions  
$$
{\Bbb I}_{\cal P}(l):= {\bf I}_{\cal P}(l, \ast) ~~\mbox{on $ {\cal A}_{\rm prin}^\vee$}.
$$  

In the formal setting we 
should have canonical $\Gamma$-equivariant maps 
\be \la{DCIIIaa}
\begin{split}
&{\Bbb I}_{\cal P}:  {\cal A}_{\rm prin}(\Z^t) \lra  \widehat {\cal O}( {\cal A}_{\rm prin}^\vee).
\end{split}
\ee
Under certain assumptions, pairing (\ref{DCIaa}) should produce a 
 $\Gamma$-equivariant map to regular functions, which  should parametrise a canonical 
linear basis in the  space of  functions on the target:
\be \la{DCIIIb}
\begin{split}
&{\Bbb I}_{\cal P}:  {\cal A}_{\rm prin}(\Z^t) \lra   {\cal O}( {\cal A}_{\rm prin}^\vee). 
\end{split}
\ee

The duality   ${\cal A}_{\rm prin} \leftrightarrow{\cal A}_{\rm prin}^\vee$ is {not compatible} with the 
isomorphism $i_{\cal P}$. The {\rm DT}-transformation ${\rm DT}_{\cal P}$  tells the failure of the 
isomorphism $i_{{\cal P}^\vee}$   to be compatible with the duality. Precisely, set
\be \la{Tildeisoa}
{\rm D}_{\cal P}:=  i_{\cal P}\circ {\rm DT}_{\cal P}. 
\ee
Then the duality should intertwine  
${\rm D}_{\cal P}$ with  $i_{{\cal P}^\vee}$. So,  schematically, 
we should have a diagram
\begin{displaymath} \la{CDMS1}
    \xymatrix{
        {\cal A}_{\rm prin} \ar@{<->}[r] \ar[d]_{{\rm D}_{\cal P}} & {\cal A}_{\rm prin}^\vee   \ar[d]^{i_{{\cal P}^\vee}} \\
         {\cal A}_{\rm prin}^\circ  \ar@{<->}[r]      &  {\cal A}_{\rm prin}^{\vee\circ}}
         \end{displaymath}
It becomes a commutative diagram when one of the columns 
is tropicalised, and the other is replaced by the induced map 
of algebras of functions. The horizontal arrows are the canonical maps,  going ``from the tropical column''. 
Let us state this precisely. 

\bcon \la{MCDTTRa} Let ${\cal A}_{\rm prin}$  be a   cluster ${\cal A}$-variety with principal coefficients. Then 

i) The  formal transformation ${\rm D}_{{\cal P}^\vee}$, see (\ref{Tildeisoa}),  
 makes the following diagram commutative:

    \begin{equation}\label{CDMS1}
\begin{gathered}
    \xymatrix{
        {\cal A}_{\rm prin}(\Z^t) \ar[r]^{{\Bbb I}_{{\cal P}}}   \ar[d]_{{i}^t_{\cal P}} & \widehat {\cal O}({\cal A}_{\rm prin}^\vee)   \ar[d]^{{\rm D}_{{\cal P}^\vee}} \\
         {\cal A}_{\rm prin}^\circ(\Z^t)  \ar[r]^{{\Bbb I}_{{\cal P}^\circ}}     & \widehat {\cal O}({{\cal A}_{\rm prin}^{\vee^\circ}}) }
         \end{gathered}
         \end{equation}
         
    ii) Assume  
that  the DT-transformation ${\rm DT}_{\cal P}$ is a positive rational map, so that the tropicalised transformation ${\rm D}^t_{\cal P}$ is defined. Then:
\begin{itemize}

\item The canonical pairing is ${\rm DT}$-equivariant:
\be\la{CANP1aa}
\begin{split}
&{\bf I}_{\cal P}: {\cal A}_{\rm prin}(\Z^t) \times {\cal A}_{\rm prin}^\vee \lra {\Bbb A}^1, ~~~~
{\bf I}_{\cal P}({\rm DT}_{\cal P}^t(x), {\rm DT}_{{\cal P}^\vee}(y)) = 
{\rm I}_{\cal P}(x, y).\\ 
\end{split}
\ee

\item  
There is the second commutative diagram
  \begin{equation}\label{COMMD1aa}
\begin{gathered}
    \xymatrix{
        {\cal A}_{\rm prin}(\Z^t) \ar[r]^{{\Bbb I}_{{\cal P}}}   \ar[d]_{{\rm D}^t_{\cal P}} & \widehat {\cal O}({\cal A}_{\rm prin}^\vee)   \ar[d]^{i^*_{{{\cal P}^\vee}^\circ}} \\
        {\cal A}_{\rm prin}^\circ(\Z^t)  \ar[r]^{{\Bbb I}_{{\cal P}^\circ}}    & \widehat {\cal O}({\cal A}_{\rm prin}^{\vee\circ}) }
         \end{gathered}
         \end{equation}

\end{itemize}
\econ

Few comments are in order.

\begin{enumerate}

\item The {\rm DT}-equivariance  (\ref{CANP1aa})  can be stated as follows:
\be \la{EQDTXa}
\begin{split}
&{\Bbb I}_{\cal P}({\rm DT}_{\cal P}^t(x)) = {\rm DT}_{{\cal P}^\vee}^*({\Bbb I}_{\cal P}(x)), ~~~
\mbox{\rm i.e.}~~~{\Bbb I}_{\cal P}\circ {\rm DT}_{\cal P}^t = 
{\rm DT}_{{\cal P}^\vee}^*\circ {\Bbb I}_{\cal P}.\\ 
\end{split}
\ee

\item The transformation ${\rm DT}_{\cal P}$  
is uniquely determined by (\ref{EQDTXa}), since its  action 
on the canonical  basis is determined by this conditions.

\item Commutative diagram (\ref{COMMD1aa}) plus (\ref{EQDTXa}) imply that the   diagram  (\ref{CDMS1}) is commutative. 
Indeed, the  maps $i_{\cal P}$  are involutive, in the sense that 
$$
i_{\cal D^\circ}\circ i_{\cal P} = {\rm Id}_{\cal P}.
$$
Since diagram (\ref{COMMD1a}) is commutative, 
the  map ${\rm D}^t_{\cal P}$ is involutive in the same sense. 
Thus  
\be\la{2.8.16.1aa}
i_{{\cal P}}^t = (i^t_{\cal P^\circ})^{-1} = {\rm DT}^t_{{\cal P}^\circ}\circ i_{\cal P}^t\circ {\rm DT}^t_{\cal P}.
\ee 
Using this, we have:
\be
\begin{split}
&{\Bbb I}_{{{\cal P}}^\circ}\circ i_{{\cal P}}^t \stackrel{(\ref{2.8.16.1aa})}{=} 
{\Bbb I}_{{\cal P}^\circ}\circ {\rm DT}^t_{{\cal P}^\circ}\circ i_{\cal P}^t\circ {\rm DT}^t_{\cal P} \stackrel{(\ref{EQDTXa})}{=} 
{\rm DT}_{{{\cal P}^\vee}^\circ}^*\circ {\Bbb I}_{{\cal P}^\circ} \circ  i_{\cal X}^t\circ {\rm DT}^t_{\cal P} \stackrel{(\ref{Tildeisoa})}{=}\\
&{\rm DT}_{{{\cal P}^\vee}^\circ}^*\circ {\Bbb I}_{{\cal P}^\circ} \circ  
{\rm D}_{\cal P}^t \stackrel{(\ref{COMMD1a})}{=}
{\rm DT}_{{{\cal P}^\vee}^\circ}^*\circ i^*_{{{\cal P}^\vee}^\circ}\circ {\Bbb I}_{{\cal X}} \stackrel{(\ref{Tildeisoa})}{=}
{\rm D}_{{{\cal P}^\vee}^\circ}^*\circ  {\Bbb I}_{{{\cal P}}}. 
\end{split}
\ee

\item Since the  map $i_{\cal P}$ is involutive 
and diagram (\ref{COMMD1aa}) commute, the map ${\rm D}_{\cal P}$   must be involutive: 
\be \la{Dinv}
{\rm D}_{{\cal P}^\circ} \circ {\rm D}_{\cal P} = {\rm Id}_{\cal P}. 
\ee






\end{enumerate}

\paragraph{Rationality of  ${\rm DT}$-transformation  and regular canonical bases.}
\bl \la{2.9.16.1} Let us assume Conjecture \ref{MCDTTR}i). Then:

i) If the map ${\rm DT}_{\cal X}$ is  cluster, then the formal canonical basis consists of  Laurent polynomials.
 
 ii) The same  is true if ${\rm DT}_{\cal X}$ is a positive rational map, and the canonical pairing ${\bf I}_{\cal A}$ is {\rm DT}-equivariant, see (\ref{CANP1a}). 
   \el
 
 \begin{proof}  i) The duality map is compatible with cluster transformations. 
So if the ${\rm DT}$'s are cluster, then we can compose the vertical maps in the right diagram in (\ref{COMMD1a}) with the inverse of ${\rm DT}_{{\cal X}^{\vee}}$ on the right, and with 
${\rm DT}^t_{\cal A}$ on the left.  
The tropical side becomes  $i^t_{\cal A} \circ {\rm DT}^t_{\cal A}$. 
The function side becomes $i_{{\cal X}^\vee}$. Note that $i^t_{{\cal X}^\vee}$ takes lower order terms to upper order terms. 
So all the canonical basis element are bounded by top terms and bottom terms, i.e., they are all polynomials. 

ii) Same argument using commutativity of diagram (\ref{COMMD1a}) plus {\rm DT}-equivariance (\ref{CANP1a}). 
\end{proof}
 
 \bcon \la{2.9.16.2}
 The map ${\rm DT}_{\cal X}$ (respectively ${\rm DT}_{\cal A}$) is  rational if and only if the formal canonical basis in $\widehat {\cal O}({\cal X})$ 
 (respectively $\widehat {\cal O}({\cal A})$) is regular, i.e. lies  in ${\cal O}({\cal X})$ (respectively ${\cal O}({\cal A})$).  
  \econ
 
 Here are the arguments supporting Conjecture \ref{2.9.16.2}. 
 
 a) If the map ${\rm DT}_{\cal X}$ is a cluster DT-transformation, 
 then by Theorem \ref{Th1.17}, which uses \cite[Theorem 0.10]{GHKK}, we get a canonical basis in ${\cal O}({\cal X})$. 
  
  b) Lemma \ref{2.9.16.1} tells that   if ${\rm DT}_{\cal X}$ is positive rational then, assuming Conjecture \ref{MCDTTR}, 
  there is a canonical basis in ${\cal O}({\cal X})$. 
  
  c) A map  ${\cal A} \to {\cal A}$ is rational if and only if 
  it takes the cluster coordinates $A_i$ to rational functions. So Conjecture \ref{MCDTTR} implies that 
  if  a canonical basis in ${\cal O}({\cal A})$ exists, then the map ${\rm DT}_{\cal A}$ must be rational.

\paragraph{{\rm DT}-transformations 
 and the target vector spaces in Duality Conjectures.} 
Consider  the largest  subalgebras  on which the  powers of the transformation ${\rm DT}$ act:
$$
{\cal O}_{\rm DT}({\cal A}):= \{f \in {\cal O}({\cal A}) ~|~ 
{\rm DT}_{\cal A}^n(f) \in {\cal O}({\cal A}), \forall n \in \Z\}.
$$
$$
{\cal O}_{\rm DT}({\cal X}):= \{f \in {\cal O}({\cal X}) ~|~ 
{\rm DT}_{\cal X}^n(f) \in {\cal O}({\cal X}), \forall n \in \Z\}.
$$

We  can  
state now the enhanced version of Duality Conjectures for  cluster varieties. 

\bcon \la{X-dualityconj} Let $({\cal A}, {\cal X})$ 
be a dual pair of cluster varieties. 
Assume that the DT transformation ${\rm DT}_{\cal A}$ and ${\rm DT}_{\cal X}$ are rational. Then   
there is a  $\Gamma\times {\rm DT}$-equivariant mirror duality between the spaces 
${\cal A}$ and ${\cal X}^\vee$. 
In particular there are canonical $\Gamma\times {\rm DT}$-equivariant isomorphisms
\be \la{canpair1002a}
{\Bbb I}_{\cal A}: \Z[{\cal A}(\Z^t)] \stackrel{\sim}{\lra} {\cal O}_{\rm DT}({\cal X}^\vee), ~~~~
{\Bbb I}_{{\cal X}}: \Z[{\cal X}(\Z^t)] \stackrel{\sim}{\lra} {\cal O}_{\rm DT}({\cal A}^\vee).
\ee
\econ

The very existence of the   ${\rm DT}$-equivariant pairing ${\bf I}_{{\cal X}}$
 implies that the image of the map ${\Bbb I}_{{\cal X}}$
 lies in the subspace ${\cal O}_{\rm DT}({\cal A}^\vee)$.
 Indeed, if ${\Bbb I}_{\cal X}(l) \in {\cal O}({\cal A}^\vee)$, then 
by the ${\rm DT}$-equivariance,  
$$
{\Bbb I}_{\cal X}\Bigl(({\rm DT}^t_{\cal X})^n(l)\Bigr) = {\rm DT}_{{\cal A}^\vee}^n{\Bbb I}_{\cal X}(l),~~ ~~\forall n \in \Z. 
$$
It remains to notice that the left hand side always lies in ${\cal O}({\cal A})$. 

Conjecture \ref{X-dualityconj} is a cluster generalisation of Conjecture \ref{WGAC}. 
We want to stress that the  ${\cal O}_{\rm DT}({\cal A})$ could be smaller then 
 ${\cal O}({\cal A})$. 
Let us elaborate on this. 

One can associate to the moduli space 
${\cal A}_{\G, \bS}$ several {\it a priori} different algebras. One is the algebra ${\cal O}({\cal A}_{\G, \bS})$ of 
regular functions on the moduli space 
${\cal A}_{\G, \bS}$. The other  is the algebra ${\cal O}_{\rm cl}({\cal A}_{\G, \bS})$ 
of regular functions on the corresponding cluster ${\cal A}$-variety, which in this case is  just  
the corresponding upper cluster algebra. The third one is the algebra ${\cal O}_{\rm DT}({\cal A}_{\G, \bS})$.

\bp
i) Assume that a decorated surface $\bS$ has $>1$ punctures. 
Then the algebras ${\cal O}_{\rm DT}({\cal A}_{\G, \bS})$ and  ${\cal O}_{\rm cl}({\cal A}_{\G, \bS})$ 
are smaller then  ${\cal O}({\cal A}_{\G, \bS})$. 

ii) If  $\bS$ has $1$ puncture and no special points, then  ${\cal O}_{\rm DT}({\cal A}_{SL_2, \bS})$ 
is smaller then ${\cal O}_{\rm cl}({\cal A}_{SL_2, \bS})$. 
\ep

\begin{proof} 
i) The Weyl group $W$ acts in this case by cluster transformations. 
Therefore, by the very definition,  its action preserves the algebra ${\cal O}_{\rm cl}({\cal A})$.  
Recall the potential ${\cal W}_p$ at a puncture $p$ introduced in \cite{GS}, see also (\ref{potential}). 
it is a regular function 
on ${\cal A}_{\G, \bS}$. However for any nontrivial $w \in W$, the function $w^*{\cal W}_p$ is not regular 
on ${\cal A}_{\G, \bS}$. So we conclude that
$$
{\cal W}_p \in {\cal O}({\cal A}_{\G, \bS}), ~~~~ {\cal W}_p \not \in {\cal O}_{\rm cl}({\cal A}_{\G, \bS}), 
~~~~ {\cal W}_p \not \in {\cal O}_{\rm DT}({\cal A}_{\G, \bS}).
$$
The same is true for any partial potential ${\cal W}_{p, \alpha}$.

ii) 
Indeed, ${\cal W}_p \not \in {\cal O}_{\rm DT}({\cal A}_{SL_2, \bS})$, but 
${\cal W}_p \in {\cal O}_{\rm cl}({\cal A}_{SL_2, \bS})$. This also tells that 
the Weyl group action on ${\cal A}_{SL_2, \bS}$ is not cluster. 
\end{proof} 


 \section{Properties of cluster DT-transformations}
 
 In Section \ref{sssec2.4} we discuss a  special presentation of
  quantum cluster transformations provided by the sign-coherence of the C-matix.  
In Section \ref{S2Example} we elaborate the cluster DT-transformation and 
the quantum canonical basis for the cluster ${\cal X}$-variety of type ${\rm A}_2$, demonstrating their compatibility. 

In Section \ref{sQDT} we show that the unitary operator quantizing
the  cluster {\rm DT}-transformations leads to a $\Gamma$-equivariant bilinear form 
on the Hilbert space assigned to a cluster Poisson variety. 

In Section \ref{sec10} we prove some results on  cluster 
${\rm DT}$-transformations stated in Section \ref{SSec1.2}. 

\subsection{Sign-coherence and cluster transformations} \la{sssec2.4}

\bd Let $\sigma: {\bf q } \to {\bf q}'$ be a cluster transformation. 
The matrix  $C_\sigma$ is a matrix whose $j$-th column $(c_{1j}, c_{2j}, ..., c_{nj})^T$ is given by the 
coordinates of $l_{{\bf q}, j}^+$ in 
the coordinate system ${\bf c}_{{\bf q}'}$.
\ed
By Proposition \ref{9:14:04:1}, the matrix $C_\sigma$ coincides with 
the Fomin-Zelevinsky $C$-matrix \cite{FZIV}. 

It is easy to show that ${\rm det}(C_{\sigma})=\pm 1$. Therefore $C_{\sigma}$ is invertible. 
The matrix $C_\sigma$  has many other nice properties. 
The most important  one is the sign-coherence.

\paragraph{Sign-coherence.} 
Let ${c}_i=(c_{i1},\ldots, c_{in})$ be the $i$-th row vector of $C_{\sigma}$. 
\bt[\rm \cite{DWZ2}] \la{DWZ2} 
The entries of ${c}_i$ are either all non-negative, or all non-positive. 
\et
This allows to introduce the sign ${\rm sgn}(c_i)\in \{\pm 1\}$ as the sign of the entries of the row $c_i$. 

Let $\sigma: {\bf q} \to {\bf q}'$ be a cluster transformation. Let $\varepsilon_{{\bf q}'}=(\varepsilon_{ik}')$. Let $\mu_k: {\bf q}'\ra {\bf q}''$ be a mutation.
\bl  \la{lem.c.matrix1} 
 The $i$-th row vector ${c'_i}$ of $C_{\mu_k\circ \sigma}$ is
\be \la{lem.c.matrix00} 
{c_{i}'}:=\left\{ \begin{array}{ll}
-{c_{k}} &\mbox{ if $i=k$},\\
{c_{i}}+[{\rm sgn}({c}_k)\varepsilon_{ik}']_+ {c}_k& \mbox{ if $i\neq k$}.\\
   \end{array}
   \right.
\ee 
\el
\begin{proof}   Let $C_{\mu_k\circ \sigma}=(c_{ij}')$. By definition, the  $c_{ij}'$ are given by the tropicalization formula \eqref{5.11.03.1xtr}:
\be \la{lem.c.matrix00a}
c_{ij}':=\left\{ \begin{array}{ll}
-c_{kj} &\mbox{ if $i=k$},\\
c_{ij}-\varepsilon_{ik}'\min\{0, -{\rm sgn}(\varepsilon_{ik}')c_{kj}\}={c_{ij}}+[{\rm sgn}({c}_{kj})\varepsilon_{ik}']_+ {c}_{kj} & \mbox{ if $i\neq k$}.\\
   \end{array}
   \right.
\ee
The Lemma follows immediately.
\end{proof}

Let $\sigma: {\bf q} \to {\bf q}'$ 
be a  cluster transformation,
presented as a 
sequence of mutations followed by a  permutation:
\be \la{CLTR}
\sigma: {\bf q}= {\bf q}_0 \stackrel{k_1}{\lra} {\bf q}_{1} 
\stackrel{k_2}{\lra} \ldots \stackrel{k_{m}}{\lra} {\bf q}_{m} \stackrel{\pi}{\lra} {\bf q'}.
\ee
Following  \eqref{lem.c.matrix00}, we get a sequence of $C$-matrices
\be
{\rm Id}= C_0\stackrel{{k_1}}{\lra} C_1 \stackrel{{k_2}}{\lra} \ldots \stackrel{{k_{m}}}{\lra} {C}_{m} \stackrel{\pi}{\lra} {C_\sigma}.
\ee
Denote by $c_{i}^{(s)}$ the $i$-th row vector of $C_s$.
\bd
The canonical sequence $\varepsilon=(\varepsilon_1, \ldots, \varepsilon_{m})$ of signs 
for the cluster transformation (\ref{CLTR}) is given by 
$\varepsilon_{s}:= {\rm sgn}(c_{k_s}^{(s-1)})$.
\ed

Recall the half reflections $\mu_k$ and  $\mu_k^-$ of basis defined by \eqref{12.12.04.2a} and  
\eqref{12.12.04.2ab} respectively. 

The canonical sequence $\varepsilon=(\varepsilon_1, \ldots, \varepsilon_{m})$ of signs gives rise to 
the following composition of bases mutations and permutation assigned to the
 cluster transformation (\ref{CLTR}):
\[
\sigma^{(\varepsilon)}:= \pi\circ \mu_{k_{m}}^{\varepsilon_{m}} \circ \ldots \circ \mu_{k_1}^{ \varepsilon_1}.
\]
Let $\{e_i\}$ be a basis of $\Lambda$  
 defining the quiver ${\bf q}$. Then  $\sigma^{(\varepsilon)}$ transforms it 
to a new basis $\{e_i'\}$ of $\Lambda$, defining a quiver isomorphic to ${\bf q}'$:
\be \la{SEQV}
\{e_i\} = \{e^{(0)}_i\} \stackrel{\mu_{k_1}^{ \varepsilon_1}}{\lra} 
\{e^{(1)}_i\} \stackrel{\mu_{k_2}^{ \varepsilon_2}}{\lra} \ldots 
\stackrel{\mu_{k_{m}}^{ \varepsilon_{m}}}{\lra} \{e^{(m)}_i\} \stackrel{\pi}{\lra} \{e'_i\}
\ee
\bl \la{CMATL}
The matrix $C_\sigma= (c_{ij})$ expresses the vectors $\{e'_i\}$ via   $\{e_j\}$:
\be \la{cmatrixbasis1}
e'_i = \sum_{j \in { I}}c_{ij}e_j, \hskip 7mm \forall i\in {\rm I}.           
\ee
\el
\begin{proof} Since $\varepsilon_s = {\rm sgn}(c_{k_s}^{(s-1)})$, 
the half reflection $\mu_{k_s}^{\varepsilon_s}$ coincides with the transformation \eqref{lem.c.matrix00a}. 
\end{proof}

\bc \la{cmatrixbasis22}
The skew-symmetric matrix $\varepsilon_{{\bf q}'}$ of ${\bf q}'$ is 
given by $\varepsilon_{{\bf q}'} = C_\sigma \varepsilon_{{\bf q}} C_{\sigma}^T$.
\ec
\begin{proof} It follows directly from \eqref{cmatrixbasis1} and the definition of $\varepsilon_{\bf q}$.
\end{proof}
Corollary \ref{cmatrixbasis22} asserts that $C_{\sigma}$ determines the isomorphism class of ${\bf q}'$. In particular, if $\sigma$ is a cluster DT-transformation ($i.e., C_{\sigma}=-{\rm Id}$), then ${\bf q'}={\bf q}$.

\paragraph{A sign-coherent presentation of cluster transformations.} 
 Recall the vectors 
$f_s^\varepsilon$ and the quantum cluster transformation ${\Phi}({\bf i})$, 
see  (\ref{fvect1}) and \eqref{sepf1}:
\be \la{sepf2}
\Phi({\bf i}) = {\rm Ad}_{{\bf \Psi}(X_{f^{\varepsilon}_1})^{\varepsilon_1}}\circ 
\ldots \circ {\rm Ad}_{{\bf \Psi}(X_{f^{\varepsilon}_{m}})^{\varepsilon_{m}}}\circ 
i^\varepsilon({\bf i}).
\ee
We consider the positive cone generated by the basis $\{e_i\}$
\[
\Lambda^+:=\oplus_{i\in{\rm I}}~ {\Bbb Z}_{\geq 0} e_{i}.
\]
The sign-coherence  property of $C$-matrices from Theorem \ref{DWZ2} is equivalent to 
the following.
\bt
\la{sign.seq.hh} Given a quiver cluster transformation $\sigma: {\bf q} \to {\bf q}'$, 
for the canonical sequence  $\varepsilon$  of signs we have 
\be \la{SCPCT}
f_s^\varepsilon\in \Lambda^+, \hskip 7mm \forall s \in \{1,\ldots, m\}.
\ee 
This is the unique sequence of signs for which (\ref{SCPCT}) holds. 
\et

\begin{proof} By the definition (\ref{fvect1}) of the vectors $f_s^\varepsilon$, 
they  are exactly the vectors $\varepsilon_s e_{k_s}^{(s-1)}$ in 
(\ref{SEQV}). 
\end{proof}

Recall that $\widehat{\bf T}_{\bf q}$ is the algebra of $q$-commutating power series 
in the basis $\{X_{e_i}\}$. 
Since  for the canonical sequence of signs $\varepsilon$ 
each of the $X_{f^{\varepsilon}_{s}}$ is a monomial with non-negative exponents 
in this basis, it make sense to consider 
a product of formal power series:
\be
\la{formal.power.series}
{\bf \Psi}({\bf i}):={\bf \Psi}(X_{f^{\varepsilon}_1})^{\varepsilon_1}
\ldots {\bf \Psi}(X_{f^{\varepsilon}_{m}})^{\varepsilon_{m}}  \in \widehat{\bf T}_{\bf q}.
\ee

Then the quantum cluster transformation \eqref{sepf1} can be written as 
\be
\la{sepf2}
\Phi({\bf i}) ={\rm Ad}_{{\bf \Psi}({\bf i})} \circ i^\varepsilon({\bf i}).
\ee

\paragraph{Decompositions of cluster transformations.}
One can exchange permutations $\pi$ and quiver mutations:
$\pi \circ \mu_{k}=\mu_{\pi(k)} \circ \pi$. 
Therefore every quiver transformation $\sigma$ can be decomposed as 
\[
\sigma=  \pi_\sigma\circ {\bf i}_\sigma,
\]
where ${\bf i}_\sigma$ is a sequence of cluster mutations and $\pi_\sigma$ is a permutation. By \eqref{sepf2}, we can 
decompose the quantum cluster transformation $\Phi(\sigma)$ into two parts
\be
{\Phi}(\sigma):={\rm Ad}_{{\bf \Psi}({\bf i}_{\sigma})} \circ \Sigma_{\sigma}, \hskip 7mm \Sigma_{\sigma}:=i^\varepsilon({\bf i}_\sigma) \circ \Phi(\pi_\sigma)
\ee 


By Lemma \ref{CMATL}, the  $\Sigma_{\sigma}$ corresponds to the change of basis,  also 
encoded by the ${C}$-matrix.

\bt[\cite{K13}]
\la{sepe.thm.hh}
Let $\sigma_1, \sigma_2$ be two cluster transformations starting from the same quiver. 
\begin{enumerate}
\item The following  are equivalent
\be
\Phi(\sigma_1)=\Phi(\sigma_2) ~~~\Longleftrightarrow~~~ \Sigma_{\sigma_1} = \Sigma_{\sigma_2} ~~~\Longleftrightarrow~~~C_{\sigma_1}=C_{\sigma_2}.
\ee
\item If $\sigma_1\circ \sigma_2^{-1}$ is a  permutation, then the corresponding formal power series are the same
\be
\la{dilogper}
{\bf \Psi}({\bf i}_{\sigma_1})={\bf \Psi}({\bf i}_{\sigma_2})  
\ee
\end{enumerate}
\et

\begin{remark}
The first part of Theorem \ref{sepe.thm.hh} is a reformulation of 
Theorem \ref{basiclam}.  It asserts  that the ${C}$-matrices determine cluster transformations. 
The second part asserts that every cluster transformation $\sigma:{\bf q}\to {\bf q'}$ canonically determines a formal power series
\be {\bf \Psi}_{\sigma}=1+\mbox{higher order terms} ~~~~~ \in \widehat{\bf T}_{\bf q}
\ee
which does not depend on the decomposition of $\sigma$.

As an application, the $q\to 1$ limit of ${\rm Ad}_{{\bf \Psi}_{\sigma}}$ gives rise to  the ${\rm F}$-polynomials from \cite{FZIV}. Precisely, recall the
 double of the quantum torus algebra ${\bf T}_{\bf q}$ obtained by adding new generators $\{A_{e_i}\}_{i\in {\rm I}}$ satisfying 
\[
A_{e_i} A_{e_j} = A_{e_j} A_{e_i}; \hskip 7mm  X_{e_i} A_{e_j}  = \left\{ \begin{array}{ll}
q^{2} A_{e_j} X_{e_i} &\mbox{ if $i=j$},\\
A_{e_j} X_{e_i} & \mbox{ otherwise}.\\
   \end{array}
   \right.
\]
Let $
[{\bf \Psi}_{\sigma} , A_i]:={\bf \Psi}_{\sigma} A_i {\bf \Psi}_{\sigma} ^{-1} A_i^{-1}\in  \widehat{\bf T}_{\bf q}$.  
Then the  ${\rm F}$-polynomials associated to $\sigma$ are
\[
F_i = \lim_{q\to 1} [{\bf \Psi}_{\sigma} , A_i], \hskip 7mm i\in {\rm I}.
\]
\end{remark}

Now we can state the following crucial result.
 \bt[{\cite[Th.6.5]{K12}}] \la{dilogper22} If ${\bf K}$ is a cluster DT-transformation, then the formal power series ${\bf \Psi}_{\bf K}$ is the quantum Donaldson-Thomas series \eqref{refine.dt.invariant}.  
\et

\subsection{An example: quantum cluster variety for the quiver of type ${\rm A}_2$} \la{S2Example}
Let ${\bf q}= \big(\Lambda, \{e_1, e_2\}, (\ast, \ast)\big)$ be a rank 2 quiver with $(e_1, e_2)=1$. Then 
\[
X_{e_1}X_{e_2}= q X_{e_1+e_2}= q^{2} X_{e_2} X_{e_1}.
\]
\paragraph{DT-transformation and the quantum pentagon relation.} 
Let $\sigma_1 =\mu_2\circ \mu_1$. The canonical sequence of signs for $\sigma_1$ is  $\varepsilon=\{1, 1\}$. 
The $\Sigma_{\sigma_1}$ is determined by a change of basis:
\be
\{e_1, e_2\}\stackrel{\mu_1} {\lra}  \{-e_1, e_2\} \stackrel{\mu_2}{\lra} \{-e_1, -e_2\}. 
\ee
We have
\[
f_1^\varepsilon = e_1,   \quad f_2^\varepsilon=e_2.
\]
Therefore
\[
{\bf \Psi}({\bf i}_{\sigma_1})= {\bf \Psi}(X_{e_1}){\bf \Psi}(X_{e_2}).
\]

Let $\sigma_2 =\pi_{12}\circ \mu_2\circ \mu_1\circ \mu_2$, where $\pi_{12}$ is the permutation exchanging $1$ and $2$. 
The canonical sequence of signs for $\sigma_2$ is  $\varepsilon=\{1, 1,1\}$. The $\Sigma_{\sigma_2}$ is determined by a change of basis:
\be
\{e_1, e_2\}\stackrel{\mu_2} {\lra}  \{e_1+e_2, -e_2\} \stackrel{\mu_1}{\lra} \{-e_1-e_2, e_1\}\stackrel{\mu_2}{\lra}\{-e_2, -e_1\}\stackrel{\pi_{12}}{\lra}\{-e_1, -e_2\}. 
\ee
We have
\[
f_1^\varepsilon = e_2,   \quad f_2^\varepsilon=e_1+e_2, \quad f_3^\varepsilon=e_1.
\]
Therefore
\[
{\bf \Psi}({\bf i}_{\sigma_2})= {\bf \Psi}(X_{e_2}){\bf \Psi}(X_{e_1+e_2}){\bf \Psi}(X_{e_1}).
\]

Note that $\Sigma_{\sigma_1}=\Sigma_{\sigma_2}$. Therefore $\sigma_1$ is equivalent to $\sigma_2$. In this case, the identity \eqref{dilogper} gives rise to the Faddev-Kashaev \emph{pentagon relation} of quantum dilogarithms
\be\la{DTquantumseries.a2}
{\bf \Psi}(X_{e_1}){\bf \Psi}(X_{e_2})={\bf \Psi}(X_{e_2}){\bf \Psi}(X_{e_1+e_2}){\bf \Psi}(X_{e_1}).
\ee

By Definition \ref{Donaldson.Thomas.transformation},  ${\sigma}_1=\sigma_2$ is  the cluster DT-transformation for ${\bf q}$. Formula \eqref{DTquantumseries.a2} factorizes the quantum DT-series ${\Bbb E}_{\bf q}$ in two different ways.

\paragraph{The canonical basis and {\rm DT}-transformation.}
There are 5 basic polynomials
\[
P_1=X_{e_1}, \hskip5mm P_2=X_{-e_2}, \hskip5mm P_3=X_{-e_1-e_2}+ X_{-e_1},\hskip 5mm
P_4=X_{-e_1}+X_{e_2-e_1}+X_{e_2}, \hskip 5mm P_5=X_{e_2}+X_{e_1+e_2}
\]
satisfying 
\[P_{i+2} P_{i}= 1+q P_{i+1}, \hskip 7mm \mbox{$i\in \Z/5\Z$.}
\]
The polynomials
\[
q^{-cd}P_{i+1}^cP_{i}^d, \hskip 7mm c, d\in \Z_{\geq 0}, ~~i\in \Z/5\Z
\]
give rise to a linear basis of ${\cal O}_q({\cal X}_{\bf q})$ parametrized by ${\cal A}_{\bf q}(\Z^t)$
\[ {\Bbb I}_{\cal A}(a,b) = \left\{ \begin{array}{ll}
         q^{ab} P_{2}^{-b}P_1^{a}& \mbox{if $a \geq 0$, $b\leq 0$};\\
         q^{-ab} P_1^a P_5^b & \mbox{if $a\geq0$, $b\geq 0$};\\
        q^{ab} P_5^b P_4^{-a}  & \mbox{if $a\leq0$, $b\geq 0$};\\
        q^{(b-a)b}P_4^{b-a}P_3^{-b} & \mbox{if $a\leq b\leq0$};\\
         q^{a(a-b)}P_3^{-a}P_2^{a-b} & \mbox{if $b\leq a\leq0$}.\\
                \end{array} \right. \] 
 The basis $\{{\Bbb I}_{\cal A}(a,b)\}$ admit the following properties:
 \begin{enumerate}
 \item${\Bbb I}_{\cal A}(a,b)= X_{ae_1+be_2}+ \mbox{higher order terms.}$
 \item ${\Bbb I}_{\cal A}(a,b)$ is selfdual for the involutive anti-automorphism $\ast$ such that 
$\ast X_v = X_v, ~\ast q = q^{-1}$. 
 
 \item $\{{\Bbb I}_{\cal A}(a,b)\}$ is compatible with cluster mutations.
 \item The {\rm DT}-transformation is a $\Z[q, q^{-1}]$-linear isomorphism preserving the basis
 \be
 {\rm DT}:  {\cal O}_q({\cal X}_{\bf q})\stackrel{\sim}{\lra} {\cal O}_q({\cal X}_{\bf q}), \hskip 7mm P_{i}\lms P_{i+3}.
 \ee
 \end{enumerate}
 
 \paragraph{The map ${\rm D}_{\cal X}$.} Consider the opposite quiver $-{\bf q}= \big(\Lambda, \{e_1, e_2\}, (\ast, \ast)_\circ\big)$ such that $(e_1, e_2)_{\circ}=-1$. Its associated
quantum torus ${\bf T}_{-{\bf q}}$ has generators $X_{v}^\circ, ~v\in \Lambda$ and relations
 \[
 X_v^\circ X_w^\circ = q ^{(v, w)_\circ} X_{v+w}^\circ.
 \]
There are 5 basic polynomials
 \[
 Q_1= X_{e_1}^\circ+ X_{e_1+e_2}^\circ, \hskip5mm Q_2=X_{-e_2}^\circ+X_{e_1-e_2}^\circ+X_{e_1}^\circ, \hskip5mm Q_3=X_{-e_1-e_2}^\circ+X_{-e_2}^\circ, \hskip5mm Q_4= X_{-e_1}^\circ, \hskip5mm Q_5=X_{e_2}^\circ.
 \]
 satisfying 
\[Q_{i+2} Q_{i}= 1+q^{-1} Q_{i+1}, \hskip 7mm \mbox{$i\in \Z/5\Z$.}
\]
Similarly the polynomials
\[
q^{cd} Q_{i+1}^cQ_{i}^d, \hskip 7mm c,d \in \Z_{\geq 0}, ~~i\in \Z/5
\]
provide a linear basis of ${\cal O}_q({\cal X}_{{\bf q}^\circ})$ 
\[ {\Bbb I}_{{\cal A}^\circ}(a,b) = \left\{ \begin{array}{ll}
          q^{-ab}Q_2^{-b}Q_{1}^a   & \mbox{if $a \geq 0$, $b\leq 0$};\\
        q^{ab}Q_1^a Q_5^b & \mbox{if $a\geq0$, $b\geq 0$};\\
        q^{-ab}Q_5^b Q_4^{-a}  & \mbox{if $a\leq0$, $b\geq 0$};\\
        q^{(a-b)b}Q_4^{b-a}Q_3^{-b} & \mbox{if $a\leq b\leq0$};\\
         q^{a(b-a)}Q_3^{-a}Q_2^{a-b} & \mbox{if $b\leq a\leq0$}.
        \end{array} \right. \] 
Here $\{{\Bbb I}_{{\cal A}^\circ}(a,b) \}$ satisfy the same properties as $\{{\Bbb I}_{{\cal A}}(a,b) \}$.       

 There is an natural isomorphism
 \[
 i: ~ {\bf T}_{-{\bf q}} \stackrel{\sim}{\lra}  {\bf T}_{{\bf q}}, \hskip 12mm X_v^\circ \lms X_{-v},~~q \lms q^{-1}.
 \]
The map ${\rm D}_{\cal X}$ is an isomorphism
\be
{\rm D}_{\cal X}:={\rm DT}_{\bf q}\circ i:~ {\cal O}_{q} ({\cal X}_{-{\bf q}}) \stackrel{\sim}{\lra} {\cal O}_{q}({\cal X}_{{\bf q}}), \hskip 12mm
{\Bbb I}_{{\cal A}^\circ}(a,b) \lms {\Bbb I}_{{\cal A}}(a,b), \hskip 7mm  q\lms q^{-1}.
\ee
    
\subsection{Canonical bilinear form on $\ast$-representations of quantum cluster  varieties}\la{sQDT}

Let ${\rm H}_{\cal X}$ be a split torus 
with the group of characters given by the kernel of the form 
$(\ast, \ast)$ on the lattice $\Lambda$. Then, see \cite[Section 2.2]{FG2},  
the cluster Poisson variety ${\cal X}$ is fibered over the ${\rm H}_{\cal X}$:
$$
\theta: {\cal X} \lra {\rm H}_{\cal X}.
$$
The subalgebra of functions $\theta^*{\cal O}({\rm H}_{\cal X})$ is the center of the Poisson algebra 
${\cal O}({\cal X})$. 

There is a $q$-deformation of the Poisson algebra  ${\cal O}({\cal X})$ is  
given by the $\ast$-algebra ${\cal O}_q({\cal X})$. 

The center of ${\cal O}_q({\cal X})$ is canonically identified 
with the algebra ${\cal O}({\rm H}_{\cal X})$ \cite[Section 3.4.1]{FG2}. 

A cluster Poisson variety ${\cal X}$ gives rise to a 
Hilbert space ${\cal H}_{\cal X}$ with the scalar product $\langle \ast, \ast\rangle$ together with 
the following quantisation data \cite{FG4}: 

\begin{itemize}

\item A unitary projective 
action of the cluster modular group $\Gamma$ in the Hilbert space 
${\cal H}_{\cal X}$. 

\item A $\Gamma$-equivariant $\ast$-representation of the $\ast$-algebra ${\cal O}_q({\cal X})$ in 
the Hilbert space ${\cal H}_{\cal X}$. 

\item A decomposition of the unitary representation ${\cal H}_{\cal X}$ 
of the cluster modular group $\Gamma$ into an integral of Hilbert spaces 
${\cal H}_{{\cal X}, \lambda}$, parametrised by the real positive points 
$\lambda \in {\rm H}_{\cal X}(\R_{+}^*)$: 
\be
{\cal H}_{\cal X} = \int_{\lambda \in {\rm H}_{\cal X}(\R_{+}^*)} {\cal H}_{{\cal X}, \lambda}d\lambda.
\ee
A point $\lambda \in {\rm H}_{\cal X}(\R_{+}^*)$ gives rise to a  
character $C_\lambda$ of the center of 
${\cal O}_q({\cal X}) $, 
given by evaluation of polynomials $P \in {\cal O}({\rm H}_{\cal X})$ on 
$\lambda$. 
The center acts on the 
${\cal H}_{{\cal X}, \lambda}$ by the character $C_\lambda$. 

\item Any cluster transformation ${\rm C}$ of ${\cal X}$ gives rise to a unitary 
operator
\be\la{12.4.15.10}
\widehat {\rm C}: {\cal H}_{{\cal X}, \lambda} \lra {\cal H}_{{\cal X}, \lambda}. 
\ee
One has $\widehat {{\rm C}_1\circ {\rm C}_2} = \mu \cdot \widehat {{\rm C}_1} \circ \widehat {{\rm C}_2}$, 
where $\mu \in \C^*, |\mu |=1$. 
\end{itemize}

In this section we establish one more feature of this picture: 

\bt Assume that the Donaldson-Thomas transformation for 
a cluster Poisson variety ${\cal X}$ is a cluster transformation, denoted by ${\rm DT}_{\cal X}$. Then there is  
\begin{itemize}

\item A $\Gamma$-equivariant non-degenerate 
bilinear form, symmetric up to a unitary scalar $\mu_{\cal X}\in {\rm U}(1)$:
\be
{\rm B}_{\cal X}: {\cal H}_{\cal X} \otimes {\cal H}_{\cal X^\circ} \lra \C, 
~~~~{\rm B}_{\cal X}(v, w) = \mu_{\cal X}\cdot \overline {{\rm B}_{\cal X^\circ}(w, v)}, ~~
|\mu_{\cal X}|=1.
\ee
It provides a non-degenerate pairings 
\be
{\rm B}_{{\cal X}, \lambda}: {\cal H}_{{\cal X}, \lambda} \otimes {\cal H}_{{\cal X^\circ}, -\lambda} \lra \C. 
\ee
\end{itemize}
\et

\begin{proof} According to (\ref{12.4.15.10}), the cluster ${\rm DT}$-transformation ${\rm DT}_{\cal X}$ 
gives rise to a unitary  operator 
\be \la{12.4.15.123}
\widehat {\rm DT}_{\cal X}: {\cal H}_{\cal X} \lra {\cal H}_{\cal X}.
\ee

Since the cluster transformation ${\rm DT}_{\cal X}$ is in the center of the cluster 
modular group, the operator $\widehat {\rm DT}_{\cal X}$ commutes with the action of the cluster modular 
group $\Gamma$. 

Recall the  isomorphism of quantum spaces from Lemma \ref{Lemma3.8}:
\be \la{12.4.15.21q}
\begin{split}
&i_{\cal X}: {\cal X}_q \lra {\cal X}^{\circ}_{q}, ~~~~
i^*_{\cal X}: {\cal O}_{q}({\cal X}^{\circ}) \lra  {\cal O}_{q}({\cal X}), \\
&i^*_{\cal X}(X^\circ_i) =  X^{-1}_i, ~~i^*_{\cal X}(q) = q^{-1}. 
\end{split}
\ee

Let $\overline H$ be  the complex conjugate of a complex vector space $H$. 
It is the same real vector space
 with a new complex structure given by $c\circ v:= \overline c v$, $c\in \C$, $v \in H$.

Let $\{X'_i\}$ be cluster coordinates in ${\cal O}_{q^{-1}}({\cal X})$, and 
$\{X_i\}$ the ones in ${\cal O}_q({\cal X})$.  
Given a representation $\rho$ of the $\ast$-algebra 
${\cal O}_q({\cal X})$ in a Hilbert space ${\cal H}_{\cal X}$ 
we get a new representation $\overline \rho$ of the $\ast$-algebra ${\cal O}_{\overline q}({\cal X})$ in 
$\overline {\cal H}_{\cal X}$ by setting
$$
\overline \rho(X_i'):= \rho(X_i), ~~~~ X'_i \in {\cal O}_{\overline q}({\cal X}), ~~X_i \in 
{\cal O}_{\overline q}({\cal X}). 
$$
Indeed, we have 
$$
\overline \rho(X'_aX'_b - \overline q^{(a,b)}X'_{a+b}) = \rho(X_a)\rho(X_b) - q^{(a,b)}\rho(X_{a+b})=0.
$$

This construction is compatible with  interwiners between the 
Hilbert spaces assigned to quivers, lifting  cluster transformations. 
In particular the assignment ${\cal H}_{\cal X} \to \overline {\cal H}_{\cal X}$ 
is $\Gamma$-equivariant. 

Assume now that $|q|=1$. Then the $\Gamma$-equivariant representation $\rho$ of the $\ast$-algebra 
${\cal O}_q({\cal X^{\circ}})$ in the Hilbert space ${\cal H}_{\cal X^{\circ}}$ 
gives rise to  a $\Gamma$-equivariant 
representation $\overline \rho$ of the $\ast$-algebra ${\cal O}_{q^{-1}}({\cal X^{\circ}})$ in 
$\overline {\cal H}_{\cal X^{\circ}}$. 
Applying the isomorphism  
$i_{\cal X}^*$ we get a $\Gamma$-equivariant unitary isomorphism
\be \la{12.4.15.12}
\widehat i_{\cal X}: 
{\cal H}_{\cal X} \lra \overline {\cal H}_{\cal X^\circ}. 
\ee


The composition of the operators (\ref{12.4.15.123}) and (\ref{12.4.15.12}) is a $\Gamma$-equivariant 
unitary operator
$$
\widehat {\rm D}_{\cal X}:= \widehat i_{\cal X}\circ \widehat {\rm DT}_{\cal X}: ~~
{\cal H}_{\cal X} {\lra} \overline {\cal H}_{\cal X^\circ}. 
$$
Therefore we get a complex bilinear $\Gamma$-equivariant non-degenerate form
\be
{\rm B}_{\cal X}: {\cal H}_{\cal X} \otimes {\cal H}_{\cal X^\circ} \lra \C,~~~~
{\rm B}_{\cal X}(v, w):= 
\langle v, \widehat  {\rm D}_{\cal X}(w)\rangle.
\ee

Since ${\rm D}_{\cal X^\circ} \circ {\rm D}_{\cal X} = {\rm Id}_{\cal X}$ is the identity map, 
the composition of the unitary operators 
$\widehat {\rm D}_{\cal X}$ and $\widehat {\rm D}_{\cal X^\circ}$ 
is the identity map up to a unitary constant $\mu_{\cal X}\in {\rm U}(1)$:
\be \la{12.4.15.1}
\widehat {\rm D}_{\cal X^\circ} \circ \widehat {\rm D}_{\cal X}: 
{\cal H}_{\cal X} \lra  {\cal H}_{\cal X}, ~~~~
\widehat {\rm D}_{\cal X^\circ} \circ \widehat {\rm D}_{\cal X} = \mu_{\cal X}\cdot {\rm Id}, ~~
|\mu_{\cal X}|=1.
\ee
Therefore
\be
\begin{split}
&{\rm B}_{\cal X}(v, w) = \langle v,  \widehat  {\rm D}_{\cal X}(w)\rangle = 
\langle \widehat {\rm D}_{\cal X^\circ} (v),  \widehat {\rm D}_{\cal X^\circ} 
\circ \widehat {\rm D}_{\cal X}(w)\rangle \stackrel{(\ref{12.4.15.1})}{=} \\
&\mu_{\cal X}\cdot \langle \widehat {\rm D}_{\cal X^\circ} (v),  w\rangle = 
\mu_{\cal X}\cdot \overline {\langle w,  \widehat {\rm D}_{\cal X^\circ} (v)\rangle} = \mu_{\cal X}\cdot \overline {{\rm B}_{\cal X^\circ}(w, v)}.\\
\end{split}
\ee
\end{proof}

\paragraph{Applications.} Given an oriented decorated surface $\bS$, 
we denote by $\bS^\circ$ 
the decorated surface $\bS$ with the opposite orientation. Then we have a canonical isomorphism:  
$$
{\cal X}^{\circ}_{G, \bS} = {\cal X}_{G, \bS^\circ}. 
$$
Therefore the map ${\rm D}_{\cal X}^*= (i_{\cal X}\circ {\rm DT})^*$ provides us an involution
$$
{\rm D}_{\cal X}^*: 
{\cal O}_q({\cal X}_{G, \bS}) \lra {\cal O}_{q^{-1}}({\cal X}^{\circ}_{G, \bS}) = 
{\cal O}_{q^{-1}}({\cal X}_{G, \bS^\circ}). 
$$





\subsection{Proof of Theorems  \ref{DT.contrav.2.1.hh}, \ref{universal.dt}} \la{sec10} 

\paragraph{Proof of Theorem \ref{DT.contrav.2.1.hh}.}
By definition,  ${\bf K}$ is a cluster DT-transformation if and only of $C_{\bf K}={-\rm Id}$. 
The rest follows directly from the next Lemma.
\bl[{\cite[Th.1.2, (1.12)]{NZ}}] 
\la{lemma.NZ} For any cluster transformation $\sigma$, we have $
C_{F(\sigma)}C_{\sigma}={\rm Id}.$
\el

\paragraph{Proof of Theorem \ref{universal.dt}.}
It suffices to show that
\[
C_{\bf K}=-{\rm Id}  ~~~\Longrightarrow ~~~~ C_{\sigma\circ{\bf K}\circ \sigma^{-1}}=-{\rm Id}.
\]
It is trivial when $\sigma$ is a permutation. We prove the case when $\sigma=\mu_k$ is a cluster mutation.

Set $\varepsilon_{\bf q}=(\varepsilon_{ij})$.
By Lemma \ref{lem.c.matrix1} we get 
\be
C_{\mu_k\circ {\bf K}}=-{\rm Id}+D, \hskip 5mm \mbox{where~}D=(D_{ij}), \hskip 4mm D_{ij}:=\left\{ \begin{array}{ll} 0 &\mbox{ if $j\neq k$},\\ {[-\varepsilon_{ik}]}_+ & \mbox{ if $j= k, i\neq k$},\\2 & \mbox{ if $j=k, i=k$}. \\ \end{array}   \right.
\ee
Note that $D^2=2D$. Therefore
$
\big({C_{\mu_k\circ {\bf K}}}\big)^2=\big(-{\rm Id}+D\big)^2={\rm Id}-2D+D^2={\rm Id}.
$ 
By Lemma \ref{lemma.NZ}, 
$$
C_{F({\bf K})\circ \mu_k}=C_{F(\mu_k\circ {\bf K})}=\big({C_{\mu_k\circ{\bf K}}}\big)^{-1}=-{\rm Id}+D.
$$
Here $F({\bf K}) \circ \mu_k$ is a cluster transformation from $-\mu_k({\bf q})$ to $-{\bf q}$.
Note that $\varepsilon_{-{\bf q}}=(-\varepsilon_{ij})$. Using Lemma \ref{lem.c.matrix1} again, we get $C_{\mu_k\circ F({\bf K})\circ \mu_k}=-{\rm Id}$.
By Lemma \ref{lemma.NZ}, we have 
 $$C_{\mu_k\circ {\bf K}\circ \mu_k}=(C_{F(\mu_k\circ {\bf K}\circ \mu_k)})^{-1}
=(C_{\mu_k\circ F({\bf K})\circ \mu_k})^{-1}=-{\rm Id}.
 $$

\paragraph{A combinatorial characterization of cluster DT-transformations \cite{K12}.}
A transformation $\sigma: {\bf q}\to {\bf q}'$ is {\it reddening}  if all the entries of  $C_{\sigma}$ are non-positive. 

\bl \la{reddening.f.pro}
If $\sigma$ is reddening, then $F(\sigma)$ is  reddening.
\el
\begin{proof} 
If $F(\sigma)$ is not reddening, then at least one of the entries of $C_{F(\sigma)}$ (say ${d_{ij}}$) is positive. By the sign-coherence of $C$-matrix, the entries on the $i$-th row of $C_{F(\sigma)}$  are all non-negative.   
Since $\sigma$ is reddening, the entries on  $i$-th row of the product $C_{F(\sigma)}C_{\sigma}$  are all non-negative, which contradicts the fact that $C_{F(\sigma)}C_{\sigma}={\rm Id}$.
\end{proof}

\bp 
\la{isom.quiver.1}
If $\sigma: {\bf q}\ra {\bf q}'$ is  reddening, then there exists a unique permutation $\pi: {\bf q} \ra {\bf q'}$ such that ${\pi\circ \sigma}$ is the cluster DT-transformation for ${\bf q}$.  
\ep

\begin{proof} Let ${c}_{i}$ be the $i$-th row vector of $C_{\sigma}$. Let ${d}_{i}=(d_{i1}, \ldots d_{in})$ be the $i$-th row vector of $C_{F(\sigma)}$. Let ${\bf e}_i=(0,..., 1, ..., 0)$ be the $i$-th unit vector. By Lemma \ref{lemma.NZ}
\[
d_{i1} c_{1}+d_{i2}c_2+\ldots +d_{in}c_n ={\bf e}_i, \hskip 7mm \forall i\in \{1,..., n\}.
\]
By Lemma \ref{reddening.f.pro}, every $d_{ik}{c}_{k}\in ({\Bbb Z}_{\geq 0})^n$. Therefore for each $i$ there is a unique $j:=\pi(i)$ such that $d_{ij}c_{j}={\bf e}_i$. Since $d_{ij}\in {\Bbb Z}_{\leq 0}$ and $c_{j}\in ({\Bbb Z}_{\leq 0})^n$, we get 
$d_{ij}=-1$, $c_{j}=-{\bf e}_i$. 
Thus $C_{\pi\circ \sigma}=-{\rm Id}$. 
\end{proof}

\bc If $\sigma: {\bf q}\ra {\bf q}'$ is  reddening, then the  formal power series ${\bf \Psi}_\sigma$ in \eqref{dilogper} is the quantum DT-series for ${\bf q}$.
\ec
\begin{proof} Follows directly from Theorem \ref{dilogper22} and  Proposition \ref{isom.quiver.1}
\end{proof}

The following Proposition provides a criterion for recognizing permutations. Its proof goes the same line as that of Proposition \ref{isom.quiver.1}.
\bp
\la{positive.analogue.12}
A cluster transformation $\sigma: {\bf q}\to {\bf q}'$ is a permutation if and only if all the entries of $C_{\sigma}$ are non-negative.
\ep

\begin{proof} We prove the ``if" part. The other direction is clear. Using the same argument as in the proof of Lemma \ref{isom.quiver.1}, it follows that there is a cluster permutation $\tau$ such that $C_{\tau \circ \sigma}={\rm Id}$. By Corollary \ref{cmatrixbasis22},  $\tau\circ \sigma$ maps the quiver ${\bf q}$ to itself. By Theorem \ref{basiclam}, $\tau\circ \sigma$ is an identical map.
\end{proof}




\section{Two geometric ways to determine  cluster {\rm DT}-transformations for ${\cal X}_{{\rm PGL}_2, \bS}$}

\subsection{${\cal X}$-laminations and  cluster {\rm DT}-transformations 
for ${\cal X}_{{\rm PGL}_2, \bS}$} \la{sec3.2}

We start with a geometric interpretation of  integral tropical points 
of the space ${\cal X}_{{\rm PGL}_2, \bS}$. 

Laminations on closed surfaces were defined  by Thurston. 
An important subclass of laminations is given by integral laminations.  
There are two kinds of laminations on decorated surfaces, 
discussed in \cite[Section 12]{FG1}, \cite{FG1a}. Let us recall the  important for us integral 
${\cal X}$-laminations (also called unbounded measured laminations in {\it loc.cit.}). 

We alter a decorated surface $\bS$ by cutting little discs 
around the punctures. Abusing notation, we denote it by $\bS$. 
We define the {\it punctured boundary}  of $\bS$ as the boundary of $\bS$ minus the special points. 
It is a union of {\it punctured boundary components}, 
which are either {\it boundary circles} or {\it boundary intervals}. The boundary circles are parametrized 
by the punctures on the original decorated surface. Each boundary interval is 
bounded by special points.

\bd
An integral ${\cal X}$-lamination on a decorated surface $\bS$ is a union of 
finitely many non-intersecting and non-self-intersecting simple closed loops and arcs connecting punctured boundary components, 
each considered with a positive integral weight, plus
\begin{itemize}
\item A choice of an orientation of each boundary circle of $\bS$ which 
bears an arc of the lamination.
\end{itemize}
In addition, we require that 
\begin{itemize}
\item There are no trivial arcs between the neighboring boundary intervals.
\item There are no loops homotopy equivalent to boundary circles.
\end{itemize} 
Denote by ${\cal L}_{\cal X}(\bS; \Z)$ the set of integral ${\cal X}$-lamination on $\bS$. 
\ed

The group $(\Z/2\Z)^n$ acts on  ${\cal L}_{\cal X}(\bS; \Z)$ by altering 
orientations of the boundary circles bearing arcs of the laminations. 
The action of $(\Z/2\Z)^n$ on ${\cal X}_{{\rm PGL}_2, \bS}$ by altering framings assigned to punctures is positive. Therefore it can be tropicalized and acts on the set 
 ${\cal X}_{{\rm PGL}_2, \bS}(\Z^t)$. 

The mapping class group $\Gamma_{\bS}$ acts on both  ${\cal L}_{\cal X}(\bS; \Z)$ and ${\cal X}_{{\rm PGL}_2, \bS}$.

The following result is  \cite[Theorem 12.1]{FG1} in the case of rational laminations on 
a surface with punctures. See the general case in  \cite{FG1a}.

\bt
There is a canonical $\Gamma_\bS\times (\Z/2\Z)^n$-equivariant isomorphism of sets 
$$
{\cal L}_{\cal X}(\bS; \Z) \stackrel{\sim}{\lra} {\cal X}_{{\rm PGL}_2, \bS}(\Z^t).
$$
\et

Each ideal triangulation ${\cal T}$ of $\bS$ gives rise to a cluster Poisson coordinate system 
$\{X_E\}$ of ${\cal X}_{{\rm PGL}_2, \bS}$,  parametrized by the edges $E$ of ${\cal T}$.  
The tropicalization of these coordinates is a cluster tropical coordinate system $\{x_E\}$ on the set 
${\cal L}_{\cal X}(\bS; \Z)$, defined as follows.

Let $l\in {\cal L}_{\cal X}(\bS; \Z)$.
We twist each of its arcs ending at a boundary circle infinitely many times along the orientation of this boundary circle entering the definition of  $l$,  see Figure \ref{spiral}.

\begin{figure}[ht]
\epsfxsize130pt
\centerline{\epsfbox{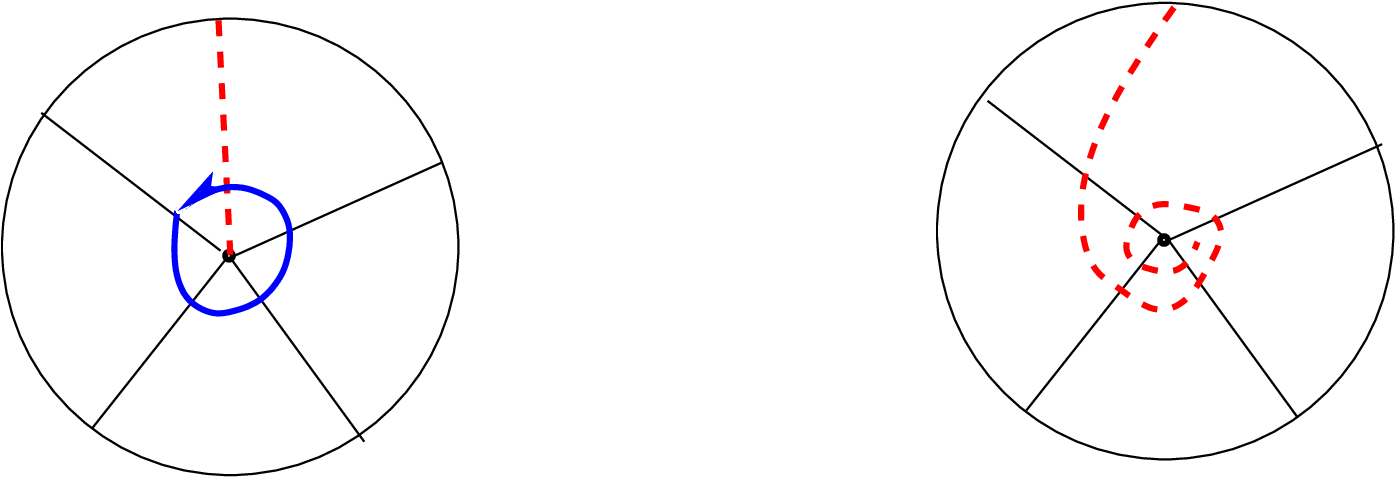}}
\caption{Rotating arcs of a lamination ending on a boundary circle. 
}
\label{spiral}
\end{figure}

We count the minimal intersection number of each connected component of the 
obtained finite collection of curves as explained on Figure \ref{xcoord3}, and multiply it by the weight of the component. 
Note that the part of a curve rotating around a vertex does not contribute to the coordinates. 
In particular, the infinite number of intersections of an arc circling around a vertex do not count.  
\begin{figure}[ht]
\epsfxsize70pt
\centerline{\epsfbox{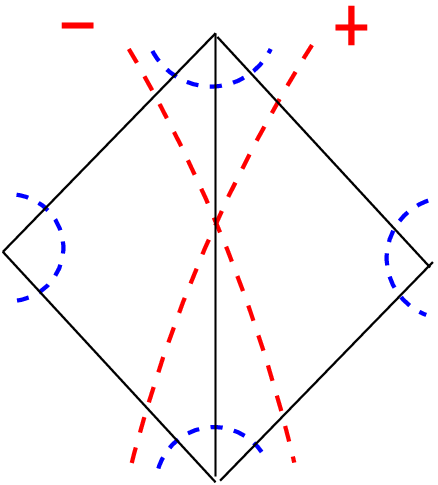}}
\caption{Counting contribution of an arc to the 
coordinate $x_E$ assigned to a diagonal $E$ of a quadrilateral. 
An arc going around a single vertex contributes $0$. An arc crossing the diagonal 
left-to-right contributes $+1$, 
and the arc crossing the diagonal  right-to-left contributes $-1$. These rules do not require an orientation of the arc. }
\label{xcoord3}
\end{figure}


Each ideal triangulation ${\cal T}$ of $\bS$ gives rise to a collection of {\it basic laminations} 
$\{l_E^+\}$ assigned to the edges $E$ of ${\cal T}$. 
The lamination $l_E^+$ is a single arc on $\bS$ with multiplicity $1$, defined as follows. 
Take an end of $E$. If it goes to a puncture, we just add the 
canonical orientation of the corresponding boundary circle, 
determined by the orientation of $\bS$. 
If it ends at a special point, we rotate the end  slightly 
following the orientation of the boundary.  

Let $\{x_F\}$ be the coordinate system assigned to 
the triangulation ${\cal T}$.
By Figure \ref{xcoord3},  we have 
\[
x_E(l_E^+) =1,\hskip 7mm  \mbox{otherwise~~} x_F(l_E^+) =0. 
\]
It means that  $\{l_E^+\}$ are the basic positive laminations for the 
ideal triangulation ${\cal T}$. The terminology ``basic positive laminations'' 
was suggested by this example.

\vskip 2mm

The longest element ${\bf w}_0=(1,1,...,1)\in  (\Z/2\Z)^n$ acts on ${\cal L}_{\cal X}(\bS; \Z)$ by altering the orientations assigned to all boundary circles.
Recall the cyclic shift by one action ${r}_\bS\in \Gamma_{\bS}$. We consider
\be
{\rm C}_{\bS}= r_\bS\circ {\bf w}_0 \in \Gamma_\bS \times (\Z/2)^n.
\ee

\bp \la{PBLNBL2}
The map ${\rm C}_{\bS}$ sends a positive basic 
 lamination $l_E^+$ to a negative one:
\[
x_E({\rm C}_{\bS}(l_E^+))=-1, ~~~~\mbox{otherwise }~~x_F({\rm C}_{ \bS}(l_E^+))=0.
\]
\ep

\begin{proof} Given an edge $E$ of the ideal triangulation ${\cal T}$, there are three cases: 

1. The edge $E$ connects two boundary circles, which could coincide. 
Altering orientations of the boundary circles we change the counting sign for 
the arc intersecting  $E$, see Figure \ref{xcoord3}.

2. The edge $E$ connects two boundary intervals. The rotation ${r}_{\bS}$ by one  changes  the multiplicity 
$+1$ intersection to the multiplicity $-1$ intersection, see Figure \ref{xcoord2}. 
\begin{figure}[ht]
\epsfxsize130pt
\centerline{\epsfbox{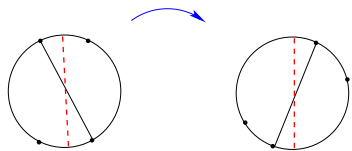}}
\caption{An edge $E$ connecting two boundary intervals. 
}
\label{xcoord2}
\end{figure}

3. The edge $E$ connects a boundary circle with a boundary interval. 
The resulting lamination ${\rm C}_{\bS}(l_E^+)$ is shown on the right of Figure \ref{xcoord1}. 
\begin{figure}[ht]
\epsfxsize130pt
\centerline{\epsfbox{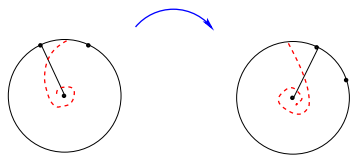}}
\caption{An edge $E$ connecting a boundary circle with a boundary interval.}
\label{xcoord1}
\end{figure}

\end{proof}



 The following result is the $\G={\rm PGL}_2$ case of Theorem \ref{THEmain.result.}.
 
\bt  \la{PBLNBL3} If $\bS$ is admissible, then 
the action of ${\rm C}_{\bS}$ on ${\cal X}_{{\rm PGL}_2, \bS}$ 
is the cluster ${\rm DT}$-transformation. 
\et

\begin{proof} If $\bS$ is admissible, then 
the action of ${\rm C}_{\bS}$ is a cluster transformation (\cite{FG1}).
The rest follows from  Theorem \ref{basiclamDT}  and Proposition \ref{PBLNBL2}.
\end{proof}

When $\bS$ has a 
single puncture and no special points, i.e., $\bS$ is not admissible, 
the action of ${\rm C}_{\bS}=w_0\in \Z/2\Z$ is not a cluster transformation. For example, 
when $\bS$ is a punctured torus, 
all cluster transformations preserve each of two tropical hemispheres,
but the action of $w_0$ 
interchanges them ({\it cf.} \cite{FG3}).

\subsection{Cluster divisors at infinity and  cluster {\rm DT}-transformations for ${\cal X}_{{\rm PGL}_2, \bS}$}\la{sec3.3}
We recall the correspondence between basic laminations and cluster divisors of ${\cal X}$-variety at infinity borrowed from \cite{FG3}. Using this correspondence, we give an alternative (rather simple) proof of Proposition  \ref{PBLNBL2} for the case when $\bS$ is a disk without punctures. We wish to apply the same approach to more general cases in the future.

\paragraph{Basic laminations and cluster divisors at infinity.}
Let $X_{k}$  be a cluster coordinate on a cluster Poisson variety ${\cal X}$, assigned 
to a basis vector $e_k$ of a quiver ${\bf q} = (\Lambda, \{e_i\}, (\ast, \ast))$. 
Deleting $e_k$, we obtain a subquiver ${\bf q}_{e_k}$, whose lattice is spanned by the 
basis vectors 
different from $e_k$, with the induced form. Mutating the quiver ${\bf q}_{e_k}$ we get a cluster 
Poisson variety ${\cal X}_{e_k}$ of dimension one less than that of  ${\cal X}$. 

The variety ${\cal X}_{e_k}$ sits naturally on the boundary of  ${\cal X}$ in two different ways:
\be \la{compX}
{\cal X}_{e_k} ^+ ~\subset~ \overline {\cal X} ~\supset~{\cal X}_{e_k}^-.
\ee
Namely,  adding the coordinate $X_k$ to 
any cluster coordinate system $\{X_j'\}$ on ${\cal X}_{e_k}$,  
we get a rational cluster coordinate system on ${\cal X}$. The cluster divisor ${\cal X}_{e_k} ^+$ is given in this coordinate 
system by the equation $X_k=0$. Mutations at the other directions do not change the equation 
$X_k=0$, as is clear from (\ref{5.11.03.1x}). Similarly, the cluster divisor ${\cal X}_{e_k}^-$ is obtained  by setting $X_k=\infty$.


Recall the basic laminations $l^+_{X_k},~ l^-_{X_k}\in {\cal X}(\Z^t)$ associated to the coordinate $X_k$.
For a generic  $p\in {\cal X}_{e_k}^+$ one has $X_k(p)=0$, while the values $\{X_j(p)\}$ of the rest coordinates $\{X_j\}$ are well defined and non-zero. Thus the irreducible divisor ${\cal X}_{e_k}^+$ can be naturally identified with  $l^+_{X_k}$:
\[
{\rm ord}_{{\cal X}_{e_k}^+} (X_k)= X_k^t(l_{X_k}^+)=1,
\hskip 7mm
{\rm ord}_{{\cal X}_{e_k}^+} (X_j) = X_j^t (l_{X_k}^+)=0. 
\]
Equivalently, $l^+_{X_k}$ is represented by a generic path $p(t)$ in ${\cal X}$ which approaches $p$ as $t\to 0$. 
Similarly, the divisor ${\cal X}_{e_k}^-$ is identified with $l_{X_k}^-$.

If the DT transformation ${\rm DT}_{\cal X}$ of ${\cal X}$ is a cluster transformation, then it maps $l_{X_k}^+$ to $l_{X_k}^-$. Therefore it gives rise to a birational map from ${\cal X}_{e_k}^+$ to ${\cal X}_{e_k}^-$. Note that ${\cal X}_{e_k}^+$ and  ${\cal X}_{e_k}^-$ are both isomorphic to the cluster Poisson variety ${\cal X}_{e_k}$. Therefore ${\rm DT}_{\cal X}$ induces a birational map 
\be
\la{boundary.dt.conjecture.11}
\widetilde{{\rm DT}_{\cal X}}: ~ {\cal X}_{e_k} \lra {\cal X}_{e_k}.
\ee
 \bcon The induced map \eqref{boundary.dt.conjecture.11} is the DT-transformation of ${\cal X}_{e_k}$.
 \econ
 \begin{remark} If $e_k$ is a sink of the quiver ${\bf q}$, $i.e., (e_k, e_j)\geq 0$ for all $j$, then the Conjecture follows directly from \cite[p.138, Proposition 16]{KS1}. 
\end{remark}

\paragraph{The cluster divisors at infinity for the space ${\cal X}_{{\rm PGL}_2, \bS}$ \cite{FG3}.} 
Let us adopt a dual  
point of view on the definition of the space ${\cal X}_{{\rm PGL}_2, \bS}$. 
We assume that given a framed ${\rm PGL}_2$-local system on $\bS$, the framing is defined as a covariantly 
constant section of the associate ${\Bbb P}^1$-bundle over the punctured boundary of $\bS$, 
that is over $\partial \bS - \{\mbox{special points}\}$.\footnote{In the original definition,  the framing is a reduction to a Borel subgroup near every marked point.}
A  dual ideal triangulation ${\cal T}$ of $\bS$ is a triangulation of $\bS$  
whose vertices are either 
in the boundary circles, or inside of the boundary intervals, so that each boundary interval 
carries just one vertex of ${\cal T}$.

\begin{figure}[ht]
\epsfxsize200pt
\centerline{\epsfbox{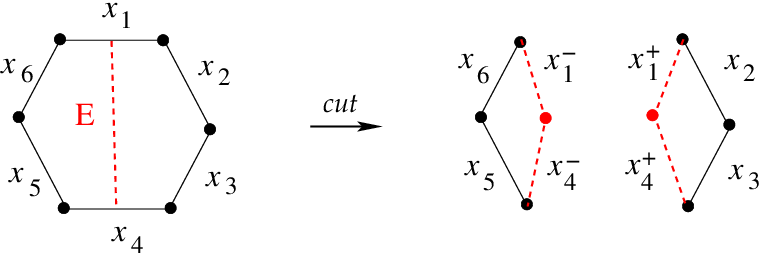}}
\caption{Cutting a polygon along the ideal edge $E$, getting two new special points.}
\label{cluste3}
\end{figure}
\begin{figure}[ht]
\epsfxsize200pt
\centerline{\epsfbox{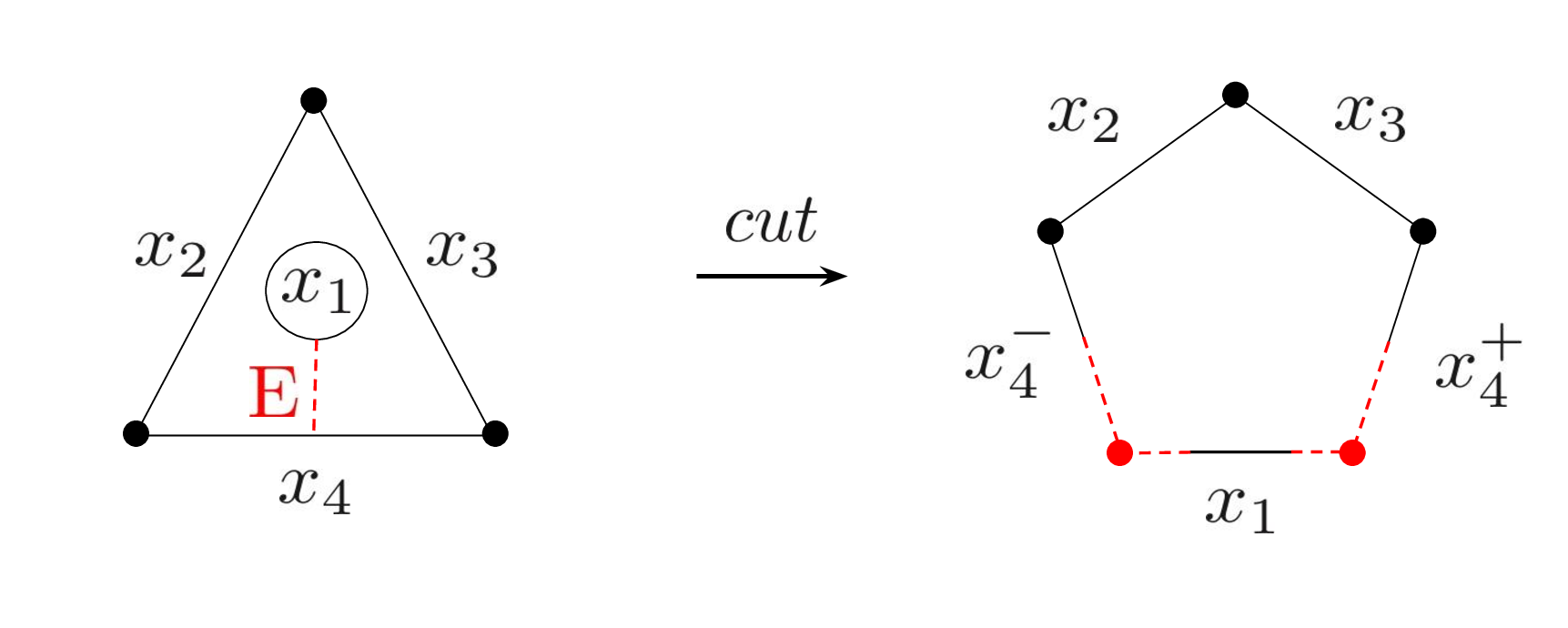}}
\caption{Cutting a punctured disk along the ideal edge $ E$, getting two new special points.}
\label{clustercut31}
\end{figure}

Given an ideal edge $E$ of $\bS$, the cluster variety assigned to $E$, sitting on the boundary of the space ${\cal X}_{{\rm PGL}_2, \bS}$, is described as follows, see Figures \ref{cluste3}-\ref{clustercut31}. 
Cut the surface $\bS$ along the edge $E$, getting a new decorated surface $\bS_E$, 
which may be disconnected.  
Its special points are the ones inherited from $\bS$ plus two new ones: 
the centers of the two new edges obtained by cutting the edge $E$. 
If $E$ ends at a boundary circle, then cutting along $E$, the boundary circle becomes a boundary interval ending at the two new special points. The pair of divisors at infinity assigned to $E$ are both identified with the moduli space 
${\cal X}_{{\rm PGL}_2, \bS_E}$:
$$
{\cal X}_{{\rm PGL}_2, \bS_E}^+ \subset \overline{ {\cal X}_{{\rm PGL}_2, \bS}} \supset {\cal X}_{{\rm PGL}_2, \bS_E}^-
$$

Proposition  \ref{PBLNBL2} is a direct consequence of  Proposition \ref{prop}. 

\bp  \la{prop} For any decorated surface $\bS$, the rational functions 
${\rm C}_\bS^*(X_i)$ on the  space ${\cal X}_{{\rm PGL}_2, \bS}$ have the following property:
\begin{enumerate}

\item 
If  $i$ is different from $k$, the function ${\rm C}_\bS^*(X_i)$, evaluated 
on a generic path $p(t)$ representing the basic lamination $l_{X_k}^+$, has a finite nonzero limit as $t\to 0$.  

\item The function ${\rm C}_\bS^*(X_k)$, evaluated 
on such a path $p(t)$, has a simple pole as $t\to 0$. 

\end{enumerate}
\ep

\bp
Proposition \ref{prop} is true when  $\bS$ is a disk without punctures.
\ep

\begin{proof}
Let $C$ be a finite subset  of the circle $S^1$. Let $R$ 
be an edge with the vertices $\{r_\bullet, r_\circ\} \subset S^1-C$. It seperates $S^1$ into two arcs. 
The points of $C$ on  the arc going clockwise from $r_\bullet$ form a subset $C_{\bullet}$. The rest points of $C$ form another subset $C_\circ$. Therefore $C = C_\bullet\cup C_\circ$.

Denote by ${\cal X}_C$ the space of configurations of points on ${\Bbb P}^1$, parametrized by the set $C$.
Its partial compactification $\overline {\cal X}_C$ has  a divisor ${\cal X}_C^+(R)$ consisting of the configurations, of which the points parametrized by $C_\bullet$ are ``very close" to a given point $x_\bullet\in {\Bbb P}^1$. 
Equivalently, the points parametrized by $C_\circ$ are ``very close" to a given point on $x_\circ\in {\Bbb P}^1$, see Figure \ref{xcoord5}.

\begin{figure}[ht]
\epsfxsize250pt
\centerline{\epsfbox{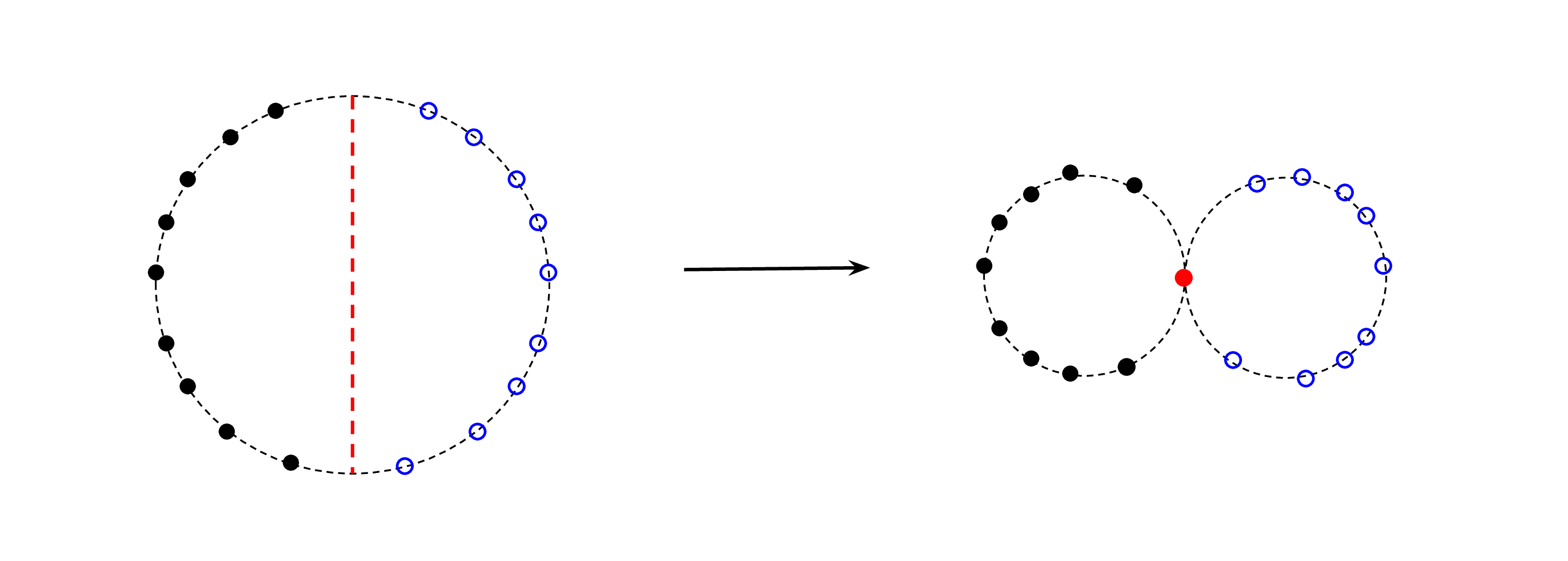}}
\caption{The red dashed edge $R$ describes a divisor at infinity.
}
\label{xcoord5}
\end{figure}


Consider the two connected intervals of $S^1 -C$ containing the points $r_\bullet, r_\circ$. Their ends are two $\bullet$-vertices and two $\circ$-vertices.  They form a quadrilateral. Let $E$ and $F$ be its diagonals. 
We assume that $E$ crosses $R$ ``from left to right'', see Figure \ref{xcoord4}.

\begin{figure}[ht]
\epsfxsize190pt
\centerline{\epsfbox{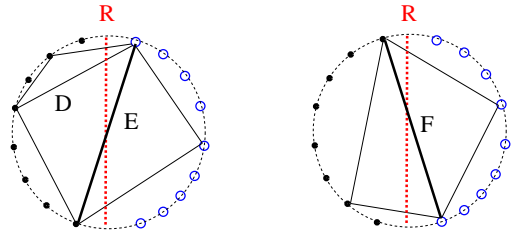}}
\caption{Rotating the set $C$ clockwise 
and keeping the 
triangulation intact,  amounts to rotating the triangulation counterclockwise. So the edge $E$ moves to the edge $F$.}
\label{xcoord4}
\end{figure}

Let ${\cal T}$ be a triangulation of the disc with vertices at  $C$, containing the edge $E$. 
Let $D$ be an edge of ${\cal T}$. Denote by $Q_D$ the quadrilateral of ${\cal T}$ containing  $D$ as a diagonal. 
Every configuration in ${\cal X}_C$ assigns to the vertices of $Q_D$  a quadruple
$x_1, ..., x_4 \in {\Bbb P}^1$ so that  $D= (x_1, x_3)$.  Its corresponding cluster ${\cal X}$-coordinate is  the cross-ratio of $x_1, ..., x_4$:
\be \la{CRN}
X_D= r^+(x_1, x_2, x_3, x_4):= \frac{(x_1-x_2)(x_3-x_4)}{(x_2-x_3)(x_1-x_4)}.
\ee

Let $p(t)$ be a path approaching a generic point $p$ of ${\cal X}_C^+(R)$. We consider the limit of $X_D(p(t))$ as $p(t)$ approaches $p$.
If $D$ is different from $E$, then $Q_D$ cannot contain two $\bullet$-vertices and two $\circ$-vertex simultaneously. Therefore $X_D(p(t))$ has a finite nonzero limit. For the edge $E$, the quadrilateral $Q_E$ contains two  $\bullet$-vertices and two $\circ$-vertex.
By  \eqref{CRN}, $X_E(p(t))\to 0$ as $p(t)\to p$. Thus the path $p(t)$ represents the basic positive lamination $l_{X_E}^+$.

\vskip 2mm

Let ${\cal T}'$  be the  ``counterclockwise rotation by one'' of  ${\cal T}$. It contains the edge $F$.  
By \eqref{CRN}, 
\[
X_F(p(t))= {\rm C}_{\bS}^* (X_E)(p(t)) \to \infty, \hskip 7mm \mbox{as }~~p(t) \to p.
\] 
For any edge $D'$ of ${\cal T}'$ different from $F$,  the limit of $X_{D'}(p(t))$ is finite and nonzero. 
These are precisely the properties we needed in Proposition \ref{prop}. 
\end{proof}


\section{Birational Weyl group action  on the space ${\cal A}_{\G, \bS}$} \la{sec1}
Let $\G$ be a split semi-simple group. Let us assume that  $\bS$ has $n$ many punctures.

The canonical central element $s_\G \in \G$ is the image of $\small{\begin{pmatrix}
-1 & 0 \\
0 & -1\end{pmatrix}}$ 
under a principal embedding ${\rm SL}_2 \hra \G$. 
A {\it twisted} $\G$-local system on $\bS$ is a $\G$-local system on the 
tangent bundle of $\bS$ minus zero section with monodromy $s_\G$ around a 
loop given by rotating 
a tangent vector by $360^\circ$ at a point of $\bS$. Since $s_\G^2=1$, the loop orientation is irrelevant. 

Recall the principal affine space ${\cal A} = \G/\U$. Elements of ${\cal A}$ are called {\it decorated flags}. 

\bd
The moduli space ${\cal A}_{\G, \bS}$ parametrizes twisted 
$\G$-local systems  on $\bS$ with an additional data, a {\it decoration}, given by 
 a reduction to a  decorated flag near each marked point.
\ed
This implies that the monodromy around each puncture is unipotent. 
However, thanks to the freedom of choices of decorations, the dimension of ${\rm dim}{\cal A}_{\G, \bS}$ will not decrease. For example, if $\bS$ has only punctures, then one has
\[
{\rm dim}{\cal A}_{\G, \bS} =  {\rm dim}{\cal X}_{\G, \bS}.
\]



If the group $\G$ has trivial center,  then the principal affine space ${\cal A}$ is the moduli space of 
pairs $(\U, \chi)$, where $\U$ is a maximal unipotent subgroup of $\G$, and 
$\chi:\U\to {\Bbb A}^1$ is a non-degenerate character ({\it cf.} \cite[Section 1]{GS}). 
For general $\G$, there is a 
canonical non-degenerate character $\chi_\A$ assigned to a decorated $\A \in {\cal A}$.

The group $\G$ acts on ${\cal A}$ on the left. The stabilizer ${\rm U}_{\A}$ 
of $\A\in {\cal A}$ is a maximal unipotent subgroup of $\G$. 
Recall the set ${I}$ indexing simple positive roots of $\G$. 
The character $\chi_{\A}$ provides an isomorphism
\be \la{isoia}
i_{\A}: {\rm U_{\A}/[U_{\A},U_{\A}]} \stackrel{\sim}{\lra} {\Bbb A}^I. 
\ee
Let $\Sigma: {\Bbb A}^I\to {\Bbb A}^1$ be the sum map. Then $\chi_\A = \Sigma\circ i_\A$. 
 This characterizes the map $i_{\A}$.

Let $p$ be a puncture. 
A decoration at $p$ is 
a decorated flag $\A_p$ in the fiber of  ${\cal L}_{\cal A}$  
near $p$, invariant under the monodromy around $p$. It defines a conjugacy class in the 
 unipotent subgroup ${\rm U}_{\A_p}$ preserving $\A_p$. So we get a regular map 
\be \la{mapmup}
\mu_p: {\cal A}_{{\rm G},   \bS} \lra {\rm U_{A_p}/[U_{\A_p},U_{\A_p}]}\stackrel{i_{{\rm A}_p}}{=}
{\Bbb A}^I.  
\ee
The composition $\mu_p\circ \sigma$ is called 
the {\it total potential} ${\cal W}_p$ at the puncture $p$. It is a sum of the components, called 
{\it partial potentials},  
parametrized by the simple positive roots $\alpha$:
\be \la{potential}
{\cal W}_p = \sum_{\alpha\in I} {\cal W}_{p, \alpha}. 
\ee


Let $\rm R$ be the lattice spanned by the simple positive roots of $\G$. 
For any abelian group ${\bf A}$, we have ${\rm Hom}({\rm R}, {\bf A}) = {\bf A}^I$. 
Using the embedding ${\Bbb G}_m\hra {\Bbb A}$ we get an open embedding 
\[
i: {\rm T}:= {\rm Hom}({\rm R}, {\Bbb G}_m) =({\Bbb G}_m)^I \hra {\Bbb A}^I. 
\]
The monodromy map, followed by the birational isomorphism $i^{-1}$, provides a rational map
\be \la{ACTMAP1}
\widetilde \mu_p:= i^{-1}\circ \mu_p: {\cal A}_{\G, \bS} \lra {\rm T}. 
\ee
Summarizing, we arrive at a commutative diagram, related to a puncture $p$ on $\bS$:
\[
\xymatrix{
&{\rm T} = {\rm Hom}({\rm R}, {\Bbb G}_m)\ar[d]_i& \\
{\cal A}_{\G, \bS}\ar[ru]^{{\widetilde\mu_p}} \ar[r]^-{\mu_p} & {\rm U}/[{\rm U}, {\rm U}] = {\Bbb A}^I\ar[r]^-{\sum}& 
{\Bbb A}
  }
\]

The Cartan group ${\rm H}$  of $\G$ acts from the  right on ${\cal A}$. 
Therefore the group  
${\rm H}^n$ acts on  ${\cal A}_{\G, \bS}$ by rescaling  decorations at punctures
\be \la{ACTMAP}
{\rm H}^n \times {\cal A}_{\G, \bS} \lra {\cal A}_{\G, \bS}, \hskip 7mm (h, a) \lms h\cdot a
\ee
Forgetting the decorations near punctures, we get
 a  principal ${\rm H}^n$-fibration over the moduli space ${\rm Loc}^{\rm un}_{\G, \bS}$ of twisted $\G$-local systems on $\bS$ with unipotent 
monodromies around the punctures: 
\be \la{CAPR}
p_{\cal A}: {\cal A}_{\G, \bS} \lra {\rm Loc}^{\rm un}_{\G, \bS}. 
\ee
The projection $p_{\cal A}$ and the rational map ${\widetilde \mu}=:\prod {\widetilde \mu}_{p}$ provide a diagram
\be \la{Adiag}
\begin{gathered}
\xymatrix{
{\cal A}_{\G, \bS}  \ar[d]_{p_{\cal A}} \ar[r]^-{{\widetilde \mu}} & {\rm T}^n \\
{\rm Loc}^{\rm un}_{\G, \bS} &   }
\end{gathered}
\ee

The Weyl group acts on the lattice $\rm R$, and hence on the torus ${\rm T}$. 
\bt \la{Weylgac}
For each puncture  $p$ of  $\bS$, there is a canonical birational action of the Weyl group $W$ 
on the space ${\cal A}_{{\rm G}, \bS}$ 
such that 
\begin{enumerate} 
\item The group $W$ acts along  the fibers of the projection $p_{\cal A}$. 
It alters only the decoration at $p$. 

\item The projection $\widetilde \mu_p$ is $W$-equivariant. 

\item The actions at different punctures  commute. So the group $W^n$ acts birationally on  ${\cal A}_{\G, \bS}$. 

\item The action $\circ$ of the  group $W^n$ intertwines the action $\cdot$ of ${\rm H}^n$ on 
${\cal A}_{{\rm G}, \bS}$:
\[
{\rm H}^n \times {\cal A}_{{\rm G}, \bS} \lra {\cal A}_{{\rm G}, \bS}, ~~~~
w\circ (h\cdot a) = w(h) \cdot (w \circ a), \hskip 4mm  w\in W^n, ~h\in {\rm H}^n, ~a \in {\cal A}_{\G, \bS}.
\]
\end{enumerate}

\et

\begin{proof}

The map $p_{\cal A}$ has a  multivalued ``section'' 
$\widetilde\mu^{-1}(1)$. 
For any generic $u\in {\rm Loc}^{\rm un}_{\G, \bS}$, we
choose an element 
\[
s \in \widetilde\mu_p^{-1}(1) \cap p_{\cal A}^{-1}(u).
\]  
Since the fiber $p_{\cal A}^{-1}(u)$ is an ${\rm H}^n$-torsor, the $s$ chosen induces an isomorphism
\be \la{4.1.15.1}
{\rm H}^n \lra p_{\cal A}^{-1}(u), ~~~~ h \lms h \cdot s. 
\ee
We define an action $\circ$ of the 
group $W^n$ on the fiber $p_{\cal A}^{-1}(u)$, making it $W^n$-equivariant,   setting 
\be \la{4.1.15.2}
w \circ (h \cdot s):= w(h) \cdot s. 
\ee

It remains to show that the action $\circ$  does not depend on the choice of $s$. 
Let $X^*({\rm Y})$ be the character group of a split torus ${\rm Y}$, and  $X_*({\rm Y})$ 
the cocharacter group. 
There is a natural embedding $X^*({\rm T}^n) \subset X^*({\rm H}^n)$. 
Its dual $X_*({\rm H}^n) \ra X_*({\rm T}^n)$ provides an isogeny: 
\be \la{ACTMAP2}
{\rm H} ^n= X_*({\rm H}^n) \otimes {\Bbb G}_m \lra {\rm T}^n = X_*({\rm T}^n) \otimes {\Bbb G}_m.  
\ee
Following the definition of $\tilde{\mu}$, it is easy to show that \eqref{ACTMAP2} coincides with the map
\be
\la{com.iso.4.1.15.2.mu}
{\rm H}^n \stackrel{\eqref{4.1.15.1}}{\xrightarrow{\hspace*{8mm}}} p_{\cal A}^{-1}(u)\stackrel{\tilde{\mu}}{\lra} {\rm T}^n.
\ee
The choices of $s$ are differed by kernel elements of \eqref{ACTMAP2}. Since (\ref{ACTMAP2}) is $W^n$-equivariant, its kernel is $W^n$-invariant. 
Thus the $W^n$-action $\circ$ in (\ref{4.1.15.2}) does not depend on the choices of $s$. 
\end{proof} 

\paragraph{Comparing with the rational Weyl group action  on the space ${\cal X}_{\G, \bS}$.} 
For any $\G$, the group $W^n$ acts by birational 
automorphisms of the  space  ${\cal X}_{\G, \bS}$, see Section \ref{sec1.1}.  
Although  it looks like the $W^n$-action  on the space ${\cal A}_{\G, \bS}$ has nothing to 
do with it, 
 they are closely related. 

Here is an analog of  diagram (\ref{Adiag}) for  the space ${\cal X}_{\G, \bS}$. 
There is a $W^n$-equivariant projection 
\be \la{ACTMAPP}
\pi:  {\cal X}_{\G, \bS} \lra {\rm H}^n,  
\ee
given by 
the semisimple part of the monodromies at the punctures,   
enhanced by  framings. 
For example, when $\G={\rm GL}_m$, a generic monodromy a each puncture has $m$ many different eigenvalues. A framing near the puncture is equivalent to an ordering of these eigenvalues.  It gives rise to the projection \eqref{ACTMAPP}. The Weyl group acts on ${\cal X}_{\G, \bS}$ by changing the ordering.

Forgetting the framing, we get
 a  projection onto the moduli space of $\G$-local systems 
\be \la{CAPR}
p_{\cal X}: {\cal X}_{\G, \bS} \lra {\rm Loc}_{\G, \bS}. 
\ee
The projection $p_{\cal X}$ and the map $\pi$  provide a diagram
\be \la{Pdiag}
\begin{gathered}
\xymatrix{
{\cal X}_{\G, \bS} \ar[r]^\pi  \ar[d]_{p_{\cal X}} & {\rm H}^n &\\
{\rm Loc}_{\G, \bS} &   &}
\end{gathered}
\ee
The Weyl group acts along  the fibers of the map $p_{\cal X}$. 
The projection $\pi$ is $W^n$-equivariant.  
The projection $p_{\cal X}$ is a Galois cover  over the generic point with the Galois group $W^n$.

\paragraph{Weyl group action for $\G={\rm SL}_2$ and tagged triangulations.} \la{WAsec}
The action of the group $(\Z/2\Z)^n$ on the space ${\cal A}_{{\rm SL}_2, \bS}$ was introduced in \cite[p.186]{FG1}. 
It  is very closely related to 
the ideal {\it tagged} triangulations of Fomin-Shapiro-Thurston \cite{FST}, \cite{FT}. 
Namely, any ideal triangulation ${\cal T}$ of $\bS$  provides a cluster coordinate system
 ${\cal C}_{\cal T}$ 
on ${\cal A}_{{\rm SL}_2, \bS}$. However if $\bS$ has more than one puncture, 
not all cluster coordinate systems on the space ${\cal A}_{{\rm SL}_2, \bS}$ 
can be interpreted  this way. 
In this case the group $(\Z/2\Z)^n$ acts by cluster transformations, 
and so for any element $w \in (\Z/2\Z)^n$ there is a new cluster coordinate system 
${\cal C}_{w{\cal T}}:= w^*{\cal C}_{{\cal T}}$. 

Any element $w\in  (\Z/2\Z)^n$ is determined uniquely by 
a subset ${\cal P}_w$ of the punctures, so that
$$
w = \prod_{p\in {\cal P}_w}w_p. 
$$
Here $w_p $ is the generator of $\Z/2\Z$ assigned to a puncture $p$. The cluster coordinate 
system ${\cal C}_{w{\cal T}}$ coincides with Fomin-Shapiro-Thurston's cluster coordinate system 
assigned to the tagged triangulation obtained by putting tags 
at the ends of all arcs of $T$ ending at the punctures $p \in {\cal P}_w$. 
This way we get almost all cluster coordinate systems, but not all of them. 
The exceptional ones are assigned to tagged triangulations which can have a puncture with just two 
arcs entering, which must be isotopic arcs, 
one is tagged, one is not, so that their other ends are either both tagged, or not. 
We discussed tagged triangulation in detail in Section 7.2.

The very existence of the $W^n$-action suggests a generalization of 
majority of tagged triangulations for the group ${\rm SL}_m$: 
they are obtained by the action of elements of the group $W^n$ 
on the cluster coordinate systems assigned to ideal webs on $\bS$ studied in  \cite{G}. 



\section{Example of cluster {\rm DT}-transformations}
We consider a basic example of the ${\rm DT}$-transformation for the punctured disk, which  serves as a basic model for studying the cluster nature of Weyl group actions in the next Section.

\subsection{Cluster {\rm DT}-transformation for the punctured disc}
\paragraph{Cluster set-up.} 
A quiver ${\bf q}$  can be described by 
 a skew-symmetric matrix $\varepsilon_{\bf q}=(\varepsilon_{ij})$, where
 \[
i,j\in  {\rm I}=\{1,\ldots,N\}, \hskip 7mm
\varepsilon_{ij}= \#\{\mbox{arrows from $i$ to $j$}\}  - \#\{\mbox{arrows from $j$ to $i$}\}.
\]
Let ${\cal F}_{\bf q}:=\Q(X_1,\ldots, X_N, A_1, \ldots, A_N)$ be the field of rational functions associated to  ${\bf q}$. 


Each $k\in  {\rm I}$ gives rise to a mutated quiver ${\bf q}'=\mu_{k}({\bf q})$ such that 
$\varepsilon_{{\bf q}'}$ is given by Formula \eqref{epsilon.mutate}. 
Let ${\cal F}_{{\bf q}'}:=\Q(X_1',\ldots, X_N', A_1', \ldots, A_N')$. Consider  
an isomorphism  $\mu_{k}^*: {\cal F}_{{\bf q'}} \to {\cal F}_{\bf q}$:
\begin{align}
\mu_{k}^\ast X_i':&=\left\{ \begin{array}{ll}
      X_k^{-1}  & \mbox{if $i=k$} \\
      X_i(1+X_k^{-{\rm sgn}(\varepsilon_{ik})})^{-\varepsilon_{ik}} & \mbox{if $i\neq k$}, \\
   \end{array}
   \right.  \\
\mu_k^*A_i' :&=\left\{ 
 \begin{array}{ll}  A_k^{-1}{({ \prod_{j|\varepsilon_{ij}<0}A_j^{-\varepsilon_{ij}} +\prod_{j|\varepsilon_{ij}>0}A_j^{\varepsilon_{ij}}})} & \mbox{if $i=k$} \\
      A_i & \mbox{if $i\neq k$}. \\
   \end{array}
   \right.
\end{align}
The map $\mu^*_k$ is the {\it cluster mutation} at the direction $k$.
It is involute.  

Let $\pi$ be a bijection from ${\rm I}$ to itself. Let ${\bf q}'=\pi({\bf q})$ be the quiver obtained via relabeling the vertices $i$ of ${\bf q}$ by $\pi(i)$. It induces an isomorphism $\pi^*:
{\cal F}_{{\bf q}'} \to {\cal F}_{{\bf q}}$,  called a {\it cluster permutation}:
$$
\pi^* X_{i}'=X_{\pi^{-1}(i)}, \quad \pi^* A_{i}'=A_{\pi^{-1}(i)}.
$$


\paragraph{Basic example.}
Let $N\geq 2$. Let ${\bf q}_N$ be a quiver of  $N$ vertices. When $N=2$, it has no arrows. When $N>2$, it is a cycle with vertices labelled clockwise from 1 to $N$.
 See Figure \ref{quiverN}. \footnote{
Starting from a Dynkin diagram of type $D_N$, we assign orientations to each edge, obtaining  quivers called the Dynkin quivers of type $D_N$. It is known that all the Dynkin quivers of type $D_N$ are equivalent to ${\bf q}_N$.}

\begin{figure}[h]
\epsfxsize300pt
\centerline{\epsfbox{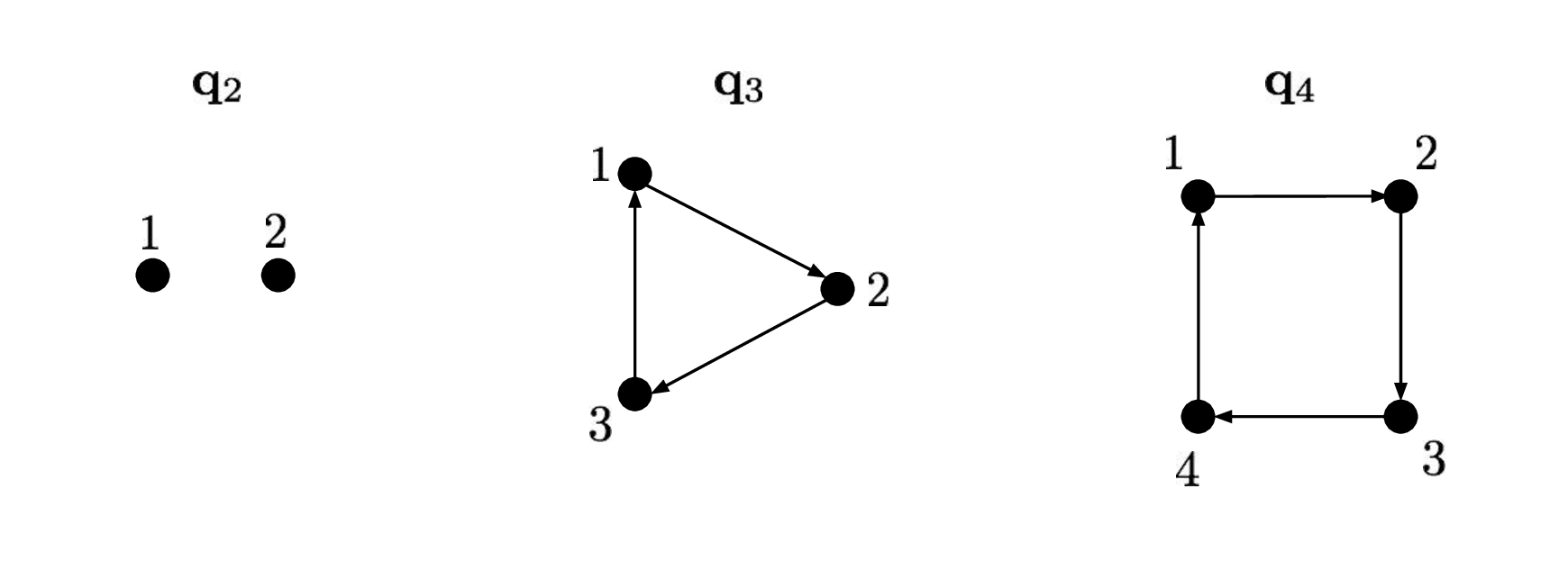}}
\caption{The quivers ${\bf q}_2$, ${\bf q}_3$, and ${\bf q}_4$. }
\label{quiverN}
\end{figure}

Let $(i_1,i_2,\ldots, i_N)$ be a permutation of $\{1, 2, \ldots, N\}$. Define a cluster 
transformation of ${\bf q}_N$
\be
\la{cluster.trans.tau}
\tau_N:=\mu_{i_1}\circ \ldots \circ \mu_{{i_{N-1}}}\circ \pi_{i_{N-1}, i_{N}} \circ \mu_{i_{N-1}}\circ \ldots \circ \mu_{i_1},
\ee
where $\mu_{k}$ is the cluster mutation at the directions $k$, and $\pi_{i_{N-1}, i_{N}}$ is the cluster permutation switching the labels $i_{N-1}$, $i_N$. We frequently write $\tau$ instead of $\tau_N$.

\bt
\la{1.5.8.35}
The cluster transformation $\tau_N$ does not depend on the choices of permutation. It maps the quiver ${\bf q}_N$ to itself. Thus $\tau_N$ is an order 2 element of the cluster modular group $\Gamma_{{\bf q}_N}$. The induced isomorphism $\tau_N^*$ of the field ${\cal F}_{{\bf q}_N}$ is determined by
\be
\tau_N^* A_i = A_i W, \hskip 7mm \tau_N^*X_i= \frac{F_i}{X_{i-1}F_{i-2}},
\ee
where
\[
W=\sum_{j=1}^N\frac{1}{A_{j}A_{j+1}}, \hskip 7mm  F_i=1+X_i+X_iX_{i-1}+\ldots+X_{i}X_{i-1}\ldots X_{i-N+2}.
\]
\et

Let $r$ be the cluster permutation that relabels the vertex $i$ of ${\bf q}_N$ by $i-1$. Since 
 ${\bf q}_{N}$ is rotation invariant, $r\in \Gamma_{{\bf q}_N}$.
\bc The composition ${\bf K}:=r\circ \tau_N$ is the {\rm DT}-transformation of ${{\bf q}_N}$.
\ec

\begin{proof}
By Theorem \ref{1.5.8.35}, we have
\be
\la{basic.example.dt.cyc}
{\bf K}^*X_{i}=\tau_N^* X_{i+1}=\frac{F_{i+1}}{X_{i}F_{i-1}}.
\ee
By the explicit formulas of $F_{i}$, we have
$
F_i^t(l_j^+)=0$ for all $i, j\in I$. 
Therefore
\[({\bf K}^*X_{i})^t (l_j^+)=
F_{i+1}^t(l_j^+)-F_{i-1}^t(l_j^+)-X_i^t(l_j^+)=-X_i^t(l_j^+)=X_i^t(l_j^-).\]
\end{proof}

\subsection{Tagged ideal triangulations of a once-punctured disk.}
We present a combinatorial model for the cluster transformation $\tau_N$.

Let $D_N$ be a once-punctured disk with $N$ special points on its boundary. 
The  special points of $D_N$ divide its boundary into $N$  {\it boundary intervals}.
Let ${\bf m}$ be the set of  special points and the puncture of $D_N$.  
An ideal arc $\gamma$ is a curve up to isotopy in $D_N$ such that:
\begin{itemize}
\item the endpoints of $\gamma$ are two different points\footnote{We require that the endpoints of $\gamma$ are different in the case of once punctured disk. For a general decorated surface $\bS$, the the endpoints of $\gamma$ may coincide. } in ${\bf m}$;
\item $\gamma$ does not intersect itself;
\item except for the endpoints, $\gamma$ is disjoint from ${\bf m}$ and  the  boundary of $D_N$;
\item $\gamma$ is not isotopic to a boundary interval of $D_N$.
\end{itemize}
Following \cite[Definition 7.1]{FST}, a tagged arc $\gamma$ is an ideal arc $\gamma$ tagged  with some extra combinatorial data such that:
\begin{itemize}
\item if $\gamma$ connects the puncture and a special point,  we tag $\gamma$ either {\it plain} or {\it notched};
\item if $\gamma$ connects two special points, then we do not assign any data. 
\end{itemize}
In the figures, the plain tags are omitted and the notched tags are presented by the $\bowtie$ symbol.

\vskip 2mm

Denote by ${\bf A}^{\bowtie}(D_N)$ the set of tagged arcs of $D_N$. Two different tagged arcs $\alpha, \beta\in{\bf A}^{\bowtie}(D_N)$ are called {\it compatible} when one of the following cases holds:
\begin{itemize}
\item if $\alpha$ and $\beta$ both contain the puncture, then $\alpha$ and $\beta$ are tagged in the same way unless they correspond to the same ideal arc;
\item if either $\alpha$ or $\beta$ is disjoint from the puncture, then we require that $\alpha$ and $\beta$ are disjoint except for their endpoints.
\end{itemize}

\bd[{\cite{FST}}]
A {\it tagged ideal triangulation} of ${D}_N$ is a maximal collection of piecewise compatible tagged arcs. Let $\gamma$ be a tagged arc contained in a tagged ideal triangulation ${\cal T}$ of $D_N$. A \emph{flip} of ${\cal T}$ at $\gamma$ is a transformation of ${\cal T}$ that removes $\gamma$ and replace it with a $(unique)$ different tagged arc ${\gamma'}$ that, together with the remaining arcs, forms a new tagged ideal triangulation ${\cal T}'$.
\ed

\begin{figure}[ht]
\epsfxsize100pt
\centerline{\epsfbox{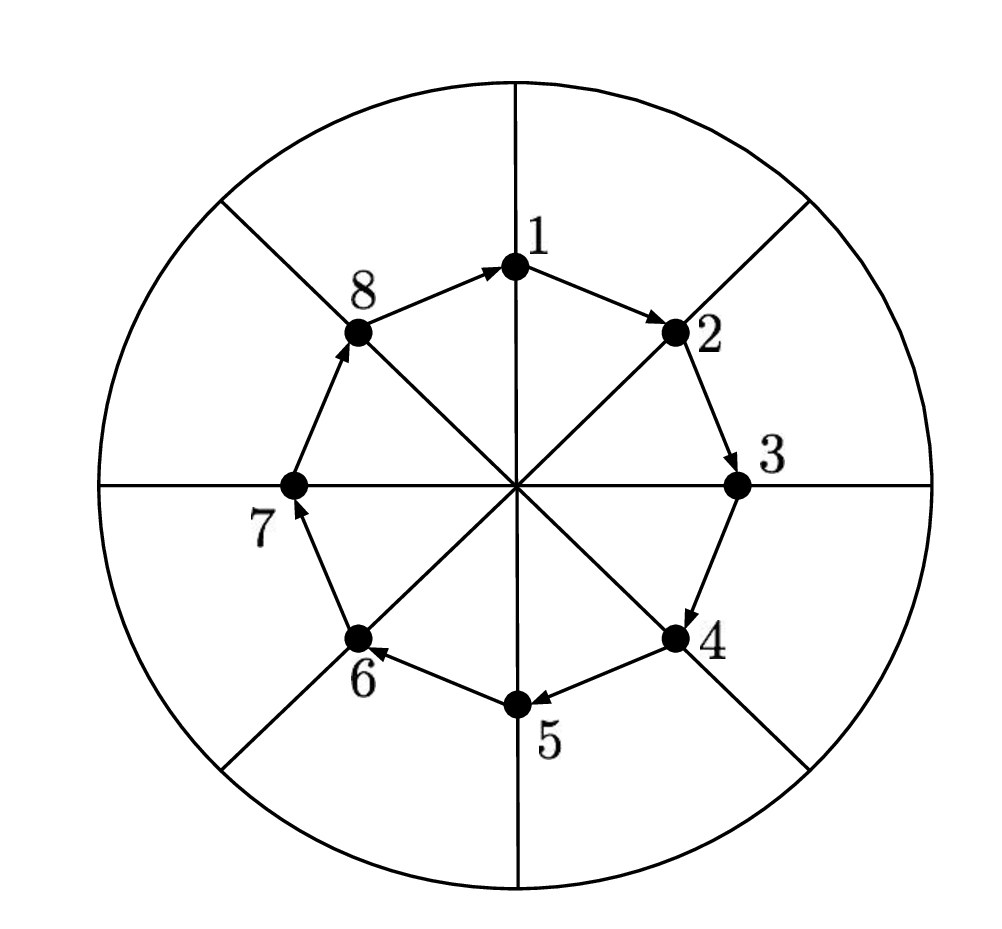}}
\caption{A tagged ideal triangulation corresponding to the quiver ${\bf q}_N$}
\label{taggedtriN}
\end{figure}

 The cardinality of every tagged ideal triangulation of $D_N$ is $N$ (\cite[Theorem 7.9]{FST}). Let us fix a tagged ideal triangulation ${\cal T}$. Let us label the tagged arcs in ${\cal T}$ by 1 through $N$.  It gives rise to a quiver ${\bf q}$ by placing a vertex in the midpoint of each tagged arc  and assigning an arrow from the vertex $i$ to the vertex $j$ if the corresponding tagged arc $i$ is to the right of the tagged arc $j$. For example, the tagged ideal triangulation in Figure \ref{taggedtriN} gives rise to the quiver ${\bf q}_N$.

\vskip 2mm
  
 It is easy to show that a flip at  a tagged arc is equivalent to the cluster mutation at the corresponding vertices of the corresponding quiver. 
Therefore the transformation $\tau_N$ in Theorem \ref{1.5.8.35} can be presented by a sequence of flips of the tagged ideal triangulations of $D_N$.

\paragraph{Example.}
If $N=2$, then $\tau_{2}=\mu_1\circ \pi_{1,2}\circ \mu_{1}$. As shown on Figure \ref{tagged2}, we start from a tagged ideal triangulation of $D_2$ with arcs labelled by $1$ and $2$. The first cluster mutation $\mu_1$ removes the plain arc 1 and replaces it with the notched arc 1 on the second graph. The permutation $\pi_{1,2}$ exchanges the labels of these two arcs. The last cluster mutation $\mu_1$ removes the plain arc 1 and replaces it with the notched arc 1 on the last graph. To summarize, the transformation $\tau_2$ preserves the underlying triangulation but replaces each plain arc by a notched one. 
\begin{figure}[ht]
\epsfxsize500pt
\centerline{\epsfbox{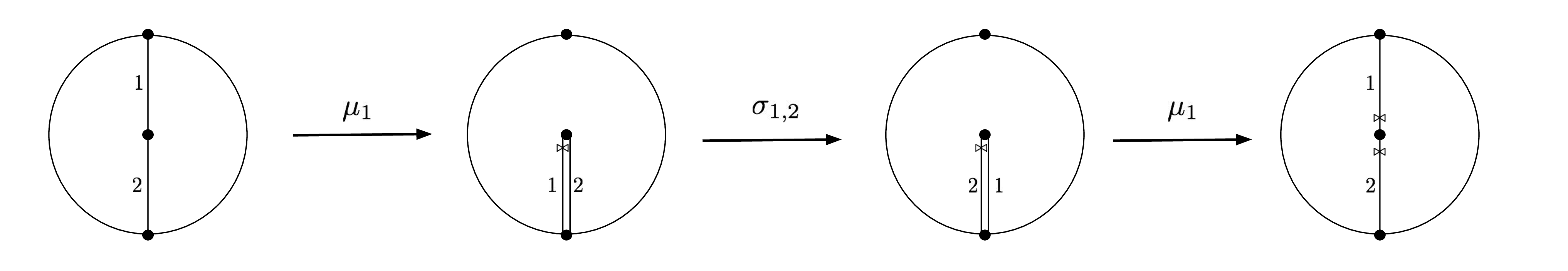}}
\caption{The transformation $\tau_2= \mu_1\circ \pi_{1,2}\circ \mu_{1}$.}
\label{tagged2}
\end{figure}

Similarly, for $N>2$,  $\tau_N$ notches all plain arcs, and therefore preserves ${\bf q}_N$. See  Figure \ref{taggedN}.

\begin{figure}[ht]
\epsfxsize220pt
\centerline{\epsfbox{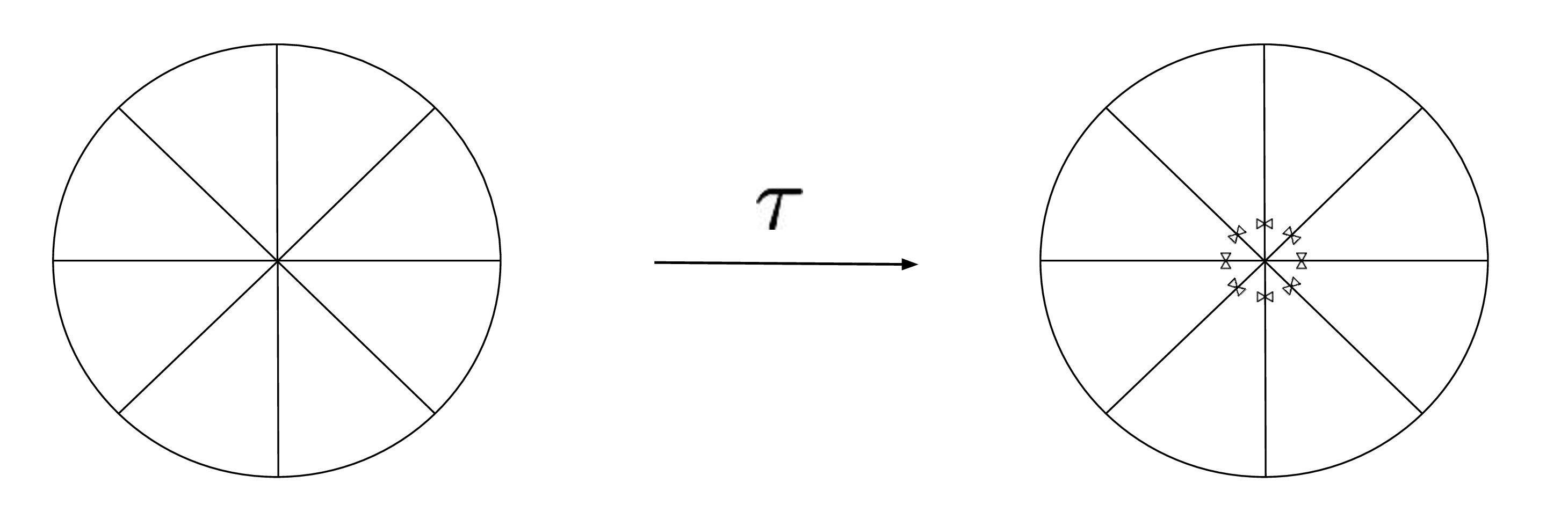}}
\caption{The transformation $\tau_N$.}
\label{taggedN}
\end{figure}

\subsection{Proof of Theorem \ref{1.5.8.35}} 
\la{proofofthm15835}
We prove Theorem \ref{1.5.8.35} in a more general setting for future use.

\bd
\la{quiver.gen.h}
Let ${\bf q}$ be a quiver containing ${\bf q}_N$.
The vertices of ${\bf q}$ are labelled by ${\rm J}$. 
The vertices of ${\bf q}_N$ are labelled by ${\rm I}=\{1,\ldots, N\}\subset {\rm J}$. We further assume that  for each vertex $k$ not in ${\bf q}_N$, the number of arrows from $k$ to ${\bf q}_N$  equals the number of arrows from ${\bf q}_N$ to $k$, i.e.,
\be
\la{zerocd}
\forall k \in {\rm J}-{\rm I}, \hskip 7mm  \sum_{i\in {\rm I}}\varepsilon_{ki}=0.
\ee
\ed

\bl Let ${\bf q}$ be as above. For $k\in {\rm J}-{\rm I}$, there is a unique
$
{\bf c}_k=(c_{k1}, \ldots, c_{kN}) \in \Z^N
$
such that
\begin{align}
&~\varepsilon_{ki}=c_{ki}-c_{k,i-1},\hskip 7mm  \forall i\in {\rm I}; \la{beta.eps.rel1}\\
&\min\{c_{k1},\ldots, c_{kN}\}=0. \la{mincd}
\end{align}
\el
\begin{proof}  The existence of ${\bf c}_k$ follows from \eqref{zerocd}. The uniqueness of ${\bf c}_k$ follows from \eqref{mincd}.
\end{proof}

\begin{example} The quiver ${\bf q}$ on the left of Figure \ref{quivergen} contains ${\bf q}_5$ and satisfies \eqref{zerocd}. We have 
\[ 
{\bf c}_6=(0,1,0,0,0),\hskip 1cm {\bf c}_7=(1,1,1,0,2).
\]
Mutating ${\bf q}$ at the direction $5$, we obtain a new quiver $\tilde{\bf q}$   on the right. Note that  $\tilde{\bf q}$ contains ${\bf q}_4$ and satisfies condition  \eqref{zerocd}. We have
\[ 
{\tilde{\bf c}}_5=(0,0,0,1), \hskip 7mm {\tilde{\bf c}}_6=(0,1,0,0), \hskip 7mm {\tilde{\bf c}}_7=(1,1,1,\min\{0,2\})=(1,1,1,0). 
\]
\begin{figure}[h]
\epsfxsize350pt
\centerline{\epsfbox{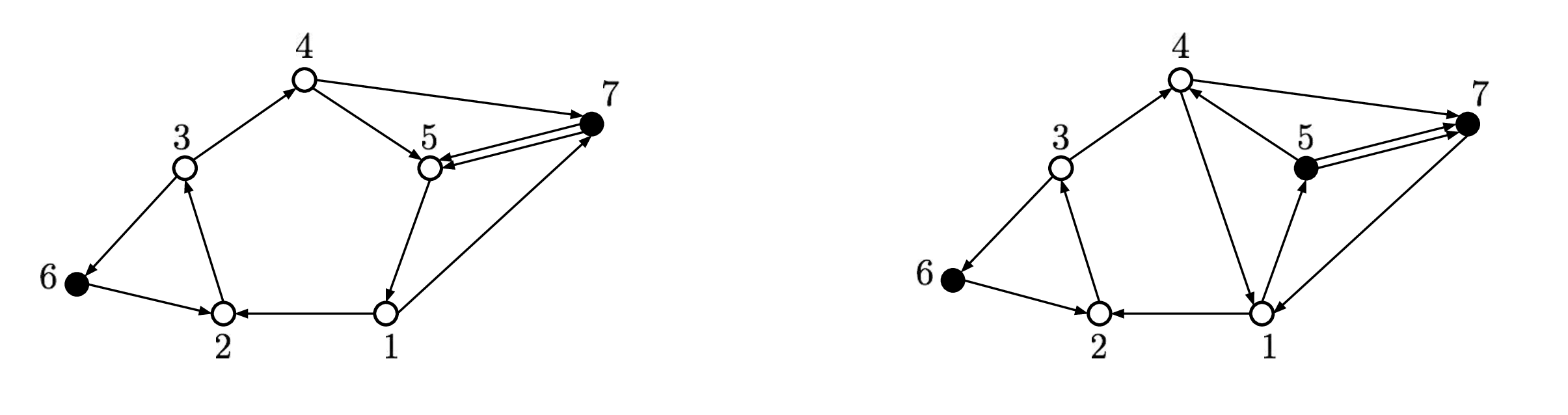}}
\caption{The right quiver is mutation of the left one at vertex $5$. Both satisfy condition \eqref{zerocd}.}
\label{quivergen}
\end{figure}
\end{example}

\bt 
\la{2.8.21.24.h}
Let ${\bf q}, {\bf c}_k$ be as above. The cluster transformation $\tau_N$ applying on the sub quiver ${\bf q}_N$  maps ${\bf q}$ to itself. The induced isomorphism of ${\cal F}_{\bf q}$ is given by
$$
\tau_N^*A_j=\left\{ \begin{array}{ll}
      A_jW & \mbox{if ${j}\in {\rm I}$} \\
     A_j & \mbox{if $j\notin {\rm I}$}, \\
   \end{array}
   \right.     
   \hskip 10mm
\tau_N^*X_j= \left\{ \begin{array}{ll}
      \frac{X_j}{Y_{j}Y_{j-1}} & \mbox{if $j\in {\rm I}$} \\
     X_j \prod_{i\in I} Y_{i}^{c_{ji}} & \mbox{if $j\notin {\rm I}$}, \\
   \end{array}
   \right.       
$$
where
\begin{align}
W:&=\sum_{i\in {\rm I}}\frac{Q_i}{A_iA_{i+1}},\hskip 7mm &Q_i:&=\prod_{k\in {\rm J}-{\rm I}}A_k^{c_{ki}};
\\
\la{induction.Yifunction}
Y_i:&=\frac{X_iF_{i-1}}{F_i},\hskip 7mm &F_i:&=1+X_i+X_iX_{i-1}+\ldots+X_iX_{i-1}\ldots X_{i-N+2}.
\end{align}
\et

\begin{proof} Theorem  \ref{2.8.21.24.h} is a generalization of Theorem \ref{1.5.8.35}.
We prove it by induction on $N$. 

\paragraph{1. Theorem \ref{2.8.21.24.h} holds for $N=2$.}
The cluster transformation $\tau_2$ mutates the quiver at the vertices labelled by 1 and 2, then switch them.  Condition \eqref{zerocd} asserts that
\[
\varepsilon_{k1}+\varepsilon_{k2}=0, \quad \forall k\in {\rm J}.
\]
It follows directly $\tau_2({\bf q})={\bf q}$. By definition, we have
\[
{\bf c}_k=(c_{k1},c_{k2})=(\max\{0, \varepsilon_{k1}\}, ~\max\{0, \varepsilon_{k2}\}).
\]
Therefore
\[
\tau_2^*A_1=\mu_2^*A_2=\frac{\prod_{k|\varepsilon_{k2}>0}A_k^{\varepsilon_{k2}}+\prod_{k|\varepsilon_{k2}<0}A_k^{-\varepsilon_{k2}}}{A_2}=A_1\big(\frac{\prod A_k^{c_{k2}}}{A_2A_1}+\frac{\prod A_k^{c_{k1}}}{A_1A_2}\big)=A_1W.
\]
Similarly, $\tau_2^*A_2=A_2W$. The rest $A_j$ remain intact.

For the ${\cal X}$-part,
note that $F_i=1+X_i$ for $i\in\{1,2\}$. So
\[
Y_1=\frac{1+X_2}{1+X_1^{-1}}, \hskip 7mm Y_2=\frac{1+X_1}{1+X_2^{-1}}, \hskip 7mm Y_1Y_2=X_1X_2.
\]
Therefore
\[
\tau_2^*X_1=\mu_2^*X_2=X_2^{-1}=\frac{X_1}{Y_1Y_2}, \quad \quad \tau_2^*X_2=X_1^{-1}=\frac{X_2}{Y_2Y_1}.
\]
For $k\in {\rm J}-\{1,2\}$, we have
\[
\tau_2^*X_k=X_k(1+X_1^{-{\rm sgn}(\varepsilon_{k1})})^{-\varepsilon_{k1}}(1+X_2^{-{\rm sgn}(\varepsilon_{k2})})^{-\varepsilon_{k2}}.
\]
If $\varepsilon_{k1}\geq 0$, then ${\bf c}_k=(c_{k1}, c_{k2})=(\varepsilon_{k1}, 0)$. Therefore
$$
\tau_2^*X_k=X_k(1+X_1^{-1})^{-c_{k1}}(1+X_2)^{c_{k1}} =X_kY_1^{c_{k1}}Y_2^{c_{k2}}. 
$$
The same formula holds for $\varepsilon_{k2}\geq 0$. 

\paragraph{2. The transformation $\tau_N$ maps ${\bf q}$ to itself.}
When $N>2$, without loss of generality, let us first mutate the quiver ${\bf q}$ at the direction $N$, obtaining a new quiver $\mu_{N}({\bf q})=\tilde{\bf q}$. 
Note that $\tilde{\bf q}$ contains ${\bf q}_{N-1}$ and satisfies condition \eqref{zerocd}. Let $\tau_{N-1}$ be the cluster transformation applying on ${\bf q}_{N-1}$. Using induction, we have $\tau_{N-1}(\tilde{\bf q})=\tilde{\bf q}$.
Therefore
$$
\tau_{N}({\bf q})=\mu_{N}\circ \tau_{N-1}\circ \mu_{N}({\bf q})=\mu_{N}^2({\bf q})={\bf q}.
$$

\paragraph{3. Proof of the ${\cal A}$-part.} By \eqref{epsilon.mutate}, the vectors  ${\tilde{\bf c}}_k$ of $\tilde{\bf q}=\mu_N({\bf q})$ are
\be
\la{tilda.beta.ind}
{\tilde{\bf c}}_N=(0,\ldots, 0, 1),\hskip 1cm {\tilde{\bf c}}_{k}=(c_{k1},\ldots, c_{k,N-2}, \min\{c_{k,N-1},c_{kN}\}),\quad \forall k\in {\rm J}-{\rm I}.
\ee
See Figure \ref{quivergen} for example.

Let $\{\widetilde{{X}_i}, \widetilde{{A}_i}\}$ be pairs of variables assigned to vertices $i$ of $\tilde{\bf q}$. By induction, we have
\be
\la{induction.a.part}
\tau_{N-1}^*\widetilde{{A}_j}=\left\{ \begin{array}{ll}
      \widetilde{{A}_j}\widetilde{W} & \mbox{if $j\in \{1,\ldots, N-1\}$} \\
     \widetilde{{A}_j} & \mbox{otherwise.}\\
   \end{array}
   \right.     
\ee
where
$$
 \widetilde{W}=\frac{\widetilde{{Q}_{N-1}}}{\widetilde{{A}_{N-1}}\widetilde{{A}_1}}+\sum_{i=1}^{N-2}\frac{\widetilde{Q_i}}{\widetilde{{A}_i}\widetilde{{A}_{i+1}}},\hskip 7mm \widetilde{{Q}_i} = \prod_{k\notin \{1,\ldots, N-1\}}\widetilde{{A}_k}^{\tilde{{c}}_{ki}}.
$$

Now we compute $\mu_N^* \widetilde{W}$. 
Note that
\be
\la{tras.a.mu.n}
\mu_N^*\widetilde{{A}_k} =\left\{ \begin{array}{ll}
A_k  &\mbox{ if } k\neq N\\
 \frac{A_1}{A_N}{\displaystyle \prod_{k\in {\rm J}-{\rm I}}}A_k^{-\min\{0, \varepsilon_{kN}\}}+\frac{A_{N-1}}{A_N}{\displaystyle\prod_{k\in {\rm J}-{\rm I}}}A_k^{-\min\{0, -\varepsilon_{kN}\}}  &\mbox{ if } k=N.\\
   \end{array}
   \right.
   \ee  
 It follows directly from \eqref{tilda.beta.ind}, \eqref{tras.a.mu.n}  that
\be
\forall i\in\{1,\ldots, N-2\}, \hskip 7mm
\mu_N^*\widetilde{{Q}_i} =Q_i.
\ee
Meanwhile
\begin{align}
\mu_N^*\widetilde{{Q}_{N-1}}&=\mu_N^*\Big(\prod_{k\notin \{1,\ldots, N-1\}}\widetilde{{A}_k}^{\tilde{{c}}_{k,N-1}}\Big)\stackrel{\eqref{tilda.beta.ind}\eqref{tras.a.mu.n}}{=\joinrel=\joinrel=}\Big(\mu_N^*\widetilde{{A}_{N}}\Big)\prod_{k\in {\rm J}-{\rm I}}A_k^{\min\{{c_{k,N-1},c_{kN}}\}}\nonumber\\
&=\frac{A_1}{A_N}\prod_{k\in {\rm J}-{\rm I}}A_k^{\min\{{c}_{k, N-1}, {c}_{kN}\}-\min\{0, \varepsilon_{kN}\}}+\frac{A_{N-1}}{A_N}\prod_{k\in {\rm J}-{\rm I}} A_k^{\min\{{c}_{k, N-1}, {c}_{kN}\}-\min\{0, -\varepsilon_{kN}\}}\nonumber\\
&\stackrel{\eqref{beta.eps.rel1}}{=\joinrel=}\frac{A_1\prod_{k\in {\rm J}-{\rm I}}A_k^{{c}_{k,N-1}}}{A_N}+\frac{A_{N-1}\prod_{k\in {\rm J}-{\rm I}}A_k^{{c}_{kN}}}{A_N} \nonumber\\
&=\frac{A_1Q_{N-1}}{A_N}+\frac{A_{N-1}Q_{N}}{A_N}
\end{align}
Therefore
\[
\mu_N^*\widetilde{W}=\sum_{i\in {\rm I}}\frac{Q_i}{A_iA_{i+1}}=W.
\]

We consider the following cases.
\begin{enumerate}
\item[(a)] If $j\notin {\rm I}$, then clearly $\tau_N^* A_j=A_j$.
\item[(b)] If $j\in  \{1,\ldots, N-1\},$ then 
\be \tau_{N}^* A_j =\mu_{N}^*\circ \tau_{N-1}^*\circ \mu_{N}^*(A_j)=\mu_N^* \big( \tau_{N-1}^* \widetilde{{A}_j}\big)\stackrel{\eqref{induction.a.part}}{=\joinrel=}\mu_N^*(\widetilde{{A}_j}\widetilde{W})= A_j W.
\ee
 \item[(c)] If $j=N$, then 
 \be \mu_{N}^*A_N=\frac{\widetilde{{A}_1}\prod_{k\in {\rm J}-{\rm I}}\widetilde{{A}_k}^{-\min\{0, \varepsilon_{kN}\}}+\widetilde{{A}_{N-1}}\prod_{k\in {\rm J}-{\rm I}}\widetilde{{A}_k}^{-\min\{0, -\varepsilon_{kN}\}}}{\widetilde{{A}_N}}.
 \ee
 Note that $\tau_{N-1}^*\widetilde{{A}_{k}}=\widetilde{{A}_k}$ for $k\notin \{1,\ldots, N-1\}$. Therefore
\be
\tau_{N}^*A_N=\mu_N^*\Big(\frac{\widetilde{{A}_1}\prod_{k\in {\rm J}-{\rm I}}\widetilde{{A}_k}^{-\min\{0, \varepsilon_{kN}\}}+\widetilde{{A}_{N-1}}\prod_{k\in {\rm J}-{\rm I}}\widetilde{{A}_k}^{-\min\{0, -\varepsilon_{kN}\}}}{\widetilde{{A}_N}}\cdot \widetilde{W}\Big)
=A_NW. 
\ee
\end{enumerate}
\paragraph{4. Proof of the ${\cal X}$-part.} 
By induction, we  have
\be
\la{induction.x.part.trans}
\tau_{N-1}^*\widetilde{{X}_{k}}
=\left\{ \begin{array}{ll}
\frac{\widetilde{{X}_1}}{\widetilde{{Y}_1}\widetilde{{Y}_{N-1}}} &\mbox{if } k=1,\\
\frac{\widetilde{{X}_k}}{\widetilde{{Y}_k}\widetilde{{Y}_{k-1}}} & \mbox{if } k \in\{2,\ldots, N-1\},\\
\widetilde{{X}_N} \widetilde{{Y}_{N-1}}& \mbox{if } k=N,\\
\widetilde{{X}_k}\prod_{j=1}^{N-1}\widetilde{{{Y}_{j}}}^{\tilde{{c}}_{kj}} &\mbox{otherwise}. \\
   \end{array}
   \right.
\ee
Here $\widetilde{{Y}_i}$ is defined similarly via \eqref{induction.Yifunction}. Now we compute $\mu_N^* \widetilde{{Y}_i}$. 
Note that
\be
\mu_N^*\widetilde{{X}_k}=\left\{ \begin{array}{ll}
X_k(1+X_N^{-{\rm sgn}(\varepsilon_{kN})})^{-\varepsilon_{kN}}  &\mbox{ if } k\in {\rm J}-{\rm I},\\
X_N^{-1}  &\mbox{ if } k=N,\\
X_1(1+X_N) &\mbox{ if } k=1,\\
X_{N-1}(1+X_N^{-1})^{{-1}} & \mbox{ if } k=N-1,\\
X_k & \mbox{ if } k \in\{2,\ldots, N-2\}.\\
   \end{array}
   \right.
\ee
By explicit calculations one obtains
$$
\mu_N^*\widetilde{{F}_i}=\left\{ \begin{array}{ll}
F_i  &\mbox{ if } i\in\{1,\ldots, N-2\},\\
F_N(1+X_N)^{-1}&\mbox{ if } i=N-1.\\
   \end{array}
   \right.
$$
Therefore
$$
\mu_N^*\widetilde{{Y}_{i}}=\left\{ \begin{array}{ll}
Y_i  &\mbox{ if } i\in\{1,\ldots, N-2\},\\
Y_NY_{N-1}&\mbox{ if } i=N-1.\\
   \end{array}
   \right.
$$
We consider the following cases.
\begin{enumerate}
\item[(a)]
If $k\in \{2,\ldots, N-2\}$, then  $\mu_N^* X_k =\widetilde{{X}_k}$. Therefore
\be
\tau_{N}^* X_i=\mu_N^*\circ \tau_{N-1}^*\widetilde{{X}_k}=\mu_N^*\Big(\frac{\widetilde{{X}_k}}{\widetilde{{Y}_k}\widetilde{{Y}_{k-1}}}\Big)=\frac{X_k}{Y_kY_{k-1}}.
\ee

\item[(b)] If $k=N$, then
\be
\la{case.k=N}
\tau_{N}^*X_N=\mu_N^*\circ \tau_{N-1}^*\circ \mu_N^*({X}_N)=\mu_N^*\circ \tau_{N-1}^*(\widetilde{{X}_N}^{-1})=\frac{1}{\mu_N^*(\widetilde{{X}_N} \widetilde{{Y}_{N-1}})}=\frac{X_N}{Y_N Y_{N-1}}.
\ee

\item[(c)]
If $k\in {\rm J}-{\rm I}$, then 
\be
\la{taunkinji}
\tau_{N}^*X_k=\mu_N^*\circ\tau_{N-1}^*\Big((1+\widetilde{{X}_N}^{{\rm sgn}(\varepsilon_{kN})})^{\varepsilon_{kN}}\widetilde{{X}_k}\Big)
\ee
Note that
\begin{align}
\mu_N^*\circ\tau_{N-1}^*(1+\widetilde{{X}_N})&=1+{\frac{Y_{N}Y_{N-1}}{X_N}}=1+\frac{X_{N-1}F_{N-2}}{F_{N}}=\frac{F_{N}+X_{N-1}F_{N-2}}{F_{N}} \nonumber\\
&=\frac{(1+X_N)F_{N-1}}{F_N}=(1+X_N^{-1})Y_N. \la{tau.action.xn}
\end{align}
If $\varepsilon_{kN}={c}_{kN}-{c}_{k,N-1}>0$, by \eqref{induction.x.part.trans}\eqref{taunkinji}\eqref{tau.action.xn}, we get
\begin{align}
\tau_N^*X_k&=\Big((1+X_N^{-1})Y_N\Big)^{\varepsilon_{kN}}\mu_N^*\Big(\widetilde{{X}_k}\prod_{j=1}^{N-1}\widetilde{{Y}_{j}}^{\tilde{{c}}_{kj}}\Big) \nonumber\\
&=\Big((1+X_N^{-1})^{\varepsilon_{kN}}\mu_N^*\widetilde{{X}_k}\Big)\Big(Y_N^{{c}_{kN}-{c}_{k,N-1}}\mu_N^*(\prod_{j=1}^{N-1}\widetilde{{Y}_{j}}^{\tilde{{c}}_{kj}})\Big)\nonumber\\
&=X_k \prod_{j=1}^N Y_j^{{c}_{kj}}.
\end{align}
By the same argument, the same formula holds for $\varepsilon_{kN}\leq0$.

\item[(d)]
If $k=N-1$, by \eqref{induction.x.part.trans}, we have
\be
\la{20.32.11.8.h}
\mu_N^*\circ \tau_{N-1}^*(\widetilde{{X}_{N-1}})=\mu_N^*\Big(\frac{\widetilde{{X}_{N-1}}}{\widetilde{{Y}_{N-2}}\widetilde{{Y}_{N-1}}}\Big)=\frac{X_{N-1}(1+X_N^{-1})^{-1}}{Y_NY_{N-1}Y_{N-2}}.
\ee
Note that $\mu_N^* X_{N-1}=\widetilde{{X}_{N-1}}(1+\widetilde{{X}_N})$.
Therefore
\begin{align}
\tau_N^* X_{N-1}&=\mu_N^*\circ \tau_{N-1}^*\Big(\widetilde{{X}_{N-1}}(1+\widetilde{{X}_N})\Big){\stackrel{\eqref{20.32.11.8.h}\eqref{tau.action.xn}}{=\joinrel =\joinrel=}}\frac{X_{N-1}(1+X_N^{-1})^{-1}}{Y_NY_{N-1}Y_{N-2}}(1+X_N^{-1})Y_N \nonumber\\
&=\frac{X_{N-1}}{Y_{N-1}Y_{N-2}}. 
\end{align}
If $k=1$, then by similar calculations we get
$\tau_N^* X_{1}=\frac{X_1}{Y_1Y_N}.$  \end{enumerate}
\end{proof}

\section{The Weyl group acts on  ${\cal X}_{{\rm PGL}_m, \bS}$ and ${\cal A}_{{\rm SL}_m, \bS}$ by cluster transformations}
\la{weyl.group.action.hhs}

Let $\bS$ be an admissible decorated surface. 
We recall the construction of cluster coordinates of the pair $({\cal X}_{{\rm PGL}_m, \bS}, {\cal A}_{{\rm SL}_m, \bS})$ introduced in \cite[Section 9, 10]{FG2}. 
If $\bS$ is a sphere with 3 punctures, then we assume $m >2$.
We show that the Weyl group actions on both ${\cal X}_{{\rm PGL}_m, \bS}$ and ${\cal A}_{{\rm SL}_m, \bS}$ are cluster transformations.

\begin{figure}[ht]
\epsfxsize 120pt
\centerline{\epsfbox{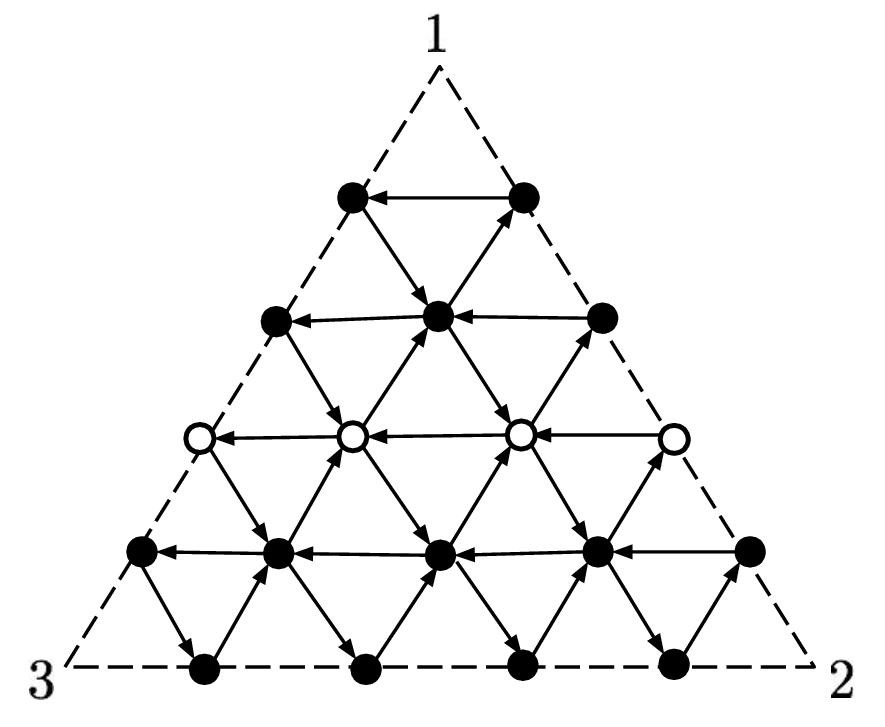}}
\caption{A 5-triangulation. The $\circ$- vertices are of distance $3$ to the vertex  1.}
\label{5tri}
\end{figure}

An $m$-triangulation of a triangle gives rise to a quiver whose vertices are parametrized  by 
\be
\la{GAMMA.m.lh}
\Gamma_m=\{(a,b,c)~|~ a+b+c=m, ~~a,b, c\in \Z_{\geq 0}\} -\{(m,0,0), (0,m,0), (0,0,m)\},
\ee
and arrows compatible with the orientation of the triangle. The vertices  $(a, b,c)$ with $a, b, c\in \Z_{>0}$ are called inner vertices. The other vertices are on the edges of the triangle.
See Figure \ref{5tri}.  

From now on, let us fix a puncture $p$ of $\bS$.
An ideal triangulation of $\bS$ is a triangulation of $\bS$ whose vertices are marked points  (i.e., punctures or special points) of $\bS$. 
 Since  $\bS$ is admissible, it admits an ideal triangulation ${\cal T}$ such that 
 \begin{itemize}
 \item ${\cal T}$ contains no {\it self-folded} triangles. See Figure \ref{selffold}.
 \begin{figure}[ht]
\epsfxsize40pt
\centerline{\epsfbox{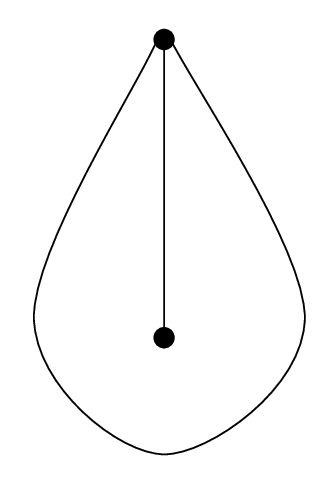}}
\caption{Self-folded triangle}
\label{selffold}
\end{figure}
 \item
 ${\cal T}$ contains no edge whose both vertices are $p$. 
 \end{itemize}
The ideal triangles in ${\cal T}$ surrounding  $p$ gives rise to a punctured disk.

We assign an $m$-triangulation to each triangle ${\rm t}\in {\cal T}$, {getting} a quiver ${\bf q}$. 
Let $i\in \{1,\ldots, m-1\}$. Denote by ${\bf q}_{p, i}$ the subquiver consists of vertices of distance $m-i$ to the puncture $p$. Note that ${\bf q}_{p, i}$ is a cycle. 
The pair $({\bf q}, {\bf q}_{p, i})$ satisfies conditions in Definition \ref{quiver.gen.h}. 

Denote by $\tau_{p,i}$ the cluster transformation \eqref{cluster.trans.tau} on the subquiver  ${\bf q}_{p, i}$.

\begin{example} Let $m=4$. If there are 4 ideal triangles surrounding $p$, then the quiver ${\bf q}$ locally looks like Figure \ref{waction}. The $\circ$- quiver ${\bf q}_{p, 1}$ consists of vertices of distance $3$ to the puncture $p$.
\end{example}

\begin{figure}[ht]
\epsfxsize150pt
\centerline{\epsfbox{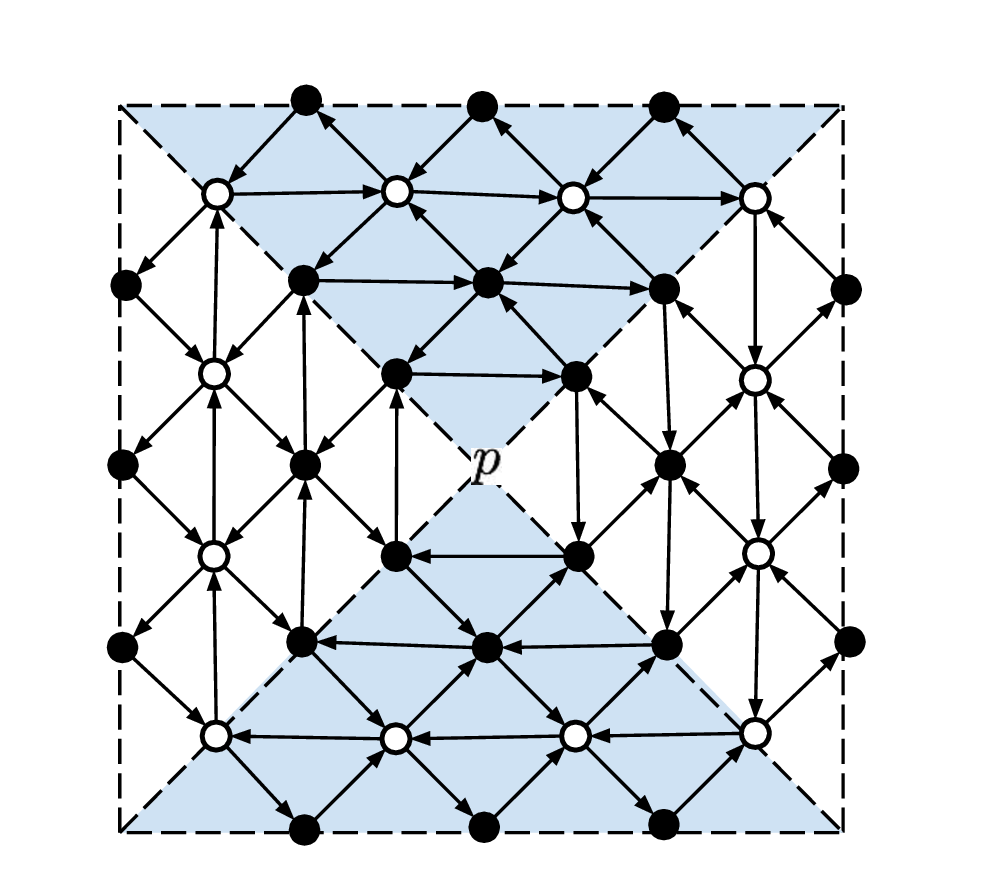}}
\caption{}
\label{waction}
\end{figure}

In this section, we assign to each vertex $v$ of  ${\bf q}$ a function $A_v$ (respectively $X_v$) of the space ${\cal A}_{{\rm SL}_m, \bS}$ (respectively ${\cal X}_{{\rm PGL}_m, \bS}$). The set $\{A_v\}$ (respectively $\{X_v\}$) provides a cluster coordinate system for ${\cal A}_{{\rm SL}_m, \bS}$ (respectively ${\cal X}_{{\rm PGL}_m, \bS}$).

Recall that the puncture $p$  corresponds  to a Weyl group (of type $A_{m-1}$) action on both ${\cal A}_{{\rm SL}_m, \bS}$ and ${\cal X}_{{\rm PGL}_m, \bS}$.
The Weyl group is generated by simple reflections $s_{p,i}$, $i\in\{1,\ldots, m-1\}$. 

\bt 
\la{2.17.22.42h}
The map $s_{p,i}$ is  exactly the cluster transformation $\tau_{p,i}$, i.e.,
\be
\la{2.17.12.58h}
s_{p,i}^* A_v=\tau_{p,i}^* A_v, \quad 
s_{p,i}^* X_v=\tau_{p,i}^* X_v.\quad \quad \forall v\in \{\mbox{vertices of ${\bf q}$.}\}
\ee
\et

We prove Theorem \ref{2.17.22.42h} in the rest of this section.

\subsection{The moduli space ${\cal A}_{{\rm SL}_m, \bS}$}
\la{aspace.def.weyl.action}

\paragraph{The decorated flag variety ${\cal A}_{{\rm SL}_m}$.}
Let $V_m$ be an $m$-dimensional vector space  with a volume form $\omega\in {\rm det} V_m^*$. A flag $F_{\bullet}$ is a collection of subspaces in $V_m$:
\be \la{FLAG}
F_1\subset F_2\subset \ldots \subset F_{m-1},\quad \dim F_i=i.
\ee
A decorated flag  is a flag ${F}_{\bullet}$ with a choice of non-zero vectors $f_{(i)}\in \wedge^i F_i$ for each $i=1,\ldots, m-1$ called {\it decorations}. The decorated flag variety ${\cal A}_{{\rm SL}_m}$ parametrizes decorated flags for ${\rm SL}_m{:={\rm SL}(V_m)}$. The group ${\rm SL}_m$ acts on ${\cal A}_{{\rm SL}_m}$ on the left. The Cartan subgroup of ${\rm SL}_m$ acts on ${\cal A}_{{\rm SL}_m}$ on the right by rescaling the decorations. Note that ${\cal A}_{{\rm SL}_m}\stackrel{\sim}{=}\G/ {\rm U}$.

\paragraph{Additive characters associated to decorated flags.}
Let ${\rm F}\in {\cal A}_{{\rm SL}_m}$ be a decorated flag. Its  stabilizer ${\rm U}_{\rm F}$ is a unipotent subgroup of ${\rm SL}_m$.  A representative of ${\rm F}$ is a linear basis $(f_1,\ldots, f_m)$ of $V_m$ which gives rise to decorations of ${\rm F}$
$$ f_{(i)}:=f_1\wedge \ldots \wedge f_i\in \wedge^i F_i,\quad  \forall i\in\{1,\ldots, m-1\}; \hskip 7mm \langle f_1\wedge\ldots\wedge f_m, \omega\rangle =1.
$$

Let $u\in {\rm U}_{\rm F}$. Note that $e_{i}:= u(f_{i+1})-f_{i+1} \in F_i$. The vector $f_{(i-1)}\wedge e_i\in \wedge ^i F_i$ is independent of the representative $(f_1,\ldots, f_m)$ chosen. 
It determines a unique $\chi_{{\rm F},i}(u)\in {\Bbb A}^1$ such that
\be
f_{(i-1)}\wedge e_i =\chi_{{\rm F},i}(u) f_{(i)}, \hskip 7mm \forall i\in \{1,\ldots, m-1\}.
\ee
Therefore we associate to ${\rm U}_{\rm F}$ a set of {\it additive characters}
\be
(\chi_{{\rm F},1},\ldots, \chi_{{\rm F},{m-1}}):\quad {\rm U}_{{\rm F}}\lra {\Bbb A}^{m-1}.
\ee

\paragraph{The moduli space ${\cal A}_{{\rm SL}_m, \bS}$.} 
The moduli space ${\cal A}_{{\rm SL}_m, \bS}$ parametrizes pairs $({\cal L}, \gamma=\{{\rm F}_s\})$ where 
${\cal L}$ is a \emph{twisted} ${\rm SL}_m$-local system on $\bS$, 
and $\gamma$ assigns to every marked point $s$ a section ${\rm F}_s$ of ${\cal L}\otimes_{{\rm SL}_m}{\cal A}_{{\rm SL}_m}$. For a puncture $p$, the assigned section ${\rm F}_p$ is invariant under the monodromy $u_p$ around $p$. 
Thus $u_p$ is unipotent and belongs to the stabilizer of ${\rm F}_p$. 
The functions 
\be
\la{potential.2.19.hh}
{\cal W}_{p,i}:=\chi_{{\rm F}_p,i}(u_p), \quad i\in \{1,\ldots, m-1\}
\ee
are called {\it partial potentials} of ${\cal A}_{{\rm SL}_m, \bS}$ associated to the puncture $p$.

For each puncture $p$ of $\bS$, there is a Weyl group action on ${\cal A}_{{\rm SL}_m, \bS}$ by rescaling the decorations of the flat section ${\rm F}_p$. If the decorations of ${\rm F}_p$ are presented by the nonzero vectors $f_{(k)}, k=\{1,\ldots, m-1\}$,  then the action of the simple reflection $s_{p, i}$ changes the decorations to
\be
\la{actionsip.2.17}
f_{(1)}, \ldots, {\cal W}_{p,i} f_{(i)}, \ldots, f_{(m-1)},
\ee
and keeps the rest intact.

\paragraph{Local picture: configurations of three decorated flags.}

Let $({\rm F}, {\rm G}, {\rm H})$ be a configuration of three decorated flags, described by sets of nonzero vectors:
$$
{\rm F}=(f_{(1)},\ldots, f_{(m-1)}), \quad {\rm G}=(g_{(1)},\ldots, g_{(m-1)}), \quad {\rm H}=(h_{(1)},\ldots, h_{(m-1)}).
$$
Recall the $m$-triangulation of a triangle.
Each vertex $(a, b, c)\in \eqref{GAMMA.m.lh}$ gives rise to a function 
\be
\la{delatafgcord}
\Delta_{a, b,c}({\rm F, G,H}):=\langle f_{(a)}\wedge g_{(b)}\wedge h_{(c)}, \omega \rangle.
\ee
Forgetting the decorations, we get a natural projection $\pi: {\cal A}_{{\rm SL}_m} \rightarrow {\cal B}_{{\rm SL}_m}$. If $({\rm F, G, H})$ is generic, then there is a unique $u\in {\rm U}_{\rm F}$ such that 
$ 
u\cdot \pi({\rm H})=\pi({\rm G}). 
$ 
We define the potential
\be
\la{potential.i.12.59h}
{\cal W}_{{\rm F}, i}({\rm F, G, H}):=\chi_{{\rm F},i}(u).
\ee

\begin{figure}[ht]
\epsfxsize200pt
\centerline{\epsfbox{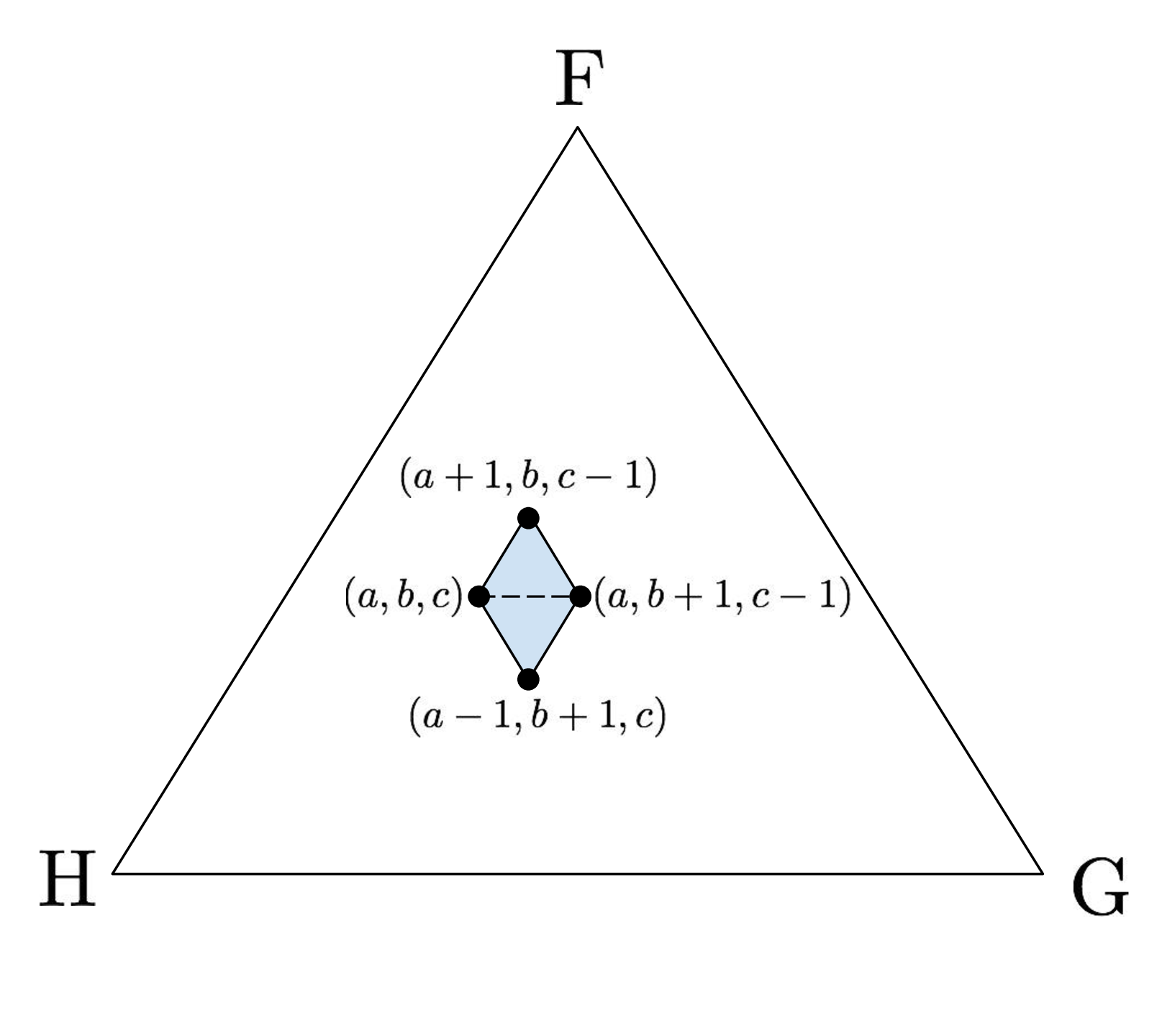}}
\caption{ }
\label{rhombi}
\end{figure}

Let $\alpha$ be the arrow $(a,b,c) \leftarrow(a, b+1,c-1 )$ in the $m$-triangulation. 
 As shown on Figure \ref{rhombi}, there is a unique rhombus with the  diagonal $\alpha$.  Its 
vertices correspond to functions in \eqref{delatafgcord}. 

Set $\Delta_{m,0,0}=\Delta_{0,m,0}=\Delta_{0,0,m}=1$. We consider the ratio

\be
\la{rhombi.function.r}
R_{\alpha}:=\frac{\Delta_{a+1,b,c-1}\Delta_{a-1,b+1,c}}{\Delta_{a,b,c}\Delta_{a,b+1,c-1}}.
\ee

\bl[{\cite[Section 3]{GS}}] 
\la{potential.i.12.59hll}
The potential \eqref{potential.i.12.59h} is 
\be
{\cal W}_{{\rm F},i}=\sum_{\alpha\in\{\mbox{arrows of row $i$\}} } R_{\alpha}.
\ee
\el

\begin{example}
Let ${\rm G}={\rm SL}_5$. There are three rhombi in row $2$ as shown on Figure \ref{potential}. Therefore
\be
{\cal W}_{{\rm F},2}=\frac{\Delta_{3,0,2}\Delta_{1,1,3}}{\Delta_{2,0,3}\Delta_{2,1,2}}+\frac{\Delta_{3,1,1}\Delta_{1,2,2}}{\Delta_{2,1,2}\Delta_{2,2,1}}+
\frac{\Delta_{3,2,0}\Delta_{1,3,1}}{\Delta_{2,2,1}\Delta_{2,3,0}}.
\ee
\end{example}

\begin{figure}[ht]
\epsfxsize100pt
\centerline{\epsfbox{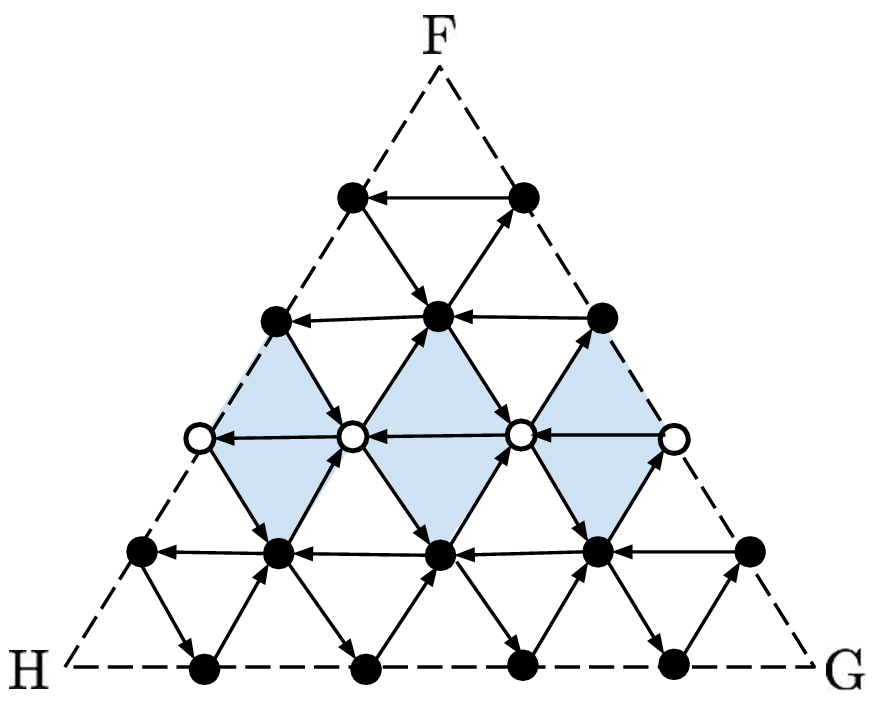}}
\caption{ }
\label{potential}
\end{figure}

\paragraph{Global picture: cluster coordinates of ${\cal A}_{{\rm SL}_m, \bS}$.} Recall the quiver ${\bf q}$ associated to an ideal triangulation ${\cal T}$ of $\bS$. 
Let $v$ be a vertex of ${\bf q}$.
Assume that $v$ is contained in a triangle ${\rm t}\in {\cal T}$ and labelled by $(a,b,c)\in \Gamma_m$.   
Restricting the data $({\cal L}, \gamma)\in {\cal A}_{{\rm SL}_m, \bS}$ to the triangle ${\rm t}$, we get  a configuration $({\rm F,G,H})$ of three decorated flags. We set
\be
\la{cluster.cor.Aspace.h}
A_v:=\Delta_{a,b,c}({\rm F,G,H}).
\ee
The set $\{A_v\}$ is a coordinate system of ${\cal A}_{{\rm SL}_m, \bS}$.

The subquiver ${\bf q}_{p,i}$ is a cycle. Every arrow $\alpha$ of ${\bf q}_{p,i}$ corresponds to a rhombi term $R_{\alpha}$. The following Lemma is a direct sequence of Lemma \ref{potential.i.12.59hll}.
\bl
The potential \eqref{potential.2.19.hh} is
$$
{\cal W}_{p,i}=\sum_{\alpha \in \{\mbox{arrows of ${\bf q}_{i,p}$}\}} R_{\alpha}.
$$
\el
\begin{example}
Let $\G={\rm SL}_4$. If there are 4 ideal triangles surrounding $p$, then the function ${\cal W}_{p,1}$ is the sum of functions $R_\alpha$ assigned to the shadowed rhombi in Figure \ref{potentialaroundp}.
\end{example}
\begin{figure}[ht]
\epsfxsize100pt
\centerline{\epsfbox{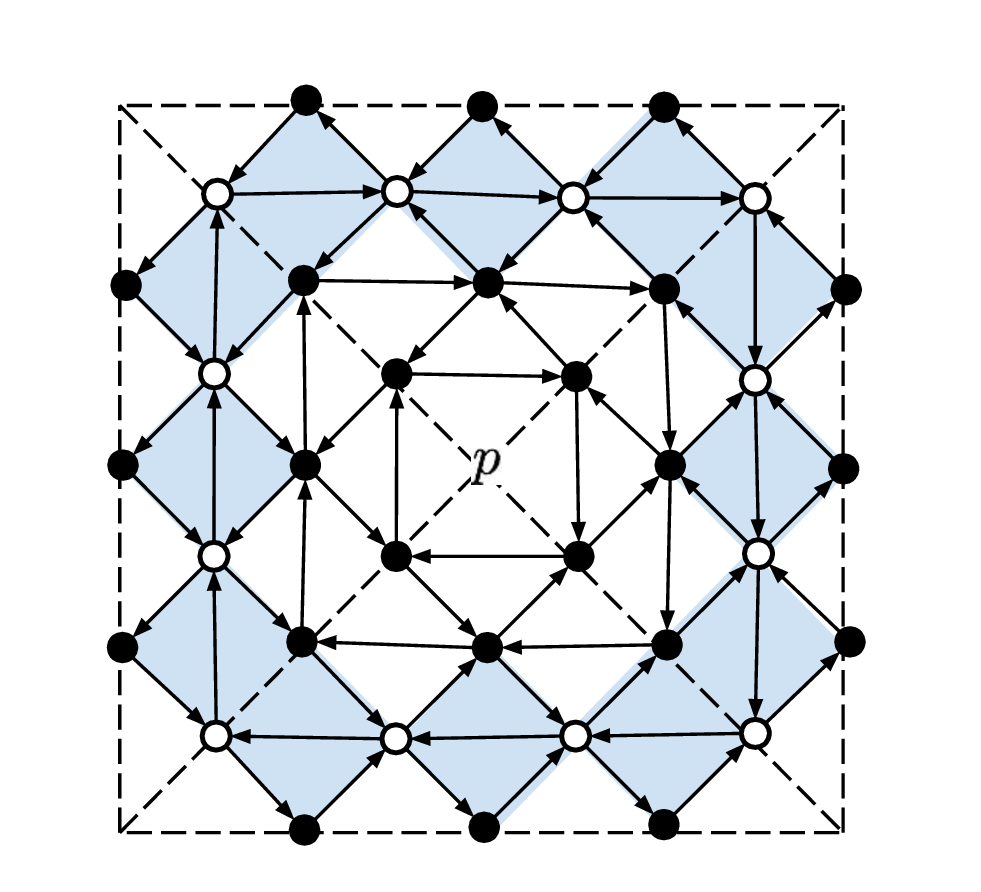}}
\caption{ }
\label{potentialaroundp}
\end{figure}

\paragraph{Proof of Theorem \ref{2.17.22.42h}: ${\cal A}$-Part.} According the definition of $A_v$ and \eqref{actionsip.2.17}, we have
$$
s_{p,i}^* A_v = \left\{ \begin{array}{ll}
      A_v{\cal W}_{p,i} & \mbox{if $j$ is vertex of ${\bf q}_{p,i}$}, \\
     A_v & \mbox{otherwise}. \\
   \end{array}
   \right.     
$$
Note that ${\cal W}_{p,i}$ is exactly the function $W$ in Theorem \ref{2.8.21.24.h}. Therefore we have
$
\tau_{p,i}^* A_v= s_{p,i}^* A_v.
$ 
 
\subsection{The moduli space ${\cal X}_{{\rm PGL}_m, \bS}$}
A flag $F_{\bullet}$ for ${\rm PGL}_m$ is a nested 
collection (\ref{FLAG}) of subspaces in a vector space $V_m$. 
 The flag variety ${\cal B}_{{\rm PGL}_m}$ parametrizes flags for ${\rm PGL}_m$.
First we consider the cases when $m=2,3$.

\paragraph{Local picture: the moduli space ${\cal X}_{{\rm PGL}_2, D_n}$.} 
The flag variety ${\cal B}_{{\rm PGL}_2}$ parametrizes lines in $V_2$. Let $(L_1,\ldots, L_4)$ be a quadruple of lines. Let $\omega \in {\rm det} V_2^*$ be a volume form. We choose  nonzero vectors $l_i\in L_i$. Let $\Delta(l_i\wedge l_j):=\langle l_i\wedge l_j, \omega\rangle$.
 We set the cross ratio
\be
\la{cross.ratio}
r^+(L_1, L_2, L_3, L_4):=\frac{\Delta(l_1\wedge l_2)\Delta(l_3\wedge l_4)}{\Delta(l_1\wedge l_4)\Delta(l_2\wedge l_3)}. 
\ee

Let $D_n$ be a punctured disk with $n$ special points on its boundary.  We label the 
special points clockwise from 1 to $n$.
The space ${\cal X}_{{\rm PGL}_2, D_n}$ parametrizes data $({\cal L}, \gamma=\{L_p, L_1,\ldots, L_n\})$, where ${\cal L}$ is a ${\rm PGL}_2$-local system on $D_n$, and $\gamma$ assigns to the puncture $p$ a flat section ${L_p}$ of  ${\cal L}\otimes_{{\rm PGL}_2} {\cal B}_{{\rm PGL}_2}$ invariant under the 
monodromy around $p$, and to each special point $i$ a flat section ${L}_i$. 
We connect each special point and the puncture, obtaining a triangulation of $D_n$. We restrict the pair $({\cal L}, \gamma)$ to the ideal quadrilateral with vertices $p$, $i-1$, $i$, $i+1$. We consider 
\be
\la{2.17.12.55h}
X_i:=r^+(L_p, L_{i-1},L_{i}, L_{i+1}).
\ee
The set $\{X_1,\ldots, X_n\}$ gives rise to a coordinate system of ${\cal X}_{{\rm PGL}_2, D_n}$. 

If the monodromy around the puncture is generic, then there is another flat section ${L}_p'$ invariant under the monodromy. We get a $\Z/2$-action on ${\cal X}_{{\rm PGL}_2, D_n}$ via replacing $L_p$ by ${L}_p'$.
Let
$$
{X}_i':=r^+( {L}_p', L_{i-1}, L_{i}, L_{i+1}), \quad {Y}_i:=r^+(L_{i}, L_{p}, L_{i+1}, {L}_p').
$$

\bl[{\cite[Lemma 12.3]{FG1}}] 
\la{Lemma12.3.FG1}
We have
\[
Y_i=\frac{X_iF_{i-1}}{F_i},\quad 
{X}_i'=\frac{X_i}{Y_iY_{i-1}}=\frac{F_i}{X_{i-1}F_{i-2}},\quad \mbox{ where } F_i=1+X_i+X_{i}X_{i-1}+\ldots+X_{i}\ldots X_{i-n+2}.
\]
\el

\paragraph{Cross-ratio versus triple ratio.} 
We consider a triple of flags for ${\rm PGL}_3$
$$
F_{\bullet}=(F_1\subset F_2),\quad G_{\bullet}=(G_1\subset G_2), \quad H_{\bullet}=(H_1\subset H_2).
$$
Let $\omega\in {\rm det} V_3^*$ be a volume form. We choose  nonzero vectors $f_1\in F_1$, $f_2\in \wedge^2 F_2$ and the same for $G_{\bullet}$ and $H_{\bullet}$. The following triple ratio  is independent of the choices of $\omega$ and $f_i, g_i, h_i$,
\be
\la{triple.ratio}
r_3^+(F_{\bullet}, G_{\bullet}, H_{\bullet}):=\frac{\langle f_1\wedge g_2,\omega\rangle~\langle g_1\wedge h_2,\omega\rangle~\langle h_1\wedge f_2, \omega\rangle}{\langle f_1\wedge h_2,\omega\rangle~\langle g_1\wedge f_2,\omega\rangle~\langle h_1\wedge g_2,\omega\rangle}.
\ee
If the triple $(F_{\bullet}, G_{\bullet}, H_{\bullet})$ is of generic position, then it gives rise to a quadruple of lines in $F_2$
$$
L_1:=F_1, \quad L_2:=G_2\cap F_2, \quad L_3:=(G_1\oplus H_1) \cap F_2, \quad L_4:=H_2\cap F_2.
$$
The following Lemma was proved in \cite[Lemma 3.8]{G94}. We provide a proof for completeness. 
\bl
\la{triple.cross.ratio}
The triple ratio \eqref{triple.ratio} is equal to the cross ratio $r^+(L_1, L_2, L_3,L_4)$.
\el
\begin{proof}
We choose  $g_1\in G_1$, $h_1\in H_1$ such that
$$
\Delta(\ast)=\langle \ast\wedge g_1,\omega\rangle= \langle \ast \wedge h_1,\omega\rangle,\quad \quad \forall\ast \in\wedge^2F_2.
$$
Therefore we get
$
\langle g_1\wedge f_2, \omega\rangle=\langle h_1\wedge f_2,\omega\rangle.
$ 
Let $l_3:=g_1-h_1\in G_1\oplus H_1$. Note that $f_2 \wedge l_3=0$. So $l_3\in F_2$. Therefore $l_3\in L_3$. 
Let $l_2\in L_2$ such that $g_2:=l_2\wedge g_1$. Therefore we get
\be
\Delta(l_2\wedge l_3)=\Delta(-l_3\wedge l_2)=\langle-l_3\wedge l_2\wedge g_1,\omega\rangle=\langle-l_3\wedge g_2,\omega\rangle=\langle h_1\wedge g_2,\omega\rangle.
\ee
Let $l_4\in L_4$ such that $h_2:=l_4\wedge h_1$. Then 
$
\Delta(l_3\wedge l_4)=\langle l_3\wedge l_4\wedge h_1,\omega\rangle=\langle l_3\wedge h_2,\omega\rangle=\langle g_1\wedge h_2,\omega\rangle.
$ 
Let  $l_1=f_1\in F_1$. Then we get
$$
\Delta(l_1\wedge l_2)=\langle l_1\wedge l_2\wedge g_1,\omega\rangle=
\langle f_1\wedge g_2,\omega\rangle,\hskip 7mm
\Delta(l_1\wedge l_4)=\langle l_1\wedge l_4\wedge h_1,\omega\rangle=\langle f_1\wedge h_2,\omega\rangle. 
$$
Combining the above equations, the Lemma is proved.
\end{proof}

\paragraph{Global picture: cluster coordinates of ${\cal X}_{{\rm PGL}_m, \bS}$.} Recall that
the moduli space ${\cal X}_{{\rm PGL}_m, \bS}$ parametrizes pairs $({\cal L}, \gamma=\{F_p\})$, where ${\cal L}$ is a ${\rm PGL}_m$-local system on $\bS$, and $\gamma$ assigns to each puncture $p$ a flat section $F_p$ of ${\cal L}\otimes_{{\rm PGL}_m}{\cal B}_{{\rm PGL}_m}$ invariant under the monodromy around $p$.

Recall the quiver ${\bf q}$ associated to an ideal triangulation ${\cal T}$ of $\bS$. We assign a function $X_v$ of ${\cal X}_{{\rm PGL}_m, \bS}$ to each vertex $v$ of ${\bf q}$. There are two cases.

\begin{itemize}
\item[1.] The vertex $v$ is an inner vertex of a triangle ${\rm t}\in {\cal T}$.  By restricting a generic pair $({\cal L}, \gamma)\in {\cal X}_{{\rm PGL}_m, \bS}$ to the triangle ${\rm t}$, we obtain a configuration  $(F_{\bullet}, G_{\bullet}, H_{\bullet})$ of flags for ${\rm PGL}_m$.
Let us choose decorations for each flag. Recall the function \eqref{delatafgcord}.
If $v$ is labelled by $(a, b,c)\in \Gamma_m$, then there are 6 vertices in the $m$-triangulation adjacent to the vertex $v$. See Figure \ref{verexfunc}. 
We consider the triple ratio 
\be
\la{triple.ratio.delta}
X_v:=\frac{\Delta_{a,b-1,c+1}\Delta_{a-1,b+1,c}\Delta_{a+1,b,c-1}}{\Delta_{a-1,b,c+1}\Delta_{a,b+1,c-1}\Delta_{a+1,b-1,c}}.
\ee
Note that $X_v$ is independent of the choices of decorations. So it is a function of ${\cal X}_{{\rm PGL}_m, \bS}$.

\begin{figure}[ht]
\epsfxsize330pt
\centerline{\epsfbox{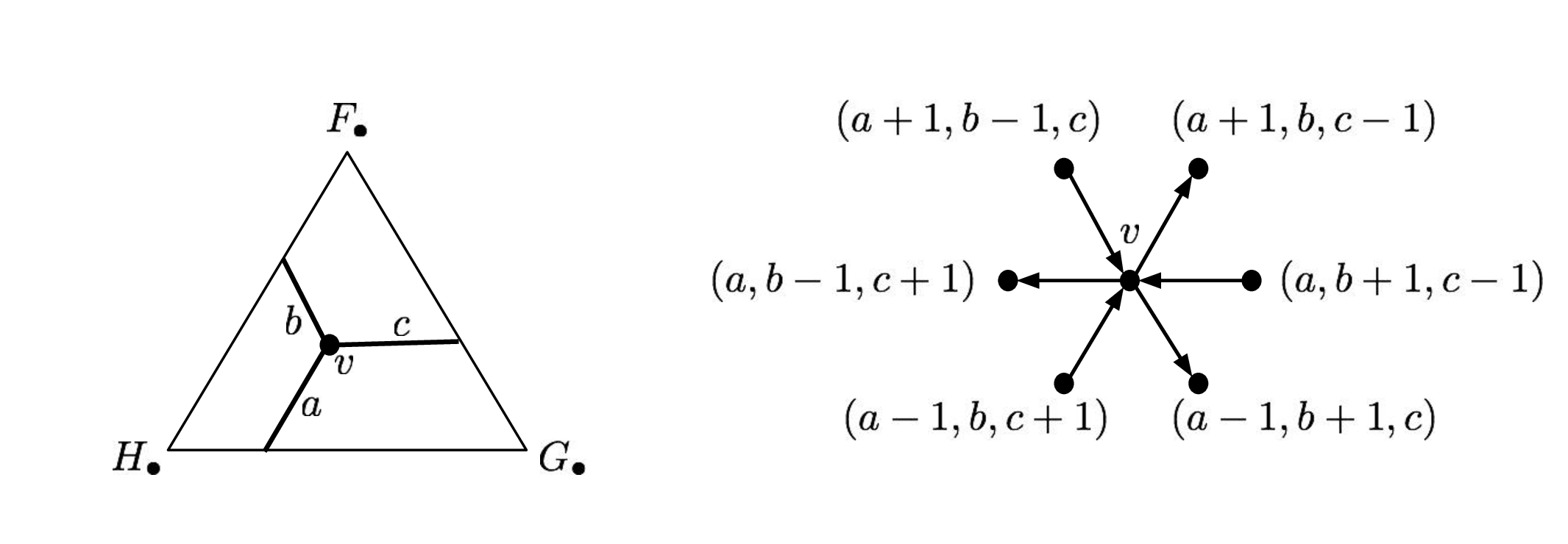}}
\caption{Triple ratio corresponding to an inner vertex.}
\label{verexfunc}
\end{figure}

We consider the following lines in the quotient $F_{a+1}/F_{a-1}$
$$
L_1=F_{a}/F_{a-1},\quad L_2=\Big((G_{b-1}\oplus H_{c+1})\cap F_{a+1}\Big)/F_{a-1},
$$
$$
 L_3=\Big((G_{b}\oplus H_{c})\cap F_{a+1}\Big)/F_{a-1},\hskip 7mm L_4=\Big((G_{b+1}\oplus H_{c-1})\cap F_{a+1}\Big)/F_{a-1}.
$$

\bl 
\la{2.17.12.50h}
We have $
X_v=r^+(L_1,L_2,L_3,L_4).$ 
\el

\begin{proof}
We project $(F_{\bullet}, G_{\bullet}, H_{\bullet})$ onto the quotient 
$$
\frac{V_m}{F_{a-1}\oplus G_{b-1}\oplus H_{c-1}},
$$
obtaining a configuration of flags for ${\rm PGL}_3$
$$
\overline{F}_{\bullet}=(F_{a}/F_{a-1}\subset F_{a+1}/F_{a-1}), ~~\overline{G}_{\bullet}=(G_{b}/G_{b-1}\subset G_{b+1}/G_{b-1}),~~ \overline{H}_{\bullet}=(H_{c}/H_{c-1}\subset H_{c+1}/H_{c-1}).
$$
Clearly
$
X_v=r_3^+(\overline{F}_{\bullet}, \overline{G}_{\bullet}, \overline{H}_{\bullet}). 
$ 
By Lemma \ref{triple.cross.ratio}, 
$
r_3^+(\overline{F}_{\bullet}, \overline{G}_{\bullet}, \overline{H}_{\bullet})=r^+(L_1,L_2,L_3,L_4).
$
\end{proof}

\item[2.] The vertex $v$ belongs to an edge $e$ in ${\cal T}$. 
Restricting a generic  $({\cal L}, \gamma)\in {\cal X}_{{\rm PGL}_m, \bS}$ to the unique 
 ideal quadrilateral containing $e$ as a diagonal, we get a configuration  $(F_{\bullet}, G_{\bullet}, H_{\bullet}, E_{\bullet})$. 
Let us choose decorations $f_{(k)}, k=1,\ldots, m-1$ for $F_{\bullet}$, and 
similarly for $G_{\bullet}, H_{\bullet}, E_{\bullet}$.
There are 4 vertices adjacent to $v$ in the quiver, see Figure \ref{verexfunc2}. 
We consider the cross ratio
\be
\la{crosss.ratio.delta} 
X_v:=\frac{\langle f_{(a)}\wedge h_{(b-1)}\wedge e_{(1)}, \omega\rangle ~\langle f_{(a-1)}\wedge g_{(1)}\wedge h_{(b)}, \omega\rangle}{\langle f_{(a-1)}\wedge h_{(b)}\wedge e_{(1)}, \omega\rangle ~\langle f_{(a)}\wedge g_{(1)}\wedge h_{(b-1)}, \omega\rangle}.
\ee
Note that $X_v$ is independent of the decorations chosen. So it is a function of ${\cal X}_{{\rm PGL}_m, \bS}$.
\begin{figure}[ht]
\epsfxsize230pt
\centerline{\epsfbox{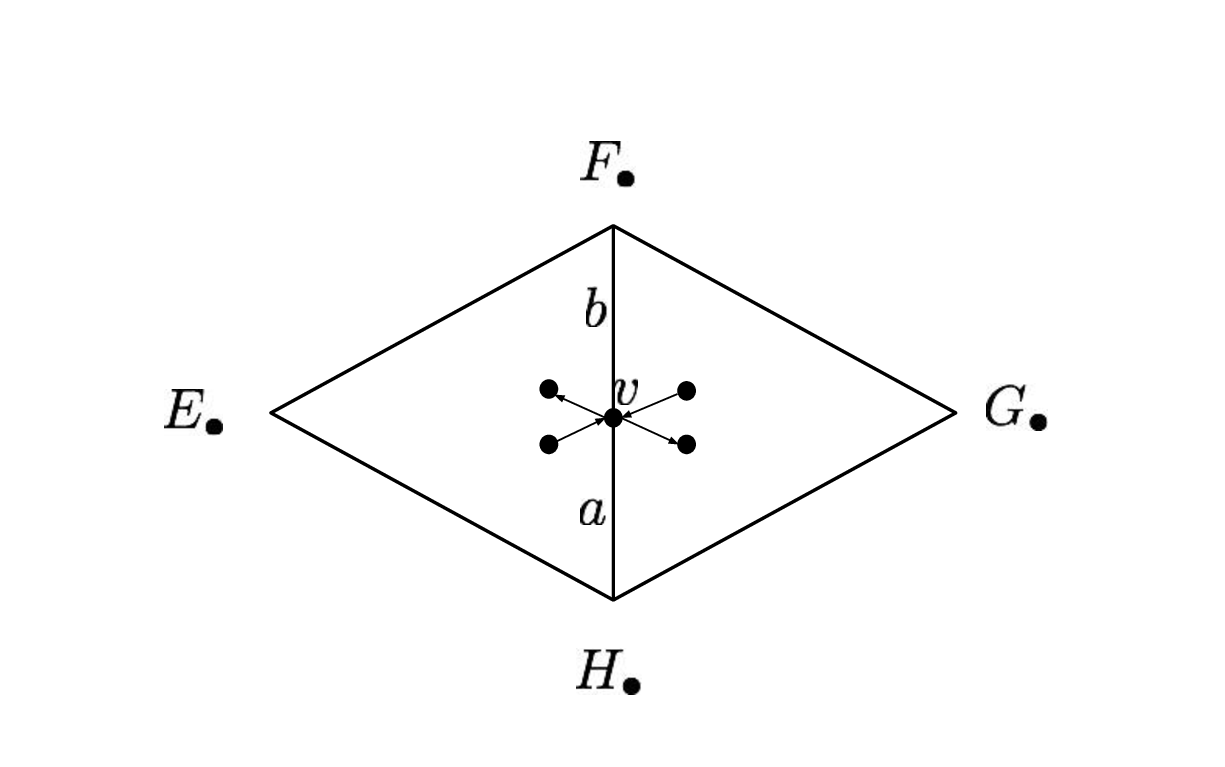}}
\caption{Cross ratio corresponding to an edge point.}
\label{verexfunc2}
\end{figure}

We consider the following lines in the quotient $F_{a+1}/F_{a-1}$:
$$
L_1=F_a/F_{a-1}, \quad L_2= \Big((G_1\oplus H_{b-1})\cap F_{a+1}\Big)/F_{a-1},
$$
$$
L_3=(H_b\cap F_{a+1})/ F_{a-1}, \quad L_4=\Big((H_{b-1}\oplus E_1)\cap F_{a+1}\Big)/F_{a-1}.
$$

\bl 
\la{tech.lemma.2.20.12.51h}
We have $X_v=r^+(L_1, L_2, L_3, L_4).$
\el

\begin{proof}
We project $(F_{\bullet}, G_{\bullet}, H_{\bullet}, E_{\bullet})$ on to the quotient
\be
\la{quoten2.hh}
\frac{V_m}{F_{a-1}\oplus H_{b-1}},
\ee  
obtaining 4 lines
$
\overline{F}=F_a/F_{a-1}, \overline{G}=G_1,  \overline{H}=H_b/H_{b-1}, \overline{E}=E_1.
$ 
Clearly we have
$
X_v=r^+(\overline{F}, \overline{G}, \overline{H},\overline{E}). 
$ 
We project $F_{a+1}$ onto \eqref{quoten2.hh}, identifying
$
F_{a+1}/F_{a-1}={V_m}/{(F_{a-1}\oplus H_{b-1})}.
$ 
It follows directly that
$
r^+(\overline{E}, \overline{F},  \overline{G}, \overline{H})=r^+(L_1, L_2, L_3, L_4).
$
\end{proof}
\end{itemize} 
The functions $\{X_v\}$ provide a coordinate system for the space ${\cal X}_{{\rm PGL}_m, \bS}$.

\paragraph{The  map $\pi_{p,i}$.} 
Let $v_1,\ldots, v_n$ be the vertices of the subquiver ${\bf q}_{p,i}$.  
Let us define a  map 
$$
\pi_{p, i}:~~
{\cal X}_{{\rm PGL}_m, \bS}\lra {\cal X}_{{\rm PGL}_2, D_n}, \quad ({\cal L}, \gamma)\lra ({\cal L}_{p,i}, \{L_p, L_{v_1}, \ldots, L_{v_n}\}).
$$
Take a generic  $({\cal L}, \gamma)\in {\cal X}_{{\rm PGL}_m, \bS}$.
The framing of ${\cal L}$ near $p$ is given by  
a  flag of local subsystems of ${\cal L}$ near $p$, or, what is the same, by a flat section of 
the local system of flags associated to ${\cal L}$:
\be \la{FP}
F_p=(F_1\subset F_2\subset \ldots \subset F_{m-1}).
\ee
The two dimensional subquotient $F_{i+1}/F_{i-1}$ of the local system 
${\cal L}$ near $p$ provides us with a ${\rm PGL}_2$-local system ${\cal L}_{p,i}$ 
of a punctured disk and an invariant line $L_p:=F_{i}/F_{i-1}$. 

Let $v_k$ belong to an $m$-triangulation of a triangle ${\rm t}\in {\cal T}$, locally labelled by $(i,b,c)\in \Gamma_m$.
The data $({\cal L},\gamma)$ restricts to a configuration $(F_{\bullet}, G_{\bullet}, H_{\bullet})$. 
We assign to $v_k$ a line
$$
L_{v_k}:=\Big((G_{b}\oplus H_c)\cap F_{i+1}\Big)/F_{i-1}.
$$
The data $({\cal L}_{p,i}, \{L_p,L_1,\ldots, L_n\})$ defines the map $\pi_{p,i}$.

\bl
\la{Lemma.proj.2.21.h}
In the   coordinate systems  of ${\cal X}_{{\rm PGL}_m, \bS}$ and ${\cal X}_{{\rm PGL}_2, D_n}$, 
the map $\pi_{p,i}$ is a projection 
$$
\pi_{p,i}: ~{\cal X}_{{\rm PGL}_m, \bS}\lra {\cal X}_{{\rm PGL}_2},\quad (X_{v_1},\ldots, X_{v_n},\ldots)\lms (X_{v_1},\ldots, X_{v_n})
$$
\el
\begin{proof}
Follows from Lemmas \ref{2.17.12.50h}, \ref{tech.lemma.2.20.12.51h}.
\end{proof}

\paragraph{Weyl group action on ${\cal X}_{{\rm PGL}_m, \bS}$.} 
The Weyl group acts on ${\cal X}_{{\rm PGL}_m, \bS}$ via changing the flat section $F_p$ around $p$, 
see (\ref{FP}),  and keeping the rest intact. The simple reflection $s_{p,i}$ maps $F_{p}$ to
\be
\la{fpprime}
F_p'=(F_1\subset \ldots F_{i-1}\subset  {F}_i'\subset F_{i+1}\subset \ldots F_{m-1})
\ee
such that ${F}_p'$ is invariant under the monodromy around $p$. 

Recall the $\Z/2$-action $s$ on ${\cal X}_{{\rm PGL}_2, D_n}$. By definition, the following map commutes
\begin{equation}\label{eq78}
\begin{gathered}
\xymatrix{
{\cal X}_{{\rm PGL}_m,\bS} \ar[r]^{s_{p,i}} \ar[d]_{\pi_{p,i}}& {\cal X}_{{\rm PGL}_m,\bS} \ar[d]^{\pi_{p,i}} &\\
{\cal X}_{{\rm PGL}_2, D_n} \ar[r]_{s} & {\cal X}_{{\rm PGL}_2, D_n}  &}
\end{gathered}
\end{equation}

\paragraph{Proof of Theorem \ref{2.17.22.42h}: ${\cal X}$-Part.}
We  consider the following cases.
\begin{enumerate}

\item
The vertex $v$  belongs to ${\bf q}_{p,i}$. By Lemma \ref{Lemma.proj.2.21.h} and \eqref{eq78}, we reduce the case to ${\cal X}_{{\rm PGL}_2, D_n}$. Comparing transition maps in Lemma \ref{Lemma12.3.FG1} and Theorem \ref{2.8.21.24.h}, 
$
s_{p,i}^* X_v =s^* X_v = \tau_{p,i}^* X_v.
$ 

\item The vertex $v$ is of distance $m-i-1$ to the puncture $p$. 
If $v$ is an inner point of an ideal triangle ${\rm t}\in {\cal T}$ labelled by $(i+1, b,c)$, then 
the function $X_v$ is defined by \eqref{triple.ratio.delta}. Let $(F_p, G_{\bullet}, H_{\bullet})$ be the configuration obtained by restricting $({\cal L},\gamma)$ on ${\rm t}$. 
Let us choose decorations for each flag.  The action $s_{p,i}$ maps $F_p$ to $F_p'$ as in \eqref{fpprime}. Let us pick an nonzero vector $f_{(i)}'\in \wedge^iF_i'$. Together with $f_{(k)}\in \wedge^k F_{k}, k\neq i$, it gives rise to decorations of $F_p'$. Set
$$\Delta_{i,j,k}':=\langle f_{(i)}'\wedge g_{(j)}\wedge h_{(k)}, \omega\rangle. $$
Using \eqref{triple.ratio.delta}, we get
$$
\frac{s_{p,i}^*X_v}{X_v}=\frac{\Delta_{i, b+1, c}'\Delta_{i, b, c+1}}{\Delta_{i, b+1, c}\Delta_{i, b, c+1}'}=r^+(L_1, L_p, L_2, L_p'),
$$
where $L_p=F_{i}/F_{i-1}$, $L_p'=F_i'/F_{i-1}$, and 
$$
L_1=\Big((G_{b+1}\oplus H_{c})\cap F_{i+1}\Big)/F_{i-1}, ~~~
L_2=\Big((G_{b}\oplus H_{c+1})\cap F_{i+1}\Big)/F_{i-1}$$
are lines in the quotient $F_{i+1}/F_{i-1}$.
Comparing Theorem \ref{2.8.21.24.h} and Lemma \ref{Lemma12.3.FG1}, we get
$
s_{p,i}^*X_v=\tau_{p,i}^* X_v.
$ 
Similarly, the same formula holds when $v$ belongs to an edge in ${\cal T}$.

\item  The vertex $v$ is of distance $m-i+1$ to the  $p$. By a similar argument, 
$s_{p,i}^*X_v=\tau_{p,i}^* X_v.$

\item For the rest $v$, we have $s_{p,i}^*X_v=\tau_{p,i}^* X_v=X_v$.
\end{enumerate}

\section{The $\ast$-involution and its cluster nature}
In this section, $\bS$ is a decorated surface which admits an ideal triangulation without self-folded triangles.

Let $\alpha_i$ $(i\in I)$ be simple positive roots. There is a Dynkin diagram automorphism such that $\alpha_{i^*}=-w_0(\alpha_i)$. Let us fix a pinning of $\G$. 
We get an involution $\ast: \G \to \G$ defined in 
(\ref{inv}). 

The involution of $\G$ preserves the subgroups $\B$ and $\U$. Therefore it acts on the moduli spaces 
${\cal X}_{{\rm PGL}_m, \bS}$ and ${\cal A}_{{\rm SL}_m, \bS}$. Indeed, they are defined as the local systems 
on $\bS$ with a chosen reduction the subgroups $\B$ or $\U$ near the marked points. 
Since all pinnings in $\G$ are $\G$-conjugated, this does not depend on the choice of pinning which we use to define $\B$ or $\U$.  
Abusing notation, we denoted all of these actions by $\ast$.

Recall the cluster structure of ${\cal X}_{{\rm PGL}_m, \bS}$ and ${\cal A}_{{\rm SL}_m, \bS}$ in the previous section.

\bt
\la{cluster.ast.6.16.16.07h}
The involution $\ast$ on $({\cal A}_{{\rm SL}_m, \bS}, {\cal X}_{{\rm PGL}_m, \bS})$ is a cluster transformation. 
\et
We prove Theorem \ref{cluster.ast.6.16.16.07h} in  Sections \ref{proof.inv.cluster1}-\ref{proof.inv.cluster2}. 
We give a ${\rm GL}_m$-specific proof since we feel that it may contain 
more information that just the claim. We present an explicit sequence of cluster transformations equivalent to the involution $\ast$.

\subsection{Involution on ${\rm Conf}_n({\cal A}_{{\rm SL}_m})$}
\la{proof.inv.cluster1}
We give an equivalent definition of the involution $\ast$ on ${\rm Conf}_n({\cal A}_{{\rm SL}_m})$.
\paragraph{The moduli space ${\rm Conf}_n({\cal A}_{{\rm SL}_m})$.} 
Let $V$ be an $m$-dimensional vector space with a volume form $\omega$. 
A decorated flag ${\rm F}=(F_\bullet, \{f_{(k)}\})$ on $(V, \omega)$ is a decorated flag in $V$ with 
 $\langle f_{(m)}, \omega\rangle=1$.  See Section \ref{aspace.def.weyl.action}.
Denote by ${\cal A}_{V,\omega}$ the space of decorated flags on $(V, \omega)$. 
The group ${\rm Aut}(V, \omega)={\rm SL}(V)$ 
 acts on it  on the left. Set
\be
\la{confvspace}
{\rm Conf}_n({\cal A}_{V, \omega}):= {\rm Aut}(V, \omega) \backslash \big({\cal A}_{V, \omega}\big)^n.
\ee
An isomorphism $g: (V, \omega) \to (V', \omega')$ 
induces an isomorphism ${\cal A}_{V, \omega} \to {\cal A}_{V', \omega'}$ and therefore an isomorphism 
\be
\la{confvpacemap}
{\rm Conf}_n({\cal A}_{V, \omega})\stackrel{\sim}{\lra}{\rm Conf}_n({\cal A}_{V', \omega'}).
\ee 
Different isomorphisms $g$ differ by an automorphism of the  $(V, \omega)$. 
Since ${\rm Conf}_n({\cal A}_{V, \omega})$ is the space of ${\rm Aut}(V, \omega)$-coinvariants, 
 isomorphism \eqref{confvpacemap} does not depend on  $g$. We set 
\be
\la{confnspace.5.31.2}
{\rm Conf}_n({\cal A}_{{\rm SL}_m}):=
{\rm Conf}_n({\cal A}_{V, \omega}).
\ee

\paragraph{The dual decorate flags.}
Let $V^\ast$ be the dual vector space of $V$.  
For each $k\in \{1,\ldots, m\}$ there is a non-degenerate bilinear map
\be
\la{dualmap.5.31.3h}
\langle - , -\rangle: {\bigwedge}^k V \times {\bigwedge}^k V^\ast \lra {\Q}, \quad \quad \langle v_1\wedge\ldots \wedge v_k, \phi_1\wedge \ldots \wedge \phi_k\rangle=\det(\langle v_i, \phi_j\rangle).
\ee
There is a  canonical isomorphism
\be
\la{dualmap.5.31.1}
\ast: {\bigwedge}^{k}V \stackrel{\sim}{\lra} {\bigwedge}^{m-k} V^\ast, ~~~~\mbox{such that} ~~
\langle v, \ast u\rangle=
\langle u\wedge v, \omega \rangle.  
\ee

Let $W$ be a $k$-dimensional subspace of $V$. Set
$
W^\perp:=\{ \phi \in V^\ast ~|~ \langle w, \phi\rangle =0~\mbox{for all } w\in W\}.
$ 
\bl 
\la{lem.6.16.14.25.15hh}
 If $u\in {\bigwedge}^k W$, then $\ast u\in {\bigwedge}^{m-k}W^\perp$.
\el

\begin{proof} 
Let us choose a linear basis $(e_1,\ldots, e_m)$ of $V$ such that
$u=e_1\wedge e_2\wedge \ldots\wedge e_k$, and $
\langle e_1\wedge e_2 \wedge \ldots\wedge e_m, \omega\rangle=1$. 
Thus $(e_1,\ldots, e_k)$ is a linear basis of  $W$. Let 
$(e^1, e^2, \ldots, e^m)$ be the basis of $V^\ast$ 
dual to $(e_1,\ldots, e_m)$. Then 
$
\omega= e^1\wedge \ldots \wedge e^m. 
$ 
Therefore
\be
\la{ast.fk.6.16.21.04h}
\langle  u\wedge v,\omega\rangle
 =\langle u\wedge v, e^1\wedge \ldots e^{k}\wedge e^{k+1}\wedge\ldots \wedge e^m\rangle
 =\langle v, e^{k+1}\wedge \ldots \wedge e^m 
  \rangle. 
\ee
Since $W^\perp$ is the linear span of $(e^{k+1},\ldots, e^{m})$, by \eqref{dualmap.5.31.1} we get
$
\ast u= e^{k+1}\wedge \ldots \wedge e^{m}\in {\bigwedge}^{m-k}W^\perp.
$ 
\end{proof}

\bl 
\la{reverse.ast.property} 
Let $\omega^\ast$ be the  volume form of $V^\ast$ dual to $\omega$, i.e. $\langle \omega^*, \omega\rangle=1.$
Then one has
\be
\langle u \wedge v , \omega \rangle =\langle \omega^\ast, \ast u \wedge \ast v \rangle,
\hskip 1cm \forall u \in{\bigwedge}^{m-k}V, \quad \forall v\in {\bigwedge}^{k}V.
\ee
\el
\begin{proof} It suffices to the prove  for $v=e_1\wedge e_2\wedge \ldots \wedge e_k$. By the proof of Lemma \ref{lem.6.16.14.25.15hh}, we set
\[
\omega^\ast = e_1\wedge e_2 \wedge \ldots\wedge e_m \hskip 1cm \ast v= e^{k+1}\wedge \ldots \wedge e^{m}.
\]
Therefore
\[
\langle \omega^\ast, \ast u \wedge \ast v \rangle=\langle e_1 \wedge \ldots\wedge e_m , \ast u \wedge e^{k+1}\wedge \ldots \wedge e^{m} \rangle= \langle e_1 \wedge \ldots\wedge e_k , \ast u\rangle= \langle v, \ast u\rangle= \langle u \wedge v , \omega \rangle .
\]
\end{proof}

The dual flag $F_\bullet^\perp$  is a flag on $V^\ast$
\be
\la{dual.flag}
F_{m-1}^\perp\subset \ldots \subset F_2^\perp\subset F_1^\perp.
\ee
The isomorphism
$
\ast: {\cal A}_{V, \omega}\stackrel{\sim}{\to}{\cal A}_{V^\ast, \omega^\ast}$, 
$(F_{\bullet},\{f_{(k)}\})\lms (F_{\bullet}^\perp, \{\ast f_{(m-k)}\})
$ from Lemma \ref{lem.6.16.14.25.15hh} provides a canonical isomorphism
$
\ast: {\rm Conf}_n({\cal A}_{V, \omega})\stackrel{\sim}{\lra}{\rm Conf}_n({\cal A}_{V^\ast, \omega^\ast}).
$ 
So we get a canonical involution
\be
\la{ast.6.16.21.14h}
\ast:  {\rm Conf}_n({\cal A}_{{\rm SL}_m})\stackrel{\sim}{\lra}{\rm Conf}_n({\cal A}_{{\rm SL}_m}).
\ee
Using \eqref{ast.fk.6.16.21.04h}, 
it is easy to show that \eqref{ast.6.16.21.14h} is the involution defined via the involution $\ast$ 
in (\ref{inv}).

\paragraph{Example.} When $\dim V=3$, decorated flags in ${\cal A}_{V, \omega}$ are canonically identified with pairs 
\be
(v,\phi)\in V\times V^\ast, \hskip 10mm
v\neq 0, \quad \phi \neq 0, \quad \langle v, \phi \rangle =0.
\ee

Switching $v$ and $\phi$, we get the map
$
\ast: {\cal A}_{V,\omega}\lra {\cal A}_{V^\ast, \omega^\ast}, \quad (v,\phi)\lms (\phi, v).
$ 

In particular, it acts on the triples of decorated flags as follows:  
$$
\ast: \big((v_1, \phi_1), (v_2,\phi_2), (v_3,\phi_3)\big)~
\lms ~
\big((\phi_1, v_1), (\phi_2,v_2), (\phi_3,v_3)\big). 
$$

\begin{figure}[h]
\epsfxsize450pt
\centerline{\epsfbox{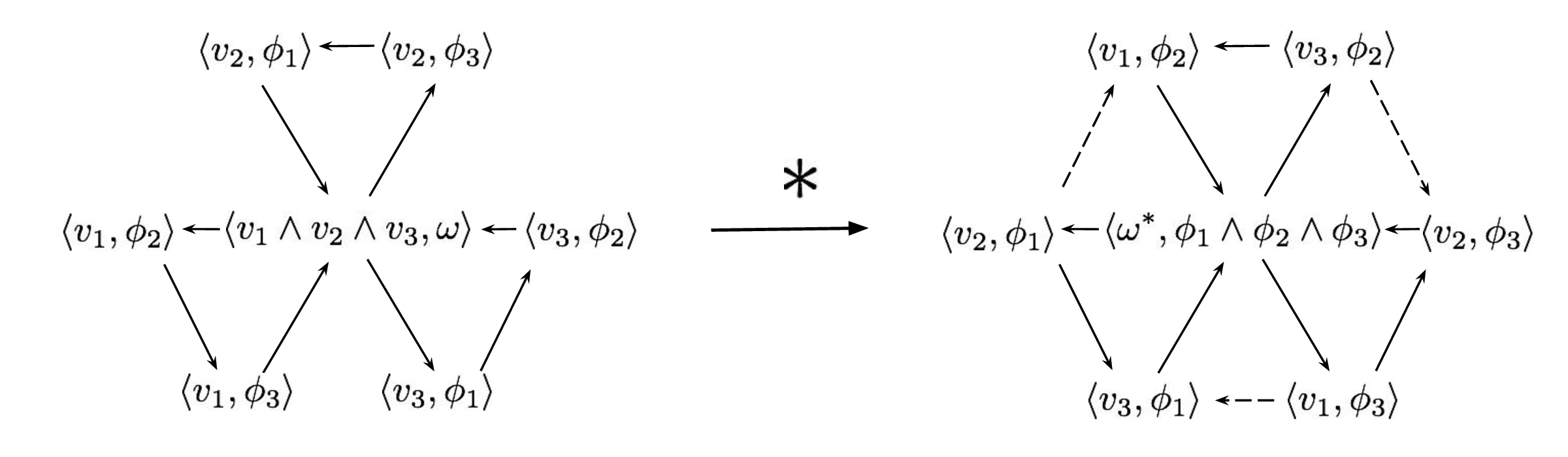}}
\caption{The involution $\ast$ on  ${\rm Conf_3({\cal A}_{{\rm SL}_3})}$ is a cluster transformation. The dashed arrows connect frozen vertices.}
\label{astm3}
\end{figure}

\bl 
\la{sl3ast.h}
We have
\be
\langle v_1\wedge v_2\wedge v_3, \omega\rangle\langle\omega^\ast, \phi_1\wedge \phi_2\wedge \phi_3\rangle =\langle v_1, \phi_2\rangle \langle v_2,\phi_3\rangle \langle v_3,\phi_1\rangle + \langle v_2, \phi_1\rangle \langle v_3,\phi_2\rangle \langle v_1,\phi_3\rangle.
\ee
Therefore the involution $\ast$ is a cluster transformation 
of ${\rm Conf_3({\cal A}_{{\rm SL}_3})}$, which mutates 
the inner vertex of the left quiver on Figure \ref{astm3}, and then switches the pair of 
vertices on each edge. 
\el
\begin{proof} By \eqref{dualmap.5.31.3h} we have
\begin{align}
\langle v_1\wedge v_2\wedge v_3, \omega\rangle\langle\omega^\ast, \phi_1\wedge \phi_2\wedge \phi_3\rangle &=\langle \omega^*, \omega\rangle \langle v_1\wedge v_2\wedge v_3,  \phi_1\wedge \phi_2\wedge \phi_3\rangle  \nonumber\\
&= \det   \begin{pmatrix} 
      0 &  \langle v_1,\phi_2 \rangle & \langle v_1,\phi_3\rangle \\
      \langle v_2, \phi_1\rangle & 0  & \langle v_2, \phi_3\rangle\\
      \langle v_3, \phi_1\rangle & \langle v_3,\phi_2\rangle & 0\\ 
   \end{pmatrix} \nonumber\\
   &=\langle v_1, \phi_2\rangle \langle v_2,\phi_3\rangle \langle v_3,\phi_1\rangle + \langle v_2, \phi_1\rangle \langle v_3,\phi_2\rangle \langle v_1,\phi_3\rangle. \nonumber
\end{align}
\end{proof}

\paragraph{Intersection of decorated flags.}
Let us fix a generic triple of decorated flags in ${\cal A}_{V, \omega}$
\[
{\rm F}=(F_\bullet, \{f_{(k)}\}),\quad {\rm G}=(G_{\bullet}, \{g_{(k)}\}), \quad {\rm H}=(H_{\bullet}, \{h_{(k)}\}).
\]
The pair $({\rm F, G})$ determines a basis $(f_1,\ldots, f_m)$ of $V$ such that
\be
f_{(k)}=f_1\wedge f_2\wedge \ldots \wedge f_k,\quad \quad f_{k}\in { F}_k\cap { G}_{m+1-k}.
\ee
The pair $({\rm G, H})$ determines a basis $(h_m,\ldots, h_1)$ of $V$ such that
\be
h_{(k)}=h_k\wedge h_{k-1}\wedge \ldots \wedge h_1,\quad \quad h_{k}\in  { G}_{n+1-k} \cap {H}_k.
\ee
For convenience, the subscripts of the wedge product decomposition of $h_{(k)}$ is reversed.
We set
\be
f_{s, (k)}:=f_{s+1}\wedge f_{s+2}\wedge\ldots \wedge  f_{s+k}. \hskip 10mm
h_{(k), s}:=h_{s+k}\wedge \ldots \wedge h_{s+1}.
\ee
Let $(a, b, c, s)$ be a quadruple of nonnegative integers such that 
\be
\la{6.23.12.27h}
a+b+c=m-s.
\ee
By definition, 
$
f_{s,(a)}\wedge g_{(b)}\wedge h_{(c),s}\in {\bigwedge}^{m-s}G_{m-s}.
$ 
We set
\be
\la{deltaabcs.h}
\Delta_{a,b,c}^s:=\langle f_{s, (a)}\wedge g_{(b)}\wedge h_{(c), s}, \omega_s\rangle,
\ee
where $\omega_s$ is a volume form of $G_{m-s}$ such that
\be
\la{omegas6.20.10.h}
\langle g_{(m-s)}, \omega_s\rangle :=\langle f_{(s)}\wedge g_{(m-s)}, \omega\rangle \langle g_{(m-s)}\wedge h_{(s)}, \omega\rangle.
\ee
\bl 
Let us assume that $a,b,c>0$. One has
\be
\la{oct.formula.6.17.11.44h}
\Delta_{a,b,c}^s\Delta_{a,b-1,c}^{s+1}=\Delta_{a+1, b-1, c}^s\Delta_{a-1, b, c}^{s+1}+ \Delta_{a, b-1, c+1}^s\Delta_{a, b, c-1}^{s+1}.
\ee
\el
\begin{proof} Note that the vector $f_{s+1, (a-1)}\wedge g_{(b)} \wedge h_{(c-1), s+1}$ belongs to the linear span of the vectors $ f_{s+1, (a)}\wedge g_{(b-1)}\wedge h_{(c-1), s+1}$ and $f_{s+1, (a-1)}\wedge g_{(b-1)} \wedge h_{(c), s+1}$. Let us set
\[
f_{s+1, (a-1)}\wedge g_{(b)} \wedge h_{(c-1), s+1}:=\alpha f_{s+1, (a)}\wedge g_{(b-1)}\wedge h_{(c-1), s+1} +\beta f_{s+1, (a-1)}\wedge g_{(b-1)} \wedge h_{(c), s+1}.
\]
Then
\begin{align}
f_{s, (a)}\wedge g_{(b)}\wedge h_{(c), s}&=\alpha f_{s, (a+1)}\wedge g_{(b-1)} \wedge h_{(c), s}+\beta f_{s, (a)}\wedge g_{(b-1)} \wedge h_{(c+1), s}, \nonumber\\
f_{s+1, (a-1)}\wedge g_{(b)}\wedge h_{(c), s+1}&=\alpha f_{s+1, (a)}\wedge g_{(b-1)} \wedge h_{(c), s+1}, \nonumber\\
f_{s+1, (a)}\wedge g_{(b)}\wedge h_{(c-1), s+1}&=\beta f_{s+1, (a)}\wedge g_{(b-1)} \wedge h_{(c), s+1}. \nonumber
\end{align}
Therefore
$$
\Delta_{a, b, c}^s=\alpha \Delta_{a+1, b-1, c}^s +\beta \Delta_{a, b-1, c+1}^s, \hskip 7mm\Delta_{a-1, b, c}^{s+1}=\alpha \Delta_{a, b-1,c}^{s+1},  \hskip 7mm
\Delta_{a, b, c-1}^{s+1} = \beta \Delta_{a,b-1, c}^{s+1}. 
$$
Plugging them to \eqref{oct.formula.6.17.11.44h}, we get the Lemma.
\end{proof}
\begin{remark}
Consider the tetrahedron
\be
{\bf T}_m:=\{(x_1,x_2, x_3, x_4)\in \R^4 ~|~ \sum_{i=1}^4 x_i =m, \quad x_i\geq 0\}. 
\ee
The quadruples $(a,b,c,s)$ satisfying \eqref{6.23.12.27h} are the integral points inside of ${\bf T}_m$. Therefore the functions $\Delta_{a,b,c}^s$ can be attached to the integral points of ${\bf T}_m.$ The functions appearing in \eqref{oct.formula.6.17.11.44h} correspond to the vertices of an octahedron as illustrated by Figure \ref{tetra.astm1}. Therefore we call Formula \eqref{oct.formula.6.17.11.44h} the octahedral relation.\footnote{A similar but different octahedral relation was studied in \cite[Sect.10]{FG1}.}
\begin{figure}[h]
\epsfxsize200pt
\centerline{\epsfbox{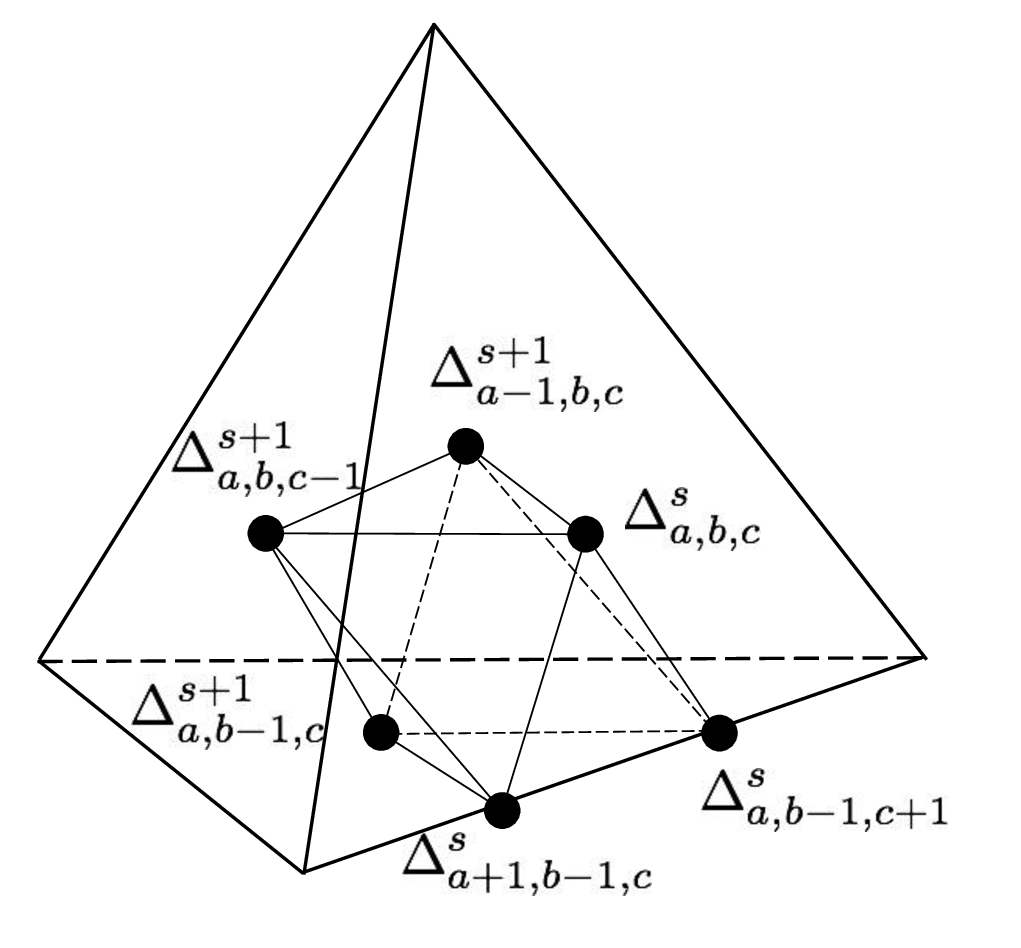}}
\caption{The octahedral relation.}
\label{tetra.astm1}
\end{figure}
\end{remark}

We show that all the functions $\Delta_{a,b,c}^s$ can be expressed in terms of  \eqref{delatafgcord}.
\bl 
\la{lem.boundary.6.20.10.46h}
One has
\be
\la{boundary.6.20.10.46h}
\Delta_{0,t, m-s-t}^s=\Delta_{s,m-s,0}\Delta_{0,t,m-t}, \hskip 12mm \Delta_{m-s-t,t,0}^s=\Delta_{m-t,t,0}\Delta_{0, m-s,s}
\ee
\el
\begin{proof}
Set
$
g_{(t)}\wedge h_{(m-s-t),s}:=\alpha g_{(m-s)}
$ 
Then
$$
g_{(t)}\wedge h_{(m-t)}=\big(g_{(t)}\wedge h_{(m-s-t),s}\big)\wedge h_{(s)}=\alpha g_{(m-s)}\wedge h_{(s)}.
$$
Therefore
\be
\la{6.22.11.59h}
\alpha=\frac{\Delta_{0,t,m-s-t}^{s}}{\Delta_{0,m-s,0}^{s}}=\frac{\Delta_{0,t,m-t}}{\Delta_{0,m-s,s}}.
\ee
By \eqref{omegas6.20.10.h} we have 
$
\Delta_{0,m-s,0}^{s}=\Delta_{s,m-s,0}\Delta_{0,s,m-s}.
$ 
Plugging it to \eqref{6.22.11.59h}, we get 
the first identity. The second  follows by a similar argument.
\end{proof}
When $s=0$, the functions \eqref{deltaabcs.h} equal $\Delta_{a,b,c}$ in \eqref{delatafgcord}. They correspond to the integral points on one face of the tetrahedron ${\bf T}_m$. The functions \eqref{boundary.6.20.10.46h} correspond to the integral points on two other faces of ${\bf T}_m$. See Figure \ref{Tfaces}. All of them can be expressed in terms of \eqref{delatafgcord}.
\begin{figure}[h]
\epsfxsize500pt
\centerline{\epsfbox{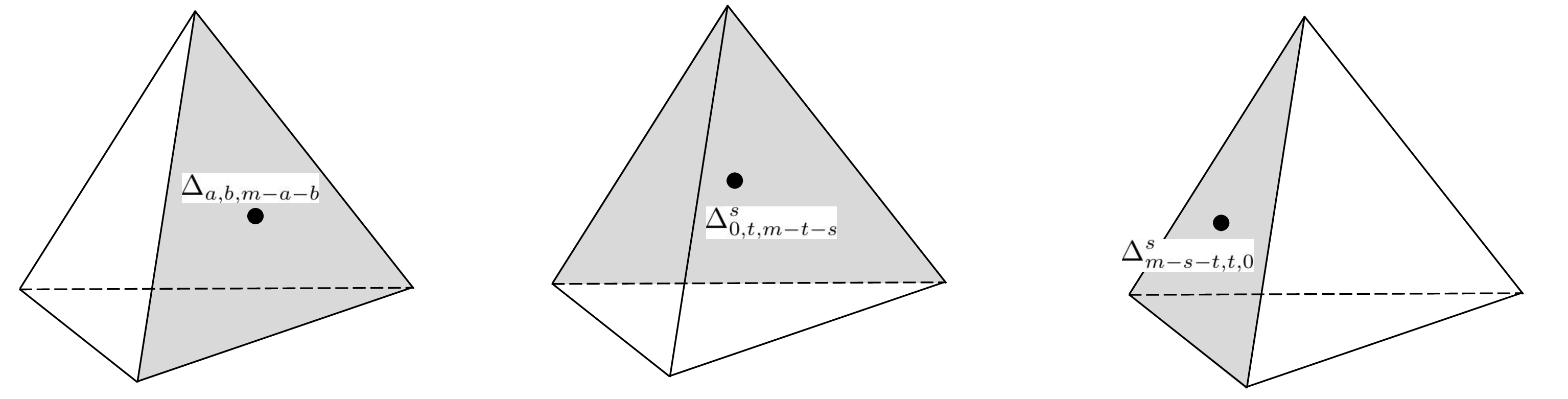}}
\caption{}
\label{Tfaces}
\end{figure}

Using the octahedral relations \eqref{oct.formula.6.17.11.44h} repeatedly, we express $\Delta_{a,b,c}^s$ in terms of \eqref{delatafgcord} layer by layer as illustrated by Figure \ref{TLayers}.
\begin{figure}[h]
\epsfxsize400pt
\centerline{\epsfbox{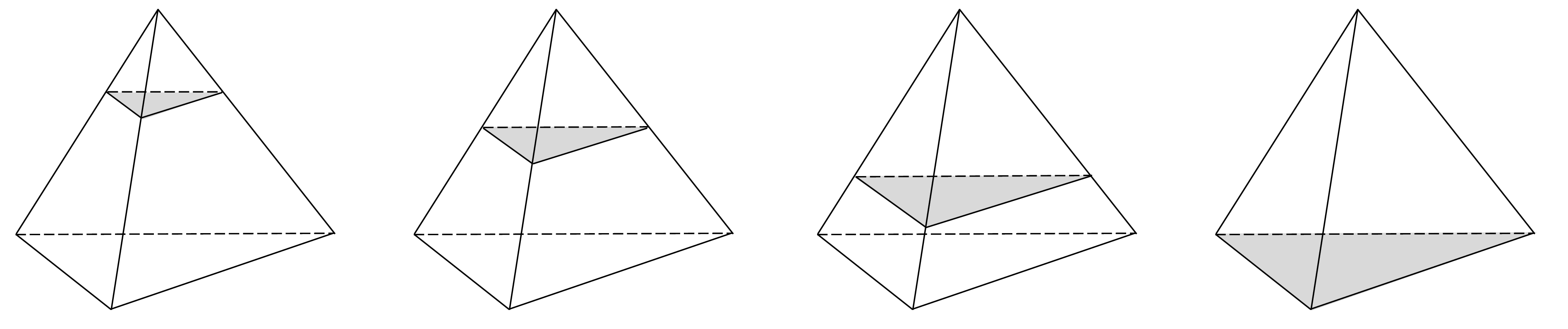}}
\caption{}
\label{TLayers}
\end{figure}

\paragraph{Coordinates of the dual configurations.}
Recall the set $\Gamma_m$ in \eqref{GAMMA.m.lh}.  We set
\be
\Delta_{a,b,c}^\ast :=\Delta_{a,b,c}(\ast{\rm F}, \ast{\rm G}, \ast{\rm H})=\langle \omega^\ast, \ast f_{(m-a)}\wedge \ast g_{(m-b)} \wedge \ast h_{(m-c)}\rangle, \hskip 10mm \forall (a,b,c)\in {\Gamma}_m.
\ee
\bl 
\la{6.25.15.16h}
One has
\be \la{6.25.15.17h}
\Delta_{a,b,c}^\ast =\Delta_{c, 0, a}^b, \hskip 10mm \forall (a,b,c)\in {\Gamma}_m.
\ee
\el 
\begin{remark} Note that the functions $\Delta_{c, 0, a}^b$ correspond to the integral points on the base of ${\bf T}_m$. Using the process illustrated by Figure \ref{TLayers}, we express $\Delta_{a,b,c}^\ast$ in terms of \eqref{delatafgcord}.
\end{remark}
\begin{proof} 
By  Lemma \ref{lem.boundary.6.20.10.46h} and Lemma \ref{reverse.ast.property}, we have
$$
\Delta_{0,0,m-b}^b=\Delta_{b,m-b,0}\Delta_{0,0,m}=\Delta_{b,m-b,0}=\Delta_{m-b,b,0}^\ast.
$$
By moving $\ast h_{(m-c)}$ to the left, we get
$$
\Delta_{a,b,c}^\ast =(-1)^{c(m-c)}\langle \omega^\ast,  \ast h_{(m-c)}\wedge\ast f_{(m-a)}\wedge \ast g_{(m-b)}\rangle=(-1)^{c(m-c)}\langle g_{(m-b)},   \ast h_{(m-c)}\wedge  \ast f_{(m-a)} \rangle
$$
By definition
$
\langle h_{(c), m-c}, \ast h_{(m-c)} \rangle= 
\langle h_{(m-c)}\wedge h_{(c), m-c}, \omega \rangle=(-1)^{c(m-c)}.
$ 
Therefore
\begin{align}
\frac{\Delta_{a,b,c}^\ast }{\Delta_{c,0,a}^b}&=
\frac{\Delta_{a,b,c}^\ast}{\Delta_{m-b, b, 0}^\ast}\cdot\frac{ \Delta_{0,0,m-b}^b}{\Delta_{c,0,a}^b}
=(-1)^{c(m-c)}\frac{\langle g_{(m-b)}, \ast h_{(m-c)}\wedge \ast f_{(m-a)}\rangle}{\langle g_{(m-b)}, \ast f_{(b)}\rangle}\cdot\frac{ \langle h_{(m-b), b}, \omega_b\rangle}{ \langle f_{b, (c)}\wedge h_{(a), b}, \omega_b \rangle}
\nonumber \\
&=(-1)^{c(m-c)}\frac{\langle h_{(m-b), b},  \ast h_{(m-c)}\wedge \ast f_{(m-a)}\rangle }{\langle f_{b, (c)}\wedge h_{(a), b}, \ast f_{(b)}\rangle } \nonumber\\
&=(-1)^{c(m-c)}\frac{\langle h_{(c), m-c}\wedge h_{(a),b},  \ast h_{(m-c)}\wedge \ast f_{(m-a)}\rangle }{ \langle f_{(b)}\wedge f_{b, (c)}\wedge h_{(a), b}, \omega \rangle } \nonumber\\
&=(-1)^{c(m-c)}\langle h_{(c), m-c}, \ast h_{(m-c)} \rangle\cdot \frac{\langle h_{(a), b},  \ast f_{(m-a)}\rangle}{\langle f_{(m-a)} \wedge h_{(a), b}, \omega \rangle }=1 .\nonumber
\end{align}
\end{proof}

\paragraph{The involution $\ast$ of ${\rm Conf}_3({\cal A}_{{\rm SL}_m})$ is a cluster transformation.}
Recall the quiver associated to the $m$-triangulation of a triangle. See the left graph of Figure \ref{astclumut3}. Denote by $\mu_{a,b,c}$ the cluster mutation at the vertex $(a,b,c)\in \Gamma_m$. 
The $\circ$-vertices on edges are frozen vertices. 
We mutate the $\bullet$-vertices only. 
\begin{figure}[h]
\epsfxsize350pt
\centerline{\epsfbox{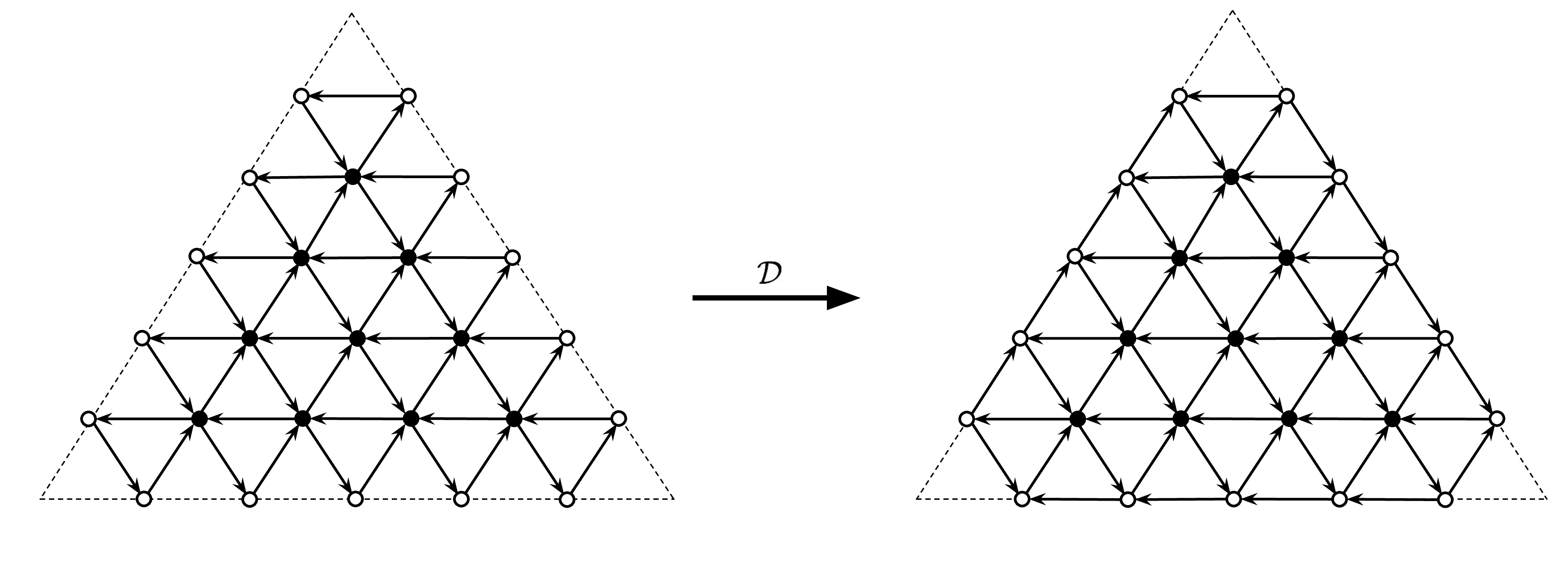}}
\caption{The cluster transformation ${\cal D}$.}
\label{astclumut3}
\end{figure}
 Let $i\in \{1,\ldots, m-2\}$. We introduce several cluster transformations: 
  \begin{enumerate}
 \item The sequence of cluster mutations at the $\bullet$-vertices in row $i$ from the left to the right: 
\be \la{cluster.row.i}
{\cal E}_{i}:=\mu_{1,b,i}\circ \ldots \circ \mu_{i-1,b,2}\circ \mu_{i,b,1}, \hskip 1cm \mbox{where~}b=m-i-1. 
\ee
\item The cluster transformation presented by a sequence of ${\cal E}_i$:
\be
{\cal S}_{i}:={\cal E}_1\circ {\cal E}_2\circ \ldots \circ{\cal E}_{i}.
\ee
It corresponds to a sequence of cluster mutations at the vertices included in the top triangle of size $i$ starting from the left bottom. See Figure \ref{astclumut4}.
\item The cluster transformation presented  by a sequence of ${\cal S}_{i}$:
\be \la{cluster.sch}
{\cal C}:={\cal S}_{m-2}\circ{\cal S}_{m-1}\circ \ldots \circ {\cal S}_1.
\ee
\item The cluster permutation $\sigma$ induced by an involution $\sigma$ of $\Gamma_m$ such that
\be
\la{perm.sigma.9.5}
\sigma: {\Gamma}_m\stackrel{\sim}{\lra}{\Gamma_m}, \hskip 1cm
\sigma(a,b,c)= \left\{ \begin{array}{ll}
    (0, c, b) & \mbox{if $a=0$}, \\
     (b, a, 0) & \mbox{if $c=0$}, \\  
     (c,b,a)  & \mbox{else}. \\
   \end{array}
   \right.
\ee
\end{enumerate}

\bp 
\la{basic.ast.prop.510}
The cluster transformation ${\cal D}:=\sigma\circ {\cal C}$ maps the left quiver of Figure \ref{astclumut3} to the right. It creates arrows between frozen vertices on the edges, and keeps the rest intact. Recall the cluster ${\cal A}$-coordinates $\{\Delta_{a,b,c}\}$ associated to the left quiver. We have
\be
\Delta^\ast_{a,b,c}:=
{\cal D}^\ast \Delta_{a,b,c}, \hskip 1cm \forall (a,b,c)\in \Gamma_m
\ee
\ep

\begin{proof} We start with proving the first part of the proposition. The proof is combinatorial and based on several pictures below. 

Note that the transformation ${\cal E}_k$ is a sequence of cluster mutations at $\bullet$-vertices in a row. Locally, the corresponding quiver mutation of ${\cal E}_k$ is illustrated by Figure \ref{astclumut}. In particular, we switch  the vertex $i$ (respectively $i'$) and the vertex $j$ (respectively $j'$) on the right quiver. 

\begin{figure}[H]
\epsfxsize350pt
\centerline{\epsfbox{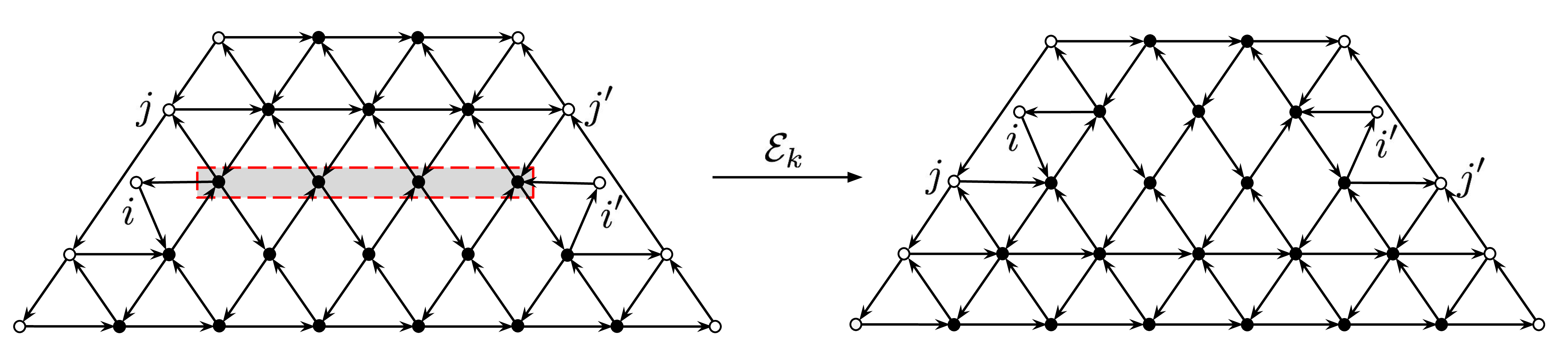}}
\caption{The cluster transformation ${\cal E}_k$.}
\label{astclumut}
\end{figure}

The cluster transformation ${\cal S}_i$ is a sequence of ${\cal E}_k$. Using the above process repeatedly, the corresponding quiver mutation of ${\cal S}_i$ is illustrated by Figure  \ref{astclumut5}. Note that at the last step we switch the vertex $1$ and $1'$.  Eventually ${\cal S}_i$ take the bottom vertex on the side to the top of the other side.
The resulted quiver looks similar to the original one, but its size is enlarged by $1$. 

\begin{figure}[H]
\epsfxsize450pt
\centerline{\epsfbox{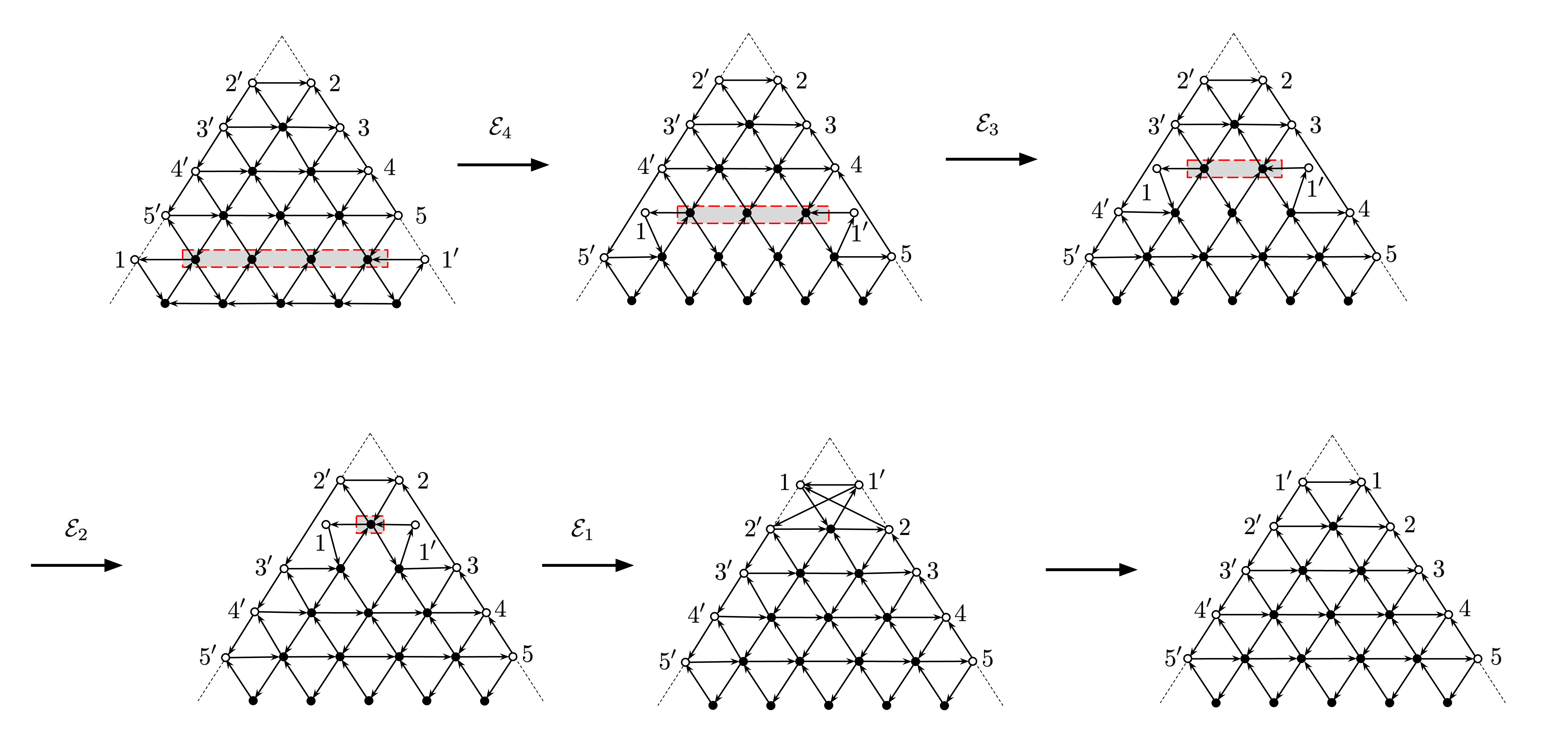}}
\caption{The cluster transformation ${\cal S}_4={\cal E}_1\circ {\cal E}_2 \circ {\cal E}_3 \circ {\cal E}_4$.}
\label{astclumut5}
\end{figure}

The cluster transformation ${\cal C}$ is a sequence of ${\cal S}_i$. Using the above process inductively, the corresponding quiver mutation of ${\cal C}$ is illustrated by 
Figure \ref{astclumut4}. After the action of ${\cal C}$, the orientation of all the arrows are reversed. In the last step we flip the whole quiver horizontally. It is equivalent to the permutation $\sigma$. Eventually we obtain the quiver after the action of ${\cal D}$.  

\begin{figure}[H]
\epsfxsize450pt
\centerline{\epsfbox{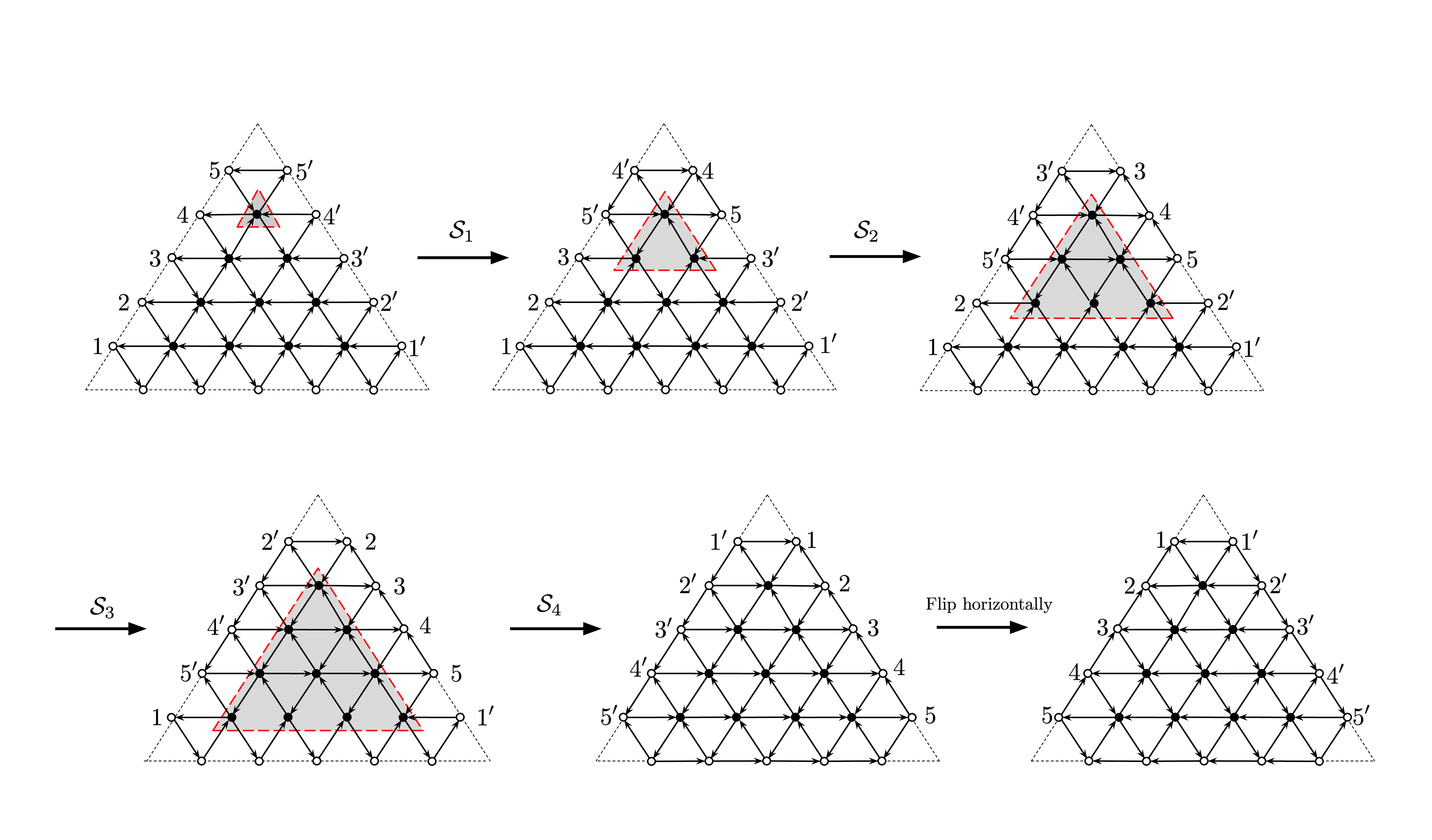}}
\caption{The cluster transformation ${\cal D}=\sigma \circ {\cal C}=\sigma\circ{\cal S}_{m-2}\circ \ldots \circ {\cal S}_2\circ {\cal S}_1$.}
\label{astclumut4}
\end{figure}

\vskip 3mm

To prove the second part of the proposition, we compare the above process with Figure \ref{TLayers}. Indeed, the action of ${\cal S}_i$ is equivalent to the transition from the $i$-th layer to the $(i+1)$-th layer in Figure \ref{TLayers}. In particular, the octahedral relation 
\eqref{oct.formula.6.17.11.44h} is compatible with the rule of cluster mutation. For the cluster mutations at the leftmost and the rightmost $\bullet$-vertices, besides the octahedral relation, we also need the identities \eqref{boundary.6.20.10.46h}. 
Note that we exchange the subscripts $a$ and $c$ in \eqref{6.25.15.17h}. Therefore, by Lemma \ref{6.25.15.16h}, after flipping the quiver horizontally, we get the function $\Delta_{a,b,c}^\ast$ eventually. 
\end{proof}


\begin{remark} The permutation \eqref{perm.sigma.9.5} can be decomposed as $\sigma:=\sigma_e\circ \sigma_i$, where $\sigma_e$ switches vertices on the edges only:
$$
\sigma_e(a,b,c):=\left\{ \begin{array}{ll}
    (0, c, b) & \mbox{if $a=0$}, \\
    (c,0,a) & \mbox{if $b=0$},\\
     (b, a, 0) & \mbox{if $c=0$}, \\  
     (a,b,c)  & \mbox{else}, \\
   \end{array}
   \right.
$$
and $\sigma_i$ exchanges the inner vertices only:
$$
\sigma_i(a,b,c):=\left\{ \begin{array}{ll}
    (c, b, a) & \mbox{if $a,b,c>0$}, \\
    (a,b,c) & \mbox{else}. \\
   \end{array}
   \right.
$$
We consider the following cluster transformation applying only on the inner vertices
\be
\la{9.6.19.48.15h}
{\cal C}_t:=\sigma_i\circ {\cal C}.
\ee
Note that the definition of ${\cal C}_t$ is not symmetric: one has to pick an angle of the triangle first.
\bc  Up to equivalence (in the sense of Definition \ref{Def2.7}),
 the cluster transformation ${\cal C}_t$ is independent of the angle chosen.\ec
\begin{proof}
The cluster transformation ${\cal D}$ in Proposition \ref{basic.ast.prop.510} can be rewritten as
${\cal D}:={\sigma}_e\circ {\cal C}_t$. The corollary is clear since ${\cal D}$ and ${\sigma}_e$ are independent of the angle chosen.
 \end{proof}
\end{remark}

\subsection{Proof of Theorem \ref{cluster.ast.6.16.16.07h}} \la{proof.inv.cluster2}
\paragraph{The ${\cal A}$--part.} Let us fix an ideal triangulation ${\cal T}:=(E, T)$ for the decorated surface $\bS$. Here $E$ is the set of all edges in ${\cal T}$, and $T$ is the set of all triangles in ${\cal T}$. We assign an $m$-triangulation to each triangle $t\in T$, obtaining a quiver ${\bf q}$. We define the cluster transformation
\be
{\cal D}: ={\sigma}_E\circ{\cal C}_T
\ee
where $\sigma_E$ is a permutation switching vertices of ${\bf q}$ on each edge  $e\in E$, and ${\cal C}_T$ is a cluster transformation applying \eqref{9.6.19.48.15h} on each triangle $t\in T$.

If $e\in E$ is a boundary edge, then the vertices of ${\bf q}$ on $e$ are frozen. As illustrated by Figure \ref{ast4clumut4}, the cluster tranformation ${\cal D}$ will create new arrows among the frozen vertices.

If $e\in E$ is an internal edge, then it belongs to two different triangles $t_1, t_2\in T$. The extra arrows on $e$ created by ${\cal C}_{t_1}$ and ${\cal C}_{t_2}$ will cancel. 

Summarizing, the cluster transformation ${\cal D}$ only add new arrows among frozen vertices.
\begin{figure}[h]
\epsfxsize350pt
\centerline{\epsfbox{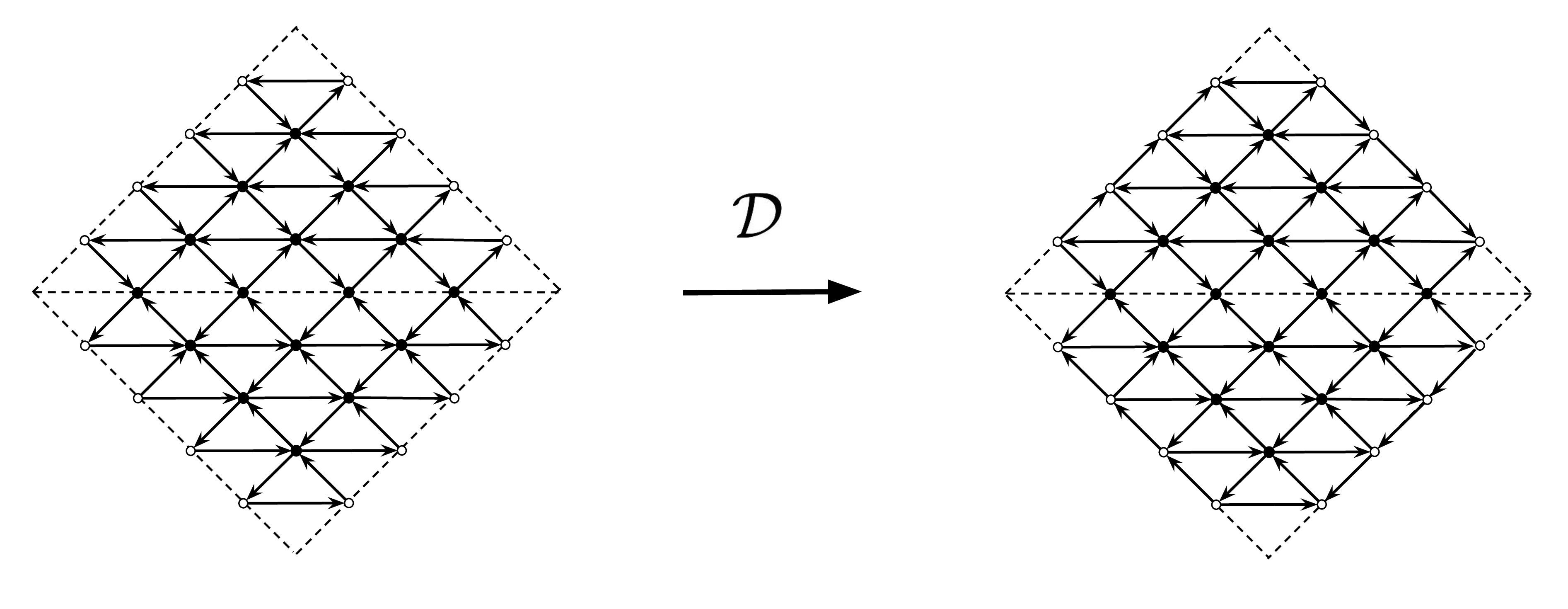}}
\caption{The involution ${\cal D}$ on ${\rm Conf}_4({\cal A}_{{\rm SL}_m})$.}
\label{ast4clumut4}
\end{figure}

\bt 
\la{ast.clu.tran.conf.n.93}
The action $\ast$ on ${\cal A}_{{\rm SL}_m,\bS}$ is exactly the cluster transformation ${\cal D}$.
\et
\begin{proof}
Note that the action $\ast$ is local, i.e., the following diagram commutes
\begin{equation}\label{mcdia}
\begin{gathered}
\xymatrix{
{\cal A}_{{\rm SL}_m,\bS} \ar[r]^{\ast} \ar[d]& {\cal A}_{{\rm SL}_m,\bS}\ar[d]  &\\
{\displaystyle \prod_{t\in T}}{\rm Conf}_3({\cal A}_{{\rm SL}_m}) \ar[r]^{\ast} &  {\displaystyle \prod_{t\in T}}{\rm Conf}_3({\cal A}_{{\rm SL}_m}) &}
\end{gathered}
\end{equation}
Meanwhile, the cluster coordinates \eqref{cluster.cor.Aspace.h} of ${\cal A}_{{\rm SL}_m,\bS}$ are also local. 
The Theorem is a direct consequence of Proposition \ref{basic.ast.prop.510}.
\end{proof}

\paragraph{The ${\cal X}$--part.}  Let ${\bf q}$ be a quiver with vertices parametrized by ${\rm I}$. Set $\varepsilon_{\bf q}=(\varepsilon_{ij})$.  Deleting frozen vertices of ${\bf q}$, we get a quiver $\overline{\bf q}$ with vertices parametrized by  ${\rm J}\subset {\rm I}$. Consider the map
\be
\la{natural.cluster.map.atox}
p: ~{\cal A}_{\bf q}\lra {\cal X}_{\overline{\bf q}}, \hskip 7mm ~~p^\ast X_j =\prod_{i\in {\rm I}}A_i^{\varepsilon_{ij}},~~\forall j\in {\rm J}.
\ee
It is known that  $p$ commutes with cluster permutations and cluster mutations
\begin{equation}\label{mcdia2h}
\begin{gathered}
\xymatrix{
{\cal A}_{\bf q} \ar[r]^{\mu_k} \ar[d]_p &{\cal A}_{\mu_k({\bf q})} \ar[d]^p  &\\
{\cal X}_{\bf \overline{q}} \ar[r]^{\mu_k} & {\cal X}_{\mu_k({\bf \overline{\bf q})}}  &}
\end{gathered}
\end{equation}
Therefore $p$ commutes with all cluster transformations.
\bt 
\la{ast.clu.tran.conf.n.93xpart}
The action $\ast$ on ${\cal X}_{{\rm PGL}_m,\bS}$ is exactly the cluster transformation ${\cal D}$.
\et
\begin{proof}
Since the action $\ast$ on ${\cal X}_{{\rm PGL}_m, \bS}$ is local, it suffices to prove the case when  ${\cal X}_{{\rm PGL}_m, \bS}={\rm Conf}_4({\cal B}_{{\rm PGL}_m})$. 
 Recall the projection from ${\cal A}_{{\rm SL}_m}$ to ${\cal B}_{{\rm PGL}_m}$ by forgetting decorations. It induces a projection 
\be
\la{projection.11.12.11.10.15}
p:~{\rm Conf}_4({\cal A}_{{\rm SL}_m})\lra {\rm Conf}_4({\cal B}_{{\rm PGL}_m}).
\ee
By \eqref{triple.ratio.delta}\eqref{crosss.ratio.delta}, the projection \eqref{projection.11.12.11.10.15} coincides with the map \eqref{natural.cluster.map.atox}.
Thus $p$ commutes with  ${\cal D}$:
$$
p^\ast ({\cal D}^\ast X_v)={\cal D}^\ast (p^\ast X_v).
$$
By the definition of $\ast$-involution, the following diagram commutes
\begin{equation}\label{mcdia3h}
\begin{gathered}
\xymatrix{
{\rm Conf}_4({\cal A}_{{\rm SL}_m})  \ar[r]^{\ast} \ar[d]_p & {\rm Conf}_4({\cal A}_{{\rm SL}_m})   \ar[d]^p  &\\
{\rm Conf}_4({\cal B}_{{\rm PGL}_m})  \ar[r]^{\ast} & {\rm Conf}_4({\cal B}_{{\rm PGL}_m})   &}
\end{gathered}
\end{equation}
Thus
$
p^\ast(\ast X_v)=\ast (p^\ast X_v).
$ 
By Theorem  \ref{ast.clu.tran.conf.n.93}, 
$
\ast (p^\ast X_v)={\cal D}^\ast (p^\ast X_v).
$ 
Hence
$p^\ast(\ast X_v)=p^\ast ({\cal D}^\ast X_v).$
Since $p$ is onto in this case,  $p^\ast$ is an injection. We get
$
\ast X_v={\cal D}^\ast X_v.
$
\end{proof}

\subsection{The Sch\"utzenberger involution}
\paragraph{The involution ${\bf S}$.}
Let $V$ be an $m$-dimensional vector space with a volume form $\omega$. Set
$
\omega':=(-1)^{m(m+1)/2}\omega.
$ 
We consider the  isomorphism
$$
t: ~{\cal A}_{V, \omega}\stackrel{\sim}{\lra}{\cal A}_{V, \omega'}, \hskip 7mm (F_{\bullet}, \{f_{(k)}\})\lms (F_{\bullet}, \{(-1)^{k(k+1)/2}f_{(k)}\}).
$$
Let $({\rm F}, {\rm G},{\rm H})$ be a triple of decorated flags. It is clear that\footnote{Note that we change the order of the decorated flags here.}
$$
\Delta_{a,b,c}({\rm F}, {\rm G},{\rm H})=\Delta_{c,b,a}(t({\rm H}), t({\rm G}),t({\rm F})), \hskip 7mm \forall (a,b,c)\in \Gamma_m.
$$
Let us compose the involution $\ast$ with the map $t$:
\be
{\bf S}: ~{\rm Conf}_3({\cal A}_{{\rm SL}_m})\stackrel{\sim}{\lra}{\rm Conf}_3({\cal A}_{{\rm SL}_m}),\hskip 7mm \big({\rm F}, {\rm G},{\rm H}\big)\lms \big(t(\ast{\rm H}), t(\ast{\rm G}),t(\ast{\rm F})\big)
\ee
\bl 
\la{15.9.10.9.24hh}
One has
${\bf S}^*\Delta_{a,b,c}=\Delta_{a,0,c}^b$ for all $(a,b,c)\in \Gamma_m$. 
\el
\begin{proof} It follows directly from Lemma \ref{6.25.15.16h}.
\end{proof}

\paragraph{The space ${\rm Conf}({\cal A}_{{\rm SL}_m}, {\cal B}_{{\rm SL}_m}, {\cal A}_{{\rm SL}_m})$.}
Recall the configuration space
\be
\la{confaabslm}
{\rm Conf}({\cal A}_{{\rm SL}_m}, {\cal B}_{{\rm SL}_m}, {\cal A}_{{\rm SL}_m}):={\rm SL}_m \backslash \big({\cal A}_{{\rm SL}_m}\times {\cal B}_{{\rm SL}_m}\times {\cal A}_{{\rm SL}_m}\big).
\ee
Let $(a,b,c)$ be a triple of nonnegative integers such that $a+b+c=m-1$. The functions 
\be
{\rm R}_{a,b,c}:=\frac{\Delta_{a,b,c+1}}{
\Delta_{a+1,b,c}}
\ee
form a coordinate system on \eqref{confaabslm}, referred to as the {\it special coordinate system}.

\bt[{\cite[Theorem 3.2]{GS}}] The special coordinate system on ${\rm Conf}({\cal A}_{{\rm SL}_m}, {\cal B}_{{\rm SL}_m}, {\cal A}_{{\rm SL}_m})$ together with the potential ${\cal W}=\chi_{\A_1}+\chi_{\A_3}$ provide a canonical isomorphism 
$$
\{\mbox{Gelfand-Tsetlin's patterns for ${\rm PGL}_m$}\}~~=~~{\rm Conf}^+({\cal A}_{{\rm SL}_m}, {\cal B}_{{\rm SL}_m}, {\cal A}_{{\rm SL}_m})(\Z^t).
$$
\et
\paragraph{The Sch\"utzenberger involution.}
Using the special coordinate system, we study the involution
\be
\la{sch.9.10.9.40hh}
\begin{split}
{\bf S}:~{\rm Conf}({\cal A}_{{\rm SL}_m}, {\cal B}_{{\rm SL}_m}, {\cal A}_{{\rm SL}_m})&\stackrel{\sim}{\lra}{\rm Conf}({\cal A}_{{\rm SL}_m}, {\cal B}_{{\rm SL}_m}, {\cal A}_{{\rm SL}_m}),\\
\big(\A_1, \B_2, \A_3\big)&\lms\big(t(\ast \A_3), \ast\B_2, t(\ast \A_1)\big)
\end{split}
\ee
Let $(a,b,c,s)$ be a quadruple of nonnegative integers such that 
$
a+b+c+s=m-1.
$ 
Let us set
\be
{\rm R}_{a,b,c}^s:=\frac{\Delta_{a,b,c+1}^s}{
\Delta_{a+1,b,c}^s}.
\ee
By definition, when $s=0$, we have
$
{\rm R}_{a,b,c}^0={\rm R}_{a,b,c}.
$ 
\bl One has 
${\bf S}^\ast {\rm R}_{a,b,c}={\rm R}_{a,0,c}^b$ for all $(a,b, c)\in \Gamma_{m-1}$.
\el
\begin{proof}
It follows directly from Lemma \ref{15.9.10.9.24hh}.
\end{proof}
\bl
Let us assume that $a,b,c>0$. One has
\be
{\rm R}_{a,b-1,c}^{s+1}{\rm R}_{a,b,c}^s=\frac{{\rm R}_{a-1,b,c}^{s+1}+{\rm R}_{a,b-1,c+1}^s}{({\rm R}_{a,b,c-1}^{s+1})^{-1}+({\rm R}_{a+1,b-1,c}^{s})^{-1}}.
\ee
\el
\begin{figure}[h]
\epsfxsize180pt
\centerline{\epsfbox{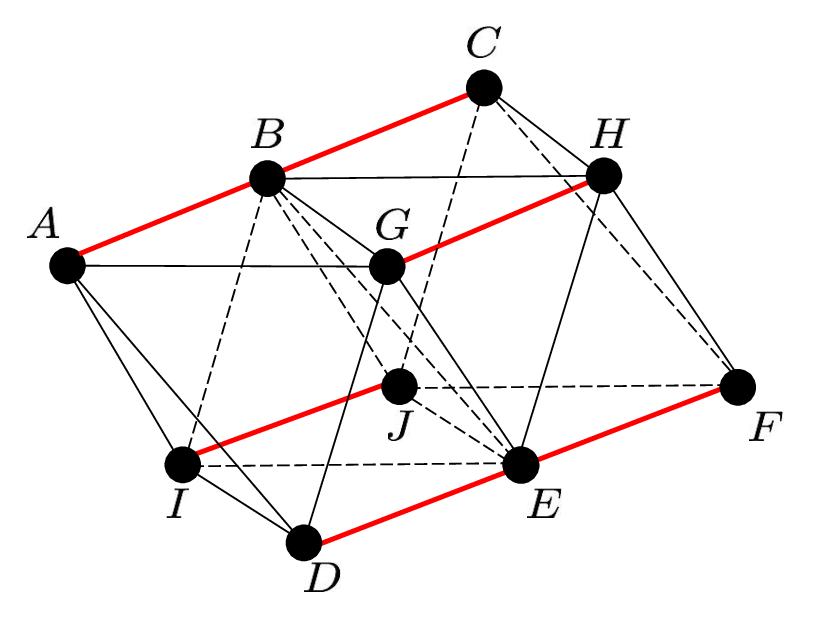}}
\caption{}
\label{CopyoftetraAST}
\end{figure}
\begin{proof}
Using Figure \ref{CopyoftetraAST},
let us assign variables $(A,B,\ldots, J)$ to the 10 vertices satisfying the octahedral relations
$$
JH=BF+EC, \hskip 7mm IG=BD+EA.
$$ 
Let us assign ratios to the 6 red edges
$$
R_1=\frac{J}{I}, ~~R_2=\frac{H}{G}, ~~R_3=\frac{B}{A}, ~~R_4=\frac{C}{B}, ~~R_5=\frac{E}{D},~~ R_6=\frac{F}{E}.
$$
Then
\be
\la{schutzenberger.eq}
R_1R_2=\frac{JH}{IG}=\frac{BF+EC}{BD+EA}=\frac{F/E+C/B}{D/E+A/B}=\frac{R_6+R_4}{R_5^{-1}+R_3^{-1}}
\ee
Recall the octahedral relations \eqref{oct.formula.6.17.11.44h} illustrated by Figure \ref{tetra.astm1}. The Lemma follows by working with the subscripts of ${\rm R}_{a,b,c}^s$ carefully.
\end{proof}
We consider the special cases when $a=0$ or $c=0$.
\bl We have
\begin{align}
\la{schutzenberger.eq2.1} {\rm R}_{0,b-1,c}^{s+1}{\rm R}_{0,b,c}^s&=\frac{{\rm R}_{0,b-1,c+1}^s}{({\rm R}_{0,b,c-1}^{s+1})^{-1}+({\rm R}_{1,b-1,c}^{s})^{-1}}, \hskip 7mm &s=m-1-b-c,~c>0, \\
\la{schutzenberger.eq2.2}{\rm R}_{a,b-1,0}^{s+1}{\rm R}_{a,b,0}^s&=\frac{{\rm R}_{a-1,b,0}^{s+1}+{\rm R}_{a,b-1,1}^s}{({\rm R}_{a+1,b-1,0}^s)^{-1}},    &s=m-1-a-b, ~a>0,\\
\la{schutzenberger.eq2.3}{\rm R}_{0,b-1,0}^{s+1}{\rm R}_{0,b,0}^s&=\frac{{\rm R}_{0,b-1,1}^s}{({\rm R}_{1,b-1,0}^s)^{-1}},    &s=m-1-b.
\end{align}
\el
\begin{proof} 
Using Figure \ref{CCtetraAST},
let us assign variables $(B,\ldots, J)$ to the 9 vertices satisfying the relations
$$
JH=BF+EC, \hskip 7mm IG=BD.
$$ 
Let us assign ratios to the 5 red edges
$$
R_1=\frac{J}{I}, ~~R_2=\frac{H}{G}, ~~R_3=\frac{C}{B}, ~~R_4=\frac{E}{D},~~ R_5=\frac{F}{E}.
$$
\begin{figure}[h]
\epsfxsize180pt
\centerline{\epsfbox{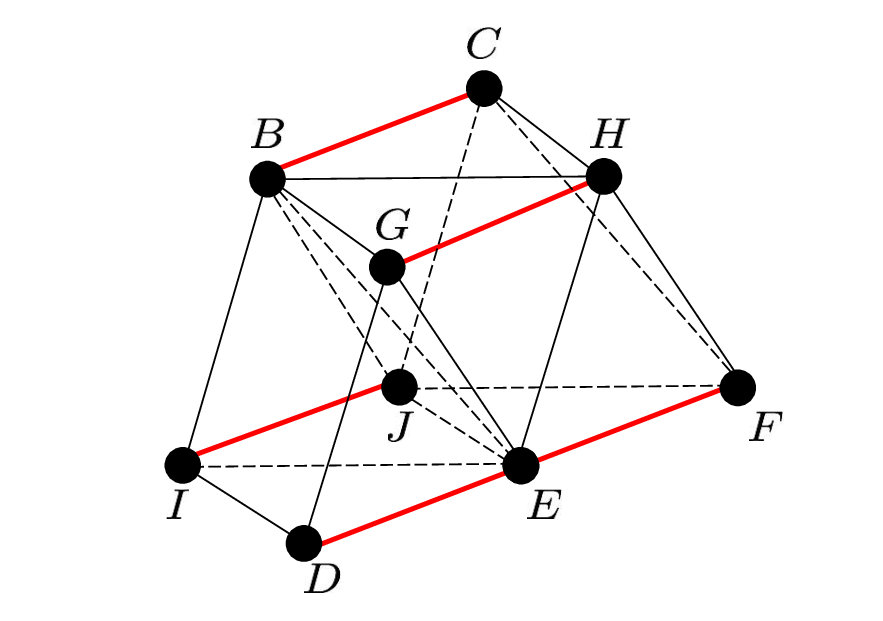}}
\caption{}
\label{CCtetraAST}
\end{figure}
Then
\be
\la{schutzenberger.eq2}
R_1R_2=\frac{JH}{IG}=\frac{BF+EC}{BD}=\frac{F/E+C/B}{D/E}=\frac{R_5+R_3}{R_4^{-1}}
\ee
For functions assigned to vertices on the face $c=0$, by Lemma \ref{lem.boundary.6.20.10.46h}, we have 
\be
\Delta_{a,b,0}^{m-a-b}\Delta_{a,b-1,0}^{m+1-a-b}=\Delta_{a+1, b-1, 0}^{m-a-b}\Delta_{a-1, b, 0}^{m+1-a-b}.
\ee
Combining with the octahedral relations \eqref{oct.formula.6.17.11.44h}, we get the second identity. The proofs for the first and the third identities are similar. 
\end{proof}

\bt
The tropicalization of the involution \eqref{sch.9.10.9.40hh} is the Sch\"utzenberger involution of the Gelfand-Tsetlin's patterns.
\et

\begin{proof} 
Recall the Sch\"utzenberger involution $\eta$ defined by Berenstein-Zelevinsky \cite[(8.5)]{BZ}.
Tropicalizing the formula \eqref{schutzenberger.eq}, we get
\begin{align}
\la{9.10.9.57.15hhh}
R_1^t&=\min\{R_6^t, ~R_4^t\}-\min\{-R_5^t, ~-R_3^t\}-R_2^t 
=\min\{R_6^t, ~R_4^t\}+\max\{R_5^t, ~R_3^t\}-R_2^t.
\end{align}
 Note that Formula \eqref{9.10.9.57.15hhh} is exactly Formula (8.4) in {\it loc.cit.}. 
The tropicalizations of \eqref{schutzenberger.eq2.1}-\eqref{schutzenberger.eq2.3} give (degenerate) formulas of \eqref{9.10.9.57.15hhh}.
Recall the cluster transformation ${\cal C}$ in \eqref{cluster.sch}. 
The Theorem is proved by comparing ${\cal C}$ with the involution $\eta$ of Berenstein-Zelevinsky.
\end{proof}

\section{Donaldson-Thomas transformation on ${\cal X}_{{\rm PGL}_m, \bS}$}\la{secDT}
 Let $\bS$ be an admissible decorated surface.
 Recall the transformation
\[
{\rm C}_\bS:=\ast \circ r_\bS\circ {\bf w}_0.
\]

\bt
\la{clus.C.6.28.4.18h}
The action ${\rm C}_\bS$ is a cluster transformation. 
\et
\begin{proof} 
Note that $r_\bS$ is an element of the mapping class group of $\bS$. By Corollary \ref{cluster.nature.of.flip}, $r_\bS$ is a
 cluster transformation. By Theorem \ref{2.17.22.42h}, the action ${\bf w}_0$ on ${\cal X}_{{\rm PGL}_m, \bS}$ is a cluster transformation.\footnote{When $\bS$ is a sphere with 3 punctures and $\G={\rm PGL}_2$, the corresponding quiver has 3 vertices but no arrows. In this case, the Weyl group action at a single puncture is not cluster. However, the action ${\bf w}_0$ is still cluster, which mutates at each of the three vertices once.}
By Theorem \ref{cluster.ast.6.16.16.07h}, the involution $\ast$ is a cluster transoformation. 
\end{proof}

\bt
\la{main.result.11.14hh}
The action ${\rm C}_\bS$ is the Donaldson-Thomas transformation on ${\cal X}_{{\rm PGL}_m, \bS}$.
\et
We prove Theorem \ref{main.result.11.14hh} in the rest of this section.

By Theorem \ref{clus.C.6.28.4.18h}, it suffices to prove that ${\rm C}_\bS$ maps basic positive laminations to basic negative laminations. 
The latter follows from  Theorem \ref{covering.map.dt},  Theorem \ref{main.thm.1}, and Theorem \ref{ast.trop.17.43t}.

 \subsection{Cluster nature of the mapping class group action}
 \la{cluster.mapping.class.group}
Let ${\cal T}$ and ${\cal T}'$ be two ideal triangulations of $\bS$ without self-folded triangles. 
We assign an $m$-triangulation to each ideal triangle in ${\cal T}$, obtaining a quiver ${\bf q}$. Each vertex $v$ of ${\bf q}$ gives rise to a function $X_v$ of ${\cal X}_{{\rm PGL}_m,\bS}$. The set  ${\bf c}_{\bf q}=\{X_v\}$ is a coordinate chart of ${\cal X}_{{\rm PGL}_m,\bS}$. In the same way, the refined $m$-triangulation ${\bf q}'$ of ${\cal T}'$ gives rise to a chart ${\bf c}_{{\bf q}'}$ for ${\cal X}_{{\rm PGL}_m, \bS}$.

\bt[{\cite[Section 10]{FG1}}] 
\la{cluster.nature.of.flip}
There is a cluster transformation from ${\bf q}$ to ${\bf q}'$ such that the transition map between ${\bf c}_{{\bf q}}$ and ${\bf c}_{{\bf q}'}$ coincides with the one provided by the cluster transformation.
\et

\begin{remark} Let $e$ be a diagonal of an ideal quadrilateral in ${\cal T}$. A {\it flip at  $e$} removes  $e$ and adds the other diagonal of the ideal quadrilateral to ${\cal T}$. 
Note that any two ideal triangulations without self-folded triangles can be connected by a sequence of flips that only involves ideal triangulations without self-folded 
triangles. 
Therefore it suffices to show that every flip in the sequence is a cluster transformation. 
Since a flip only involves a local quadrilateral, it is enough to prove Theorem \ref{cluster.nature.of.flip} for the case when $\bS$ is a quadrilateral.
\end{remark}
\begin{figure}[h]
\epsfxsize250pt
\centerline{\epsfbox{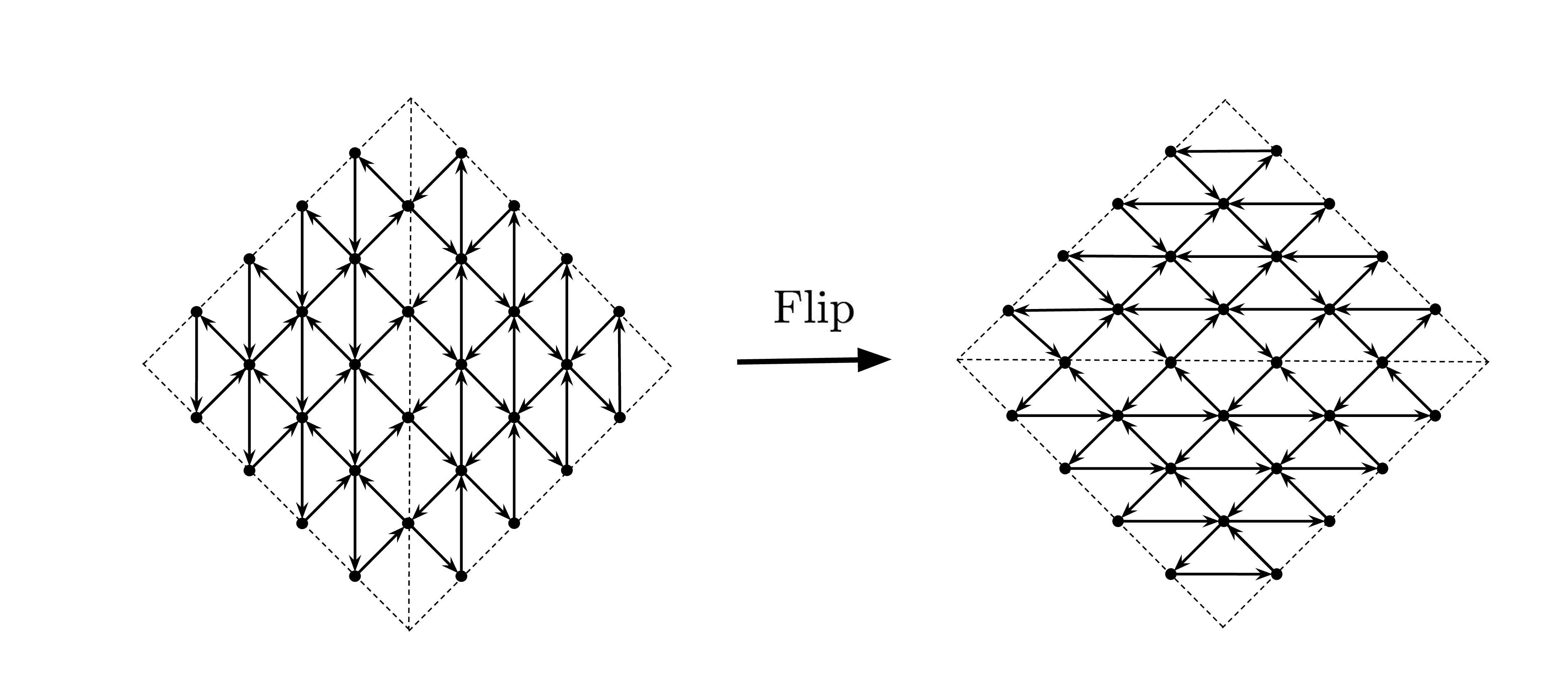}}
\caption{Flip.}
\label{flip2}
\end{figure}

\begin{proof} 
For future use,  we present below a sequence of quiver mutations that takes left quiver on Figure \ref{flip2} to the right  by induction on $m$. 
We refer the reader to \cite[Section 10]{FG1} for showing that it gives the transition map between ${\bf c}_{\bf q}$ and ${\bf c}_{{\bf q}'}$.
 
Consider the integral points inside of the tetrahedron
\[
{\bf T}_m:=\{(x_1,x_2, x_3, x_4)\in \R^4 ~|~ \sum_{i=1}^4 x_i=m, ~~x_i\geq 0\}.
\]
We identify the vertices of the left quiver on Figure \ref{flip2} with the integral points on the faces of ${\bf T}_m$ when $x_1=0$ or $x_2=0$,  
and the vertices of the right  with the integral points on the the faces  when $x_3=0$ or $x_4=0$. 
First we focus on the top tetrahedron of size $m-1$.  By induction, after a sequence of mutations, we obtain the second graph of Figure \ref{tetraflip}. Then we mutate at the vertices on the last layer, obtaining the third graph.
\begin{figure}[h]
\epsfxsize300pt
\centerline{\epsfbox{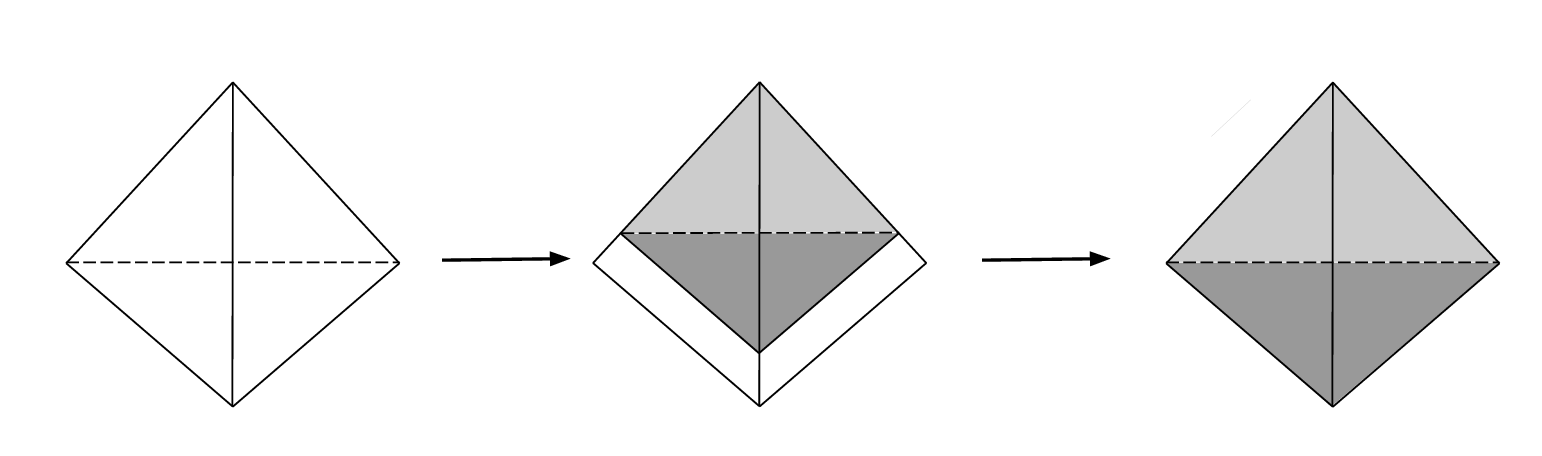}}
\caption{}
\label{tetraflip}
\end{figure}
Using  the language of quivers, we first mutate the sub quiver consisting of vertices on the top square of size $m-1$. See the first quiver on Figure \ref{Flip27}. By induction, we obtain the second quiver on Figure \ref{Flip27}. Then we mutate at the vertices contained in the bottom triangle, in the order of row by row, from bottom left to top right, obtaining the final quiver. \footnote{In fact, every quiver mutation in the sequence gives a two by two move on the bipartite graph of Figure \ref{gra10i}.}

\begin{figure}[h]
\epsfxsize350pt
\centerline{\epsfbox{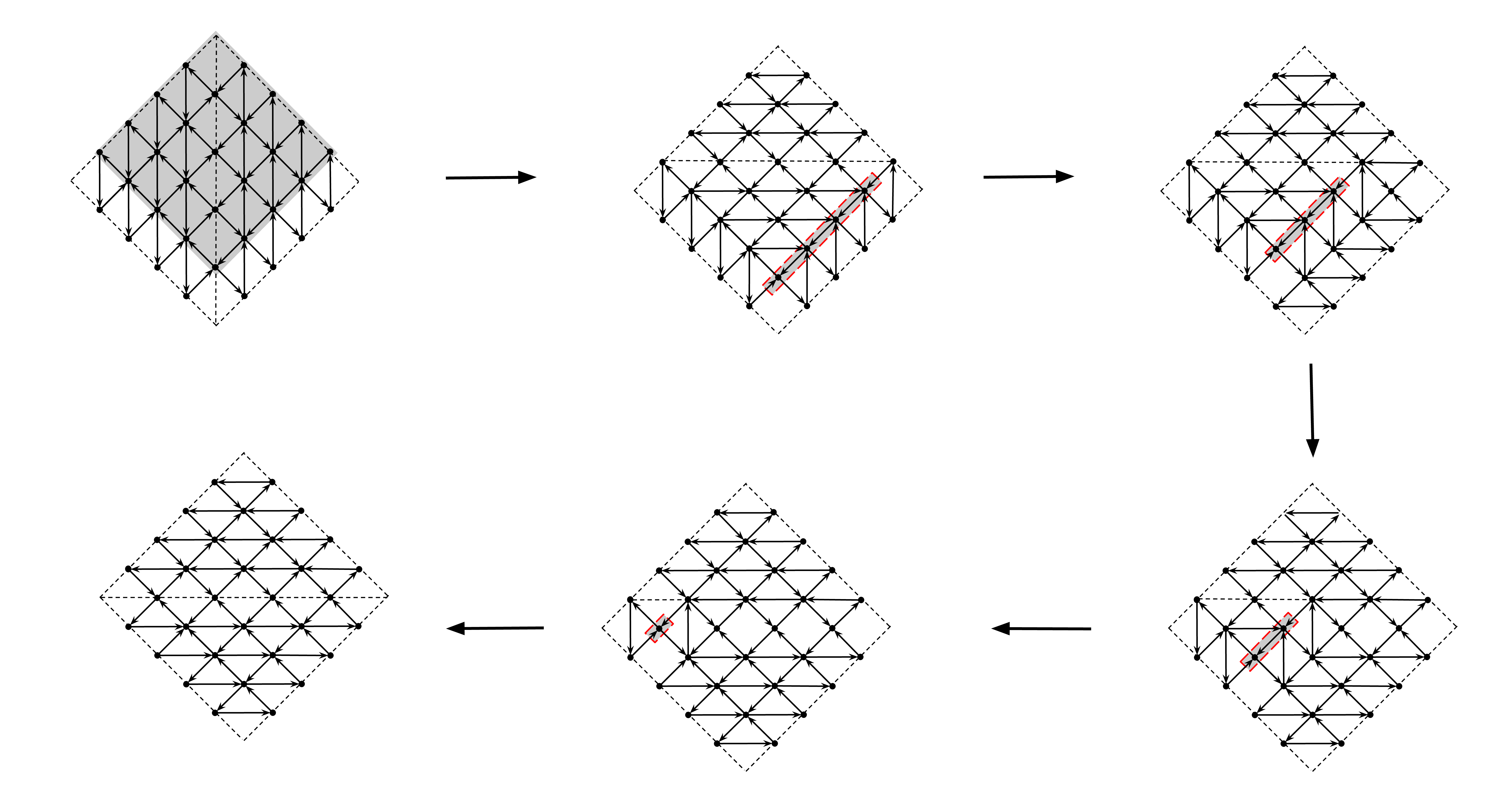}}
\caption{}
\label{Flip27}
\end{figure}
\end{proof}

\bc 
\la{cluster.nature.of.flip22}
The mapping class group $\Gamma_{\bS}$ of $\bS$ acts on ${\cal X}_{{\rm PGL}_m, \bS}$ by cluster transformations.
\ec

\begin{proof} Let ${\cal T}$ be an ideal triangulation of $\bS$ without self-folded triangles. Each element $\gamma\in \Gamma_{\bS}$ maps ${\cal T}$ to another ideal triangulation $\gamma({\cal T})$ without self-folded triangles. By Theorem \ref{cluster.nature.of.flip}, ${\cal T}$ and $\gamma({\cal T})$ are connected by cluster transformations.
\end{proof}

From now on let us assume that ${\cal T}'$ is obtained from ${\cal T}$  by a flip at $e$.
Recall the basic laminations  $l_v^+$  in the coordinate chart ${\bf c}_{\bf q}$ (Definition \ref{def.basic.laminations.t}).
We study their coordinates in ${\bf c}_{{\bf q}'}$. 
\paragraph{Notation.} The coordinates of ${\cal X}$-laminations will be  illustrated as in Figure \ref{flip1}: the $\circ$-vertices  with $``+"$ give 1, the $\circ$-vertices with $``-"$ give -1, and the rest give 0.
\bl
\la{trop.flip.bas.t}
\begin{enumerate}
\item If  $v$ is an inner vertex of an ideal triangle containing $e$, then the coordinates of $l_v^{+}$ in ${\bf c}_{{\bf q}'}$ are illustrated  by the second graph of Figure \ref{flip1}.
\item If  $v$ is on the edge $e$, then the coordinates of $l_v^+$ in ${\bf c}_{{\bf q}'}$ are illustrated by the fourth graph of Figure \ref{flip1}.
\item For the rest vertices $v$, the coordinates of $l_v^+$ remain intact.
\end{enumerate}
\el
 \begin{figure}[h]
\epsfxsize500pt
\centerline{\epsfbox{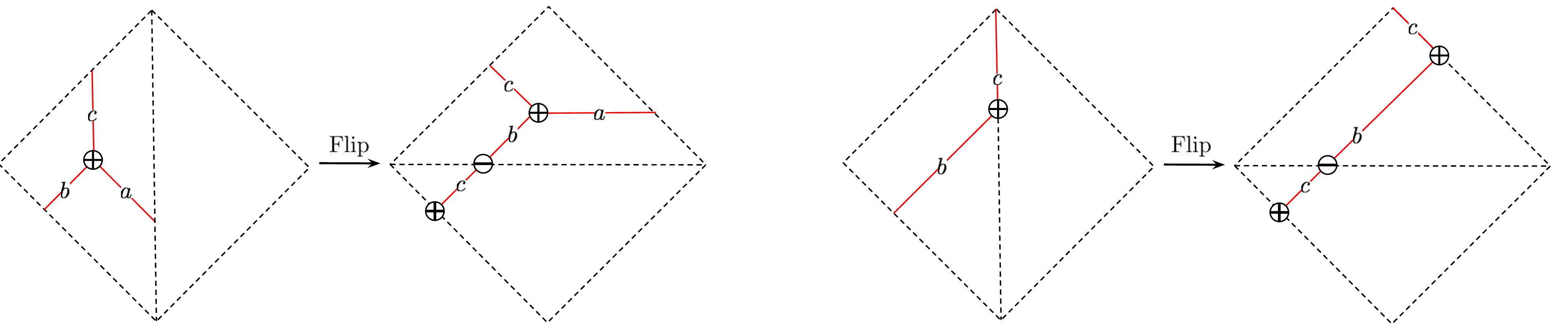}}
\caption{Basic laminations under a flip.}
\label{flip1}
\end{figure}

\begin{proof} Part 3 is clear. Part  2 is a special case of 1 when  $a=0$. 

The proof uses  induction on $m$ in the same way as the proof of Theorem \ref{cluster.nature.of.flip}. 

If $m=2$, then $(a,b,c)=(0,1,1)$. The Lemma follows due to direct calculation.

If $m>2$, we prove the case when $b>1$ (the proof for $b=1$ is similar but easier). 
First we apply cluster mutations to the top square of size $m-1$. Using induction, the coordinates of the basic lamination are shown on the second graph of Figure \ref{tetraflip2}. Then we mutate at the vertices contained  in the bottom triangle, in such an  order illustrated by Figure \ref{Flip27}. By an easy calculation, we get the last graph of Figure \ref{tetraflip2}.
\begin{figure}[H]
\epsfxsize450pt
\centerline{\epsfbox{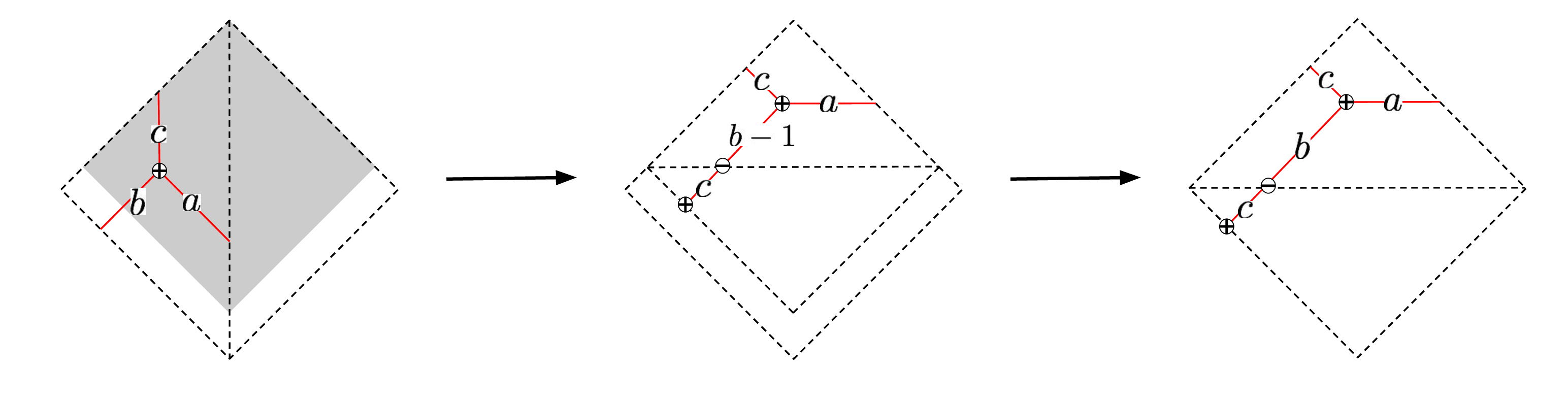}}
\caption{}
\label{tetraflip2}
\end{figure}
\end{proof}

\subsection{Covering map of decorated surfaces}

\paragraph{The cone of positive laminations.} Let ${\cal X}$ be a cluster Poisson variety assigned to a quiver ${\bf q}$ with $N$ vertices indexed by $I=\{1, \ldots, N\}$. The chart ${\bf c}_{\bf q}:=\{X_i\}$ provides a bijection
\[
{\bf c}_{\bf q}^t:  {\cal X}(\Z^t)\stackrel{\sim}{\lra} \Z^N, \hskip 7mm l \lms (X_1^t(l), \ldots, X_N^t(l)).
\] 
Let ${\Bbb N}$ be the set of non-negative integers. We consider the cone of positive laminations in ${\bf c}_{\bf q}$
\[
{\cal X}_{\bf q}^+(\Z^t): =  \big({\bf c}_{\bf q}^t\big)^{-1}({\Bbb N}^N).
\]

\bl 
\la{lem.dt.pos.cone}
If ${\bf K}$ is the cluster DT-transformation on ${\cal X}$, then it maps positive laminations to negative laminations:
\be
{\bf c}_{\bf q}^t ({\bf K}^t(l)) =- {\bf c}_{\bf q}^t(l), \hskip 7mm 
\forall l \in {\cal X}_{\bf q}^+(\Z^t).
\ee
\el
\begin{proof} By the commutative version of Formula \eqref{sepf2}, we have \footnote{See \cite[Prop 3.13]{FZIV} for the commutative version. See \cite[Theorem 1.7]{DWZ2} for the proof that $F$-polynomials have constant 1.  As an example, if ${\bf q}$ is a cycle,  then ${\bf K}^\ast X_i $ is given by \eqref{basic.example.dt.cyc}.}
\be
{\bf K}^\ast X_i = X_i^{-1} \prod_{j\in I} F_{j}^{\varepsilon_{ij}}, \hskip 7mm \mbox{where the $F$-polynomials $F_j$ are of constant term 1. }
\ee
If $l\in {\cal X}_{\bf q}^+(\Z^t)$, then $F^t_j(l)=0$. We have
\[ X_i^t({\bf K}^t(l))=\big( {\bf K}^\ast X_i \big)^t(l) = -X_i^t(l)+\sum_j {\varepsilon_{ij}} F_j^t(l)= -X_i^t(l).\]
\end{proof}

\paragraph{Covering map.}
Let $\pi: \tilde{\bS}\ra {\bS}$ be a covering map of decorated surfaces. By pulling back, it induces a natural positive embedding 
$j:  ~{\cal X}_{\G, \bS}\rightarrow {\cal X}_{\G, \tilde{\bS}}.$

\bl 
\la{commu.diag.3.10}
The following diagram commutes
\begin{displaymath}
    \xymatrix{
    {\cal X}_{\G, \bS}    \ar[d]_{{\rm C}_{\bS}}   \ar[r]^j &  {\cal X}_{\G, \tilde{\bS}} \ar[d]^{{\rm C}_{\tilde{\bS}}} \\
       {\cal X}_{\G, \bS}       \ar[r]^j   &  {\cal X}_{\G, \tilde{\bS}} }.
\end{displaymath}
\el
\begin{proof}
It follows directly from the geometric meaning of ${\rm C}_{\bS}$.
\end{proof}

\bt
\la{covering.map.dt}
If ${\rm C}_{\tilde{\bS}}$ is the cluster DT-transformation on ${\cal X}_{\G, \tilde{\bS}}$, then so is ${\rm C}_{{\bS}} $ on ${\cal X}_{\G, {\bS}}$.
\et

\begin{proof}
Thanks to Theorem \ref{clus.C.6.28.4.18h}, it remains to prove that ${\rm C}_\bS$ maps basic positive laminations to basic negative laminations. 

Let ${\cal T}$ be an ideal triangulation of $\bS$ without self-folded triangles. Its $m$-refined triangulation gives a quiver ${\bf q}$ with vertices indexed by $I$.
By pulling back to $\tilde{\bS}$, we get an ideal triangulation $\tilde{\cal T}$ of $\tilde{\bS}$, and a quiver $\tilde{\bf q}$ with vertices indexed by $\tilde{I}$. There is a natural projection
$\pi: \tilde{I}\rightarrow I$. 
Using the coordinate charts ${\bf c}_{\bf q}$ and ${\bf c}_{\tilde{\bf q}}$, the embedding $j$ is  given by
\be
\la{j.coord. 11.14h}
j:~ ~{\cal X}_{\G, \bS}\lra {\cal X}_{\G, \tilde{\bS}}, \hskip 7mm j^*(X_{v})=X_{\pi(v)}, ~~\forall v\in \tilde{I}.
\ee
Its tropicalization is an injection
$j^t: ~{\cal X}_{\G, \bS}(\Z^t)\hookrightarrow {\cal X}_{\G, \tilde{\bS}}(\Z^t).
$

Let $l_i^\pm~(i\in I)$ be  basic  laminations  in the coordinate chart ${\bf c}_{\bf q}$. By \eqref{j.coord. 11.14h},  $j^t(l_i^+)$ is a positive lamination in the cone ${\cal X}_{\tilde{\bf q}}^+(\Z^t)$.  
If $ {\rm C}_{\tilde{\bS}}$ is the cluster DT-transformation, by Lemma \ref{lem.dt.pos.cone}, we have
$
{\rm C}_{\tilde{\bS}}^t \big(j^t(l_i^+)\big)=j^t(l_i^-).
$
By the commutative diagram in Lemma \ref{commu.diag.3.10}, we have
$
j^t\big({\rm C}_{\bS}^t(l_i^+)\big) = {\rm C}_{\tilde{\bS}}^t \big(j^t(l_i^+)\big).
$
Therefore $j^t\big({\rm C}_{\bS}^t(l_i^+)\big)= j^t(l_i^-)$. Since $j^t$ is an injection, we get ${\rm C}_{\bS}^t(l_i^+)=l_i^-$.
\end{proof}

\subsection{The action ${\bf w}:=r_{\bS} \circ {\bf w}_0$ on ${\cal X}_{{\rm PGL}_m, \bS}$}
\la{sec.action.rot.weyl}

From now on, let us assume that $\bS$ admits an ideal triangulation ${\cal T}$ such that every edge in ${\cal T}$ connects two different marked points. By Theorem \ref{covering.map.dt}, it is enough to prove Theorem \ref{main.result.11.14hh} for such decorated surfaces.

Let ${\bf c}_{\bf q}:=\{X_v\}$ be the cluster chart of ${\cal X}_{{\rm PGL}_m, \bS}$ given by the $m$-triangulation of ${\cal T}$.
The tropicalization of ${\bf w}:=r_{\bS} \circ {\bf w}_0$ is an isomorphism
\be
{\bf w}^t:~ {\cal X}_{{\rm PGL}_m, \bS}(\Z^t)\stackrel{\sim}{\lra} {\cal X}_{{\rm PGL}_m, \bS}(\Z^t)
\ee
We consider the images of  the basic positive ${\cal X}$-laminations $l_v^+$ under ${\bf w}^t$.
There are two cases.
\bt
\la{main.thm.1}
1. If $v$ is  a vertex on an edge $e$ of ${\cal T}$, labelled by $(a,b)$, then  the coordinates of ${\bf w}^t(l_{v}^{+})$ are illustrated by Figure \ref{edgept}. 
\begin{figure}[H]
\epsfxsize120pt
\centerline{\epsfbox{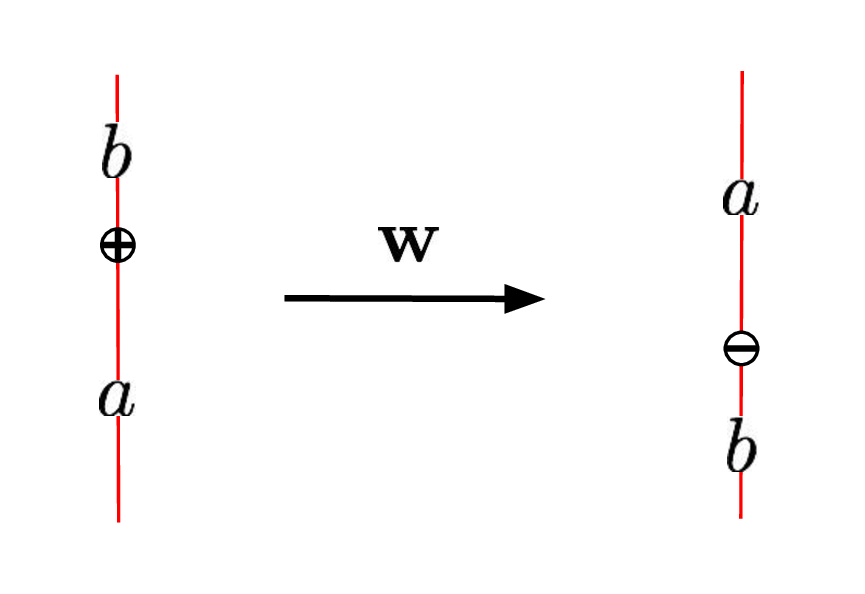}}
\caption{ }
\label{edgept}
\end{figure}

2. If $v$ is  inside an ideal triangle ${\rm t}$ of ${\cal T}$, labelled by $(a,b,c)\in {\Gamma}_m$, then the coordinates of ${\bf w}^t(l_{v}^{+})$ are illustrated by Figure \ref{innerpt}. 
 \begin{figure}[H]
\epsfxsize250pt
\centerline{\epsfbox{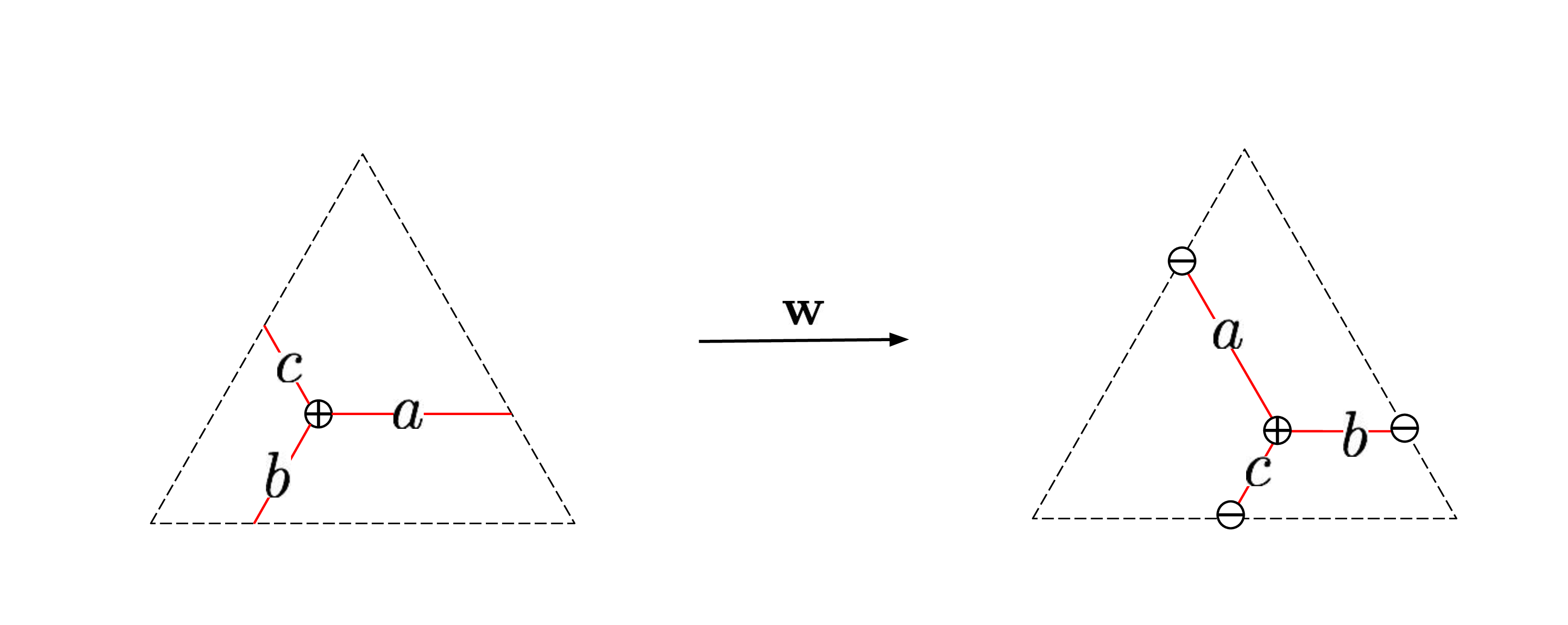}}
\caption{ }
\label{innerpt}
\end{figure}
\et

We prove Theorem \ref{main.thm.1} in the rest of Section \ref{sec.action.rot.weyl}.

 \subsubsection{Tropicalization of Weyl group actions}
Recall the quivers ${\bf q}_N\subset {\bf q}$ in Theorem \ref{2.8.21.24.h}. 
The vertices of ${\bf q}_N$ are labelled by ${\rm I}=\{1,..., N\}$ clockwise.
\bl
\la{6.25.10.16t}
Let $i\in I$ be a vertex of ${\bf q}_N$. Let $l\in {\cal X}_{|{\bf q}|}(\Z^t)$ such that
$$
X_k^t(l)=\left\{ \begin{array}{ll}
      1 & \mbox{if $k=i$}, \\
      0 & \mbox{if $k\in {\rm I}$ and $k\neq i$}, \\
      l_k & \mbox{if $k\notin {\rm I}$}.\\
   \end{array}
   \right.       
$$
We have
$$
X_k^t\big(\tau^t(l)\big)=\left\{ \begin{array}{ll}
      -1 & \mbox{if $k=i+1$}, \\
      0 & \mbox{if $k\in {\rm I}$ and $k\neq i+1$}, \\
      l_k+c_{ki} & \mbox{if $k\notin {\rm I}$}.\\
   \end{array}
   \right.       
$$
\el
\begin{proof} Recall $F_j$, $Y_j$ in Theorem \ref{2.8.21.24.h}. By definition, 
$
F_j^t(l)=0$ for all $j\in {\rm I}$. 
Therefore 
$$
Y_i^t(l)=1; \hskip 7mm Y_j^t(l)=0,\hskip 7mm \forall j\in {\rm I}-\{i\}. 
$$
Note that $X_k^t\big(\tau^t(l)\big) = (\tau^* X_k)^t(l)$. The Lemma follows from Theorem \ref{2.8.21.24.h}.
\end{proof}

\vskip 3mm

Recall the Weyl group action on ${\cal X}_{{\rm PGL}_m, \bS}$ assigned to  a puncture  $p$  of $\bS$.
By Theorem \ref{2.17.22.42h}, the action of the simple reflection $s_{p,i}$ is exactly the cluster transformation $\tau_{p,i}$. Set
\be
\la{weyl.action.d.t}
d_{k,p}:=s_{p, m-k}\circ \ldots \circ s_{p,m-2}\circ s_{p, m-1}.
\ee
Let $v$ be a vertex of distance $(m-i)$ to $p$. 
Using Lemma \ref{6.25.10.16t} repeatedly,  the coordinates of $l_v^+$ under the action $d_{m-1,p}$  are illustrated by Figure \ref{ww2}.

 \begin{figure}[H]
\epsfxsize450pt
\centerline{\epsfbox{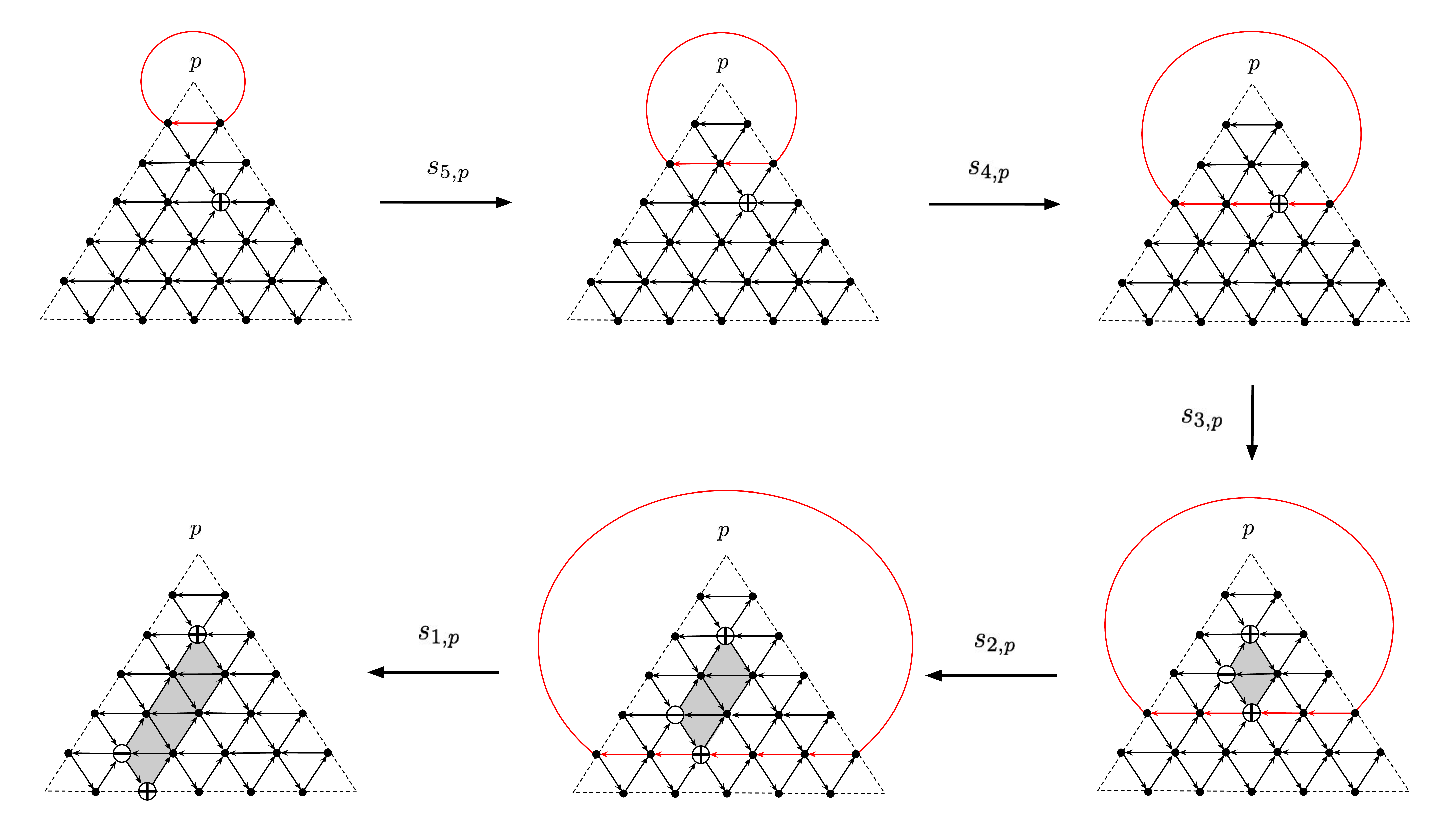}}
\caption{Here $m=6$, $d_{5,p}=s_{p,1}\circ \ldots \circ s_{p,5}$. }
\label{ww2}
\end{figure}

\subsubsection{Part 1 of Theorem \ref{main.thm.1}.}
We have the following three cases.
\paragraph{1. The edge $e$ connects two different special points $m_1$ and $m_2$. }
Let $m_i'$ be the previous special point of $m_i$. By flips at edges other than $e$, we get an ideal triangulation ${\cal T}'$ containing the ideal quadrilateral of vertices $(m_1, m_1', m_2, m_2')$ as the left graph of Figure \ref{mm}. Note that flips at edges other than $e$ keep the coordinates of $l_v^+$ intact. So $l_v^+$ is still a basic positive lamination  in the coordinate chart given by ${\cal T}'$.
 
 We flip at the edge $e$. By Lemma \ref{trop.flip.bas.t},  the coordinates of $l_v^+$ are illustrated by the second  graph\footnote{The edges connecting $m_i$ and $m_i'$ are boundary intervals. Since we consider the moduli space ${\cal X}_{{\rm PGL}_m, \bS}$, there are no frozen vertices assigned to the boundary intervals.} of Figure \ref{mm}.  
The action $r_\bS$ of transporting framings rotates the edge $m_1'm_2'$ back to $e$. Finally, again by flips at edges other than $e$, we return to  the original ideal triangulation ${\cal T}$. 

The actions at other marked points preserve the lamination of the last graph of Figure \ref{mm}. 
So the coordinates of ${\bf w}_0^t(l_v)$ is the same as predicted by the Theorem.
\begin{figure}[H]
\epsfxsize350pt
\centerline{\epsfbox{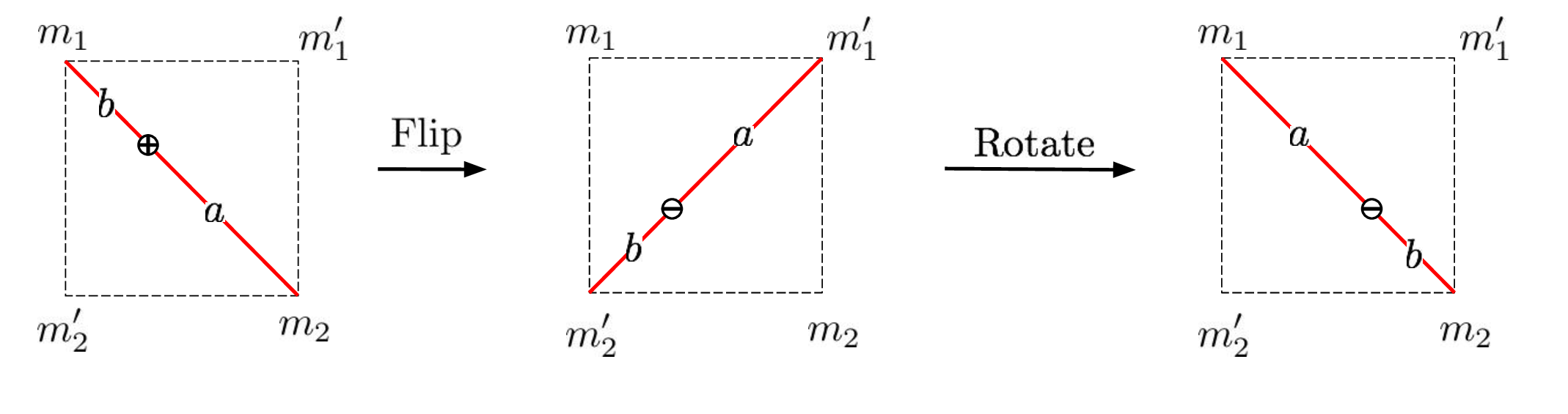}}
\caption{ }
\label{mm}
\end{figure}
\paragraph{2. The edge $e$ connects a puncture $p$ and a special point $m$. }
We consider the ideal triangle of vertices $p, m, m'$. Recall the action $d_{k,p}$ in \eqref{weyl.action.d.t}. The action of the longest Weyl group element on $p$ is
\be
\la{logn.wel.6.28.4}
w_{0,p}:=d_{1,p}\circ d_{2,p}\circ \ldots \circ d_{m-1,p}
\ee
Using repeatedly the process illustrated by Figure \ref{ww2}, $w_{0,p}$  maps $l_v^+$ to the lamination shown on the second graph of Figure \ref{pm}. The action $r_\bS$ rotates the edge $pm'$ back to $e$.
\begin{figure}[H]
\epsfxsize350pt
\centerline{\epsfbox{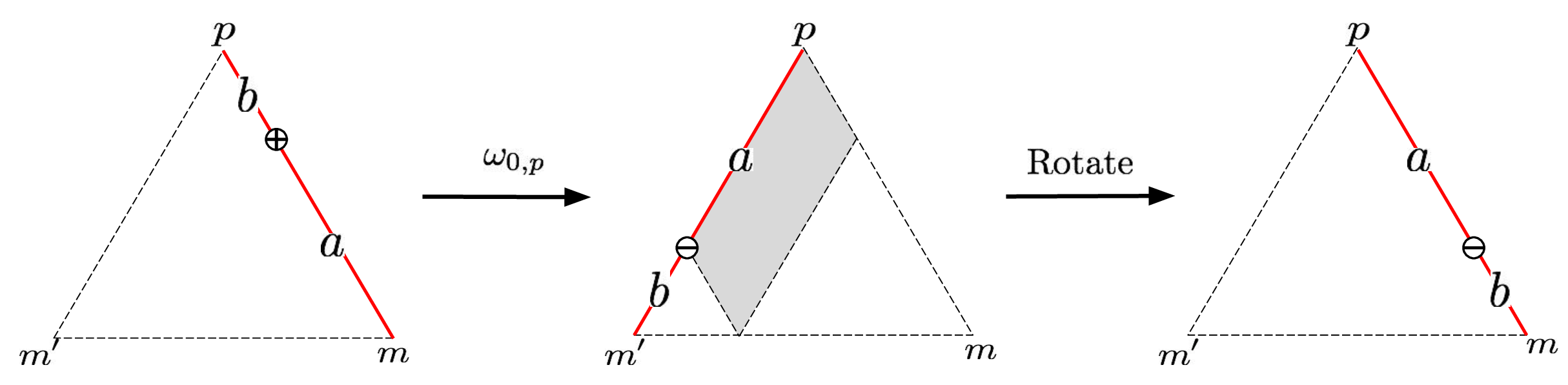}}
\caption{ }
\label{pm}
\end{figure}
\paragraph{3. The edge $e$ connects two different punctures $p_1$ and $p_2$.}
By the same process described in Case 2,  the action of $w_0$ on $p_1$ maps $l^+_v$ to the second graph of Figure \ref{w2}. The action of $w_0$ on $p_2$ maps it to the last graph. The action at other point preserve the lamination of the last graph.
 \begin{figure}[H]
\epsfxsize350pt
\centerline{\epsfbox{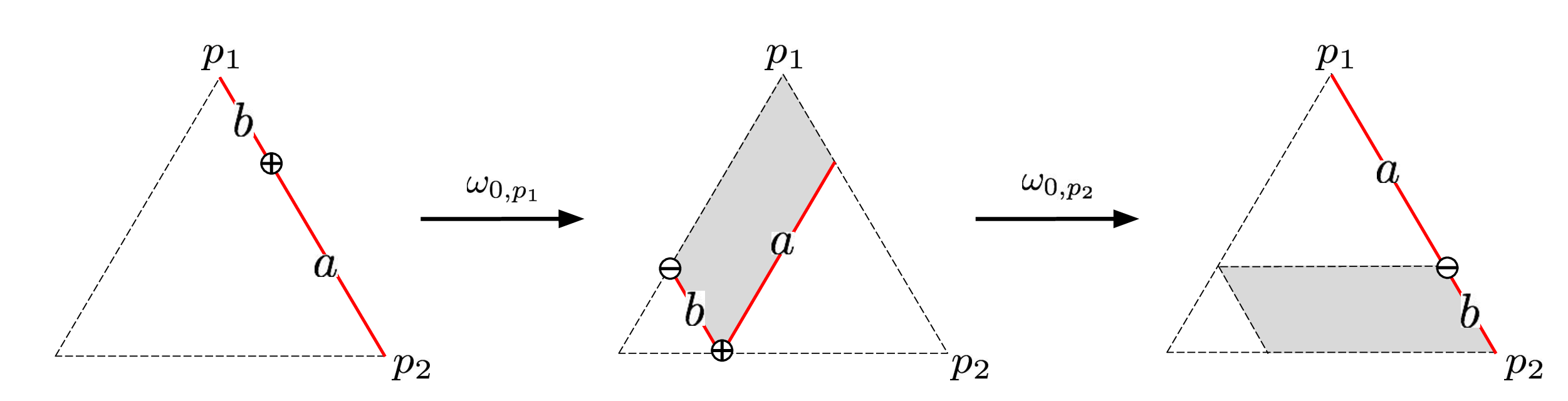}}
\caption{l}
\label{w2}
\end{figure}

\subsubsection{Part 2 of Theorem \ref{main.thm.1}.}
We have the following four cases. 

\paragraph{1. The vertices of ${\rm t}$ consists of three  different punctures $p_1, p_2, p_3$. }
Set
\[
w_{2,p_2}:=d_{m-c,p_2}\circ \ldots \circ d_{m-1,p_2}, \hskip 1cm w_{1, p_2}=d_{1,p_2}\circ \ldots \circ d_{a+b-1,p_2}.
\]
By \eqref{logn.wel.6.28.4}, the action of $w_0$ on $p_2$  is
\[
w_{0,p_2}=w_{1,p_2}\circ w_{2, p_2}.
\]
The actions on the other punctures/special points will not change the coordinates of $l_v$. So it suffices to consider the action $w_0$ on $p_1, p_2, p_3$. Note that Weyl group actions on different punctures always commute. The action
$
w_{1,p_2}\circ w_{0, p_1}\circ w_{0,p_3}\circ w_{2, p_2}
$
is illustrated by Figure \ref{w1}, which coincides with Figure \ref{innerpt}. 
\begin{figure}[H]
\epsfxsize350pt
\centerline{\epsfbox{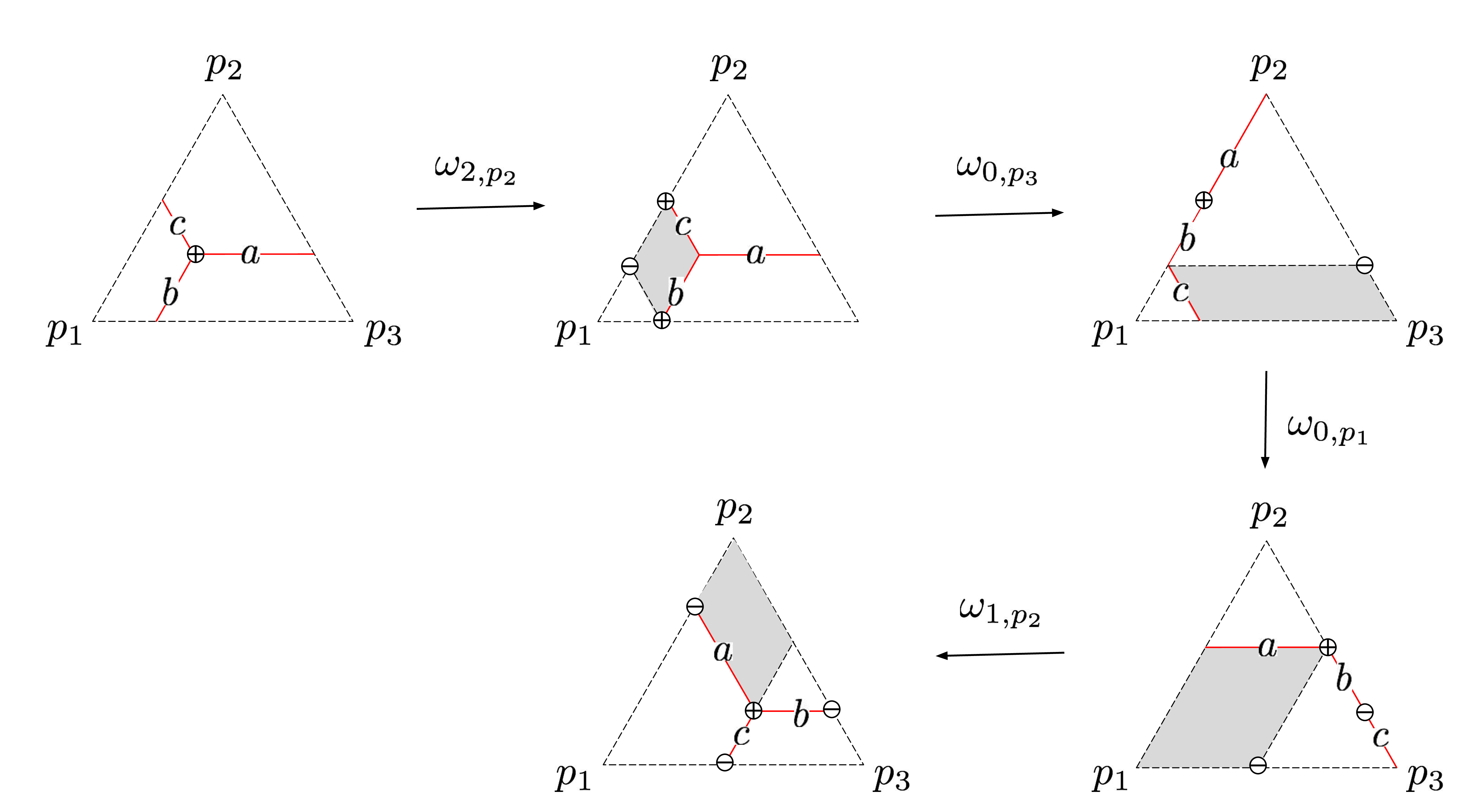}}
\caption{ }
\label{w1}
\end{figure}

\paragraph{2. The vertices of ${\rm t}$ consist of two punctures $p_1,p_2$ and a special point $m$. }
Let $m'$ be the previous special point of $m$. Set
\[
w_{2,p_2}:=d_{m-b,p_2}\circ \ldots \circ d_{m-1,p_2}, \hskip 1cm w_{1, p_2}=d_{1,p_2}\circ \ldots \circ d_{a+c-1,p_2}.
\]
It suffices to consider the ideal quadrilateral of vertices $(p_1,p_2,m,m')$. Figure \ref{ppm} illustrates the change of coordinates of $l_v^+$ after  $w_0$ actions on $p_1$ and $p_2$ and a flip the edge $p_1m$. The action $r_\bS$ rotates the triangle $p_1p_2m'$ back to ${\rm t}$. The actions on the other marked points will preserve the coordinates of $l_v^+$. 
\begin{figure}[H]
\epsfxsize350pt
\centerline{\epsfbox{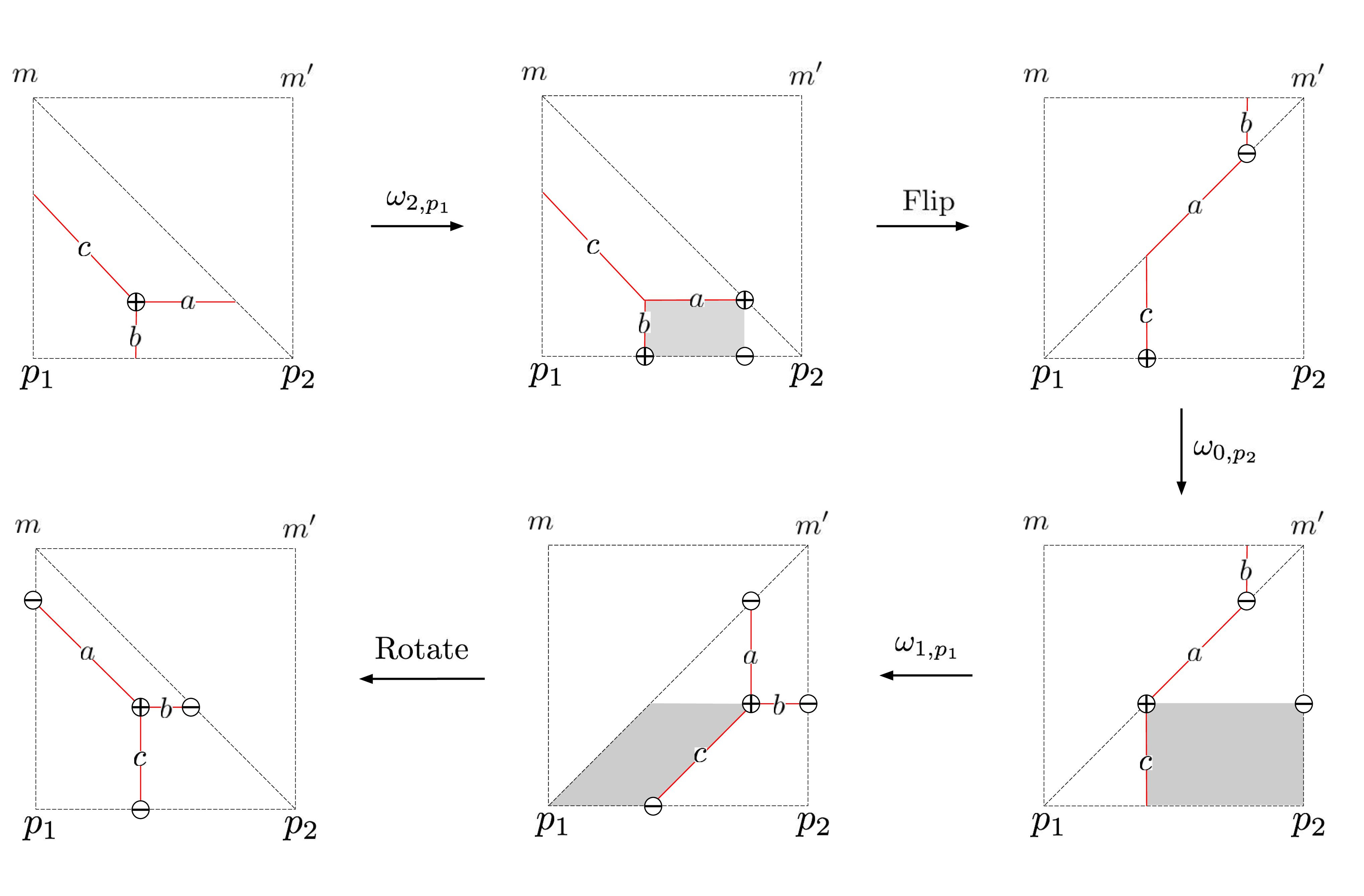}}
\caption{ }
\label{ppm}
\end{figure}

\paragraph{3. The vertices of ${\rm t}$ consist of a puncture $p$ and two special point $m_1$, $m_2$.} Figure \ref{pmm} illustrates the change of coordinates of $l_v^+$ after the $w_0$ action on $p$ and flips at two edge. The action $r_\bS$ rotates the triangle $pm_1'm_2'$ back to ${\rm t}$. The actions on the other marked points will preserve the coordinates of $l_v^+$. 
\begin{figure}[H]
\epsfxsize375pt
\centerline{\epsfbox{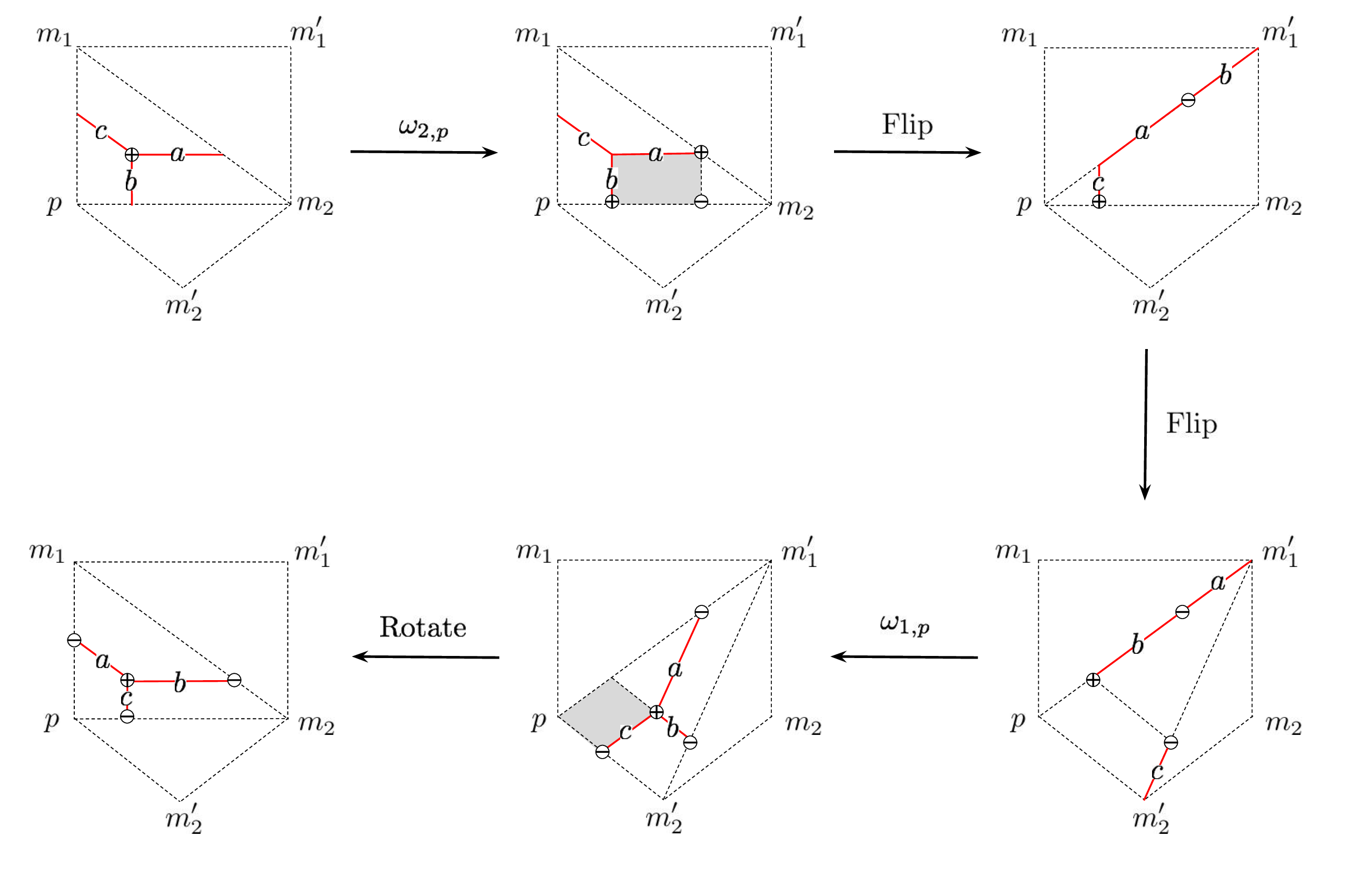}}
\caption{ }
\label{pmm}
\end{figure}

\paragraph{4. The vertices of ${\rm t}$ consist of three special point $m_1$, $m_2$ and $m_3$.} Figure \ref{pmm} illustrates the change of coordinates of $l_v$ after flips at four edges. The action $r_\bS$ rotates the triangle $m_1'm_2'm_3'$ back to ${\rm t}$. The actions on the other marked points will preserve  the coordinates of $l_v$. 
\begin{figure}[H]
\epsfxsize400pt
\centerline{\epsfbox{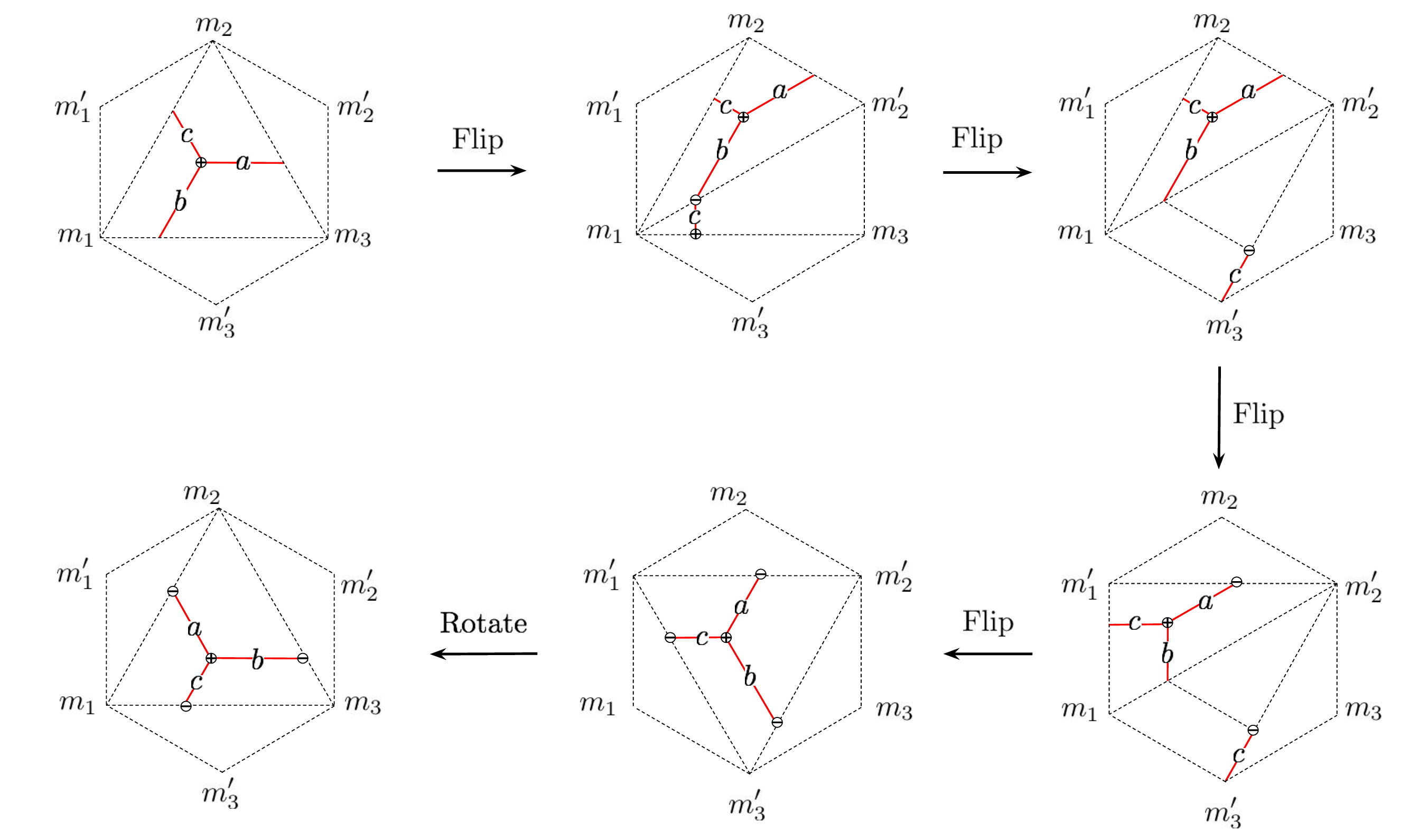}}
\caption{ }
\label{mmm}
\end{figure}

\subsection{The involution $\ast$ on ${\cal X}_{{\rm PGL}_m, \bS }$}
\bt 
\la{ast.trop.17.43t}
The involution $\ast$ maps the laminations in Figures \ref{ast1} to basic negative laminations.
\begin{figure}[h]
\epsfxsize350pt
\centerline{\epsfbox{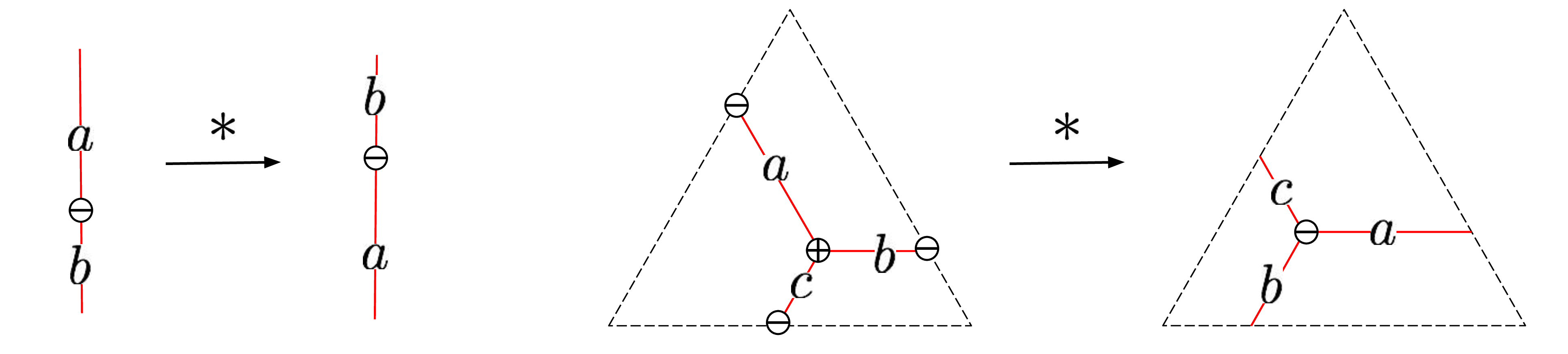}}
\caption{ }
\label{ast1}
\end{figure}
\et

\begin{proof} The first case is clear. We prove the second case by induction on $m$.

If $m=3$, then $a=b=c=1$. The second case is clear.

If $m>3$, without loss of generality, let us assume that $c>1$. By Proposition \ref{basic.ast.prop.510}, the involution $\ast$ locally equals
$${\cal D}:=\sigma\circ {\cal S}_{m-2}\circ {\cal S}_{m-3} \ldots \circ {\cal S}_1.$$ 
Let us assume that Theorem \ref{ast.trop.17.43t} holds for $m-1$. Then ${\cal C}':={\cal S}_{m-3} \ldots \circ {\cal S}_1$ maps the first graph to the second one. Recall the exact sequence of ${\cal S}_{m-2}$ as illustrated by Figure \ref{astclumut5}. It follows directly that ${\cal S}_{m-2}$ maps the second graph to the third one. Finally, we flip the third graph horizontally, getting the last graph that we want.  

\begin{figure}[H]
\epsfxsize500pt
\centerline{\epsfbox{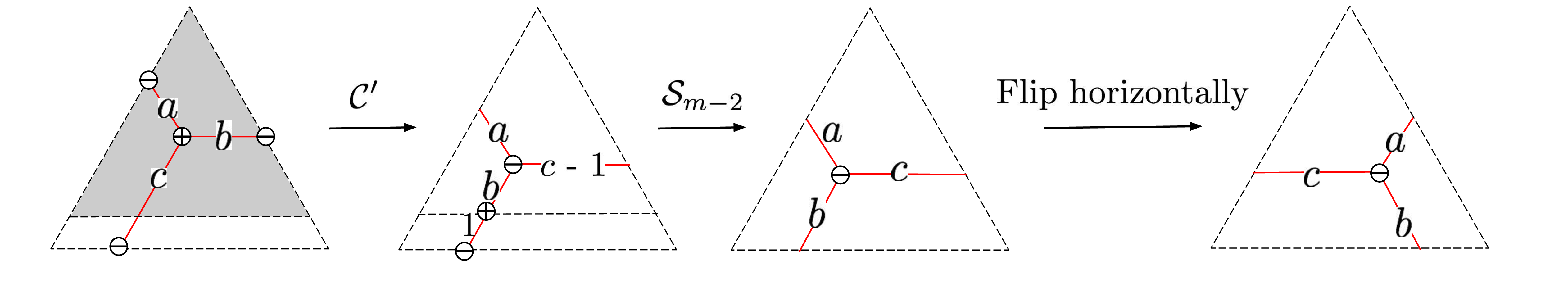}}
\caption{ }
\label{ast2}
\end{figure}
\end{proof}

\end{document}